\documentclass[a4paper]{amsart}
\usepackage[ps2pdf]{hyperref}
\usepackage{amsmath,amssymb}
\usepackage{bbold}
\usepackage{psfrag}
\usepackage{graphicx}
\usepackage{leftidx}
\usepackage[all,poly]{xy}
\usepackage[latin1]{inputenc}
\usepackage[dvipsnames]{xcolor}
\usepackage{tikz-cd}
\usepackage{caption}

\newtheorem{thm}{Theorem}[section]

\newtheorem{lem}[thm]{Lemma}
\newtheorem{claim}[thm]{Claim}

\definecolor{myblue}{cmyk}{1,0.9,0,0}
\definecolor{mylightblue}{cmyk}{0.7,0.2,0,0}
\definecolor{myred}{cmyk}{0,1,1,0}
\definecolor{mygreen}{cmyk}{.9,0.3,1,0.3}

\newcommand{\Euler}{\mathrm{Eul}}
\newcommand{\smallspace}{\hspace{.1em}}
\newcommand{\smallnegspace}{\hspace{-.1em}}
\newcommand{\card}{\mathrm{card}}

\newcommand{\TOP}{\mathrm{TOP}}
\newcommand{\Yet}{\mathrm{Yetter}}

\newcommand{\FTop}{\mathrm{FTop}}
\newcommand{\EB}{\mathrm{EB}}
\newcommand{\Ball}{\mathrm{Ball}}
\newcommand{\colr}{\alpha}
\newcommand{\coleb}{\beta}
\newcommand{\CM}{\chi}

\newcommand{\suc}{\mathrm{succ}}
\newcommand{\pred}{\mathrm{pred}}

\newcommand{\co}{\colon}
\newcommand{\id}{\mathrm{id}}

\newcommand{\opp}{\mathrm{op}}

\newcommand{\cc}{\mathcal{C}}
\newcommand{\tg}{\mathcal{G}}
\newcommand{\cch}{\mathcal{C}_{\mathrm{hom}}}

\newcommand{\dd}{\mathcal{D}}

\newcommand{\ee}{\mathcal{E}}
\newcommand{\rr}{\mathcal{R}}

\newcommand{\mm}{\mathcal{M}}

\newcommand{\sss}{\mathcal{S}}

\newcommand{\lhs}{\mathcal{L}}

\newcommand{\CC}{\mathbb{C}}

\newcommand{\ZZ}{\mathbb{Z}}
\newcommand{\RR}{\mathbb{R}}

\newcommand{\mti}{\,\mbox{-}\,}

\newcommand{\XMod}{\mathrm{CRMod}}
\newcommand{\Mod}{\mathrm{Mod}}
\newcommand{\Ker}{\mathrm{Ker}}
\newcommand{\Coker}{\mathrm{Coker}}

\newcommand{\lmm}[1]{\mathcal{M}_\kk^{#1}}

\newcommand{\col}{\mathrm{Col}}

\newcommand{\kk}{\Bbbk}
\newcommand{\lact}[2]{\leftidx{^{#1}}{{#2}}{}}
\newcommand{\elact}[2]{\leftidx{^{#1}}{{\hspace{-.08em} #2}}{}}
\newcommand{\kt}{$\Bbbk$\nobreakdash-\hspace{0pt}}

\newcommand{\ti}{\mbox{-}\,}

\newcommand{\Z}{\mathbb{Z}}
\newcommand{\Q}{\mathbb{Q}}
\newcommand{\R}{\mathbb{R}}
\newcommand{\un}{\mathbb{1}}

\newcommand{\Ob}{\mathrm{Ob}}
\newcommand{\Int}{\mathrm{Int}}
\newcommand{\Reg}{\mathrm{Reg}}

\newcommand{\Aut}{\mathrm{Aut}}

\newcommand{\End}{\mathrm{End}}
\newcommand{\Hom}{\mathrm{Hom}}

\newcommand{\Ima}{{\mathrm{Im}}}

\newcommand{\vect}{\mathrm{vect}}

\newcommand{\Tr}{\mathrm{Tr}}
\newcommand{\tr}{\mathrm{tr}}

\newcommand{\inv}{\mathbb{F}}

\newcommand{\lev}{\mathrm{ev}}

\newcommand{\rev}{\widetilde{\mathrm{ev}}}
\newcommand{\lcoev}{\mathrm{coev}}
\newcommand{\rcoev}{\widetilde{\mathrm{coev}}}
\newcommand{\ldual}[1]{#1^{*}}

\newcommand{\scaledraw}[1]{A}
\newcommand{\scaleraisedraw}[2]{A}
\newcommand{\rsdraw}[3]{\raisebox{-#1\height}{\scalebox{#2}{\includegraphics{#3.eps}}}}
\newcommand{\labela}{\renewcommand{\labelenumi}{{\rm (\alph{enumi})}}}
\newcommand{\labeli}{\renewcommand{\labelenumi}{{\rm (\roman{enumi})}}}

\begin{document}

\title[Monoidal categories graded by crossed modules and HQFTs]{Monoidal categories graded by crossed modules and 3-dimensional HQFTs}
\author{K\"{u}r\c{s}at S\"{o}zer}
 \address{K\"{u}r\c{s}at S\"{o}zer\newline
  \indent  \'Ecole Polytechnique F\'ed\'erale de Lausanne, 1015 Lausanne, Switzerland\\}
\email{kursat.sozer@epfl.ch}
\author{Alexis Virelizier}
\address{Alexis Virelizier\newline
\indent Univ. Lille, CNRS, UMR 8524 - Laboratoire Paul Painlev\'e, F-59000 Lille, France\\}
\email{alexis.virelizier@univ-lille.fr}

\date{\today}

\begin{abstract}
Given a crossed module $\chi$, we introduce $\chi$-graded monoidal categories and $\chi$-fusion categories.
We use spherical $\chi$-fusion categories to construct (via the state sum method) 3-dimensional Homotopy Quantum Field Theories with target the classifying space $B\chi$ of the crossed module $\chi$ (which is a homotopy 2-type).
\end{abstract}

\maketitle

\setcounter{tocdepth}{1}
\tableofcontents

\section{Introduction}\label{sec-Intro}

Homotopy Quantum Field Theory (HQFT) is a branch of quantum topology concerned with  maps from manifolds
to a fixed target space $X$. The aim   is to define and to study
homotopy invariants of such maps using  methods of quantum topology.
The formal notion of a $d$-dimensional HQFT with target $X$ was introduced in
\cite{Tu1}. Roughly, it is a TQFT for closed oriented $(d-1)$-dimensional
manifolds and compact oriented $d$-dimensional cobordisms endowed with maps
to~$X$. Such an HQFT yields numerical homotopy invariants of maps from
closed oriented $d$-dimensional manifolds to $X$. Note that 1-dimensional HQFTs with target $X$ correspond bijectively to finite-dimensional representations of the fundamental group
of $X$ or, equivalently, to finite-dimensional flat vector bundles over $X$. This allows one to view HQFTs as high-dimensional generalizations of flat vector bundles.  The case of 2-dimensional HQFTs with target $X$ is quite well understood: they  are classified by certain classes of Frobenius algebras which are in particular graded by the fundamental group of $X$ (see \cite{BT,Tu1,PT,ST},  and \cite{So} for the extended setting).

The present paper focuses on 3-dimensional HQFTs. The case where the target space $X$  is a point (i.e., a connected homotopy 0-type)  corresponds to 3-dimensional TQFTs and has been intensively studied. In particular, there are two fundamental constructions of 3-dimensional TQFTs: the Turaev-Viro TQFT which consists of state sums on triangulations (or more generally on skeletons) of 3-manifolds and uses spherical fusion categories, and the Reshetikhin-Turaev TQFT which is based on a surgery approach and uses modular categories (see for example the monographs \cite{Tu1,TVi5} and the references therein). The case where the target space $X$ is a connected homotopy 1-type, that is, an  Eilenberg-MacLane space $K(G,1)$ where~$G$ is a group, has been studied by Turaev and the second author: both the state-sum and the surgery approaches can be extended using spherical and modular fusion $G$-graded categories (see \cite{TVi1,TVi3,TVi4}). Here, a $G$-graded category is a monoidal category whose objects are equipped with a degree in $G$ (which is multiplicative with respect to the monoidal product) and whose morphisms preserve the grading (i.e., there are no nonzero morphisms between objects with distinct degrees). In this paper, we extend the state sum approach to the case where the target space $X$  is a connected homotopy 2-type.

Following MacLane and Whitehead \cite{MLW}, we model pointed connected homotopy 2-types by crossed modules: if $X$ is a homotopy 2-type, then $X$ is homotopy equivalent to the classifying space of a crossed module. Recall that a crossed module is a group homomorphism $\CM\co E \to H$, with $H$ acting on the left on $E$, such that $\CM$ is $H$-equivariant (where $H$ acts on itself by conjugation) and satisfies the Peiffer identity. Such a crossed module has a classifying space $B\CM$ such that
$$
\pi_1(B\CM)=\Coker(\CM), \quad \pi_2(B\CM)=\Ker(\CM), \quad \pi_k(B\CM)=0 \quad \text{for $k \geq 3$.}
$$

We introduce the notion of a $\CM$-graded monoidal category in which not only objects have a degree but also morphisms. More precisely, each (homogeneous) object of such a category~$\cc$ is equipped with a degree $|X| \in H$ (which is multiplicative with respect to the monoidal product) and the Hom-sets of $\cc$ are $E$-graded $\kk$-modules:
$$
\Hom_\cc(X,Y)=\bigoplus_{e \in E} \Hom_\cc^e(X,Y).
$$
For homogeneous objects $X,Y$, we have $\Hom_\cc^e(X,Y)=0 $ whenever $|Y|\neq \CM(e)|X|$. In particular, there may be nonzero morphisms between homogeneous objects with distinct degrees.  The  composition of morphisms is multiplicative in degree:
$$
(\beta,\alpha) \in\Hom_\cc^f(Y,Z)\times\Hom_\cc^e(X,Y) \;\, \Rightarrow \;\, \beta\circ \alpha \in  \Hom_\cc^{fe}(X,Z).
$$
The monoidal product of morphisms is compatible with the gradings as follows:
$$
(\alpha,\beta) \in \Hom_\cc^e(X,Y)\times\Hom_\cc^f(Z,T) \;\, \Rightarrow \;\,  \alpha \otimes \beta \in  \Hom_\cc^{e\lact{|X|}{\!f}}(X\otimes Z,Y \otimes T).
$$
For example, the $\kk$-linearization $\kk\tg_\CM$ of the 2-group $\tg_\CM$ associated with $\CM$ is a $\CM$\ti graded monoidal category. Other instances of $\CM$-graded monoidal categories are given by the categories of representations of Hopf $\CM$-(co)algebras (see~\cite{SV}).
Note that if $G$ is a group, then the group homomorphism $1 \to G$ is a crossed module and $(1 \to G)$-graded monoidal categories correspond to $G$-graded monoidal categories. Any cohomology class in $H^3(B\CM,\kk^*)$ can always be represented by a normalized 3-cocycle $\omega\co H^3 \times E^3 \to \kk^*$ which allows to twist the composition, the monoidal product, and the associator of any $\CM$-graded monoidal category $\cc$ to produce another $\CM$-graded monoidal category $\cc^\omega$. Also, the push-forward~$\phi_*(\cc)$ of a $\CM$-graded monoidal category $\cc$ along a crossed module morphism $\phi\co \CM \to \CM'$ is a $\CM'$-graded monoidal category.

Next, using the state sum approach, we derive a 3-dimensional HQFT $|\cdot|_\cc$ with target $B\CM$ from any spherical $\CM$-fusion category $\cc$ whose neutral component has invertible dimension (see Theorems~\ref{thm-state-3man-Xi} and ~\ref{thm-state-sum-Xi-HQFT}). To this end, we represent 3-manifolds by their skeletons and the maps to $B\CM$ by certain $H$-labels on the faces and $E$-labels on the edges of the skeletons. In particular, this HQFT induces a scalar invariant $|M,g|_\cc$ of closed $\CM$-manifolds, that is, of pairs $(M,g)$ where $M$ is closed oriented 3-manifold and $g$ is a homotopy class of maps $M \to B\CM$. We prove by examples that this invariant is nontrivial. It may even distinguish homotopy classes of phantom maps $M \to B\CM$  (i.e., of maps inducing trivial homomorphisms on homotopy groups). For the spherical $\CM$-fusion category $\kk\tg_\CM^\omega$ obtained by twisting the category $\kk\tg_\CM$ with a 3-cocycle  $\omega$ for $\CM$, we conjecture that
$$
  |M,g|_{\kk\tg_\CM^\omega}=\langle g^*([\omega]), [M] \rangle,
$$
where $g^*([\omega])  \in H^3(M,\kk^*)$ is the pullback of  the cohomology class $[\omega] \in H^3(B\CM,\kk^*)$ induced by the cocycle  $\omega$,  $[M] \in H_3(M,\ZZ)$ is the fundamental class of $M$,  and $\langle \,,\, \rangle \co  H^3(M,\kk^*) \times H_3(M,\ZZ) \to \kk$ is the Kronecker pairing. Moreover, the push-forward $\phi_*(\cc)$ of a $\CM$-fusion category $\cc$ along a crossed module morphism $\phi\co \CM \to \CM'$ (satisfying some finiteness conditions) is a $\CM'$-fusion category, and we relate the invariant $\vert M,g'\vert_{\phi_*(\cc)}$ of a closed $\CM'$-manifold $(M,g')$ with the invariants $ \{\vert M,g\vert_{ \cc }\}_g$ and the fundamental groups $\{\pi_1(\TOP_*(M,B\CM),g)\}_g$, where $g$ runs over homotopy classes of maps $M \to B\CM$ such that $B\phi \circ g=g'$ (see Theorem~\ref{thm-pushforward-3man-Xi}). Here $\TOP_*(M,B\CM)$ is the mapping space of pointed maps and $B\phi \co B \CM \to B\CM'$ is the map induced by $\phi$.
In particular, if $\phi$ is an equivalence of crossed modules (so that~$B\phi$ is a homotopy equivalence), then $\vert M,g'\vert_{\phi_*(\cc)}=\vert M,(B\phi)^{-1} \circ g'\vert_{ \cc }$.

The paper is organized as follows. In Section~\ref{sect-crossed-modules}, we  review crossed modules and their classifying spaces.
Section~\ref{sect-E-enreiched-categories} is dedicated to categories whose Hom-sets are graded by a group.
In Section~\ref{sect-chi-graded-categories}, we introduce the notion of a monoidal category graded by a crossed module $\CM$ and
discuss various classes of such categories including pivotal, spherical, and  $\CM$-fusion ones. In Section~\ref{sec-multimodulesandgraphs}, we first introduce colored $\CM$\ti cyclic sets and their multiplicity modules, and then we define an isotopy invariant of colored $\CM$-graphs which is used in our state sums (as a replacement of $6j$-symbols).
In Section~\ref{sect-skeletons-maps}, we discuss skeletons of 3-manifolds and
presentations of maps to $B\CM$ by labelings of skeletons.  We use these presentations in Section~\ref{sect-state-sum-invariants-closed} to derive from any  spherical $\CM$-fusion category a numerical invariant  of closed $\CM$-manifolds. In Section~\ref{sect-Xi-HQFTs}, we recall the definition (adapted to our need) of a 3-dimensional HQFT with target $B\CM$. Finally, in Section~\ref{sect-statesum-HQFT}, we extend the state sum invariants of Section~\ref{sect-state-sum-invariants-closed} to an HQFT with target $B\CM$.

We fix throughout the paper a nonzero commutative ring~$\kk$. For topological spaces $X$ and $Y$, we denote by $[X,Y]$ the set of homotopy classes of maps $X \to Y$.

\section{Crossed modules and their classifying spaces}\label{sect-crossed-modules}
We recall some standard notions of the theory of crossed modules.

\subsection{Crossed modules}\label{sect-crossed-modules-def}
A \emph{crossed module} is a group homomorphism $\CM \co E \to H$ together with a left  action of $H$ on $E$ (by group automorphisms) denoted  $$(x,e) \in H \times E  \mapsto  \elact{x}{e} \in E$$ such that $\CM$ is equivariant with respect to the conjugation action of $H$ on itself:
\begin{equation}\label{eq-precrossed}
\CM(\elact{x}{e}) =x\CM(e)x^{-1}
\end{equation}
and satisfies the Peiffer identity:
\begin{equation}\label{eq-Peiffer}
\lact{{\CM(e)}}{\!f}=efe^{-1}
\end{equation}
for all $x \in H$ and $e,f \in E$.

These axioms imply that the image $\Ima(\CM)$ is normal in $H$ and that the kernel $\Ker(\CM)$ is central in $E$ and is acted on trivially by $\Ima(\CM)$. In particular,  $\Ker(\CM)$ inherits an action of $H/\Ima(\CM)=\Coker(\CM)$.

\subsection{Examples}\label{sect-crossed-modules-ex}
1. Given any  normal subgroup $E$ of a group $H$, the inclusion $E\hookrightarrow H$ is a crossed module with the conjugation action of $H$ on $E$.

2. For any group $E$, the homomorphism $E \to \Aut(E)$ sending any element of $E$ to the corresponding inner automorphism is a crossed module.

3. For any group $H$ and any left $H$-module $E$, the trivial map $E \to H$ is a crossed module.

4. Any group epimorphism $E\to H$ with kernel contained in the center of $E$ is a crossed module with trivial action of $H$ on $E$. In particular, if $A$ is an abelian group, then the trivial map $A \to 1$ is a crossed module.

5. To any map $\omega \co S \to H$, where $S$ is a set and $H$ is a group, one associates the free crossed module $\CM \co \mathrm{CM}(\omega) \to H$ as follows. The group $\mathrm{CM}(\omega)$ is generated by $S \times H$ with the relations
$$
(e,x)(f,y)=(f,x\omega(e)x^{-1}y)(e,x)
$$
for all $e,f \in S$ and $x,y \in H$. The action of $H$ on $CM(\omega)$ and the morphism $\CM$ are given for all $e \in S$ and $x,y \in H$ by
$$
\lact{x}{(e,y)}=(e,xy) \quad \text{and} \quad \partial(e,x)=x \omega(e) x^{-1}.
$$

6. A key geometric example of a crossed module is due to Whitehead. He showed that if $(X,A,x)$ is a pair of pointed topological spaces, then the homotopy boundary map $\partial \co \pi _{2}(X,A,x)\rightarrow \pi _{1}(A,x)$,
together with the standard action of $\pi _{1}(A,x)$ on $\pi _{2}(X,A,x)$,  is a crossed module.

\subsection{Crossed module morphisms}\label{sect-crossed-modules-maps}
A \emph{morphism} from a crossed module $\CM \co E \to H$ to a crossed module $\CM' \co E' \to H'$ is a pair $\phi=(\psi \co E \to E',\varphi \co H \to H')$
of group homomorphisms making the square
$$
\xymatrix@R=.7cm @C=.7cm{ E \ar[r]^-{\CM}  \ar[d]_-{\psi} & H\ar[d]^-{\varphi} \\
    E' \ar[r]_-{\CM'} & H'
    }
$$
commutative and  preserving the action in the sense that for all $x \in H$ and $e\in E$,
$$
\psi(\elact{x}{e})=\lact{\varphi(x)}{\psi(e)}.
$$

Crossed modules and crossed module morphisms form a category denoted $\XMod$. Whitehead's construction (see Example 6 of  Section~\ref{sect-crossed-modules-ex}) extends to a functor
$$
\Pi_2 \co (\text{pair of pointed topological spaces}) \to \XMod.
$$

\subsection{Classifying spaces of crossed modules}\label{sect-crossed-modules-classifying-spaces}
There is a classifying space functor~$B$ (see \cite{BH1}) which  assigns to a crossed module $\CM \co E \to H$
a connected,  reduced\footnote{A CW-complex is \emph{reduced} if it has a single 0-cell, serving then as basepoint.} CW-complex
$B\CM$ with the following properties:
\begin{enumerate}
\labeli
\item The homotopy groups of the classifying space $B\CM$ are given by
$$
\pi_1(B\CM)=\Coker(\CM), \quad \pi_2(B\CM)=\Ker(\CM), \quad \pi_k(B\CM)=0 \quad \text{for $k \geq 3$.}
$$

\item The classifying space $B\CM$ has a canonical (reduced) subcomplex $BH$ which is a classifying space of the group $H$, contains all 1-cells of $B\CM$, and satisfies
    $$
    \Pi_2(B\CM,BH)=\CM.
    $$
In particular, $BH$ is the classifying space $B(1 \to H)$ of the trivial crossed module $1 \to H$.

\item Let $X$ be a reduced CW-complex. Let  $X_1$ be the 1-skeleton of $X$ and consider the Whitehead crossed module $\partial \co \pi _{2}(X,X_1)\rightarrow \pi _{1}(X_1)$. Then there is a map
$X \to B\partial$
inducing isomorphisms on $\pi_1$ and $\pi_2$.

\end{enumerate}
It follows that crossed modules model all pointed connected homotopy 2-types (as originally proved by MacLane and Whitehead \cite{MLW}).

\subsection{2-groups and crossed modules}\label{sect-crossed-modules-two-groups}
Recall that a 2-group is a monoidal small category in which every morphism is invertible and every object has a weak inverse.
A strict 2-group is a strict monoidal small category in which every morphism is invertible and every object has a strict inverse.

Strict 2-groups are identified with crossed modules as follows. Any crossed module ~$\CM \co E \to H$ gives rise to a strict 2-group $\tg_\CM$ defined as follows. The objects of~$\tg_\CM$  are the elements of~$H$. For any $x,y \in H$,
$$
\Hom_{\tg_\CM}(x,y)=\{e \in E \, | \, y=\CM(e)x\}.
$$
The composition of morphisms is given by the product of $E$:
$$
\left(y \xrightarrow{f} z \right) \circ \left(x \xrightarrow{e} y \right) = \left(x \xrightarrow{fe} z \right) \quad \text{and} \quad \id_x=1 \in E.
$$
The monoidal product of objects is given  by the product of $H$:
$$
x \otimes y =xy \quad \text{and} \quad \un=1 \in H.
$$
The monoidal product of morphisms is induced by the left action of $H$ on $E$:
$$
\left(x \xrightarrow{e} y \right) \otimes \left(z \xrightarrow{f} t \right) = \left(xz \xrightarrow{e\lact{x}{\!f}} yt \right).
$$

Conversely, any strict 2-group $\tg$ gives rise to a crossed module $E \to H$, where~$H$ is the
set of objects of $\tg$, $E$ is the set of morphisms emanating from the unit object, and the map $E \to H$ is the target map.

\section{Categories with Hom-sets graded by a group}\label{sect-E-enreiched-categories}

In this  section, we study Hom-graded categories which are linear categories whose Hom-sets are graded by a group. Throughout this section, $E$ denotes a group.

\subsection{Linear categories}\label{sect-linear-categories}
A category~$\cc$ is \emph{\kt linear}  if for all objects $X,Y\in \cc$, the set $\Hom_\cc(X,Y)$ carries a structure of a left  \kt module so that  the composition of morphisms  is $\kk$-bilinear.

An object $X$ of a $\kk$-linear category $\cc$ is \emph{simple} if $\End_\cc(X)$ is a free \kt module of rank 1 (and
so has the basis $\{\id_X\}$).     It is clear that   an object isomorphic to a  simple object is itself  simple.

A $\kk$-linear category $\cc$ is  \emph{semisimple} if
\begin{enumerate}
  \labela
\item each object of $\cc$ is a finite direct sum of  simple objects;
\item for any non-isomorphic  simple objects $i,j  $ of $\cc$, we have $\Hom_\cc(i,j)=0$.
\end{enumerate}
Clearly, the Hom spaces in such a $\cc$ are   free $\kk$-modules of finite rank.

A set $I$ of simple objects  of a semisimple $\kk$-linear category $\cc$ is \emph{representative} if every simple object of $\cc$ is isomorphic to a unique element of~$I$. Each object of $\cc$ is then a direct sum of a finite family of elements of~$I$. Also,  $\Hom_\cc(i,j)=0$ for any distinct elements $i,j \in I $.

A monoidal category is  \emph{$\kk$-linear} if it is $\kk$-linear as a category and the monoidal product of morphisms
is \kt bilinear.

\subsection{Hom-graded categories}\label{sect-Hom-graded-categories}
By an $E$-\emph{Hom-graded category (over $\kk$)} we mean a category enriched over the monoidal category of
$E$-graded \kt modules  and  \kt linear  grading-preser\-ving homomorphisms. In other words, an \emph{$E$-Hom-graded category (over $\kk$)} is a \kt linear category $\cc$ such that:
\begin{enumerate}
\labela
\item The Hom-sets in $\cc$ are $E$-graded \kt modules: for all objects $X,Y \in \cc$,
$$
\Hom_\cc(X,Y)=\bigoplus_{e \in E} \Hom_\cc^e(X,Y).
$$
\item The composition in $\cc$ is grading-preser\-ving:
for all $X,Y, Z \in \cc$ and $e,f \in E$, it sends $\Hom_\cc^f(Y,Z) \times \Hom_\cc^e(X,Y)$ into $\Hom_\cc^{fe}(X,Z)$.
\item The identities have trivial degree: for all $X\in \cc$, $$\id_X \in \End_\cc^1(X)=\Hom_\cc^1(X,X),$$
where $1$ denotes the unit element of $E$.
\end{enumerate}
Note that for any object $X \in \cc$, the monoid $\End_\cc(X)=\Hom_\cc(X,X)$ is an $E$-graded \kt algebra.

A morphism $f \co X \to Y$ in an $E$-Hom-graded category $\cc$ is \emph{homogeneous} if $$f \in\coprod_{e \in E} \Hom_\cc^e(X,Y).$$ It is \emph{homogeneous of degree $e \in E$} if $f \in \Hom_\cc^e(X,Y)$. Note that if $f$ is nonzero, then such an $e \in E$ is unique, is called the \emph{degree} of $f$, and is denoted $e=|f|$. The objects of $\cc$ together with the homogenous morphisms of degree 1 form a \kt linear subcategory of $\cc$ called the \emph{1-subcategory} of $\cc$ and denoted $\cc^1$.

Given $e \in E$, by an \emph{$e$-isomorphism} we mean an isomorphism which is homogeneous of degree~$e$. We say that an object $X$ is \emph{$e$-isomorphic} to an object $Y$, and we write $X \cong_e Y$, if there is an $e$-isomorphism $X \to Y$. It is easy to see that for all $X,Y,Z \in \cc$ and $e,f \in E$,
\begin{equation}\label{eq-cong-E}
X \cong_1 X, \quad  X \cong_e Y \Longrightarrow Y \cong_{e^{-1}} X, \quad X \cong_e Y \text{ and } Y \cong_f Z \Longrightarrow
X \cong_{fe} Z.
\end{equation}
Also, if $ X \cong_e Y$, then for all $d \in E$, there are  \kt linear isomorphisms
\begin{equation}\label{eq-hom-graded}
\Hom_\cc^d(Y,Z) \simeq \Hom_\cc^{de}(X,Z) \quad \text{and} \quad  \Hom_\cc^d(Z,X) \simeq  \Hom_\cc^{ed}(Z,Y)
\end{equation}
which are induced by pre/post-composition with an $e$-isomorphism $X \to Y$.

\subsection{Example}\label{ex-graded-vector-all-maps}
A \kt module $M$ is \emph{$E$-graded} if it has a direct sum decomposition
$$
M= \bigoplus_{e \in E} M_e
$$
by submodules indexed by $e \in E$. We say that such a \kt module $M$ has \emph{finite support} if $M_e=0$ for all but a finite number of $e\in E$. For example, this is the case if $E$ is finite or if~$M$ is free of finite rank.
A \kt linear homomorphism $\phi \co M \to N$ between $E$-graded \kt modules is \emph{homogeneous of degree $d\in E$} if $\phi(M_e) \subset N_{de}$ for all $e \in E$. Such homomorphisms form a submodule $\Hom_\kk^d(M,N)$ of the \kt module $\Hom_\kk(M,N)$ of \kt linear homomorphisms $M \to N$. For example, the identity $\id_M \co M \to M$ is homogenous of degree 1. Clearly, the composition of homogeneous homomorphisms is multiplicative with respect to the degree. This induces an $E$-Hom-graded category $\lmm{E}$ whose objects are $E$-graded \kt modules and Hom-sets are
$$
\Hom_{\lmm{E}}(M,N)=\bigoplus_{d\in E} \Hom_\kk^d(M,N).
$$
The 1-subcategory of $\lmm{E}$ is the category of $E$-graded \kt modules and grading-preserving \kt linear homomorphisms.
Note that if  $M$ and $N$ have finite support, then
$\Hom_{\lmm{E}}(M,N)=\Hom_\kk(M,N)$.
Then $E$-graded \kt modules with finite support and \kt linear homomorphisms form an $E$-Hom-graded category (as a subcategory of $\lmm{E}$). In particular, $E$-graded free \kt modules of finite rank and \kt linear homomorphisms form an $E$-Hom-graded category.

\subsection{Direct sums in Hom-graded categories}\label{sect-Hom-graded-direct-sums}
Let $\cc$ be an $E$-Hom-graded category. Given $e \in E$, an object~$D$ of~$\cc$ is an \emph{$e$-direct sum} of a finite family $(X_\alpha)_{\alpha \in A}$ of objects of~$\cc$ if there is a family $(p_\alpha, q_\alpha)_{\alpha \in A}$ of morphisms in $\cc$ such that:
\begin{enumerate}
\labela
\item $p_\alpha\co D \to X_\alpha$ is homogeneous of degree $e^{-1}$ for all $\alpha \in A$;
\item $q_\alpha \co X_\alpha \to D$ is homogeneous of degree $e$ for all $\alpha \in A$;
\item $\id_D=\sum_{\alpha \in A} q_\alpha p_\alpha$;
\item $p_\alpha q_\beta=\delta_{\alpha,\beta} \, \id_{X_\alpha}$ for all $\alpha,\beta\in A$, where $\delta_{\alpha,\beta}$ is the Kronecker symbol  defined by $\delta_{\alpha,\beta}=1$ if $\alpha=\beta$ and $\delta_{\alpha,\beta}=0$ otherwise.
\end{enumerate}
Such an $e$-direct sum $D$, if it exists, is  unique up to a 1-isomorphism and is denoted
$$
D=  \bigoplus_{\alpha \in A}^e X_\alpha.
$$
It is easy to see that if a finite family $(X_\alpha)_{\alpha \in A}$ of objects of~$\cc$ has an $e$-direct sum
and an $f$-direct sum with $e,f \in E$, then
$$
\bigoplus_{\alpha \in A}^e X_\alpha \cong_{fe^{-1}}   \bigoplus_{\alpha \in A}^f X_\alpha.
$$
Also,  for any finite families $(X_\alpha)_{\alpha \in A}$  and $(Y_\beta)_{\beta \in B}$  of objects of~$\cc$ and for any $d,e,f \in E$, there are \kt linear isomorphisms
\begin{equation}\label{eq-Hom-direct-sum}
\Hom_\cc^d \left( \bigoplus_{\alpha \in A}^e X_\alpha,\bigoplus_{\beta \in B}^f Y_\beta \right )  \simeq
\bigoplus_{\substack{\alpha \in A \\ \beta \in B}} \Hom_\cc^{f^{-1}de}\bigl(X_\alpha,Y_\beta\bigr).
\end{equation}
By definition, a direct sum of an empty family of objects of a \kt linear category $\cc$ is a \emph{zero object}   of $\cc$, that is, an object $\mathbf{0}$ of $ \cc$ such that $\End_\cc(\mathbf{0})=0$.

An $E$-Hom-graded category~$\cc$ is \emph{$E$-additive} if any finite (possibly empty) family of objects of~$\cc$ has an $e$-direct sum in~$\cc$ for all $e \in E$.

\subsection{Semisimple Hom-graded categories}\label{sect-Hom-graded-semisimple}
We say that an $E$-Hom-graded category $\cc$ is \emph{$E$-semisimple} if:
\begin{enumerate}
\labela
\item its 1-subcategory $\cc^1$ is semisimple (see Section~\ref{sect-linear-categories});
\item for any $e \in E$, each object of $\cc$ is an $e$-direct sum of a finite family of simple objects of $\cc^1$.
\end{enumerate}
For $E=1$, we recover the usual notion of a semisimple category.
Note that a simple object of $\cc^1$ is nothing but an object of $\cc$ such that  $\End_\cc^1(X)$ is free of rank~1 (with basis $\id_X$). In particular, a simple object of $\cc^1$  may not be simple in $\cc$. Consequently, an $E$-semisimple $E$-Hom-graded category  is not necessarily  semisimple (in the usual sense of Section~\ref{sect-linear-categories}).

For example, the category of $E$-graded free \kt modules of finite rank and  all \kt linear homomorphisms (see Section~\ref{ex-graded-vector-all-maps}) is an $E$-semisimple $E$-Hom-graded category.

\begin{lem}\label{lem-E-semisimple}
Let $\cc$ be an $E$-semisimple $E$-Hom-graded category. Then:
\begin{enumerate}
\labeli
\item For any simple object $i$ of $\cc^1$ and $e \in E$, there is a simple  object $j$ of $\cc^1$ such that
$i \cong_e j$.
\item For any simple objects $i,j$ of $\cc^1$ and $e \in E$, the \kt module $\Hom_\cc^e(i,j)$ is free  of rank 1 if $i \cong_e j$ and $\Hom_\cc^e(i,j)=0$ otherwise.
\item Hom-sets in $\cc$ are free \kt modules of finite rank.
\item For any objects $X,Y\in \cc$ and $e,f \in E$, there is a   \kt linear isomorphism
$$
\Hom_\cc^{fe}(X,Y) \simeq \bigoplus_{i \in I} \Hom_\cc^f(i,Y)  \otimes_\kk \Hom_\cc^e(X,i),
$$
where  $I$ is any representative set of simple objects of $\cc^1$.
\end{enumerate}
\end{lem}
\begin{proof}
Let us prove Part (i). Pick a decomposition $(p_\alpha \co i \to i_\alpha, q_\alpha \co i_\alpha \to i)_{\alpha
\in A}$   of~$i$ as an $e^{-1}$-direct sum of a finite family $(i_\alpha)_{\alpha \in A}$ of simple objects $\cc^1$.
Formula~\eqref{eq-Hom-direct-sum} implies that
$$
\bigoplus_{\alpha,\beta \in A} \Hom_\cc^1(i_\alpha,i_\beta)\simeq\End_\cc^1 (i)
\quad
\text{and so}
\quad
\bigoplus_{\alpha \in A} \End_\cc^1(i_\alpha)\subset\End_\cc^1 (i).
$$
Since $i$ and each $i_\alpha$ are simple in $\cc^1$, taking the rank in the latter inclusion gives that the cardinal of $A$ is less than or equal to 1. Since the set $A$ is nonempty (otherwise $i$ would be the zero object which is not simple in $\cc^1$), it is a singleton $A=\{\alpha_0\}$. Set $j=i_{\alpha_0}$. Then $p=p_{\alpha_0} \co i \to j$ is a homogeneous morphism of degree~$e$ and $q=q_{\alpha_0} \co j \to i$ is a homogeneous morphism of degree $e^{-1}$ such that  $\id_i=qp$ and $pq=\id_j$. Thus~$p$ is  an $e$-isomorphism and so $i \cong_e j$.

Let us prove Part (ii). If $i \cong_e j$, then Formula \eqref{eq-hom-graded} implies that $\Hom_\cc^{e}(i,j) \simeq \End_\cc^1(i)$ is free of rank 1 (since $i$ is simple in $\cc^1$). Assume that $i \not \cong_e j$. By Part (i), there is a simple object $k$ of $\cc^1$ such that $j \cong_{e^{-1}} k$. Then $i \not \cong_1 k$ (indeed, using  \eqref{eq-cong-E} and since $k \cong_e j$, if
$i\cong_1 k$ then $i \cong_e j$). Then the simple objects $i,j$ of $\cc^1$ are not isomorphic in $\cc^1$ and so $\Hom_\cc^1(i,k)=0$ (since $\cc^1$ is semisimple). Hence, using \eqref{eq-hom-graded}, we conclude that
$\Hom_\cc^e(i,j)\simeq \Hom_\cc^1(i,k)=0$.

Since each object of $\cc$ is a $1$-direct sum of a finite family of simple objects of $\cc^1$, Part (iii) follows directly from \eqref{eq-Hom-direct-sum} and Part (ii).

Let us prove Part (iv).  Pick a decomposition $(p_\alpha \co X \to i_\alpha, q_\alpha \co i_\alpha \to X)_{\alpha
\in A}$   of~$X$ as an $e^{-1}$-direct sum of a finite family $(i_\alpha)_{\alpha \in A}$ of objects in $I$. For $i\in I$, consider the subset $A_i$ of $A$ consisting of elements $\alpha \in A$ such that $i_\alpha =i$. Note that $A_i$ is empty for all but a finite number of $i \in I$. It follows from the definition of an~$e^{-1}$-direct sum decomposition (see Section~\ref{sect-Hom-graded-direct-sums}) that the map
$$
\phi \in \Hom_\cc^{fe}(X,Y) \mapsto \sum_{i \in I} \left ( \sum_{\alpha \in A_i} (\phi q_\alpha)  \otimes_\kk p_\alpha \right ) \in \bigoplus_{i \in I} \Hom_\cc^f(i,Y)  \otimes_\kk \Hom_\cc^e(X,i)
$$
is a \kt linear isomorphism with inverse defined on the summands by
\begin{equation*}
v \otimes_\kk u\in \Hom_\cc^f(i,Y)  \otimes_\kk \Hom_\cc^e(X,i) \mapsto vu \in \Hom_\cc^{fe}(X,Y).  \qedhere
\end{equation*}
\end{proof}

\begin{lem}\label{lem-E-action-on-representatives}
Let $\cc$ be an $E$-Hom-graded category such that its 1-subcategory~$\cc^1$ is semisimple. Let $I$ be a representative set of simple objects of $\cc^1$. Then the following assertions are equivalent:
\begin{enumerate}
\labeli
\item The category $\cc$ is $E$-semisimple.
\item For all $i \in I$ and $e \in E$, there exists a unique $j \in I$ such that $i \cong_e j$.
\item There is a left action $E \times I \to I$, $(e,i) \mapsto e \cdot i$, such that:
$$
i \cong_e j \Longleftrightarrow j=e \cdot i \quad \text{for all $i,j\in I$ and $e \in E$.}
$$
\end{enumerate}
\end{lem}
\begin{proof}
Let us prove that (i) implies (ii). Let $i \in I$ and $e \in E$. By Lemma~\ref{lem-E-semisimple}(i), there is a simple object $k$ of $\cc^1$  such that $i \cong_e k$. Since $I$ is a representative set of simple objects of $\cc^1$, there is $j \in I$ such that $k \cong_1 j$. Then $i \cong_e j$ by \eqref{eq-cong-E}. If $j' \in I$ is such that $i \cong_e j'$, then $j \cong_1 j'$ by \eqref{eq-cong-E}, and so $j=j'$ since $I$ is a representative set of simple objects of $\cc^1$.

Let us prove that (ii) implies (iii). For any $i \in I$ and $e \in E$, denote by $e\cdot i$ the unique element of $I$ verifying
$i \cong_e e\cdot i$. This defines a map $E \times I \to I$. It follows from \eqref{eq-cong-E} and the uniqueness in (ii) that this map is a left action such that $i \cong_e j$ if and only if $j=e \cdot i$.

Let us prove that (iii) implies (i). Let $X$ be an object of $\cc$ and $e \in E$. We need to prove that $X$ is an $e$-direct sum of a finite family of simple objects of $\cc^1$. Since~$\cc^1$ is semisimple, there is a decomposition $(p_\alpha \co X \to j_\alpha, q_\alpha \co j_\alpha \to i)_{\alpha \in A}$ of~$X$ as a $1$-direct sum of a finite family $(j_\alpha)_{\alpha \in A}$ of simple objects.  For any $\alpha \in A$, set $i_\alpha=(e^{-1} \cdot j_\alpha) \in I$. Then $j_\alpha =e \cdot i_\alpha$ and so $i_\alpha \cong_e j_\alpha$. Pick an $e$-isomorphism $\phi_\alpha \co i_\alpha \to j_\alpha$ and set
$P_\alpha=\phi_\alpha^{-1} p_\alpha \co X \to i_\alpha$ and  $Q_\alpha=q_\alpha\phi_\alpha \co i_\alpha \to X$. It is easy to verify that $(P_\alpha, Q_\alpha)_{\alpha \in A}$ is a decomposition of~$X$ as an $e$-direct sum of the finite family $(i_\alpha)_{\alpha \in A}$.
\end{proof}

\subsection{Multiplicity numbers and partitions}\label{sect-multiplicity-partitions}
Let $\cc$ be an $E$-semisimple $E$-Hom-graded category. Given an object $X$ of $\cc$, a simple object $i$ of $\cc^1$, and $e \in E$, we denote by $N^{i,e}_X \geq 0$ the number of simple objects 1-isomorphic to~$i$  in decomposition of~$X$ as an $e$-direct sum of a finite family of simple objects of $\cc^1$. By \eqref{eq-Hom-direct-sum} and Lemma~\ref{lem-E-semisimple}(ii), this number is equal to the ranks of   the free \kt modules  $\Hom_\cc^e(i,X)$ and $\Hom_\cc^{e^{-1}}\!(X,i)$. We call $N^{i,e}_X$ the \emph{degree $e$ multiplicity index of $i$ in~$X$}. It follows from \eqref{eq-Hom-direct-sum} that if $X$ is a $f$-direct sum of a finite family $(X_\alpha)_{\alpha \in A}$ of objects of~$\cc$ with $f \in E$, then
$$
N^{i,e}_X=\sum_{\alpha \in A} N^{i,f^{-1}e}_{X_\alpha}.
$$

If $I$ is a representative set of simple objects of $\cc^1$, then for all objects $X$ of $\cc$ and $e \in E$, the family $\{N^{i,e}_X\}_{i\in I}$ has only a finite number  of nonzero terms, and Lemma~\ref{lem-E-semisimple}(iv) implies that
for all objects $X,Y$ of $\cc$ and $e,f \in E$,
$$
\mathrm{rank}_\kk\bigl(\Hom_\cc^{fe} (X,Y) \bigr)=\sum_{i\in I} N^{i,f}_Y  N^{i,e}_X.
$$

Given a simple object $i$ of $\cc^1$ and $e \in E$, an \emph{$(i,e)$-partition} of an object   $X$  of~$\cc$  is a family of morphisms  $(p_\alpha\co X \to i, \,  q_\alpha\co i \to X)_{\alpha \in \Lambda}$ such that $(p_\alpha)_{\alpha \in \Lambda}$ is a basis of $\Hom_\cc^{e^{-1}}\!(X,i)$, $(q_\alpha)_{\alpha \in \Lambda}$ is a basis of $\Hom_\cc^e(i,X)$, and $ p_\alpha q_\beta=\delta_{\alpha,\beta} \, \id_{i} $ for all $\alpha,\beta\in \Lambda$.  Note that the cardinality of $\Lambda$  is equal to the multiplicity index $N^{i,e}_X$. The existence of $(i,e)$-partitions follows from
the fact that each object of $\cc$ is an $e$-direct sum of simple objects of $\cc^1$.

Given a representative set $I$ of simple objects of $\cc^1$ and $e \in E$, an \emph{$(I,e)$-partition} of an object $X$ of $\cc$ is a decomposition $(p_\alpha\co X \to i_\alpha,q_\alpha\co i_\alpha \to X)_{\alpha \in A}$ of $X$ as an $e$-direct sum of a finite family $(i_\alpha)_{\alpha \in A}$ of objects in $I$. For such an $(I,e)$-partition and any $i \in I$, the family   $( p_\alpha  , q_\alpha)_{\alpha \in A, \, i_\alpha=i }$ is an $(i,e)$-partition of $X$. Conversely, a union of $(i,e)$-partitions of $X$  over
all $i\in I$ is  an $(I,e)$-partition of $X$.

\section{Monoidal categories graded by a crossed module}\label{sect-chi-graded-categories}

We introduce categories graded by a crossed module which are linear monoidal categories whose Hom-sets are graded by the source of the crossed module and whose objects are graded by the target of the crossed module in some compatible way (involving the crossed module axioms).

Throughout this section,  $\CM \co E \to H$ denotes a crossed module. Recall from Section~\ref{sect-crossed-modules-def} that it comes equipped with a left action $(x,e) \in H \times E \mapsto \elact{x}{e} \in E$.

\subsection{$\CM$-categories}\label{sect-crossed-module-graded-categories}
A \emph{$\CM$-graded category (over $\kk$)}, or shorter a  \emph {$\CM$-category}, is a $\kk$-linear  mo\-noi\-dal category $\cc=(\cc,\otimes,\un)$ which is $E$-Hom-graded (see Section~\ref{sect-Hom-graded-categories}) and is endowed with a subclass $\cch$ of nonzero objects of $\cc$, called \emph{homogeneous objects}, and with a map $|\cdot| \co \cch \to H$, called the \emph{degree map}, such that:
\begin{enumerate}
  \labela
  \item  Each object of $\cc$ is a $1$-direct sum of a finite family of homogeneous objects.
  \item  For all homogenous objects $X,Y$ and $e \in E$ such that $|Y|\neq \CM(e)|X|$,
  $$
  \Hom_\cc^e(X,Y)=0.
  $$
  \item  The monoidal product $X \otimes Y$ of any two homogeneous objects $X,Y$ is a $1$-direct sum of homogeneous objects of degree $|X||Y|$.
  \item The unit object $\un$ is homogeneous with trivial degree (i.e., $|\un|=1\in H$).
    \item The monoidal product $\alpha \otimes \beta$ of any two nonzero homogeneous morphisms $\alpha,\beta$ is a homogeneous morphism of degree
$$
|\alpha\otimes \beta|=|\alpha|\,\lact{|s(\alpha)|}{|\beta|},
$$
whenever the source $s(\alpha)$ of $\alpha$ is a homogeneous object. In other words, for any objects $X,Y,Z,T$ with $X$ homogeneous  and for any morphisms $\alpha \in \Hom_\cc^e(X,Y)$, $\beta \in \Hom_\cc^f(Z,T)$ with $e,f \in E$, we have:
$$
\alpha \otimes \beta \in  \Hom_\cc^{e \lact{|X|}{\!f}}(X\otimes Z,Y \otimes T).
$$
\item The associativity constraints $(X\otimes Y)\otimes Z\cong
X\otimes (Y\otimes Z)$ and the unitality constraints $X\otimes \un
\cong X\cong \un \otimes X$ of $\cc$ are all homogenous of degree $1 \in E$.
\end{enumerate}
Note that we  will suppress in our formulas  the associativity  and unitality  constraints. Indeed, this does not lead  to   ambiguity because all legitimate ways of inserting these constraints  give the same result (by  MacLane's   coherence theorem) and do not affect the degree of morphisms (by the above axioms).

Axiom (a) implies that the Hom-sets in $\cc$ are fully determined by the Hom-sets between homogeneous objects.
Axiom (b) and  \eqref{eq-Hom-direct-sum} imply that if $(X_\alpha)_{\alpha \in A}$,  $(Y_\beta)_{\beta \in B}$ are  finite families of homogeneous objects and  $e \in E$ is such that $|Y_\beta| \neq \CM(e)\,|X_\alpha|$ for all $\alpha \in A$ and $\beta \in B$, then
\begin{equation}\label{eq-dir-sum-homogeneous-objects}
\Hom_\cc^e \left( \bigoplus_{\alpha \in A}^1 X_\alpha,\bigoplus_{\beta \in B}^1 Y_\beta \right )=0.
\end{equation}
Axiom (b) implies that for all homogenous objects $X,Y$,  we have
$$
\Hom_{\cc}(X,Y)= \hspace{-1.8em} \bigoplus_{e \in \CM^{-1}(|Y||X|^{-1})} \hspace{-1.8em} \Hom_\cc^e(X,Y) \quad \text{and} \quad \End_{\cc}(X)=\!\!\bigoplus_{e \in \Ker(\CM)} \!\!\End_\cc^e(X).
$$
In particular $\Hom_\cc(X,Y)=0$ whenever $|Y||X|^{-1} \not  \in \Ima(\CM)$.
Also, if $\alpha \co X \to Y$ is a nonzero homogeneous morphism between homogeneous objects, then
\begin{equation}\label{eq-deg-morphism-objects}
|Y|=\CM(|\alpha|)\,|X|,
\end{equation}
where $|\alpha| \in E$ is the degree of $\alpha$ (see Section~\ref{sect-Hom-graded-categories}). In particular, 1-isomorphic homogenous  objects have the same degree.

Axiom (e) implies that for any homogenous objects $X,Y,Z$ of $\cc$ and $e \in E$,
$$
X \cong_e Y \quad \Longrightarrow \quad X \otimes Z \cong_e Y \otimes Z \quad \text{and} \quad  Z \otimes X \cong_{\lact{|Z|}{e}} Z \otimes Y.
$$

\subsection{Relations with group-graded monoidal categories}\label{sect-group-graded-categories}
Recall that a \kt linear monoidal category $\dd$ is \emph{graded by the group $H$}, or shorter \emph{$H$-graded}, if it decomposes as a direct sum
$$
\dd=\bigoplus_{h \in H} \dd_h,
$$
of full subcategories   such that $\un \in \dd_1$ and $\dd_g \otimes \dd_h \subset \dd_{gh}$ for all $g,h \in H$ (see for example \cite{TVi1} or \cite{EGNO}). The monoidal subcategory $\dd_1$ is called the \emph{neutral component} of $\dd$.

The 1-subcategory~$\cc^1$ of a $\CM$-graded category $\cc$ is an $H$-graded \kt linear monoidal category:
$$
\cc^1=\bigoplus_{h \in H} \cc^1_h
$$
where $\cc^1_h$ is the subcategory of $\cc$ whose objects are (finite)  $1$-direct sums of homogeneous objects of $\cc$ of degree $h$ and whose morphisms are homogeneous morphisms of degree $1\in E$.

Also, any $\CM$-graded category $\cc$ is $\Coker(\CM)$-graded:
$$
\cc=\!\!\! \bigoplus_{p \in \Coker(\CM)} \!\!\! \cc_p
$$
where $\cc_p$ is the full subcategory of $\cc$ whose objects are (finite) $1$-direct sums of homogeneous objects of $\cc$ with degree in the preimage $\pi^{-1}(p) \subset H$ of $p$ under the canonical projection $\pi \co H \to \Coker(\CM)=H/\Ima(\CM)$.

\subsection{Particular cases} 1. Given a group $H$, the trivial map $1 \to H$ is a crossed module and the notion of a $(1\to H)$-category  agrees with that of an $H$-graded \kt linear monoidal category (see Section~\ref{sect-group-graded-categories}).

2. Given an abelian group $A$, the trivial map $A \to 1$ is a crossed module and the notion of a $(A\to 1)$-category  corresponds to that of a \kt linear monoidal category whose Hom-sets are $A$-graded \kt modules so that the identities and monoidal constraints have trivial degree and the composition and monoidal product of morphisms are both multiplicative with respect to the degree: for any $\alpha \in \Hom_\cc^a(X,Y)$, $\beta \in \Hom_\cc^b(Y,Z)$, $\gamma \in \Hom_\cc^c(U,V)$ with $a,b,c \in A$, we have:
$$
\beta \circ \alpha \in  \Hom_\cc^{ba}(X,Z) \quad \text{and} \quad \alpha \otimes \gamma \in  \Hom_\cc^{ac}(X\otimes U,Y \otimes V).
$$

\subsection{Example}\label{ex-linearized-crossed-module}
Recall from Section~\ref{sect-crossed-modules-two-groups} that the crossed module $\CM$ gives rise to a strict 2-group $\tg_\CM$.
Let us consider the linearization $\kk\tg_\CM$ of this 2-group. The objects of $\kk\tg_\CM$ are those of $\tg_\CM$, that is, the elements of~$H$.
For any $x,y \in H$, the set $\Hom_{\kk\tg_\CM}(x,y)$ is the free \kt module with basis $\Hom_{\tg_\CM}(x,y)=\{e \in E \, | \, y=\CM(e)x\}$. The composition and monoidal product of $\kk\tg_\CM$ extend those of $\tg_\CM$ by bilinearity.
Then $\kk\tg_\CM$ is a \kt linear strict monoidal category. It is $\CM$-graded as follows: each object $x \in H$ is homogeneous with degree $|x|=x$, and for any $x,y \in H$ and $e \in E$,
$$
\Hom_{\kk\tg_\CM}^e(x,y)=\left\{ \begin{array}{ll} \kk e & \text{if $y=\CM(e)x$,} \\ 0 & \text{otherwise.} \end{array} \right.
$$

\subsection{Example}\label{ex-graded-vector-graded-maps}
Let $E$ be a normal subgroup of a group $H$. A \kt linear homomorphism $\alpha \co M \to N$ between $H$-graded \kt modules is \emph{homogeneous of degree $e\in E$} if $\alpha(M_h)  \subset N_{eh}$ for all $h \in H$.  Such homomorphisms form a submodule $\Hom_\kk^e(M,N)$ of $\Hom_\kk(M,N)$. A \kt linear homomorphism $\alpha \co M \to N$ between $H$-graded \kt modules is \emph{$E$-graded} if it is a linear combination of homogeneous \kt linear homomorphisms with degree in $E$. Such homomorphisms form a submodule $\Hom_\kk^E(M,N)$ of $\Hom_\kk(M,N)$. Clearly,
$$
\Hom_\kk^E(M,N)=\bigoplus_{e\in E} \Hom_\kk^e(M,N).
$$
With usual composition and identities, $H$-graded \kt modules and $E$-graded \kt linear homomorphisms form an $E$-Hom-graded category denoted $\lmm{E,H}$.
Note that $\lmm{E,E}$ is the category $\lmm{E}$ of Example~\ref{ex-graded-vector-all-maps}. We endow $\lmm{E,H}$ with   the   monoidal product    defined on objects   by
$$
M \otimes N=\bigoplus_{h\in H} \, (M \otimes N)_h \quad \text{where} \quad (M \otimes N)_h=\bigoplus_{\substack{x,y \in H \\ \, xy=h}} M_x \otimes_\kk N_y ,
$$
and on morphisms by $\alpha \otimes \beta=\alpha \otimes_\kk \beta$. This gives a \kt linear monoidal category  with unit object  $\un$ defined by $\un_1=\kk$ and $\un_h=0$ for all $h\in H\setminus \{1\}$. An $H$-graded module $M=\bigoplus_{h \in H} M_h$ is \emph{homogeneous} if it is nonzero and there is $h \in H$ such that $M_x=0$ for all $x\in H\setminus \{h\}$. Such an $h$ is then unique and called the \emph{degree} of~$M$. With these homogenous objects, $\lmm{E,H}$ becomes a $(E\hookrightarrow H)$-category, where the inclusion $E\hookrightarrow H$ is a crossed module with the conjugation action of $H$ on $E$.

\subsection{Examples from representations}\label{ex-Hopf-Xi-coalgebras}
Recall that the category of representations of a Hopf algebra is monoidal (and closed). Examples of $\CM$-graded categories can be constructed similarly: in \cite{SV}, we introduce Hopf $\CM$-(co)algebras and prove that their categories of representations are $\CM$-categories.

\subsection{Push-forward of $\CM$-categories}\label{sect-push-forward}
Let $\phi=(\psi \co E \to E',\varphi \co H \to H')$ be a morphism from a crossed module $\CM \co E \to H$ to a crossed module $\CM' \co E' \to H'$ (see Section~\ref{sect-crossed-modules-maps}). Every      $\CM$-category  $\cc$ determines a $\CM'$-category  $\phi_*(\cc)$   called the \emph{push-forward} of~$\cc$. By definition, $\phi_*(\cc)=\cc$ as  monoidal categories. The $E'$-Hom-grading of $\phi_*(\cc)$ is given for any objects $X,Y$ and $e' \in E$ by
$$
\Hom_{\phi_*(\cc)}^{e'}(X,Y)=\bigoplus_{e \in \psi^{-1}(e')} \Hom_\cc^e(X,Y).
$$
Homogenous objects of $\phi_*(\cc)$ are the same as those of $\cc$. The degree of a homogenous object $X$ of $\phi_*(\cc)$ is computed from the degree of $X$ in $\cc$ by
$$
|X|_{\phi_*(\cc)}=\varphi(|X|_\cc).
$$
Note that 1-subcategory of $\phi_*(\cc)$ has the same objects as $\cc$ and its morphisms are the linear combinations of homogenous morphisms of $\cc$ with degree in $\Ker(\psi)$.

\subsection{Pivotal $\CM$-categories}\label{sect-crossed-module-graded-pivotal}
Recall that a monoidal category $\cc=(\cc,\otimes,\un)$ is pivotal if it is endowed with a pivotal structure, that is, if  each
object $X$ of $\cc$ has a dual object~$X^*\in \cc$ and four morphisms
\begin{align*}
& \lev_X \co X^*\otimes X \to\un,  \qquad \lcoev_X\co \un  \to X \otimes X^*,\\
&   \rev_X \co X\otimes X^* \to\un, \qquad   \rcoev_X\co \un  \to X^* \otimes X,
\end{align*}
called left/right (co)evaluations, satisfying   several   conditions which say, in summary,  that the associated  left/right dual functors   coincide as   monoidal functors (see \cite{TVi5} for details). In particular, each morphism $\alpha\co X \to Y$ in~$\cc$ has a dual morphism $\alpha^*\co Y^* \to X^*$  computed by
\begin{align*}
\alpha^*&=(\lev_Y \otimes  \id_{X^*})(\id_{Y^*}  \otimes \alpha \otimes \id_{X^*})(\id_{Y^*}\otimes \lcoev_X)\\
 &= (\id_{X^*} \otimes \rev_Y)(\id_{X^*} \otimes \alpha \otimes \id_{Y^*})(\rcoev_X \otimes \id_{Y^*}).
\end{align*}
The \emph{left} and  \emph{right traces} of an endomorphism $\alpha$ of an object
$X$ of a pivotal category~$\cc$ are defined by
$$
\tr_l(\alpha)=\lev_X(\id_{\ldual{X}} \otimes \alpha) \rcoev_X  \quad {\text {and}}\quad \tr_r(\alpha)=  \rev_X( \alpha \otimes
\id_{\ldual{X}}) \lcoev_X .
$$
Both traces take values in   the commutative monoid   $\End_\cc(\un)$ and are symmetric: $\tr_l
(\beta\gamma)=\tr_l(\gamma\beta)$ for any morphisms $\beta\co X\to Y$, $\gamma\co Y\to X$ in $\cc$
and similarly for~$\tr_r$. Also $\tr_l( \alpha^*)=\tr_r(\alpha)$ and $\tr_r( \alpha^*)=\tr_l(\alpha)$ for any endomorphism in $\cc$.
The \emph{left} and  \emph{right dimensions} of an object $X$ of~$\cc$   are defined by
$$
\dim_l(X)=\tr_l(\id_X) \in\End_\cc(\un)  \quad \text{and} \quad \dim_r(X)=\tr_r(\id_X) \in\End_\cc(\un).
$$
Clearly, $\dim_l(X^*)=\dim_r(X)$ and $\dim_r(X^*)=\dim_l(X)$ for all $X$.

A $\CM$-category is \emph{pivotal} if it is endowed with a pivotal structure such that:
\begin{enumerate}
  \labela
  \item The dual $X^*$ of any homogenous object $X$ is homogenous of degree $$|X^*|=|X|^{-1}.$$
  \item All (co)evaluations morphisms are homogenous of degree $1 \in E$.
\end{enumerate}
The second axiom implies that the 1-subcategory of a pivotal $\CM$-category is itself pivotal.
In a pivotal $\CM$-category $\cc$, the dual $\alpha^* \co Y^* \to X^*$ of a nonzero homogenous morphism $\alpha \co X \to Y$ between homogenous objects is homogenous of degree
$$
|\alpha^*|=\lact{|Y|^{-1}}{\!|\alpha|}=\lact{|X|^{-1}}{\!|\alpha|}.
$$
This follows from the above expressions of the dual morphism. (Note that the latter equality may also be deduced from \eqref{eq-deg-morphism-objects} and the crossed module axioms.) Moreover, for any $e \in E$ and any homogeneous endomorphism $\alpha \in \End_\cc^e(X)$ of a homogenous object $X$ of $\cc$,
\begin{equation*}
\tr_l(\alpha)\in \End_\cc^{\bigl(\lact{|X|^{-1}}{\!\!e}\bigr)}(\un) \quad {\text {and}}\quad\tr_r(\alpha)\in \End_\cc^e(\un) .
\end{equation*}
This together with the fact that  objects of $\cc$ are $1$-direct sum of homogeneous objects imply that both traces of homogeneous endomorphisms of degree 1 take values in  $\End_\cc^1(\un)$. In particular, both dimensions of objects of $\cc$ take values in~$\End_\cc^1(\un)$.

\subsection{Spherical $\CM$-categories}\label{sect-crossed-module-graded-spherical}
A \emph{spherical $\CM$-category} is a pivotal $\CM$-category $\cc$ such that its 1-subcategory $\cc^1$ is spherical (as a pivotal category), that is, if the left and right traces of any homogenous endomorphism of degree 1  are equal.

In a spherical $\CM$-category, the \emph{trace} of a homogeneous endomorphism~$\alpha$ of degree~1 is defined by
$$
\tr(\alpha)=\tr_l(\alpha)=\tr_r(\alpha) \in \End_\cc^1(\un).
$$
Likewise, the \emph{dimension} of an object~$X$ is  defined by
$$
\dim(X)=\tr(\id_X)=\dim_l(X)=\dim_r(X)  \in \End_\cc^1(\un).
$$

\subsection{$\CM$-fusion categories}\label{sect-crossed-module-graded-fusion}
A $\CM$-category $\cc$ is \emph{pre-fusion} if it is $E$-semisimple (see Section~\ref{sect-Hom-graded-semisimple}) and the unit object $\un$ is  simple in the 1-subcategory $\cc^1$  of $\cc$, that is, $\End_\cc^1(\un)$ is free of rank 1. The map $\kk \to \End_\cc^1(\un)$, $k \mapsto k \, \id_\un$  is then a \kt algebra isomorphism
which we use  to identify $\End_\cc^1(\un)=\kk$. In a pre-fusion $\CM$-category, if an object $X$ is a $1$-direct sum of homogeneous objects of degree in $D \subset H$ and $i$ is a homogeneous simple object of $\cc^1$, then $N^{i,e}_X=0$ for all $e \in E$ such that $\CM(e)|i| \not \in D$. The left and right dimensions of a simple object in the 1-subcategory of a pivotal pre-fusion $\CM$-category are invertible (see for example \cite[Lemma 4.2]{TVi5}).

A \emph{$\CM$-representative set} for a pre-fusion $\CM$-category $\cc$ is a  set $I$  of simple  objects of the 1-subcategory $\cc^1$ such that $\un \in I$, all elements of $I$ are homogeneous, and every simple object of $\cc^1$ is 1-isomorphic to a unique element of~$I$. Note that $I$ is then representative set of simple objects of $\cc^1$ in the sense of Section~\ref{sect-linear-categories}.
Any   such  $\CM$-representative set $I$ splits  as a disjoint union $I=\amalg_{h\in H}\, I_h$  where $I_h$ is the set of all elements of $I$ of degree $h$.  Moreover, any homogeneous object of~$\cc$ of degree $h \in H$ is a 1-direct sum of a finite family of elements in $I_h$.

By a \emph{$\CM$-fusion category (over $\kk$)} we mean a pre-fusion $\CM$-category $\cc$ (over $\kk$) such that for any $h \in H$, the set of $1$-isomorphism classes of degree $h$ homogeneous simple objects of~$\cc^1$  is finite  and non-empty. In particular, any $\CM$-representative set $I=\amalg_{h\in H}\, I_h$ for a $\CM$-fusion category is such that $I_h$ is finite for all $h \in I$.

Let $\cc$ be a pivotal $\CM$-fusion category. Recall from Section~\ref{sect-group-graded-categories} that  the 1-subcategory $\cc^1$ of $\cc$  decomposes as $\cc^1=\bigoplus_{h \in H} \cc^1_h$. Then $\cc^1$ is a pivotal $H$-fusion category in the sense of \cite{TVi1}, and its neutral component $\cc^1_1$ is a pivotal fusion category in the usual sense.
The \emph{dimension} of $\cc_1^1$ (as a pivotal fusion category) is defined by
$$
\dim(\cc_1^1)=\sum_{i\in J} \dim_l(i)\dim_r(i) \in \End_\cc^1(\un)=\kk,
$$
where $J$ is any (finite) representative set   of   simple objects of   $\cc_1^1$.
Note that  if   $\kk $ is an algebraically closed field of characteristic zero, then $\dim(\cc_1^1)\neq 0$ (see \cite{ENO}). By \cite[Lemma 4.1]{TVi1}, if $I=\amalg_{h\in H}\, I_h$ is a $\CM$-representative set for $\cc$, then for all~$h \in H$,
\begin{equation}\label{eq-computation-dimension-neutral-component}
\sum_{i \in I_h} \dim_l( i)\dim_r( i) =\dim(\cc_1^1).
\end{equation}

\subsection{Remark}\label{rk-case-injective-G-fusion}
If $\CM\co E \to H$ is injective (for example when $E=1$), then any $\CM$-fusion category $\cc$, endowed with the $\Coker(\CM)$-grading defined in Section~\ref{sect-group-graded-categories}, is a $\Coker(\CM)$-fusion category (in the sense of \cite{TVi1}) whose neutral component has dimension $\dim(\cc_1^1)$.

\subsection{Example}\label{ex-linearized-crossed-module-fusion}
The  $\CM$-category $\kk\tg_\CM$ from Example~\ref{ex-linearized-crossed-module} has a canonical pivotal structure: the dual of an object $h \in H$ is $h^*=h^{-1}$ with (co)evaluation morphisms all given by the unit element $1 \in E$. Then $\kk\tg_\CM$ is a spherical $\CM$-fusion category with $\CM$-representative set $I=H$ and with $\dim((\kk\tg_\CM)^1_1)=1_\kk$.

\subsection{Example}\label{ex-graded-vector-graded-maps-fusion}
Let $E$ be a normal subgroup of a group $H$. The inclusion $E\hookrightarrow H$ is a crossed module with the conjugation action of $H$ on $E$. Denote by $(E,H)\mti\vect_\kk$ the category of $H$-graded free \kt modules of finite rank and $E$-graded \kt linear homomorphisms. This is an $(E\hookrightarrow H)$-category (as a full subcategory of the category~$\lmm{E,H}$ from Example~\ref{ex-graded-vector-graded-maps}).
It has  a canonical pivotal structure: the dual of an object $M=\bigoplus_{h \in H} M_h $ is
$$
M^*=\bigoplus_{h \in H} (M^*)_h \quad \text{where} \quad (M^*)_h=\Hom_\kk(M_{h^{-1}},\kk)
$$
with evaluations $\lev_M\co M^* \otimes M \to \un$ and $\rev_M \co M \otimes M^* \to \un$ given  for any $h,k \in H$, $m \in M_h$, and $f \in (M^\ast)_k$ by
$$
\lev_M(f \otimes_\kk m)=\rev_M(m \otimes_\kk f)=  \delta_{k,h^{-1}}  \,f(m).
$$
Then $(E,H)\mti\vect_\kk$ is a spherical $(E\hookrightarrow H)$-fusion category. For any $h \in H$, let~$\kk_h $ be the $H$-graded free \kt module   of rank 1  defined by $(\kk_h)_h=\kk$ and $(\kk_h)_x=0$ for $x\in H\setminus \{h\}$. Note that $\kk_h$ is homogeneous of degree $h$.
Then $\{\kk_h\}_{h \in H}$ is a  $(E\hookrightarrow H)$-representative set for $(E,H)\mti\vect_\kk$.
Clearly,  $\dim(\kk_h)= 1_\kk\in \End_\cc^1(\un)=\kk$ for all $h\in H$. In particular, $\dim\bigl(((E,H)\mti\vect_\kk)^1_1 \bigr)=1_\kk$. Also,  for all $g,h\in H$ and $e \in E$,
$$
\kk_g \otimes \kk_h \cong_1 \kk_{gh}, \quad \kk_h^* \cong_1 \kk_{h^{-1}},  \quad \text{and} \quad
\kk_g \cong_e \kk_h \Longleftrightarrow  \CM ( e)=hg^{-1}.
$$
Note that by Remark~\ref{rk-case-injective-G-fusion}, $(E,H)\mti\vect_\kk$ is  $(H/E)$-fusion in the sense of \cite{TVi1} with neutral component of dimension $1_\kk$.

\subsection{Push forward of fusion $\CM$-categories}\label{sect-push-forward-fusion}
Let $\phi=(\psi \co E \to E',\varphi \co H \to H')$ be a morphism from a crossed module $\CM \co E \to H$ to a crossed module $\CM' \co E' \to H'$ (see Section~\ref{sect-crossed-modules-maps}). Assume that $\psi$ and $\varphi$ are surjective, $\Ker(\psi)\cap \Ker(\CM)=1$, and $\Ker(\varphi)$ is finite. These conditions imply that the map $\CM$ induces an injective group homomorphism $\Ker(\psi)\hookrightarrow\Ker(\varphi)$, and so $\Ker(\psi)$ is finite and its order divides that of $\Ker(\varphi)$. Then the push-forward $\phi_*(\cc)$ of any $\CM$-fusion category $\cc$ (defined in Section~\ref{sect-push-forward}) is a $\CM'$-fusion category.
Also, if $\cc$ is pivotal or spherical, then so is~$\phi_*(\cc)$ (with the same pivotal structure) and
$$
\dim\bigr(\phi_*(\cc)^1_1\bigl)=\frac{\card\bigr(\Ker(\varphi)\bigl)}{\card\bigr(\Ker(\psi)\bigl)} \, \dim(\cc^1_1).
$$
Indeed, for any homogeneous object~$X$ of $\cc$, we have:
$$
\End_{\phi_*(\cc)}^1(X)=\hspace{-2em} \bigoplus_{e \in \Ker(\psi)\cap \Ker(\CM)=1}\hspace{-2em} \End_\cc^e(X)  = \End_\cc^1(X),
$$
and so an object is simple in $\phi_*(\cc)^1$ if and only if it is simple in $\cc^1$. Then the semisimplicity of $\phi_*(\cc)^1$ follows from that of $\cc^1$. Also, the $E$-semisimplicity of $\cc$ together with the surjectivity of $\psi$ imply that $\phi_*(\cc)$ is $E'$-semisimple and so a pre-fusion $\CM'$-category. Moreover, the facts that $\cc$ is a $\CM$-fusion and $\varphi$ is surjective with finite kernel imply that $\phi_*(\cc)$ is a $\CM'$-fusion category. Finally, the formula for the dimension of $\phi_*(\cc)^1_1$ is proved in Claim~\ref{claim-dim-push} below.

\subsection{Example}\label{ex-crossed-module-vector-spaces}
A crossed module $\CM \co E \to H$ comes with a left action of $H$ on $E$ and so induces the semidirect product
group  $H\ltimes E=(H \times E,\ast)$ where
$$
(x,e)\ast (y,f)=(xy,e\smallspace\lact{x}{\!f}).
$$
Then $E=1 \times E$ is a normal subgroup of $H\ltimes E$ and so the inclusion $\iota \co E \hookrightarrow H\ltimes E$ is a crossed module (with the conjugation action of $H\ltimes E$ on $E$). Consider spherical $\iota$-fusion category $(E,H\ltimes E)\mti\vect_\kk$ from Example~\ref{ex-graded-vector-graded-maps-fusion}.
The pair
$$
\phi=(\id_E \co E \to E,\varphi \co  H\ltimes E \to H) \quad \text{with} \quad \varphi(h,e)=\CM(e)h
$$
is a crossed module morphism from $\iota$ to~$\CM$. Therefore, by Section~\ref{sect-push-forward}, the push-forward $\phi_*((E,H\ltimes E)\mti\vect_\kk)$ is a $\CM$-category denoted $\CM \ti \vect_\kk$. More explicitly, the objects of $\CM \ti \vect_\kk$ are the
$(H\ltimes E)$-graded free \kt modules of finite rank. The degree of a homogeneous object $M=M_{(h,e)}$ is $|M|=\CM(e)h$. The morphisms of $\CM \ti \vect_\kk$ are the $E$-graded \kt linear homomorphisms (see Example~\ref{ex-graded-vector-graded-maps}).
Note that~$\id_E$ and~$\varphi$ are surjective,  $\Ker(\id_E)\cap \Ker(\CM)=1$, and there is a group isomorphism
$$
\Ker(\varphi)=\{(h,e)\in H \ltimes E \, | \, \CM(e)h=1\} \cong E, \quad (h,e) \mapsto e^{-1}.
$$
Consequently, by Section~\ref{sect-push-forward-fusion}, if $E$ is finite, then $\CM \ti \vect_\kk$ is a spherical $\CM$-fusion category and
$$
\dim\bigr((\CM \ti \vect_\kk)^1_1\bigl)= \card(E) 1_\kk.
$$

\subsection{Twisting $\CM$-categories by cocycles}\label{sect-twisting-fusion-cocycles}
By a \emph{3-cocycle} for the crossed module $\CM\co E \to H$ with values in an abelian group $A$, we mean a map $\omega\co H^3 \times E^3 \to A$ such that for all $(x,y,z,t) \in H^4$ and $(a,b,c,d,e,f) \in E^6$,
\begin{gather*}
\omega\bigl(x,y,z,a,b,c\bigr) \,  \omega\bigl(x,\CM(c)yz,t,ba\smallspace\lact{x}{\smallnegspace(}c^{-1}),f,d\bigr) \,\omega\bigl(y,z,t,c,d,e\bigr)
\\
= \omega\bigl(\CM(a)xy,z,t,b,f,e\bigr) \, \omega\bigl(x,y,\CM(e)zt,a,fba\smallspace\lact{xy}{\smallnegspace(}e^{-1})a^{-1}, dc\smallspace\lact{y}{\smallnegspace(}e^{-1})\bigr).
\end{gather*}
Note that the map $\widetilde{\omega} \co H^3 \to A$,  defined by $\widetilde{\omega}(x,y,z)=\omega(x,y,z,1,1,1)$, is then a (usual) 3-cocycle for the group $H$ with values in $A$.
A 3-cocycle $\omega$ for $\CM$ is \emph{normalized} if for all $x,y \in H$ and $e \in E$,
$$
\omega(1,x,y,1,e,1)=1_A=\omega(x,1,y,1,1,1).
$$
Normalized 3-cocycles for~$\CM$ with values in $A$ can be used to describe the third cohomology classes in $H^3(B\CM,A)$, where $B\CM$ is the classifying space of $\CM$ (see \cite{FP} for the unnormalized case).

Any normalized 3-cocycle $\omega$ for $\CM$ with values in the group $\kk^*$ can be used to twist the composition of morphisms, the monoidal product of morphisms, and the associativity constraints of a $\CM$-category $\cc$ (over $\kk$) to obtain a new $\CM$-category~$\cc^\omega$  (over $\kk$). More precisely, the objects, the homogeneous objects, the degree of homogeneous objects, the monoidal product of objects, the unit object~$\un$, the Hom-sets (as $E$-graded \kt modules), and the identity morphisms for $\cc^\omega$ are the same as those for $\cc$. The composition of morphisms in  $\cc^\omega$ is given, for all $\alpha \in \Hom^e_\cc(X,Y)$ and $\beta \in \Hom^f_\cc(Y,Z)$ with $X$ homogeneous, by
$$
\beta \circ^\omega \alpha = \omega\bigl(|X|,1,1,e,f,1\bigr)^{-1}  \beta\alpha \in \Hom^{fe}_\cc(X,Z).
$$
The monoidal product of morphisms in  $\cc^\omega$ is given, for all $\alpha \in \Hom^e_\cc(X,Y)$ and $\beta \in \Hom^f_\cc(Z,T)$ with $X,Y$ homogeneous, by
$$
\alpha \otimes^\omega  \beta = \dfrac{\omega\bigl(|X|,|Y|,1,e\lact{|X|}{\!f},1,f\bigr)}{\omega\bigl(|X|,1,\CM(f)|Y|,e,1,1\bigr)}\, \alpha \otimes \beta \in \Hom^{e\lact{|X|}{\!f}}_\cc(X \otimes Z, Y \otimes T).
$$
For any homogeneous objects $X,Y,Z$, the monoidal constraint
$$
(X\otimes Y)\otimes Z\cong X\otimes (Y\otimes Z)
$$
of $\cc^\omega$ is $\widetilde{\omega}(|X|,|Y|,|Z|)$ times that of $\cc$, where  $\widetilde{\omega} \co H^3 \to \kk^*$ is the  3-cocycle for~$H$ derived from $\omega$ as above. Since $\widetilde{\omega}$ is normalized (because $\omega$ is), the unitality constraints $X\otimes \un \cong X \cong X\otimes \un$ of $\cc^\omega$ are those of $\cc$.

Note that $(\cc^\omega)^1_1=\cc^1_1$ as \kt linear monoidal categories. Also, recall that the 1-subcategory of a $\CM$-category is $H$-graded (see Section~\ref{sect-group-graded-categories}). Then $(\cc^\omega)^1$ is the twisted $H$-graded category $(\cc^1)^{\widetilde{\omega}}$, that is, the $H$-graded category~$\cc^1$ with monoidal constraints twisted by the 3-cocycle $\widetilde{\omega}$.

If $\cc$ is pivotal, then so is $\cc^\omega$ with the same dual objects and with the following twisted (co)evaluations: for any homogeneous object $X$,
\begin{align*}
& \lev^\omega_X=  \lev_X,  && \lcoev^\omega_X=  \widetilde{\omega}\bigl(|X|,|X|^{-1},|X|\bigr)^{-1}  \lcoev_X,\\
& \rcoev^\omega_X=  \rcoev_X, && \rev^\omega_X=  \widetilde{\omega}\bigl(|X|,|X|^{-1},|X|\bigr) \, \rev_X.
\end{align*}
If $\cc$ is spherical, then so is $\cc^\omega$. Also, if $\cc$ is pre-fusion, then so is $\cc^\omega$ with the same $\CM$-representative sets. Finally, if~$\cc$ is $\CM$-fusion, then so is $\cc^\omega$.

\subsection{Penrose graphical calculus}\label{sect-graphical-calculus}
We will represent morphisms in a category $\cc$ by plane   diagrams to be read from the bottom to the top.
The  diagrams are made of   oriented arcs colored by objects of $\cc$  and of boxes colored by morphisms of~$\cc$.  The arcs connect the boxes and   have no mutual intersections or self-intersections.
The identity $\id_X$ of $X\in  \cc$, a morphism $f\co X \to Y$, and the composition of two morphisms $f\co X \to Y$ and $g\co Y \to Z$ are represented as follows:
\begin{center}
 \psfrag{X}[Bc][Bc]{\scalebox{.7}{$X$}}
 \psfrag{Y}[Bc][Bc]{\scalebox{.7}{$Y$}}
 \psfrag{h}[Bc][Bc]{\scalebox{.8}{$f$}}
 \psfrag{g}[Bc][Bc]{\scalebox{.8}{$g$}}
 \psfrag{Z}[Bc][Bc]{\scalebox{.7}{$Z$}}
 $\id_X=$ \rsdraw{.45}{.9}{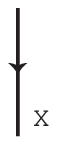}\,,\quad $f=$ \rsdraw{.45}{.9}{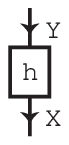} ,\quad \text{and} \quad $gf=$ \rsdraw{.45}{.9}{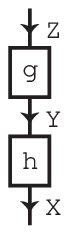}\,.
\end{center}
If $\cc$ is monoidal, then the monoidal product of two morphisms $f\co X \to Y$ and $g \co Z \to T$ is represented by juxtaposition:
\begin{center}
 \psfrag{X}[Bc][Bc]{\scalebox{.7}{$X$}}
 \psfrag{h}[Bc][Bc]{\scalebox{.8}{$f$}}
 \psfrag{Y}[Bc][Bc]{\scalebox{.7}{$Y$}}
 $f\otimes g=$ \rsdraw{.45}{.9}{morphism} \psfrag{X}[Bc][Bc]{\scalebox{.8}{$Z$}} \psfrag{g}[Bc][Bc]{\scalebox{.8}{$g$}}
 \psfrag{Y}[Bc][Bc]{\scalebox{.7}{$T$}}
 \rsdraw{.45}{.9}{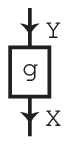}\,.
\end{center}
 Suppose that $\cc$ is      pivotal. By convention,   if an arc colored by $X\in \cc$ is oriented upwards,
then the corresponding object   in the source/target of  morphisms
is $X^*$. For example, $\id_{X^*}$  and a morphism $f\co X^* \otimes
Y \to U \otimes V^* \otimes W$  may be depicted as
\begin{center}
 $\id_{X^*}=$ \, \psfrag{X}[Bl][Bl]{\scalebox{.7}{$X$}}
\rsdraw{.45}{.9}{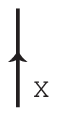} $=$  \,
\psfrag{X}[Bl][Bl]{\scalebox{.7}{$\ldual{X}$}}
\rsdraw{.45}{.9}{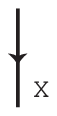}  \quad and \quad
\psfrag{X}[Bc][Bc]{\scalebox{.7}{$X$}}
\psfrag{h}[Bc][Bc]{\scalebox{.8}{$f$}}
\psfrag{Y}[Bc][Bc]{\scalebox{.7}{$Y$}}
\psfrag{U}[Bc][Bc]{\scalebox{.7}{$U$}}
\psfrag{V}[Bc][Bc]{\scalebox{.7}{$V$}}
\psfrag{W}[Bc][Bc]{\scalebox{.7}{$W$}} $f=$
\rsdraw{.45}{.9}{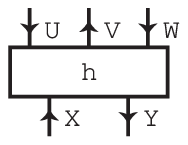} \,.
\end{center}
The duality morphisms   are depicted as follows:
\begin{center}
\psfrag{X}[Bc][Bc]{\scalebox{.7}{$X$}} $\lev_X=$ \rsdraw{.45}{.9}{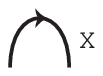}\,,\quad
 $\lcoev_X=$ \rsdraw{.45}{.9}{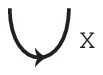}\,,\quad
$\rev_X=$ \rsdraw{.45}{.9}{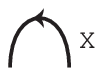}\,,\quad
\psfrag{C}[Bc][Bc]{\scalebox{.7}{$X$}} $\rcoev_X=$
\rsdraw{.45}{.9}{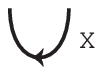}\,.
\end{center}
The dual of a morphism $f\co X \to Y$ and the traces of an endomorphism $\alpha\co X \to X$ can be depicted as
follows:
\begin{center}
 \psfrag{X}[Bc][Bc]{\scalebox{.7}{$X$}}
 \psfrag{h}[Bc][Bc]{\scalebox{.8}{$f$}}
 \psfrag{Y}[Bc][Bc]{\scalebox{.7}{$Y$}}
 \psfrag{g}[Bc][Bc]{\scalebox{.8}{$\alpha$}}
 $f^*=$ \rsdraw{.45}{.9}{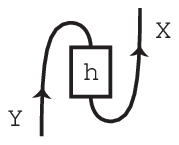}$=$ \rsdraw{.45}{.9}{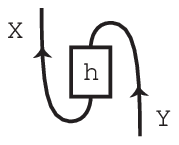}\quad \text{and} \quad
 $\tr_l(\alpha)=$ \rsdraw{.45}{.9}{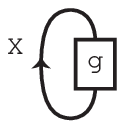}\,,\quad  $\tr_r(\alpha)=$ \rsdraw{.45}{.9}{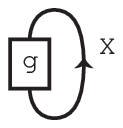}\,.
\end{center}
The morphisms represented by such diagrams are invariant under isotopies of the diagrams in $\RR^2$  keeping the bottom and   top endpoints fixed (see for example \cite{JS} or \cite[Theorem 2.6]{TVi5}).

\subsection{Enriched graphical calculus}\label{sect-enriched-graph-calc}
For a pivotal pre-fusion $\CM$-category $\cc$, we enrich  the Penrose graphical calculus   of the previous section. Let $X$ be an object of~$\cc$, $i$ be a simple object of the 1-subcategory $\cc^1$ of $\cc$, and $e \in E$.
Consider a (finite) formal sum of diagrams
\begin{equation}\label{SUMS}
 \psfrag{p}[Bc][Bc]{\scalebox{.9}{$p_\alpha$}}
 \psfrag{q}[Bc][Bc]{\scalebox{.9}{$q_\alpha$}}
 \psfrag{X}[Bl][Bl]{\scalebox{.9}{$X$}}
 \psfrag{i}[Bl][Bl]{\scalebox{.9}{$i$}}
 \sum_{\alpha \in \Lambda} \,\; \rsdraw{.45}{.9}{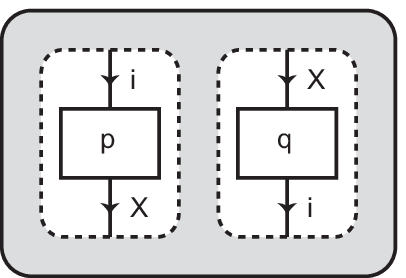}
\end{equation}
where $(p_\alpha\co X \to i, q_\alpha \co i \to X)_{\alpha \in \Lambda}$ is an $(i,e)$-partition of~$X$ (see Section~\ref{sect-multiplicity-partitions}) and the gray area represents a part of these  diagrams   independent of $\alpha\in \Lambda$ (and, in particular,   not involving    $p_\alpha, q_\alpha$).   By the Penrose graphical calculus and the \kt linearity of $\cc$,   the sum \eqref{SUMS} represents a morphism in $\cc$.
Since the  tensor
\begin{equation}\label{SUMS1}
\sum_{\alpha \in \Lambda} p_\alpha \otimes q_\alpha \in \Hom_\cc^{e^{-1}}\!(X, i)
\otimes_\kk \Hom_\cc^e(i,X)
\end{equation}
does not depend on the choice of the $(i,e)$-partition of $X$,
the morphism~\eqref{SUMS} in~$\cc$    also does not
depend on this choice. We therefore can eliminate the colors $p_\alpha, q_\alpha$ of  the two boxes, keeping in mind only   the order of the    boxes and  the fact that they jointly stand for the tensor~\eqref{SUMS1}.      We will graphically represent  this pair of  boxes by  two curvilinear boxes (a semi-disk  and a  compressed rectangle) labeled by $e$ and painted with the same
color which stand respectively for~$p_\alpha$ and~$q_\alpha$ where~$\alpha$ runs over~$\Lambda$ :
\begin{equation}\label{eq-def-extended-penrose}
 \psfrag{p}[Bc][Bc]{\scalebox{.9}{$p_\alpha$}}
 \psfrag{q}[Bc][Bc]{\scalebox{.9}{$q_\alpha$}}
 \psfrag{e}[Bc][Bc]{\scalebox{.9}{$e$}}
 \psfrag{X}[Bl][Bl]{\scalebox{.9}{$X$}}
 \psfrag{i}[Bl][Bl]{\scalebox{.9}{$i$}}
 \rsdraw{.45}{.9}{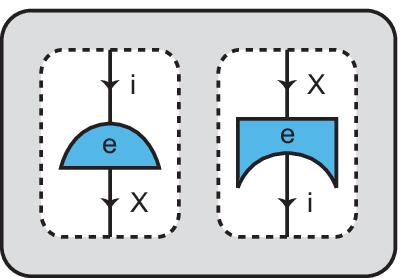} \, =\, \sum_{\alpha \in \Lambda} \,\; \rsdraw{.45}{.9}{tensor0} \, .
\end{equation}
The   gray  areas in the picture are the same as above.
We will also use similar  notation    obtained from \eqref{eq-def-extended-penrose}  by reorienting the $X$-colored arcs upward and replacing $(p_\alpha, q_\alpha)_{\alpha \in \Lambda}$ with an  $(i,e)$-partition of $X^*$,  or by reorienting the $i$-colored arcs upward and replacing $(p_\alpha, q_\alpha)_{\alpha \in \Lambda}$ with an  $(i^*,e)$-partition of $X$.
We   will allow  several   arcs to be attached to the bottom of the semi-disk
and to the top of the  compressed rectangle in \eqref{eq-def-extended-penrose}. The number of these arcs, their directions (up/down), and their colors should be the same.  For example,
\begin{equation*}
\psfrag{p}[Bc][Bc]{\scalebox{.9}{$p_\alpha$}}
\psfrag{q}[Bc][Bc]{\scalebox{.9}{$q_\alpha$}}
\psfrag{X}[Br][Br]{\scalebox{.9}{$X$}}
\psfrag{Y}[Bl][Bl]{\scalebox{.9}{$Y$}}
\psfrag{i}[Br][Br]{\scalebox{.9}{$i$}}
\psfrag{e}[Bc][Bc]{\scalebox{.9}{$e$}}
\rsdraw{.45}{.9}{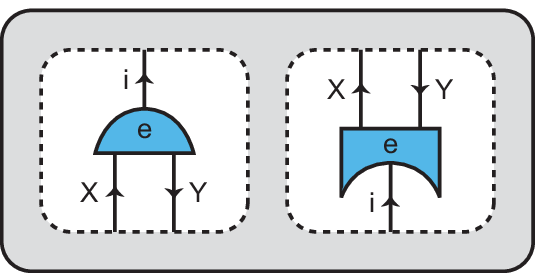} \, =\, \sum_{\alpha \in \Lambda} \,\; \rsdraw{.45}{.9}{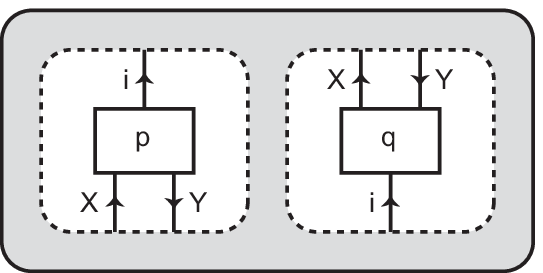} \,,
\end{equation*}
where $(p_\alpha, q_\alpha)_{\alpha \in \Lambda}$ is any $(i^*,e)$-partition of $X^*\otimes Y$.
We will allow  to erase $i$-colored arcs for $i=\un$. In particular,
\begin{equation*}
\psfrag{p}[Bc][Bc]{\scalebox{.9}{$p_\alpha$}}
\psfrag{q}[Bc][Bc]{\scalebox{.9}{$q_\alpha$}}
\psfrag{X}[Bl][Bl]{\scalebox{.9}{$X$}}
\psfrag{e}[Bc][Bc]{\scalebox{.9}{$e$}}
\rsdraw{.45}{.9}{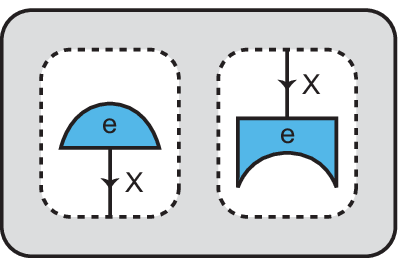} \, =\, \sum_{\alpha \in \Lambda} \,\; \rsdraw{.45}{.9}{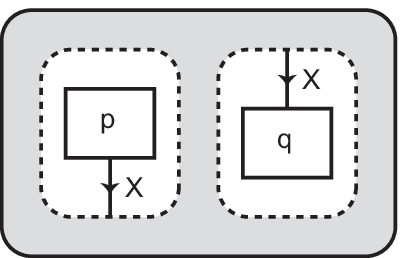} \,,
\end{equation*}
where $(p_\alpha, q_\alpha)_{\alpha \in \Lambda}$ is any $(\un,e)$-partition of $X$.
Furthermore,    in accordance with the isotopy invariance,
\begin{equation*}
\psfrag{e}[Bc][Bc]{\scalebox{.9}{$e$}}
\psfrag{a}[Br][Br]{\scalebox{.9}{$i$}}
\psfrag{R}[Br][Br]{\scalebox{.9}{$X$}}
\psfrag{X}[Bl][Bl]{\scalebox{.9}{$X$}}
\psfrag{i}[Bl][Bl]{\scalebox{.9}{$i$}} \rsdraw{.45}{.9}{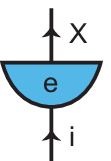} \qquad \text{will stand for} \qquad
\rsdraw{.45}{.9}{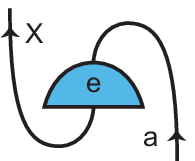} \;=\; \rsdraw{.45}{.9}{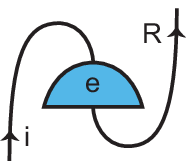}
\;\;\phantom{.}
\end{equation*}
and
\begin{equation*}
\psfrag{e}[Bc][Bc]{\scalebox{.9}{$e$}}
\psfrag{a}[Br][Br]{\scalebox{.9}{$i$}}
\psfrag{R}[Br][Br]{\scalebox{.9}{$X$}}
\psfrag{X}[Bl][Bl]{\scalebox{.9}{$X$}}
\psfrag{i}[Bl][Bl]{\scalebox{.9}{$i$}}
\rsdraw{.45}{.9}{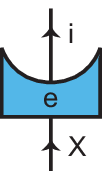} \qquad \text{will stand for} \qquad
\rsdraw{.45}{.9}{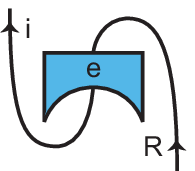} \;=\; \rsdraw{.45}{.9}{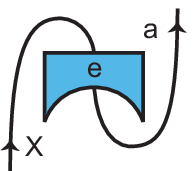}
 \;\;.
\end{equation*}
Similar  notation  will be applied when the $X$-colored  or $i$-colored   arcs are oriented downward and   when there are several arcs attached to the horizontal sides of the curvilinear boxes.

Every enriched diagram as above represents a morphism in~$\cc$ which is invariant under ambient isotopies (relative endpoints) of the diagram  in the plane.

We now state four properties of the above enriched graphical calculus associated with the pivotal pre-fusion $\CM$-category $\cc$. First, for any object $X$ of $\cc$, any simple object $i$ of $\cc^1$, and any $e \in E$,
\begin{equation}\label{eq-Nxi}
\psfrag{i}[Bl][Bl]{\scalebox{.9}{$i$}}
\psfrag{e}[Bc][Bc]{\scalebox{.9}{$e$}}
\psfrag{X}[Bl][Bl]{\scalebox{.9}{$X$}}
\rsdraw{.45}{.9}{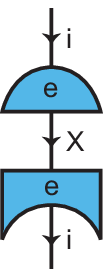}\;\, = \,N_X^{i,e}\;\; \rsdraw{.45}{.9}{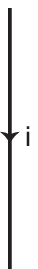}
\end{equation}
where $N^{i,e}_X$ is the degree $e$ multiplicity index of $i$ in $X$. This equality follows from the fact that for any  $(i,e)$-partition $(p_\alpha,q_\alpha)_{\alpha \in \Lambda}$   of $X$,  we have  $p_\alpha q_\alpha=\id_i$ for all $\alpha \in \Lambda$ and   $\card(\Lambda)=N_X^{i,e}$.  Second,  pick a $\CM$-representative set~$I$  for~$\cc$.
For any object $X$ of $\cc$ and $e \in E$,
\begin{equation}\label{eq-pre-fus-sum}
\psfrag{i}[Bl][Bl]{\scalebox{.9}{$i$}}
\psfrag{X}[Bl][Bl]{\scalebox{.9}{$X$}}
\psfrag{e}[Bc][Bc]{\scalebox{.9}{$e$}}
\sum_{i \in I} \;\, \rsdraw{.45}{.9}{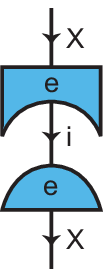}\;\, = \;\, \rsdraw{.45}{.9}{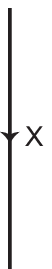}\;\,.
\end{equation}
This follows from the fact that the union of $(i,e)$-partitions of   $X$ over all $i\in I$ is  an $(I,e)$-partition of $X$.
Third, for any homogenous morphisms $\alpha \in \Hom_\cc^e(\un,X)$ and $\beta \in \Hom_\cc^{e^{-1}}\!(X,\un)$ with $e \in E$, we have
\begin{equation}\label{eq-pre-fus-un}
\psfrag{X}[Bl][Bl]{\scalebox{.9}{$X$}}
\psfrag{f}[Bc][Bc]{\scalebox{1}{$\alpha$}}
\psfrag{e}[Bc][Bc]{\scalebox{.9}{$e$}}
\rsdraw{.45}{.9}{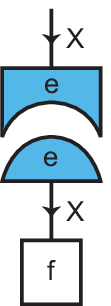}\;\, = \;\, \rsdraw{.45}{.9}{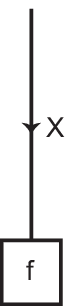}
\qquad \text{and} \qquad
\psfrag{f}[Bc][Bc]{\scalebox{1}{$\beta$}}
\rsdraw{.45}{.9}{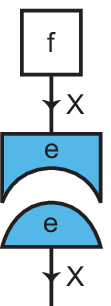}\;\, = \;\, \rsdraw{.45}{.9}{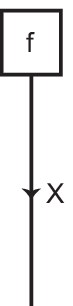} \;\,.
\end{equation}
These equalities follow from \eqref{eq-pre-fus-sum} and the fact $\Hom_\cc^1(\un,i)=0=\Hom_\cc^1(i,\un)$ for all $i\in I \setminus \{ \un\}$. Fourth, for any object $X$ of $\cc$ and $e \in E$,
\begin{equation}\label{eq-dim-mult-numbers}
\dim_l(X)=\sum_{i\in I} \dim_l(i) N_X^{i,e} \quad \text{and} \quad \dim_r(X)=\sum_{i\in I} \dim_r(i) N_X^{i,e}.
\end{equation}
Here, recall that $N^{i,e}_X=0$ for all but a finite number of $i\in I$. These equalities follow from Formulas \eqref{eq-Nxi} and \eqref{eq-pre-fus-sum} together with the isotopy invariance of the enriched graphical calculus:
\begin{gather*}
   \dim_{l}(X)=  \;
   \psfrag{i}[Bl][Bl]{\scalebox{.9}{$X$}}
\rsdraw{.45}{.9}{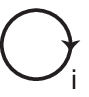}
\;=\;
\psfrag{e}[Bc][Bc]{\scalebox{.9}{$e$}}
\psfrag{i}[Bl][Bl]{\scalebox{.9}{$i$}}
\psfrag{X}[Bl][Bl]{\scalebox{.9}{$X$}}
\sum_{i \in I}\;  \rsdraw{.45}{.9}{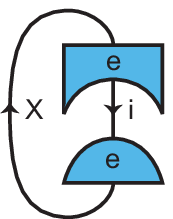}
\;=\;
\psfrag{e}[Bc][Bc]{\scalebox{.9}{$e$}}
\psfrag{i}[Bc][Bc]{\scalebox{.9}{$i$}}
\psfrag{X}[Bc][Bc]{\scalebox{.9}{$X$}}
\sum_{i \in I}  \; \rsdraw{.45}{.9}{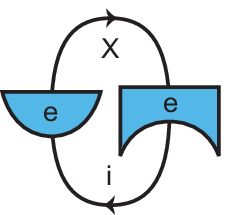}\\
\psfrag{e}[Bc][Bc]{\scalebox{.9}{$e$}}
\psfrag{i}[Bl][Bl]{\scalebox{.9}{$i$}}
\psfrag{X}[Bl][Bl]{\scalebox{.9}{$X$}}
=\;\sum_{i \in I}  \; \rsdraw{.45}{.9}{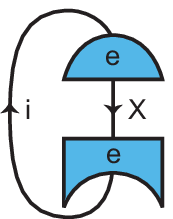}
\;=\;
\psfrag{i}[Bl][Bl]{\scalebox{.9}{$i$}}
\sum_{i \in I}\;   N^{i,e}_X\;\rsdraw{.45}{.9}{tensor-cor-3}
\;=\;
\sum_{i\in I} \dim_{l}(i) N_X^{i,e}
\end{gather*}
and similarly for the right trace.
As an application,  we prove the following claim.
\begin{lem}\label{lem-bubbleidentity}
Let $\cc$ be a pivotal $\CM$-fusion category and let $I=\amalg_{h\in H}\, I_h$ be  a  $\CM$\ti represen\-ta\-ti\-ve set for $\cc$. Then, for any homogeneous objects $X,Y$ of $\cc$, $g,h \in H$, and $e \in E$ such that $\CM(e)=|X|g|Y|h$, we have:
$$
\sum_{\substack{m\in I_g\\ n\in I_h}}\!\dim_l(m)\dim_l(n)
N_{X \otimes m \otimes Y \otimes n}^{\un,e}  =   \dim_r(X) \dim_r(Y)   \dim(\cc_1^1)
$$
and similarly with the subscripts $l,r$ exchanged.
\end{lem}
\begin{proof}
Observe first that for any object $Z$ of $\cc$,
$$
N^{\un,e}_{Z \otimes n}=N^{n^*\!,e}_Z.
$$
Indeed the free \kt modules $\Hom_\cc^e(\un,Z \otimes n)$ and $\Hom_\cc^e(n^*,Z)$ have the same rank since  the map
$\alpha \in \Hom_\cc^e(\un,Z \otimes n) \mapsto (\id_Z \otimes \rev_n)(\alpha \otimes \id_{n^*})  \in \Hom_\cc^e(n^*,Z)$ is a well-defined \kt linear isomorphism. Thus, using  the identity $\dim_l(n)=\dim_r(n^*)$, we obtain:
\begin{gather*}
\sum_{\substack{m\in I_g\\ n\in I_h}} \!\dim_l(m)\dim_l(n)
N_{X \otimes m \otimes Y \otimes n}^{\un,e}
=\!\sum_{\substack{m\in I_g\\ n\in I_h}} \! \dim_l(m) \dim_r(n^*)
N_{X \otimes m\otimes Y}^{n^*\!,e}\\
=\sum_{m\in I_g} \!\dim_l(m) \!\sum_{i\in I_{h^{-1}}}\!\!\dim_r(i) N_{X \otimes m\otimes Y}^{i,e}.
\end{gather*}
Now, for any $m \in I_g$, since $X \otimes m\otimes Y$ is a 1-direct sum of homogenous objects of degree $|X|g|Y|$, we have $N_{X \otimes m\otimes Y}^{i,e}=0$ whenever
$i \in I$ satisfies $\CM(e)|i| \neq |X|g|Y|$ or equivalently $|i| \neq h^{-1}$ (see  Section~\ref{sect-crossed-module-graded-fusion}) and so Formula \eqref{eq-dim-mult-numbers} gives that
$$
\sum_{i\in I_{h^{-1}}}\!\!\dim_r(i) N_{X \otimes m\otimes Y}^{i,e} =
\sum_{i\in I}\! \dim_r(i) N_{X \otimes m\otimes Y}^{i,e} = \dim_r(X \otimes m\otimes Y).
$$
Then, using  the $\otimes$-multiplicativity of the dimensions (which follows from the fact that the unit object $\un$ is simple in~$\cc^1$, see Sections~2.6.2 and~4.2.2 of \cite{TVi5})  and Formula~\eqref{eq-computation-dimension-neutral-component}, we obtain:
\begin{gather*}
\sum_{\substack{m\in I_g\\ n\in I_h}} \!\dim_l(m)\dim_l(n)
N_{X \otimes m \otimes Y \otimes n}^{\un,e}
=\sum_{m\in I_g} \! \dim_l(m) \dim_r(X \otimes m\otimes Y)\\
= \dim_r(X)  \dim_r(Y) \sum_{m\in I_g} \dim_l(m) \dim_r(m)=
 \dim_r(X)  \dim_r(Y) \dim(\cc_1^1).
\end{gather*}
A similar proof works  when   the subscripts $l,r$ are exchanged.
\end{proof}

In the next lemma, each equality has gray areas   which are supposed to be the same on both   sides.

\begin{lem}\label{lem-graphical-calculus-prefusion-Xi}
For any  object $X$ of $\cc$, any homogeneous simple object $i$ of $\cc^1$, and any $e \in E$, we have:
$$
   \psfrag{e}[Bc][Bc]{\scalebox{.9}{$\lact{|i|}{e}$}}
 \psfrag{i}[Bl][Bl]{\scalebox{.9}{$i$}}
 \psfrag{X}[Bl][Bl]{\scalebox{.9}{$X$}}
 \rsdraw{.45}{.9}{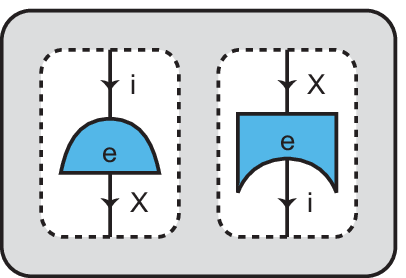}  \,
 \psfrag{e}[Bc][Bc]{\scalebox{.9}{$e$}}
\psfrag{i}[Br][Br]{\scalebox{.9}{$i$}}
\psfrag{j}[Bl][Bl]{\scalebox{.9}{$i$}}
\psfrag{Y}[Br][Br]{\scalebox{.9}{$X$}}
\psfrag{X}[Bl][Bl]{\scalebox{.9}{$X$}}
= \, \dim_l(i) \;\; \rsdraw{.45}{.9}{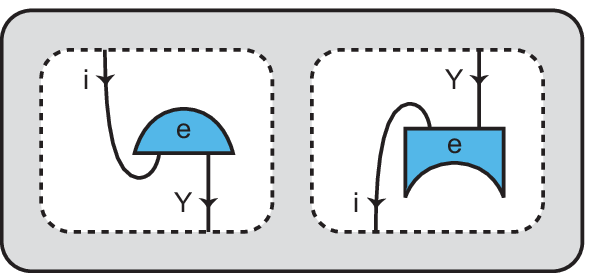}\;\,.
$$
\end{lem}
\begin{proof}
Set $a=\lact{|i|}{e}$. Pick  a   $(\un,e)$-partition $(p_\alpha,q_\alpha)_{\alpha \in \Lambda}$  of~$i^* \otimes X$. For each $\alpha  \in \Lambda$, set
$$
\psfrag{p}[Bc][Bc]{\scalebox{1}{$p_\alpha$}}
\psfrag{q}[Bc][Bc]{\scalebox{1}{$q_\alpha$}}
\psfrag{X}[Br][Br]{\scalebox{.9}{$X$}}
\psfrag{i}[Br][Br]{\scalebox{.9}{$i$}}
P_\alpha=\dim_l(i) \;\rsdraw{.5}{.9}{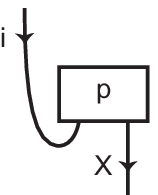}   \, \in\Hom_\cc^{\lact{|i|}{e^{-1}}}\!(X,i)=\Hom_\cc^{a^{-1}}\!(X,i)
$$
and
$$
\psfrag{p}[Bc][Bc]{\scalebox{1}{$p_\alpha$}}
\psfrag{q}[Bc][Bc]{\scalebox{1}{$q_\alpha$}}
\psfrag{X}[Br][Br]{\scalebox{.9}{$X$}}
\psfrag{i}[Br][Br]{\scalebox{.9}{$i$}}
Q_\alpha=\rsdraw{.5}{.9}{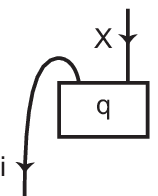}\,\in\Hom_\cc^{\lact{|i|}{e}}( i,X )=\Hom_\cc^{a}(i,X).
$$
Since  $(p_\alpha)_{\alpha \in \Lambda}$ is a basis of the \kt module $\Hom_\cc^{e^{-1}}\!(i^* \otimes X,\un)$  and $\dim_l(i) $ is invertible in $\kk$,   the family $(P_\alpha)_{\alpha \in \Lambda}$ is  a basis  of the \kt module   $\Hom_\cc^{a^{-1}}\!(X,i)$.  Since     $(q_\alpha)_{\alpha \in \Lambda}$ is a basis of the \kt module $\Hom_\cc^e(\un,i^* \otimes X)$,  the family $(Q_\alpha)_{\alpha \in \Lambda}$ is  a basis  of the \kt module   $\Hom_\cc^{a}(i,X)$. Since the object  $i$ is simple in $\cc^1$, for any  $\alpha,\beta \in \Lambda$, we have $P_\alpha Q_\beta \in \End_\cc^1(i)=\kk\,\id_i$ and so $P_\alpha Q_\beta =\lambda_{\alpha,\beta}\, \id_i$ where
$$
\lambda_{\alpha,\beta}\;\overset{(i)}{=}\;\frac{\tr_l(P_\alpha Q_\beta)}{\dim_l(i)}
\;\overset{(ii)}{=} \;\;
\psfrag{p}[Bc][Bc]{\scalebox{1}{$p_\alpha$}}
\psfrag{q}[Bc][Bc]{\scalebox{1}{$q_\beta$}}
\psfrag{X}[Br][Br]{\scalebox{.9}{$X$}}
\psfrag{i}[Bl][Bl]{\scalebox{.9}{$i$}}
\psfrag{i}[Br][Br]{\scalebox{.9}{$i$}}
\psfrag{X}[Bl][Bl]{\scalebox{.9}{$X$}}
\rsdraw{.45}{.9}{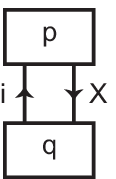} \;\; \overset{(iii)}{=} \; \delta_{\alpha,\beta}.
$$
Here $(i)$ follows from the invertibility of $\dim_l(i)$, $(ii)$ from the definitions of $P_\alpha$, $Q_\beta$ and the isotopy invariance of the graphical calculus, and $(iii)$ from the equality $p_\alpha q_\beta= \delta_{\alpha,\beta}$. Thus $(P_\alpha,Q_\alpha)_{\alpha \in \Lambda}$ is an $(i,a)$-partition of $X$. This directly implies the lemma.
\end{proof}

\section{Multiplicity modules and \texorpdfstring{$\CM$}{X}-graphs}\label{sec-multimodulesandgraphs}

Throughout this section, $\CM \co E \to H$ is a crossed module and $\cc$ is a pivotal $\CM$\ti category.
We first introduce $\cc$-colored $\CM$-cyclic sets and their multiplicity modules. Then we define an isotopy invariant of $\cc$-colored $\CM$-graphs.

\subsection{Multiplicity modules of colored $\CM$-cyclic sets}\label{sect-muliplicity-modules}
A \emph{$\CM$-cyclic set} is a totally cyclically ordered finite set $S$ endowed with maps $\alpha\co S \to H$, $\beta\co S \to E$, and $\varepsilon\co S \to \{+, - \}$ such that for any $s \in S$,
\begin{equation}\label{eq-xi-graph-1}
  \CM(\beta(s))=\alpha(s_1)^{\varepsilon(s_1)} \cdots \alpha(s_n)^{\varepsilon(s_n)}
\end{equation}
where $n$ is the number of elements of $S$ and $s =s_1<s_2< \cdots <s_n$ are the elements of $S$ ordered in the given cyclic order starting from $s$, and
\begin{equation}\label{eq-xi-graph-2}
  \beta(\suc(s))=\lact{\left(\alpha(s)^{-\varepsilon(s)}\right)}{\!\beta(s)}
\end{equation}
where  $\suc(s)$ is the successor of $s$ in the given cyclic order. Note that \eqref{eq-xi-graph-2} implies that the map $\beta$ is fully determined by its value on one element of $S$, and that if~\eqref{eq-xi-graph-1} holds for one element of $S$, then it holds for all elements of $S$.

A $\CM$-cyclic set~$S$ is \emph{$\cc$-colored} if it is endowed with a map $c \co S \to \Ob(\cc)$ such that for any $s \in S$, the object $c(s)$ is homogeneous of degree $|c(s)|=\alpha(s)$. For shortness, $\cc$-colored $\CM$-cyclic sets are also called \emph{$\CM$-cyclic $\cc$-sets}.

Each $\CM$-cyclic $\cc$-set $S=(S,\alpha, \beta, \varepsilon,c)$ determines a   \kt module   $H(S)$     called the \emph{multiplicity module} of~$S$ and defined as follows. For   $s\in S$, set
$$
H_s  =H_s(S) =\Hom_\cc^{\beta(s)} (\un, c(s_1)^{\varepsilon(s_1)} \otimes \cdots \otimes
c(s_n)^{\varepsilon(s_n)}  ),
$$
where  $s =s_1<s_2< \cdots <s_n$ are the elements of $S$ ordered in the given cyclic order starting from $s$. Here  $X^+=X$ and $X^-=X^*$ for any object $X$ of $\cc$.  If $t\in S\setminus \{s\}$, then $t=s_k$
for some    $k\in\{2,\dots,n\}$.  Set
$$
[st)  = c(s_1)^{\varepsilon(s_1)} \otimes \cdots \otimes c(s_{k-1})^{\varepsilon(s_{k-1})}
\quad \text{and} \quad
[ts)  =c(s_k)^{\varepsilon(s_k)} \otimes \cdots \otimes c(s_n)^{\varepsilon(s_n)}.
$$
Clearly,
$$
H_s=\Hom_\cc^{\beta(s)} (\un, [st) \otimes [ts)) \quad \text{and} \quad  H_t=\Hom_\cc^{\beta(t)} (\un, [ts)\otimes [st)).
$$
Set $p_{s,s}=\id_{H_s}\co {H_s}\to {H_s}$. For distinct $s,t \in S$, define $p_{s,t}\co H_s\to H_t$ by
\begin{equation}\label{eq-def-permutation}
v \in H_s \mapsto p_{s,t}(v)=\!\!
\psfrag{X}[Bl][Bl]{\scalebox{.8}{$[st)$}}
\psfrag{e}[Bc][Bc]{\scalebox{1}{$v$}}
\psfrag{Y}[Br][Br]{\scalebox{.8}{$[ts)$}}
\rsdraw{.45}{.9}{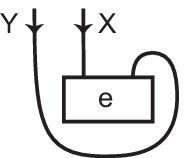} \;=\;\rsdraw{.45}{.9}{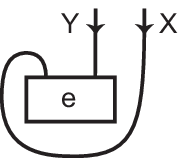}\!\!\!\in H_t.
\end{equation}
It follows from \eqref{eq-xi-graph-1}, \eqref{eq-xi-graph-2}, and the axioms of a pivotal $\CM$-category that $p_{s,t}$ is a well-defined \kt linear isomorphism and   $p_{s,t} \, p_{r,s}=p_{r,t}$ for all $r,s,t \in S$.
Thus the family $(\{H_s\}_{s \in S},\{p_{s,t}\}_{s,t \in S})$ is a projective system of \kt modules and \kt linear isomorphisms.
The    multiplicity module $H(S)$ is the projective limit    of this system:
$$
H(S)=\underleftarrow{\lim} \, H_s \, .
$$
The module $H(S)$   depends only on $S=(S,\alpha, \beta, \varepsilon,c)$ and is equipped with  a family of    \kt linear isomorphisms $\{\tau_s\co H(S) \to H_s \}_{s\in S}$ such that $p_{s,t} \, \tau_s= \tau_t$ for all $s,t \in S$.    We call~$\tau_s$   the \emph{cone isomorphism} and    the family $ \{\tau_s \}_{s\in S}$ the \emph{universal cone}.
Recall that $H(S)$ may be realized as the following submodule of the \kt module $\prod_{s \in S} H_s$:
$$
H(S)=\Bigl\{(x_s)_{s \in S} \in \prod_{s \in S} H_s \; \Big | \; p_{s,t}(x_s)=x_t \; \text{for all $s,t \in S$}\Bigr\}.
$$
The cone isomorphism $\tau_s\co H(S) \to H_s$ associated with this realization is given by $\tau_s((x_t)_{t \in S})=x_s$ and its inverse is computed by $\tau_s^{-1}(v)=(p_{s,t}(v))_{t \in S}$.

Note that if   the category   $\cc$ is pre-fusion (see Section~\ref{sect-crossed-module-graded-fusion}), then the \kt module  $H_s$ is free of finite rank for all $s\in S$, and thus so is  the   multiplicity   module $H(S)$.

An isomorphism   of $\CM$-cyclic $\cc$-sets $\phi\co S \to  S' $ is  a bijection which preserves the cyclic order and commutes with the maps to $H$, $E$, $\{+, - \}$, and $\cc$. Such  a   map  $\phi$ induces    a  \kt isomorphism  $H(\phi) \co H(S ) \to H(S' )$ in the obvious way.

The \emph{dual} of  a $\CM$-cyclic $\cc$-set  $S=(S,\alpha, \beta, \varepsilon,c)$ is the $\CM$-cyclic $\cc$-set
$$
S^\opp=(S^\opp,\alpha, \beta^\opp, -\varepsilon,c)
$$
where~$S^\opp$ is~$S$ with opposite cyclic order and, for any $s \in S$,  $ \beta^\opp(s)=\beta(\suc(s))^{-1}$. Here, $\suc(s)$ denotes the successor of $s$ in the cyclic order of~$S$.
If $\cc$ is spherical pre-fusion, then the \kt bilinear pairings
$$
\left\{\omega_{S,s}\co H_{\pred(s)}(S^\opp) \otimes_\kk H_s(S) \to \End_\cc^1(\un)=\kk, \quad u \otimes_\kk v \mapsto
  \psfrag{u}[Bc][Bc]{\scalebox{.9}{$u$}}
  \psfrag{v}[Bc][Bc]{\scalebox{.9}{$v$}}
  \rsdraw{.35}{.9}{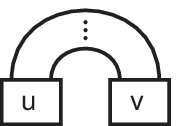}\,\right\}_{\!\!s \in S},
$$
where $\pred(s)$ denotes the predecessor in the cyclic order of~$S$,
are symmetric, non-degenerate, and compatible with the permutation maps \eqref{eq-def-permutation}, and so induce
a non-degenerate \kt bilinear  pairing $\omega_S \co H(S^\opp ) \otimes H(S) \to \kk$, where~$\otimes$ is the unordered tensor product  of \kt modules. In this case, the vector
$$
\ast_S=\Omega_S(1_\kk)\in  H(S) \otimes H(S^\opp ),
$$
where $\Omega_S  \co \kk \to H(S) \otimes H(S^\opp )$ is the inverse pairing of $\omega_S$, is called the \emph{contraction vector} of $S$. Note that it is induced from the cone isomorphisms and the vectors
$$
\left\{\,\ast_{S,s}=\Omega_{S,s}(1_\kk) \in H_s(S) \otimes_\kk H_{\pred(s)}(S^\opp) \,\right\}_{s \in S},
$$
where $\Omega_{S,s}  \co \kk \to H_s(S) \otimes_\kk H_{\pred(s)}(S^\opp)$ is the inverse pairing of $\omega_{S,s}$. A computation similar to \cite[Lemma 4.8]{TVi5} shows that
\begin{equation}\label{eq-compute-inverse-s}
\ast_{S,s}=\;\,
\psfrag{e}[Bc][Bc]{\scalebox{.9}{$e$}}
\rsdraw{.35}{.9}{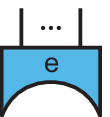} \;\, \otimes_\kk \;\, \rsdraw{.35}{.9}{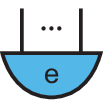} \;,
\end{equation}
where $e=\beta(s) \in E$ and the arcs are colored and oriented according to the maps $c \co S \to \Ob(\cc)$ and  $\varepsilon\co S \to \{+, - \}$.

\subsection{Graphs}\label{sect-graphs-defs}
By a  \emph{graph} we mean a topological space~$G$    obtained from   a finite number  of disjoint copies of the closed interval $[0,1]$   by identification of certain endpoints. The images of the copies of $[0,1]$ in~$G$ are called   \emph{edges} of~$G$. The   endpoints
of the edges of~$G$  (that is, the images of   $0,1\in [0,1]$)   are called \emph{vertices}  of~$G$. Each edge of~$G$
connects two (possibly, coinciding)  vertices, and  each vertex  of~$G$ is incident to at least one edge. By \emph{half-edges} of~$G$, we mean  the images of the closed intervals $[0,1/2]\subset [0,1]$ and $[1/2,1] \subset [0,1]$   in~$G$.   The number of half-edges of~$G$ incident to a vertex~$v$ of~$G$ is    greater than or equal to $1$ and is  called the \emph{valence} of~$v$.

A graph is  \emph{oriented} if all its edges are oriented.  The   empty set is viewed as an oriented graph with  no vertices and no edges.

\subsection{$\CM$-graphs}\label{sect-Xi-graphs}
Let $\Sigma$ be an oriented surface. By a \emph{graph} in~$\Sigma$, we mean a  graph embedded in~$\Sigma$. A vertex $v$ of a graph~$G$ in $\Sigma$ determines a totally cyclically ordered  set~$G_v$ consisting of the   half-edges of~$G$ incident to~$v$ with cyclic order induced by the opposite orientation of~$\Sigma$.  If $G$ is oriented, then  we  have a map $\varepsilon_v \co G_v \to \{+,-\}$ assigning $+$ to the half-edges oriented towards $v$ and $-$ to the half-edges oriented away from $v$.

A \emph{$\CM$-graph} in~$\Sigma$ is an oriented graph $G$ in~$\Sigma$ whose every edge is labeled with an element of $H$, called \emph{$H$-label}, and every half-edge is labeled with an element of $E$, called \emph{$E$-label}, in such a way that, for every  vertex $v$ of $G$, the cyclically ordered set~$G_v$ together with the maps $\alpha_v \co G_v \to H$, $\beta_v \co G_v \to E$ induced by the labels and the map $\varepsilon_v \co G_v \to \{+,-\}$ defined above is a $\CM$-cyclic set. (Here, the $H$-label of a half-edge is the $H$-label of the edge containing it.) Explicitly, this means that for any vertex $v$ of $G$ and any half-edge $s_1$ of $G$ incident to $v$,
$$
\CM(\beta_v(s_1))=\alpha_v(s_1)^{\varepsilon_v(s_1)} \cdots \alpha_v(s_n)^{\varepsilon_v(s_n)}
\quad \text{and} \quad  \beta_v(s_2)=\lact{\left(\alpha_v(s_1)^{-\varepsilon_v(s_1)}\right)}{\!\beta_v(s_1)},
$$
where $s_1<s_2< \dots < s_n$ are the half-edges of $G$ incident to $v$ ordered by the opposite orientation of~$\Sigma$
starting from $s_1$:
\begin{equation*}
 \psfrag{h}[Bc][Bc]{\scalebox{.9}{$s_1$}}
 \psfrag{k}[Br][Br]{\scalebox{.9}{$s_2$}}
 \psfrag{s}[Br][Br]{\scalebox{.9}{$s_n$}}
 \psfrag{v}[Br][Br]{\scalebox{.8}{$v$}}
 \rsdraw{.45}{.9}{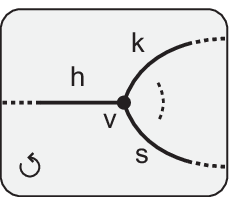}\;.
\end{equation*}

When depicting a $\CM$-graph $G$ in $\Sigma$, we draw the $H$-label of an edge next to it. By
\eqref{eq-xi-graph-2}, the $E$-labels of the half-edges incident to a vertex $v$ are fully determined by the $E$-label of one half-edge incident to~$v$. Now a dot next to~$v$ together with the orientation of $\Sigma$ determine a half-edge incident to $v$: it is the first half-edge encountered while traversing from the dot a small loop negatively encircling $v$. Thus we draw next to each vertex of $G$ a dot and  the $E$-label of the half-edge determined by this dot. For example, consider the following trivalent vertex $v$ of a $\CM$-graph $G$ in an oriented surface:
\begin{equation*}
 \psfrag{h}[Bc][Bc]{\scalebox{.9}{$h$}}
 \psfrag{k}[Bc][Bc]{\scalebox{.9}{$k$}}
 \psfrag{l}[Bc][Bc]{\scalebox{.9}{$\ell$}}
 \psfrag{e}[Bl][Bl]{\scalebox{.8}{$e$}}
 \rsdraw{.45}{.9}{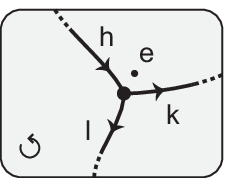}\;.
\end{equation*}
Here $h,k,\ell$ are the $H$-labels of the edges incident to~$v$, and $e$ is the $E$-label of the half-edge determined by the dot, that is, of the half-edge labeled by $k$. Condition~\eqref{eq-xi-graph-1} reduces to $\CM(e)=k^{-1}\ell^{-1}h$. By \eqref{eq-xi-graph-2}, the $E$-labels of the half-edges labeled by $\ell$ and $h$ are $\elact{k}{e}$ and $\elact{\ell k}{e}$, respectively.

\subsection{Grade of planar $\CM$-graphs}\label{sect-Xi-graphs-degree}
We always orient the plane~$\R^2$ counterclockwise. Let $G$ be a $\CM$-graph in~$\R^2$. Pick a point $m$  in the unbounded component of $\RR^2 \setminus G$.
For each vertex $v$ of $G$, pick a small loop $\ell_v$  negatively (i.e. clockwisely)  encircling~$v$ and an arc $\gamma_v$ in $\RR^2$
starting at $m$, ending in a point of $\ell_v \setminus G$, avoiding the vertices of $G$, and intersecting the edges of $G$ transversally. We require that the arcs $\{\gamma_v\}_v$  intersect each other only at their initial point $m$.
Denote by $e_v \in E$ the $E$-label of the first encountered half-edge incident to $v$ while traversing $\ell_v$ from the endpoint of~$\gamma_v$.
We assign an element~$h_v \in H$ to each arc~$\gamma_v$ as follows. Start with $h_v=1_H$. Go through the arc $\gamma_v$ starting from~$m$. Each time~$\gamma_v$ intersects an edge $s$ of~$G$ at some point~$c$, replace $h_v$ by $h_vx^\varepsilon$, where $\varepsilon=+$ if $(d_cs,d_c\gamma_v)$ is a positively oriented basis of $\RR^2$, $\varepsilon=-$ otherwise, and $x$ is the $H$-label of the edge $s$. Denote by $v_1, \dots,v_k$ the vertices of~$G$
so that $\gamma_{v_1},\dots, \gamma_{v_k}$ are the arcs successively encountered while traversing a small loop negatively encircling $m$ (starting from any point on that loop).
We define the \emph{grade} of $G$  by
$$
|G|=\lact{(h_{v_1})}{e_{v_1}} \cdots \lact{(h_{v_k})}{e_{v_k}} \in E.
$$
For example, consider the following $\CM$-graph in $\RR^2$:
$$
\psfrag{a}[Br][Br]{\scalebox{.9}{$e$}}
\psfrag{w}[Bc][Bc]{\scalebox{.9}{$f$}}
\psfrag{c}[Bl][Bl]{\scalebox{.9}{$g$}}
\psfrag{k}[Bl][Bl]{\scalebox{.9}{$k$}}
\psfrag{x}[Bc][Bc]{\scalebox{.9}{$x$}}
\psfrag{y}[Br][Br]{\scalebox{.9}{$y$}}
\psfrag{r}[Br][Br]{\scalebox{.9}{$r$}}
\psfrag{s}[Bl][Bl]{\scalebox{.9}{$s$}}
\psfrag{z}[Bl][Bl]{\scalebox{.9}{$z$}}
\psfrag{t}[Bl][Bl]{\scalebox{.9}{$t$}}
G=\,\rsdraw{.45}{.9}{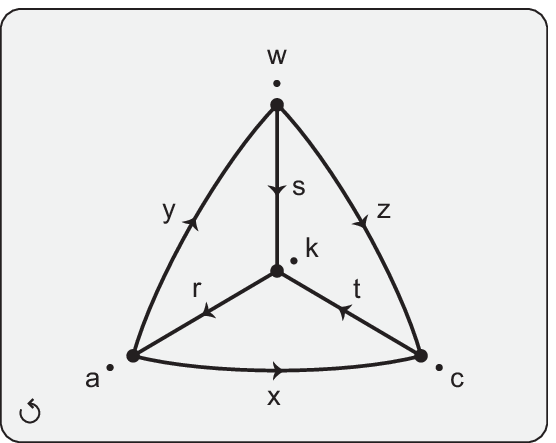} \, \;.
$$
Here $x,y,z,r,s,t$ are the $H$-labels of the edges and $e,f,g,k$ are the $E$-labels of the half-edges determined by the dots (with the graphical convention of Section~\ref{sect-Xi-graphs}).
The conditions for $G$ to be a $\CM$-graph reduce to
$$
\CM(e)=y^{-1}rx^{-1}, \quad \CM(f)=z^{-1}s^{-1}y, \quad \CM(g)=xt^{-1}z, \quad \CM(k)=tr^{-1}s.
$$
Pick the following point $m$, loops $\{\ell_i\}_{1 \leq i \leq 4}$, and arcs  $\{\gamma_i\}_{1 \leq i \leq 4}$:
$$
\psfrag{D}[Br][Br]{\scalebox{.9}{\color{mylightblue}{$\partial D$}}}
\psfrag{m}[Bc][Bc]{\scalebox{.9}{\color{myblue}{$m$}}}
\psfrag{a}[Br][Br]{\scalebox{.9}{\color{myred}{$\ell_1$}}}
\psfrag{d}[Bc][Bc]{\scalebox{.9}{\color{myred}{$\ell_2$}}}
\psfrag{c}[Bl][Bl]{\scalebox{.9}{\color{myred}{$\ell_4$}}}
\psfrag{k}[Bl][Bl]{\scalebox{.9}{\color{myred}{$\ell_3$}}}
\psfrag{u}[Bl][Bl]{\scalebox{.85}{\color{myblue}{$\gamma_1$}}}
\psfrag{v}[Bl][Bl]{\scalebox{.85}{\color{myblue}{$\gamma_2$}}}
\psfrag{n}[Bl][Bl]{\scalebox{.85}{\color{myblue}{$\gamma_3$}}}
\psfrag{o}[Bl][Bl]{\scalebox{.85}{\color{myblue}{$\gamma_4$}}}
\rsdraw{.45}{.9}{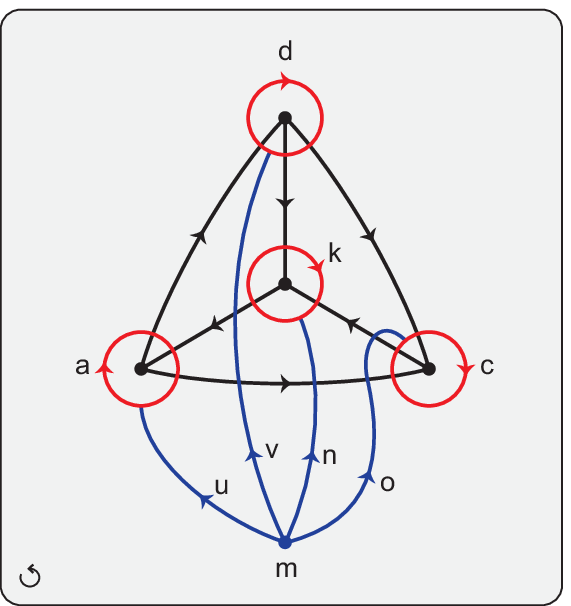} \, \;.
$$
For any $1 \leq i \leq 4$,  denote by  $e_i$ the element of $E$ associated with the vertex encircled by~$\ell_i$ and by $h_i$ the element of $H$ associated to the arc $\gamma_i$ as described above. Using \eqref{eq-xi-graph-2}, we obtain:
\begin{align*}
& e_1=e, && e_2=\lact{y}{\!f}, && e_3=\elact{t^{-1}}{k}, && e_4=\elact{z}{g}, \\
& h_1=1, && h_2=xr^{-1}, && h_3=x, && h_4=xt^{-1}.
\end{align*}
Then the grade of $G$ is
$$
|G|=(\lact{1}{e}) (\lact{xr^{-1}y}{\!f}) (\lact{xt^{-1}}{\!k}) (\lact{xt^{-1}z}{g})=e(\lact{\CM(e^{-1})}{\!f}) (\lact{\CM(g)z^{-1}}{\!k})(\elact{\CM(g)}{g}) =feg (\lact{z^{-1}}{\!k})
$$
where the last equality follows from \eqref{eq-Peiffer}. Note that this final expression for $|G|$ corresponds to the following choice of loops and arcs:
$$
\psfrag{D}[Br][Br]{\scalebox{.9}{\color{mylightblue}{$\partial D$}}}
\psfrag{m}[Bc][Bc]{\scalebox{.9}{\color{myblue}{$m$}}}
\psfrag{a}[Br][Br]{\scalebox{.9}{\color{myred}{$\ell_1$}}}
\psfrag{d}[Bc][Bc]{\scalebox{.9}{\color{myred}{$\ell_2$}}}
\psfrag{c}[Bl][Bl]{\scalebox{.9}{\color{myred}{$\ell_4$}}}
\psfrag{k}[Bl][Bl]{\scalebox{.9}{\color{myred}{$\ell_3$}}}
\psfrag{u}[Bl][Bl]{\scalebox{.9}{\color{myblue}{$\gamma_1$}}}
\psfrag{v}[Bl][Bl]{\scalebox{.9}{\color{myblue}{$\gamma_2$}}}
\psfrag{n}[Bl][Bl]{\scalebox{.9}{\color{myblue}{$\gamma_3$}}}
\psfrag{o}[Bl][Bl]{\scalebox{.9}{\color{myblue}{$\gamma_4$}}}
\rsdraw{.45}{.9}{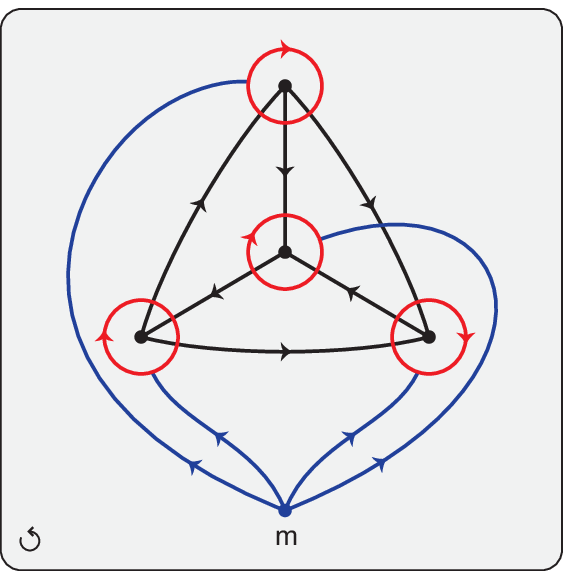} \, \;.
$$

\begin{lem}\label{lem-grade-labeled-graphs}
The grade $|G|$ is well-defined and $|G|\in \Ker(\CM)$.
\end{lem}

In particular, $|G|$ is  a central element of $E$ (because $\Ker(\CM)$ is central in $E$). Also, the grade is multiplicative with respect to the disjoint union: if $G$ and $G'$ are $\CM$-graphs in~$\R^2$ lying in disjoint disks, then $|G \amalg G'|=|G||G'|$.

\begin{proof}
We first prove that, for any choice (as above) of the arcs  $\gamma_{v_1},\dots, \gamma_{v_k}$, we have $|G| \in \Ker(\CM)$.
Denote by $G_0$ the set of vertices of $G$.
To any oriented arc~$\lambda$ in $\R^2\setminus G_0$ intersecting the edges of $G$ transversally, we associate an element $\phi(\lambda) \in H$ as follows. Start with $\phi(\lambda)=1 \in H$. Go through the arc $\lambda$ starting from the initial point of $\lambda$. Each time~$\lambda$ intersects an edge $s$ of~$G$ at some point~$c$, replace $\phi(\lambda)$ by $\phi(\lambda)x^\varepsilon$, where $\varepsilon=+$ if $(d_cs,d_c\gamma)$ is a positively oriented basis of $\RR^2$, $\varepsilon=-$ otherwise, and $x$ is the $H$-label of the edge $s$. In particular $\phi(\lambda)=1$ if~$\lambda$ does not intersect any edge of $G$. Clearly, $\phi$ is multiplicative with respect to the concatenation of arcs in $\RR^2\setminus G_0$: $\phi(\lambda \mu)=\phi(\lambda) \phi(\mu)$ if the endpoint of $\lambda$ is the starting point of $\mu$. Moreover, if $\lambda$ and $\lambda'$ are homotopic (relative endpoints) in $\RR^2\setminus G_0$, then  $\phi(\lambda)=\phi(\lambda')$. For example, for any vertex $v$ of $G$, we have $\phi(\gamma_v)=h_v$ (by definition of $h_v$). Also, viewing the loop $\ell_v$ as an oriented arc with starting point the endpoint of $\gamma_v$, it follows from $\eqref{eq-xi-graph-1}$ that $\phi(\ell_v)=\CM(e_v)$. Then
\begin{gather*}
\CM(|G|)   \overset{(i)}{=} \CM\left(\lact{(h_{v_1})}{e_{v_1}} \cdots \lact{(h_{v_k})}{e_{v_k}} \right )
 \overset{(ii)}{=} h_{v_1} \CM(e_{v_1}) h_{v_1}^{-1} \cdots h_{v_k} \CM(e_{v_k}) h_{v_k}^{-1} \\
 \overset{(iii)}{=} \phi(\gamma_{v_1}) \phi(\ell_{v_1})\phi(\gamma_{v_1})^{-1} \cdots \phi(\gamma_{v_k}) \phi(\ell_{v_k})\phi(\gamma_{v_k})^{-1} \\
 \overset{(iv)}{=} \phi\bigl(\gamma_{v_1} \ell_{v_1} \gamma_{v_1}^{-1} \cdots \gamma_{v_k} \ell_{v_k} \gamma_{v_k}^{-1} \bigr)
 \overset{(v)}{=} 1.
\end{gather*}
Here $(i)$ follows from the definition of $|G|$, $(ii)$ from the multiplicativity of $\CM$ and~\eqref{eq-precrossed}, $(iii)$ from
the fact that  $h_v=\phi(\gamma_v)$ and $\CM(e_v)=\phi(\ell_v)$ for any vertex $v$ of $G$, $(iv)$ from the multiplicativity of $\phi$ with respect to the concatenation of arcs, and $(v)$ from the fact that the concatenated arc  $\gamma_{v_1} \ell_{v_1} \gamma_{v_1}^{-1} \cdots \gamma_{v_k} \ell_{v_k} \gamma_{v_k}^{-1}$ is homotopic (relative endpoints) in $\RR^2\setminus G_0$ to an arc not intersecting the edges of $G$ (namely, to the boundary of a disk containing $m$ in its boundary and $G$ in its interior).

Next, $|G|$ is independent of the choice of the starting point on the small loop negatively encircling $m$ used to enumerate the vertices $v_1, \dots,v_k$ of $G$. This follows from the fact that for any $1 \leq r \leq k$,
\begin{gather*}
\lact{(h_{v_r})}{e_{v_r}} \cdots \lact{(h_{v_k})}{e_{v_k}} \lact{(h_{v_1})}{e_{v_1}} \cdots \lact{(h_{v_{r-1}})}{e_{v_{r-1}}}
 \overset{(i)}{=} w_r^{-1} |G| w_r  \overset{(ii)}{=} |G|,
\end{gather*}
where $w_r=\lact{(h_{v_1})}{e_{v_1}} \cdots \lact{(h_{v_{r-1}})}{e_{v_{r-1}}} \in E$. Here $(i)$ follows from the definition of $|G|$ and $(ii)$ from the fact that $|G|$ is central in $E$ (since $|G| \in \Ker(\CM)$).

Let us prove the invariance of $|G|$ on the choice of the system of arcs $\{\gamma_v\}_{v \in G_0}=(\gamma_1,\dots,\gamma_k)$ where $\gamma_i=\gamma_{v_i}$.  Any two such systems of arcs are related by a finite sequence of homotopies (relative endpoints) in $\RR^2\setminus G_0$ and of the following moves:
\begin{enumerate}
\labela
\item local move around a vertex $v$ of $G$:
$$
\psfrag{j}[Bc][Bc]{\scalebox{.9}{\color{myblue}{$\gamma'_v$}}}
\psfrag{g}[Bc][Bc]{\scalebox{.9}{\color{myblue}{$\gamma_v$}}}
\psfrag{v}[Bl][Bl]{\scalebox{.9}{\color{myblue}{$v$}}}
\psfrag{x}[Bl][Bl]{\scalebox{.9}{$x \in H$}}
\psfrag{d}[Bc][Bc]{\scalebox{.9}{\color{myred}{$\ell_v$}}}
\rsdraw{.45}{.9}{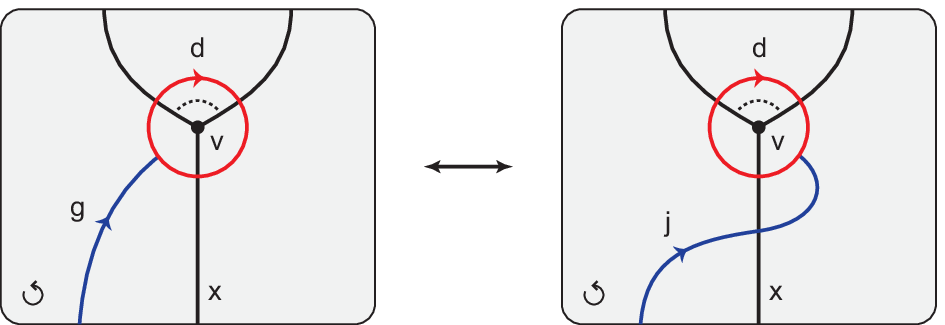} \, \;.
$$
\item global exchange move for $1 \leq i <k$:
$$
\psfrag{d}[Bc][Bc]{\scalebox{.9}{\color{myred}{$\ell_i$}}}
\psfrag{m}[Bc][Bc]{\scalebox{.9}{\color{myblue}{$m$}}}
\psfrag{u}[Bl][Bl]{\scalebox{.9}{\color{myblue}{$\gamma_i$}}}
\psfrag{v}[Bl][Bl]{\scalebox{.9}{\color{myblue}{$\gamma_{i+1}$}}}
\psfrag{a}[Bl][Bl]{\scalebox{.9}{\color{myblue}{$\gamma'_{i+1}$}}}
\psfrag{r}[Bl][Bl]{\scalebox{.9}{\color{myblue}{$\gamma'_{i}$}}}
\rsdraw{.45}{.9}{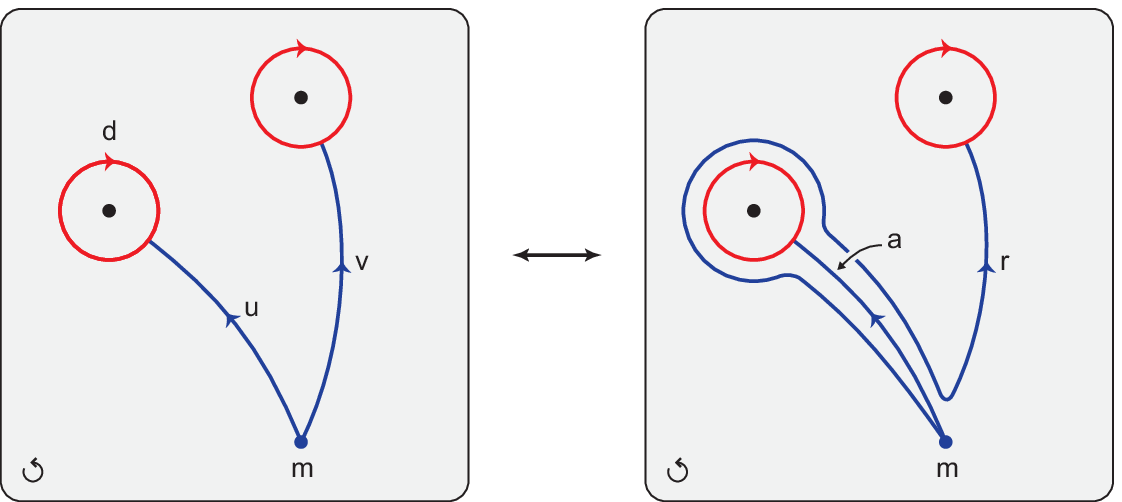}
$$
which replaces the arcs $(\gamma_i, \gamma_{i+1})$ by the arcs $(\gamma'_i, \gamma'_{i+1})$
where $\gamma'_{i+1}=\gamma_i$ and $\gamma'_i$ is the concantenation $\gamma_i \ell_i \gamma_i^{-1} \gamma_{i+1}$ (sligthly homotoped so that it becomes disjoint from $\gamma_i$ except at $m$).
\end{enumerate}
Invariance of $|G|$ under homotopies (relative endpoints) in $\RR^2\setminus G_0$ of the arcs $\gamma_v$ follows directly from the invariance of the elements  $h_v \in H$ and $e_v \in E$ under such homotopies. To prove the invariance under the move (a), denote by $h'_v \in H$ and $e'_v \in E$ the elements associated (as above) to $\gamma'_v$.
Set $\varepsilon=+$ if the edge labeled with $x \in H$ is oriented towards $v$ and  $\varepsilon=-$ otherwise. It follows directly from the definitions that $h'_v=h_vx^{-\varepsilon}$. Also $e_v=\lact{x^{-\varepsilon}}{\!e'_v}$ by \eqref{eq-xi-graph-2}, and so
$e'_v=\lact{x^\varepsilon}{\!e_v}$. Then
$$
\lact{h'_v}{e'_v}=\lact{h_vx^{-\varepsilon}x^\varepsilon}{\!e_v}=\lact{h_v}{e_v}.
$$
This proves the invariance of $|G|$ under the move (a). To prove the invariance under the move (b), denote by $h_r \in H$, $e_r \in E$ the elements associated (as above) to $\gamma_r$ and by $h'_r \in H$, $e'_r \in E$ those associated to $\gamma'_r$ for $r\in \{i,i+1\}$. Clearly, $h'_{i+1}=h_i$, $e'_{i+1}=e_i$, and $e'_i=e_{i+1}$. It follows from \eqref{eq-xi-graph-1} that $h'_i=h_i\CM(e_i) h_i^{-1}h_{i+1}$, and so $h'_i=\CM(\lact{h_i}{e_i}) h_{i+1}$ by \eqref{eq-precrossed}. Then, using \eqref{eq-Peiffer}, we obtain
$$
\lact{h'_i}{e'_i}\, \lact{h'_{i+1}}{e'_{i+1}}
=\lact{\CM(\lact{h_i}{e_i})h_{i+1}}{e_{i+1}}\, \lact{h_i}{e_i}
= \lact{h_i}{e_i}\, \lact{h_{i+1}}{e_{i+1}}.
$$
This proves the invariance of $|G|$ under the move (b).

Finally, the invariance of   $|G|$ on the choice of the point $m$ in the unbounded component of $\RR^2 \setminus G$ follows from the fact that this component is path connected.
\end{proof}

\subsection{Colored $\CM$-graphs}\label{sect-colored-Xi-graphs}
Let $\Sigma$ be an oriented surface. A $\CM$-graph in~$\Sigma$ is \emph{$\cc$-colored} if each edge is endowed with a homogeneous object of $\cc$, called its \emph{$\cc$-color}, whose degree is the $H$-label of this edge.

Let $G\subset \Sigma$ be a $\cc$-colored $\CM$-graph. The $\CM$-cyclic set $G_v$ associated with a vertex~$v$ of $G$ is $\cc$-colored by the map $c_v \co G_v \to \Ob(\cc)$ assigning to each half-edge the $\cc$-color of the edge containing it. Let
$$
H_v(G)=H(G_v )
$$
be the  multiplicity   module  of the $\CM$-cyclic $\cc$-set~$G_v$.  It can be described   as follows. Let $n \geq 1$ be the valence of~$v$ and let $s_1 < s_2 < \cdots < s_n < s_1$ be the half-edges of~$G$ incident to~$v$ with cyclic order induced by the opposite orientation of~$\Sigma$. Let $X_r=c_v(s_r)$, $e_r=\beta(s_r)$, and $\varepsilon_r=\varepsilon_v(s_r)$ be the $\cc$-color, the $E$-label, and the sign of~$s_r$, respectively. Then  we have the cone isomorphism
$$
\tau^v_{s_1} \co H_v(G) \stackrel{\simeq} {\longrightarrow}  \Hom_\cc^{e_1}(\un, X_1^{\varepsilon_1} \otimes  \cdots \otimes X_n^{\varepsilon_n} ).
$$
By the definition of $H_v(G)$, the cone isomorphisms determined by different elements of~$G_v$ are related via composition with the permutation maps~\eqref{eq-def-permutation}. We set
$$
H(G)=\otimes_v \, H_v(G),
$$
where $v$ runs over all vertices of~$G$ and   $\otimes$ is the unordered tensor product  of \kt modules.  To
emphasize the role of~$\Sigma$, we   sometimes write $H_v(G;\Sigma)$ for~$H_v(G)$ and $H(G;\Sigma)$ for~$H(G)$.   By definition,   for $G=\emptyset$, we have $H(G)=\kk$.

Any orientation preserving embedding $f $ of $\Sigma$ into an oriented surface $\Sigma'$ carries a   $\cc$-colored $\CM$-graph $G\subset \Sigma$  into  a $\cc$-colored $\CM$-graph $G'  =f(G)   \subset \Sigma'$ preserving the  vertices, the edges, the  $E$-labels of the half-edges,  and the orientations, $H$-labels,  and $\cc$-colors   of the edges. The map~$f$  induces a \kt linear isomorphism
$
H(f)\co   H(G;\Sigma ) \to H(G';\Sigma')
$
in the obvious way. This applies,   in particular, when~$f$ is an  orientation preserving self-homeomorphism of~$\Sigma$.

Given  $\cc$-colored $\CM$-graphs $G$ and $G'$ in $\Sigma$,   by an \emph{isotopy} of~$G$  to~$G'$, we mean  an  ambient    isotopy of~$G$  to~$G'$ in $\Sigma$ preserving the   vertices, the edges,   the  $E$-labels of the half-edges, and the orientations, $H$-labels,  and $\cc$-colors   of the edges. Such an isotopy  $\iota$   induces an  orientation preserving   homeomorphism  $\Sigma \to \Sigma$ carrying~$G$  to~$G'$. This  homeomorphism induces  a \kt linear  isomorphism   $  H(G) \to H(G')$ denoted $H(\iota)$.

Let $-\Sigma$ be $\Sigma$ with opposite orientation.
The \emph{dual} of $\cc$-colored $\CM$-graph $G$ in $\Sigma$ is the $\cc$-colored $\CM$-graph $G^\opp$ in $-\Sigma$ obtained from $G$ by reversing the orientation of all edges, keeping the $H$-labels and $\cc$-colors of the edges, and defining the $E$-labels of half-edges of $G^\opp$ as follows: the $E$-label of a half-edge $\ell$ of $G^\opp$ incident to a vertex $v$ is the inverse of the $E$-label of the successor of $\ell$ in $G_v$. Observe that $(G^\opp)_v=(G_v)^\opp$ for any vertex $v$ of $G$, where $(G_v)^\opp$ is the dual of $G_v$ in the sense of Section~\ref{sect-muliplicity-modules}. If $\cc$ is spherical pre-fusion, then the  (unordered) tensor product of the non-degenerate \kt bilinear pairings
$ \omega_{G_v}   \co  H(G_v^\opp ) \otimes H(G_v )    \to  \kk$ (see Section~\ref{sect-muliplicity-modules})
over all vertices $v$ of~$G$ yields a non-degenerate \kt bilinear pairing
$H (G^\opp ) \otimes H (G )  \to  \kk$,
and so these pairings  induce canonical \kt linear isomorphisms
$$
H_v(G^\opp )\simeq H_v(G )^\star \quad {\text {and}} \quad  H (G^\opp )\simeq H (G )^\star.
$$
Here we denote by $N^\star=\Hom_{\kk}(N, \kk)$ the dual of a \kt module $N$.

\subsection{An invariant of  planar $\cc$-colored  $\CM$-graphs}\label{sect-inv-Xi-graphs}
Let    $G$ be a   $\cc$-colored $\CM$-graph in~$\R^2$.  For
each vertex $v$   of~$G$,   pick a half-edge $s_v$ incident to $v$ and deform $G$ near~$v$ so that the half-edges incident to $v$   lie above $v$ with respect to the second coordinate of $\R^2$ and $s_v$ is the leftmost of them. Pick any $\alpha_v \in H_v(G)$ and replace~$v$ by a box colored with $\tau^v_{s_v}(\alpha_v)$, where $\tau^v$ is the universal cone of $H_v(G)$:
\begin{center}
  \psfrag{R}[Bc][Bc]{\scalebox{.9}{$\R^2$}}
  \psfrag{h}[Bc][Bc]{\scalebox{.9}{$\tau^v_{s_v}(\alpha_v)$}}
  \psfrag{e}[Bl][Bl]{\scalebox{.9}{$s_v$}}
  \psfrag{v}[Br][Br]{\scalebox{.9}{$v$}}
\rsdraw{.45}{.9}{inv-R2a-col} \quad \rsdraw{.45}{.6}{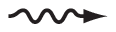} \quad
  \psfrag{e}[Br][Br]{\scalebox{.9}{$s_v$}}
  \psfrag{v}[Bc][Bc]{\scalebox{.9}{$v$}}
\rsdraw{.45}{.9}{inv-R2b-col} \quad \rsdraw{.45}{.6}{fleche} \quad \rsdraw{.45}{.9}{inv-R2c-col} \,.
\end{center}
This   transforms~$G$ into a diagram  without free ends. Let $\inv_\cc(G)(\otimes_v\alpha_v) \in \End_\cc(\un)$ be the associated morphism computed via the Penrose graphical calculus.  Since the  (co)evaluations and the morphisms $\tau^v_{s_v}(\alpha_v)$ are homogeneous morphisms, so is $\inv_\cc(G)(\otimes_v\alpha_v)$. It follows from the axioms of a pivotal $\CM$-category and of the definition of $H_v(G)$ that the homogeneous morphism $\inv_\cc(G)(\otimes_v\alpha_v)$ is independent of the choice of the half-edges $s_v$ and its degree is the grade $|G|\in \Ker(\CM)$ of $G$. Then $\inv_\cc(G)(\otimes_v\alpha_v) \in \End_\cc^{|G|}(\un)$. This   extends by linearity to  a \kt linear homomorphism
$$
\inv_\cc(G)\co H(G) =\otimes_v \, H_v(G) \to \End_\cc^{|G|}(\un).
$$
For $G=\emptyset$, we have $|G|=1$ and the map $\inv_\cc(G)\co H(G) =\kk \to \End_\cc^1(\un)$ is the \kt linear homomorphism carrying $1_\kk$ to $\id_\un$.

The   homomorphism  $\inv_\cc (G) \co H(G) \to \End_\cc^{|G|}(\un)$  is a well-defined isotopy invariant of   the   $\cc$-colored $\CM$-graph $G$ in $\R^2$. More precisely, for any isotopy~$\iota$ between   $\cc$-colored $\CM$-graphs $G$ and $G'$ in $\R^2$,
we have  $\inv_\cc(G')\, H(\iota)=\inv_\cc(G)$, where $H(\iota)\co H(G) \to H(G')$ is the  \kt linear   isomorphism induced by $\iota$. The invariance under isotopies follows from the isotopy invariance of the Penrose graphical calculus.

We state a few simple properties of $\inv_\cc$:

\begin{enumerate}
\labeli
\item (Naturality) If a $\cc$-colored $\CM$-graph $G'$ in   $\R^2$ is obtained from a
$\cc$-colored $\CM$-graph $G$ in $\R^2$ through the replacement of the $\cc$-color $X$ of an edge by a 1-isomorphic object $X'$, then any 1-isomorphism $\phi \co X\to X'$ induces a  \kt linear isomorphism $\Phi \co H(G')\to H(G)$  and $\inv_\cc(G')=\inv_\cc(G)\, \Phi$.

\item If a $\cc$-colored $\CM$-graph $G'$ in   $\R^2$ is obtained from a $\cc$-colored $\CM$-graph $G$ in $\R^2$
by  reversing the orientation of an edge and replacing the $\cc$-color  of this edge  by the dual object of~$\cc$, then
the pivotal structure of~$\cc$ induces a  \kt linear isomorphism $\Psi \co H(G')\to H(G)$ and $\inv_\cc(G')=\inv_\cc(G) \, \Psi$.

\item If   an edge~$e$ of   a $\cc$-colored $\CM$-graph $G$   in  $\R^2$    is colored with $\un$ and   the endpoints of~$e$ are also  endpoints of other edges of~$G$, then  $G'=G\setminus  \Int(e)\subset \RR^2$ inherits from~$G$ the structure of a $\cc$-colored $\CM$-graph,  there is a canonical \kt linear isomorphism $\Delta \co H(G')\to H(G)$,  and $\inv_\cc(G')=\inv_\cc(G ) \, \Delta$.

\item ($\otimes$-multiplicativity) If $G$ and $G'$ are disjoint $\cc$-colored $\CM$-graphs   in  $\R^2$  lying on different sides of a straight line,  then there is a canonical \kt linear isomorphism $\Theta \co H(G\amalg G') \to H(G) \otimes H (G')$ and $$
    \inv_\cc(G\amalg G')=\mu\bigl(\inv_\cc(G) \otimes \inv_\cc(G')\bigr)\Theta,
    $$
    where $\mu $ is  the multiplication of the \kt algebra   $\End_\cc(\un)$.
\end{enumerate}

\subsection{Example}
Consider the following $\cc$-colored $\CM$-graph   in $\R^2$ with two vertices and two edges $H$-labeled with $h,k \in H$ and $\cc$-colored with homogenous objects~$X,Y$:
$$
  \psfrag{R}[Bc][Bc]{\scalebox{.9}{$\R^2$}}
  \psfrag{X}[Br][Br]{\scalebox{.9}{$X$}}
  \psfrag{Y}[Bc][Bc]{\scalebox{.9}{$Y$}}
  \psfrag{h}[Bl][Bl]{\scalebox{.9}{$h$}}
  \psfrag{k}[Bc][Bc]{\scalebox{.9}{$k$}}
  \psfrag{e}[Br][Br]{\scalebox{.9}{$e$}}
  \psfrag{f}[Bl][Bl]{\scalebox{.9}{$f$}}
G=\,\rsdraw{.45}{.9}{inv-R2-exa-col}  \,.
$$
There are 3 half-edges $s_1, s_2,s_3$ incident to the leftmost vertex $u$ and one half-edge~$t$ incident to rightmost vertex $v$:
$$
s_1=\,\rsdraw{.45}{.9}{inv-R2-ex-ed1-col} \, , \quad s_2=\,\rsdraw{.45}{.9}{inv-R2-ex-ed2-col} \, , \quad
s_3=\,\rsdraw{.2}{.9}{inv-R2-ex-ed3-col} \; , \quad t=\,\rsdraw{.2}{.9}{inv-R2-ex-ed4-col} \;.
$$
The cyclic order   on $G_u=\{s_1,s_2,s_3\}$ is $s_1<s_2<s_3<s_1$. By the convention of Section~\ref{sect-Xi-graphs-degree}, the elements $e$ and $f$ of $E$ are the $E$-labels of $s_2$ and $t$, respectively. The conditions for $G$ to be a $\cc$-colored $\CM$-graph reduce to
$$
\CM(e)=h^{-1}k^{-1}h, \quad  \CM(f)=k , \quad |X|=h, \quad \text{and} \quad |Y|=k.
$$
The degree of $G$ is $|G|=\elact{h}{e}f$. By \eqref{eq-xi-graph-2}, the $E$-labels of the half-edges $s_1$ and $s_3$ are $\lact{h}{e}$ and $\lact{k h}{e}$, respectively.
The cone isomorphisms associated with the half-edges are:
\begin{align*}
\tau^u_{s_1} \co H_u(G) & \to \Hom_\cc^{\lact{h}{\!e}}(\un,X \otimes X^* \otimes Y^*),\\
\tau^u_{s_2} \co H_u(G) & \to \Hom_\cc^e(\un,X^* \otimes Y^* \otimes X),\\
\tau^u_{s_3} \co H_u(G) & \to \Hom_\cc^{\lact{k h}{\!e}}(\un,Y^* \otimes X \otimes X^*),\\
\tau^v_{t} \co H_v(G) & \to \Hom_\cc^f(\un,Y).
\end{align*}
The first three isomorphisms are related  to each other via composition with the permutation   maps~\eqref{eq-def-permutation}.   For  instance, for any $\alpha \in H_u(G)$,
$$
\tau^u_{s_2}(\alpha)=
  \psfrag{X}[Bl][Bl]{\scalebox{.9}{$X$}}
  \psfrag{Y}[Bl][Bl]{\scalebox{.9}{$Y$}}
  \psfrag{h}[Bc][Bc]{\scalebox{.9}{$\tau^u_{s_1}(\alpha)$}}
  \rsdraw{.39}{.9}{inv-R2-exe-col}  \;.
$$
By definition,   $H(G)=H_u(G) \otimes H_v(G)$. For any $\alpha \in H_u(G)$ and  $\beta \in  H_v(G) $,
$$
  \psfrag{X}[Bc][Bc]{\scalebox{.9}{$X$}}
  \psfrag{Y}[Bc][Bc]{\scalebox{.9}{$Y$}}
  \psfrag{k}[Bc][Bc]{\scalebox{.9}{$\tau^v_t(\beta)$}}
  \inv_\cc (G)(\alpha \otimes \beta) \, = \,
  \psfrag{h}[Bc][Bc]{\scalebox{.9}{$\tau^u_{s_1}(\alpha)$}}
  \rsdraw{.39}{.9}{inv-R2-exb-col}  \, = \,
  \psfrag{h}[Bc][Bc]{\scalebox{.9}{$\tau^u_{s_2}(\alpha)$}}
  \rsdraw{.39}{.9}{inv-R2-exc-col}  \, = \,
  \psfrag{h}[Bc][Bc]{\scalebox{.9}{$\tau^u_{s_3}(\alpha)$}}
  \rsdraw{.39}{.9}{inv-R2-exd-col}  \;.
$$

\subsection{The spherical  case}\label{sect-Xi-graphs-spherical}
Consider the 2-sphere $S^2=\R^2 \cup \{\infty\}$ endowed with   the  orientation   extending the
counterclockwise orientation in $\R^2$.  We say that a $\CM$-graph $G$ in~$S^2$ is \emph{1-spherical} if when pushing $G$ away from $\infty$, we obtain   a $\CM$-graph $G_0$ in~$\R^2$ with trivial grade.
Note that this notion does not depend on the  way $G$ is pushed away from~$\infty$. Indeed, if $G'_0$ is another $\CM$-graph in $\R^2$ obtained by pushing $G$ away from~$\infty$, then it follows from Section~\ref{sect-Xi-graphs-degree} that $|G'_0|=\lact{x}{|G_0|}$ for some $x\in H$, and so $|G'_0|=1$ if and only $|G_0|=1$.

Assume $\cc$ is spherical (see Section~\ref{sect-crossed-module-graded-spherical}).
Then the invariant $\inv_\cc$ of  $\cc$-colored $\CM$-graphs in $\RR^2$ induces an isotopy invariant of 1-spherical $\cc$-colored $\CM$-graphs in  $S^2$. Indeed,  consider  a $\cc$-colored $\CM$-graph~$G$ in~$S^2$. Pushing $G$ away from $ \infty $ by an isotopy, we obtain  a $\cc$-colored   $\CM$-graph $G_0$ in $\R^2$. The isotopy induces  a \kt linear isomorphism $H (G;S^2)\simeq H  (G_0;\R^2)$. Composing    with $\inv_\cc(G_0)\co H (G_0;\R^2)\to \End_\cc^1(\un)$ we obtain a \kt linear homomorphism
$$
\inv_\cc(G) \co  H (G;S^2)\to \End_\cc^1(\un).
$$
The sphericity of $\cc^1$ ensures that $\inv_\cc(G)$ is a well-defined isotopy invariant of $G$. The properties  of $\inv_\cc$ formulated in Section~\ref{sect-inv-Xi-graphs} extend to 1-spherical $\cc$-colored $\CM$-graphs in    $S^2 $ in the obvious way. The condition in (iv) involving a straight line  may be dropped here because for any pair of  disjoint connected graphs in $S^2$, there is an isotopy of $S^2$ in itself carrying these graphs into new positions in $\RR^2$ separated by a line.

The invariant $\inv_\cc$ further   extends to 1-spherical $\cc$-colored $\CM$-graphs in   an arbitrary oriented  surface~$\Sigma$ homeomorphic to $S^2 $. Given a 1-spherical $\cc$-colored $\CM$-graph $G$ in $\Sigma$, pick an orientation preserving homeomorphism $f \co \Sigma \to S^2$ and set
$$
\inv_\cc (G)=\inv_\cc (f(G))\, H(f) \co H(G; \Sigma) \to \End_\cc^1(\un).
$$
Since all orientation preserving homeomorphisms $\Sigma \to   S^2$  are isotopic, the homomorphism $\inv_\cc (G)$ does not depend on the choice of~$f$.

\subsection{Dual vertices}\label{sect-dula-verticies-Xi-graphs}
Consider a $\cc$-colored $\CM$-graph~$G$ in an oriented surface~$\Sigma$ and a $\cc$-colored $\CM$-graph~$G'$ in an oriented surface~$\Sigma'$.
By a \emph{duality} between a vertex~$u$ of~$G$ and a vertex $v$ of $G'$,  we mean an isomorphism  of $\CM$-cyclic $\cc$-sets $ \phi \co G_u^\opp \to G_v'$, where $G_u^\opp$ is the dual of $G_u$ (see Section~\ref{sect-muliplicity-modules}).
When $\cc$ is spherical pre-fusion, the \emph{contraction vector}   of a duality   $ \phi \co G_u^\opp \to G_v'$ is the  vector
$$
\ast_\phi=\bigl(\id_{H(G_u)} \otimes H(\phi)\bigr)(\ast_{G_u})\in H(G_u) \otimes H(G_v' )=H_u(G) \otimes H_v(G'),
$$
where
$$
H(\phi) \co H( G_u^\opp) \to H(G_v') \quad \text{and} \quad \ast_{G_u}\in  H(G_u) \otimes H(G_u^\opp )
$$
are respectively  the \kt linear isomorphism induced by $\phi$ and  the contraction vector of the $\CM$-cyclic $\cc$-set $G_u$ (see Section~\ref{sect-muliplicity-modules}).
Clearly,  $\ast_\phi =\ast_{\phi^{-1}}$.

The next lemma formulates a local relation for the invariant $\inv_\cc$ of $\cc$-colored $\CM$-graphs in $\RR^2$ involving dual vertices and contraction vectors.

\begin{lem}\label{lem-calc-diag}
Assume that $\cc$ is spherical pre-fusion and let $I$ be a $\CM$-representative set for~$\cc$. The pictures below stands  for  a    pieces   of  $\cc$-colored $\CM$-graphs in $\RR^2$. The orientation of the edges $\cc$-colored by the object $X_i$ should match in both sides of the equality, and let $\varepsilon_i=+$ if this orientation is downwards and $\varepsilon_i=-$ otherwise.
Set $h=|X_1|^{\varepsilon_1} \cdots |X_n|^{\varepsilon_n} \in H$.
Then, for all $e \in E$,
$$
\psfrag{i}[Br][Br]{\scalebox{.9}{$i$}}
\psfrag{u}[Bl][Bl]{\scalebox{.9}{$e$}}
\psfrag{v}[Bl][Bl]{\scalebox{.9}{$e^{-1}$}}
\psfrag{X}[Br][Br]{\scalebox{.8}{$X_1$}}
\psfrag{Z}[Bl][Bl]{\scalebox{.8}{$X_n$}}
\inv_\cc \left (\
\rsdraw{.45}{.9}{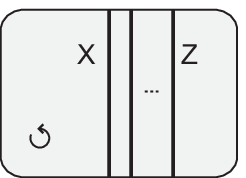}
\,\right )
= \sum_{i \in I_{h \CM(e^{-1})}} \dim(i) \;
\inv_\cc \left (\,
\rsdraw{.45}{.9}{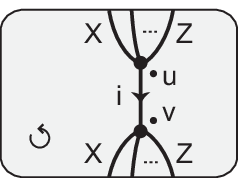}
\,\right )  (\ast ).
$$
Here, $e^{\pm 1}$ are the $E$-labels of the half-edges determined by the points (with the graphical convention of Section~\ref{sect-Xi-graphs}) and  $\ast$ is the   contraction vector provided by the duality between the two vertices induced   by the symmetry   with respect to  a horizontal  line.
\end{lem}
\begin{proof}
We only need to compare the contributions to~$\inv_\cc$   of the depicted pieces of $\cc$-colored $\CM$-graphs. The contribution of the left-hand piece is $\id_X$ where $X=X_1^{\varepsilon_1} \otimes \cdots  \otimes X_n^{\varepsilon_n}$.
Using the definitions  of $\inv_\cc$ and $\ast$, Formula~\eqref{eq-compute-inverse-s}, and considering   the isotopy
\begin{equation*}
\psfrag{i}[Br][Br]{\scalebox{.9}{$i$}}
\psfrag{u}[Bl][Bl]{\scalebox{.9}{$u$}} \psfrag{v}[Bl][Bl]{\scalebox{.9}{$v$}}
\rsdraw{.45}{.9}{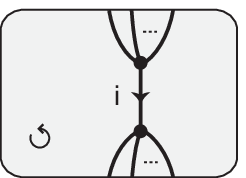}\,\; \simeq \;\, \psfrag{i}[Bc][Bc]{\scalebox{.9}{$i$}}\rsdraw{.45}{.9}{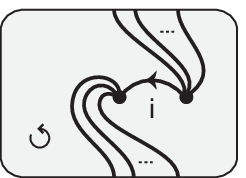}\;,
\end{equation*}
we obtain that the contribution  of the right-hand piece is equal to
\begin{gather*}
\psfrag{e}[Bc][Bc]{\scalebox{.9}{$e$}}
\psfrag{X}[Br][Br]{\scalebox{.8}{$X$}}
\psfrag{i}[Bc][Bc]{\scalebox{.9}{$i$}}
\sum_{i \in I_g} \dim(i)\!\! \psfrag{i}[Bl][Bl]{\scalebox{.9}{$i$}} \rsdraw{.45}{.9}{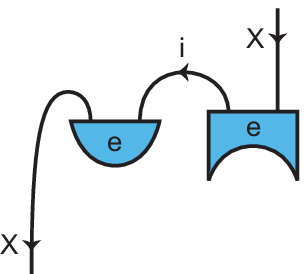} \;
\overset{(i)}{=} \;  \sum_{i \in I_g} \dim(i)\; \psfrag{i}[Bl][Bl]{\scalebox{.9}{$i$}} \rsdraw{.45}{.9}{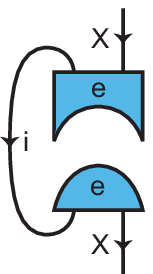}  \; \overset{(ii)}{=} \;
\psfrag{e}[Bc][Bc]{\scalebox{.9}{$\elact{g}{e}$}}
\sum_{i \in I_g} \; \rsdraw{.45}{.9}{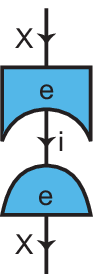}  \;\overset{(iii)}{=}  \,\id_{X}
\end{gather*}
where $g=h \CM(e^{-1})$.
Here $(i)$ follows from the isotopy invariance of the graphical calculus,  $(ii)$   from Lemma~\ref{lem-graphical-calculus-prefusion-Xi}, and $(iii)$  from  \eqref{eq-pre-fus-sum} and the fact that
$\Hom_\cc^{\elact{g}{e}}(i,X)=0$ for all $i \in I\setminus I_g$ (since $X$ is a 1-direct sum of homogenous objects of degree $h$).
\end{proof}

\section{Skeletons and maps}\label{sect-skeletons-maps}

Throughout this section, $\CM \co E \to H$ is a crossed module and $B\CM$ is the classifying space of $\CM$ (see Section~\ref{sect-crossed-modules-classifying-spaces}). We first recall  the notions of stratified 2-polyhedra and skeletons of 3-manifolds (referring to \cite{TVi5} for details). Then we explain how to encode maps from closed 3-manifolds to $B\CM$ in terms of labelings of skeletons.

\subsection{Stratified 2-polyhedra}\label{sect-stratified-polyhedron}
A \emph{2-polyhedron} is a compact topological space that can be triangulated using a finite number of simplices of dimension $ \leq 2$ so that all $0$-simplices and $1$-simplices are faces of $2$-simplices. The interior $\Int(P)$ of a $2$-polyhedron $P$ consists of all points which have a neighborhood homeomorphic to~$\RR^2$. By the definition of a 2-polyhedron, the surface $\Int(P)$ is dense in $P$.

A \emph{stratified polyhedron} is a $2$-polyhedron $P$ endowed with a graph $P^{(1)}$ embedded in~$P$  such that $P \setminus \Int(P) \subset P^{(1)}$. The vertices and edges of $P^{(1)}$ are called respectively the \emph{vertices} and \emph{edges} of $P$. Any 2-polyhedron~$P$ can be stratified. For example,   the 0-simplices and the 1-simplices of a  triangulation of~$P$ form a graph which is   a stratification of~$P$.    Any (possibly, empty) graph embedded in a   compact surface without boundary  is   a stratification of this surface.

An \emph{orientation} of a  stratified 2-polyhedron $P$ is an orientation of the surface $P \setminus P^{(1)}$. A  stratified 2-polyhedron  is \emph{oriented} if it is endowed with an orientation. Cutting a stratified 2-poly\-hed\-ron~$P$ along the graph $P^{(1)}\subset P$,  we obtain a compact surface~$\widetilde{P}$ with interior $P \setminus P^{(1)}$.
The connected components of~$\widetilde{P}$ are called the \emph{regions} of~$P$. Each component of $P \setminus P^{(1)}\subset \widetilde{P} $ is the interior of a   unique   region. The set $\Reg(P) $ of the regions of~$P$ is finite. To orient~$P$, one must orient all its regions.

A \emph{branch}   of a stratified 2-polyhedron~$P$ at a vertex~$v$ of~$P$ is a germ  at~$v$ of an adjacent region.
Similarly,  a \emph{branch}   of~$P$ at an edge~$e$ of~$P$ is a germ  at~$e$ of an adjacent region.
The set of branches of~$P$ at~$e$ is denoted~$P_e$. This set    is finite and
non-empty.  The number of elements of $P_e$ is   called    the \emph{valence} of~$e$.

The edges  of a  stratified 2-polyhedron~$P$  of valence 1 together with their vertices  form a graph called the
\emph{boundary} of~$P$ and denoted~$\partial P$. An  orientation of~$P$ induces an orientation of~$\partial P$: each  edge~$e$ of $\partial P$ is oriented so that the orientation of the
unique region  of~$P$ adjacent to~$e$   is determined by a   pair (a  vector directed outward   at an interior point of~$e$, a positive tangent vector of~$e$).

\subsection{Skeletons of closed 3-manifolds}\label{sect-skeletons}
A \emph{skeleton} of  a closed oriented
(possibly, non-connected) 3-dimensional manifold $M$  is an oriented  stratified 2-polyhedron $P\subset M$ such that $\partial P=\emptyset$ and $M\setminus P$ is a disjoint union of open 3-balls. These open 3-balls are called \emph{$P$-balls}.  The condition $\partial P=\emptyset$ ensures  that all edges of~$P $ have valence   $ \geq 2$.
An example of a skeleton of~$M$ is provided by the
(oriented) 2-skeleton~$t^{(2)}$ of a triangulation~$t$ of~$M$, where
the edges of~$t^{(2)}$ are the edges of~$t$.

Given a skeleton $P$ of $M$, a \emph{$P$-ball branch}  at an edge~$e$ of $P$ is a germ  at~$e$ of an adjacent $P$-ball.

Any    vertex   $x$ of a skeleton~$P$ of $M$  has a closed ball neighborhood $B_x \subset M$ such that $\Gamma_x=P\cap \partial B_x$ is a   non-empty graph and $ P\cap B_x $ is the cone over~$  \Gamma_x $ with summit~$x$.  The vertices of $\Gamma_x$ are the intersection   points of the 2-sphere $\partial B_x$ with the edges   of~$P$ incident to~$x$.
Each  edge of $\Gamma_x$ is the intersection  of $\partial B_x$ with a branch $b$ of $P$ at $x$, and we endow  this edge  with the orientation induced by that of $b$ restricted to $b \setminus \Int (B_x)$.
We endow~$\partial B_x\cong S^2$ with the orientation induced by that of~$M$ restricted to $M\setminus \Int(B_x)$.   In this way,
$\Gamma_x$ becomes an oriented graph in the oriented 2-sphere $\partial B_x$.  We call $B_x$ a \emph{$P$-cone neighborhood}   of~$x$ and call   $\Gamma_x$ the \emph{link graph} of~$x$.  The pair $(B_x, \Gamma_x)$  is    determined by  the triple $ (M,P,x) $  uniquely  up to    homeomorphism.  Since all edges of~$P$ have valence $\geq 2$, so do  all vertices of~$\Gamma_x$.
\begin{figure}
\begin{center}
\scalebox{.9}{\psfrag{P}[Br][Br]{\scalebox{1.111111}{$P=$}}
\psfrag{x}[Bc][Bc]{\scalebox{1}{$x$}}
\psfrag{1}[Bc][Bc]{\scalebox{.7}{$1$}}
\psfrag{2}[Bc][Bc]{\scalebox{.7}{$2$}}
\psfrag{3}[Bc][Bc]{\scalebox{.7}{$3$}}
\psfrag{M}[Br][Br]{\scalebox{.9}{$M$}}
\psfrag{S}[Bl][Bl]{\scalebox{.9}{$S^2$}}
\psfrag{G}[Br][Br]{\scalebox{1.111111}{$\Gamma_x=$}}
\psfrag{H}[Bl][Bl]{\scalebox{1.111111}{$ \subset \partial B_x \cong S^2$.}}
\rsdraw{.45}{.9}{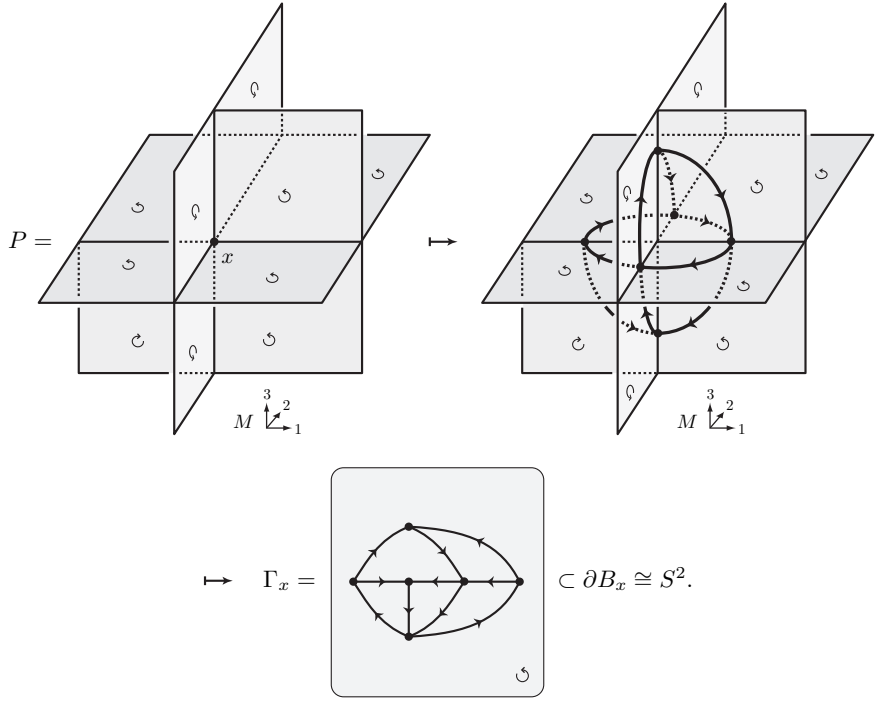}}
\end{center}
\captionsetup{justification=centering}
\caption{The link graph of a vertex}
\label{fig-link-graph}
\end{figure}
An example  is given in Figure~\ref{fig-link-graph} where there are 10 branches of $P$ at $x$, 6 $P$-ball branches at~$x$,  and $\partial B_x$ is identified with $S^2=\R^2\cup\{\infty\}$ via the stereographic projection from a pole in the upper left quarter.

\subsection{$\CM$-labelings}\label{sect-Xi-labelings}
Let $P$ be a skeleton of an oriented closed  3-dimensional manifold~$M$.
By an \emph{oriented edge} of~$P$, we mean an edge of $P$ endowed with an orientation.
For each oriented edge $e$ of $P$, the orientations of~$e$ and~$M$ determine a positive
direction on a small loop in~$M$ encircling~$e$ (so that the linking number of this loop with $e$ is
$+1$).  The  resulting oriented loop determines  cyclic orders on the set $P_e$ of branches of~$P$ at~$e$ and on the set
$P^{\text{ball}}_e$ of $P$-ball branches at~$e$.
Any $P$-ball branch $B \in P^{\text{ball}}_e$ determines a branch $b_B\in P_e$ which is the first branch of~$P$ at~$e$ encountered while traversing the oriented loop starting from $B$. This induces a bijection $P^{\text{ball}}_e \to P_e$ preserving the cyclic order.
In particular, any $P$-ball branch $B\in P^{\text{ball}}_e$ turns the cyclic order on $P_e$ into a linear order so that the  first element is $b_B$.
To each  branch $b\in P_e$ we assign a sign equal to~$ +$ if the orientation  of~$b$  induces the one of $e\subset \partial b$ (that is, the orientation of~$b$ is given by the orientation of~$e$ followed by a vector at a point of~$e$ directed inside~$b$)    and equal  to $ -$ otherwise. This gives a   map $\varepsilon_e \co P_e \to \{+, -\}$.  When orientation  of $e$ is reversed, the cyclic orders  on $P_e$ and $P^{\text{ball}}_e$ are reversed and $ \varepsilon_e$ is multiplied by $-$.

Recall that  $\Reg(P)$ denotes the (finite) set of regions of~$P$ (see Section~\ref{sect-stratified-polyhedron}).
Let $\EB(P)$ be the set of pairs $(e,B)$ where $e$ is an oriented edge of $P$ and $B$ is a $P$-ball branch at $e$. Note that
$
\EB(P)=\bigcup_{e} \, \{e \} \times P_e^{\text{ball}}
$
where $e$ runs over all oriented edges of $P$.
A \emph{pre-$\CM$-labeling} of $P$ is a pair
$
\bigl(\colr \co \Reg (P) \to H,\coleb \co \EB(P) \to E\bigr)
$
of maps such that for any $(e,B) \in \EB(P)$,
\begin{align}
 \CM(\coleb(e,B)) &=\colr(b_1)^{\varepsilon_e(b_1)} \cdots \colr(b_n)^{\varepsilon_e(b_n)},  \label{precol1} \\
\coleb(e,\suc(B)) & =\lact{\left(\colr(b_1)^{-\varepsilon_e(b_1)}\right)}{\!\coleb(e,B)}, \label{precol2} \\
\coleb(-e,B) & =\coleb(e,B)^{-1}, \label{precol3}
\end{align}
where $b_1< \cdots <b_n$ are the elements of $P_e$ enumerated in the linear order determined by $B$, $\suc(B)$ is the successor of $B$ in the cyclic order on $P^{\text{ball}}_e$, and $-e$ is $e$ with opposite orientation:
$$
\psfrag{1}[Bc][Bc]{\scalebox{.7}{$1$}}
\psfrag{2}[Bc][Bc]{\scalebox{.7}{$2$}}
\psfrag{3}[Bc][Bc]{\scalebox{.7}{$3$}}
\psfrag{M}[Br][Br]{\scalebox{.9}{$M$}}
\psfrag{a}[Bc][Bc]{\scalebox{.8}{$b_1$}}
\psfrag{b}[Bc][Bc]{\scalebox{.8}{$b_2$}}
\psfrag{c}[Bc][Bc]{\scalebox{.7}{$b_3$}}
\psfrag{d}[Bc][Bc]{\scalebox{.8}{$b_{n-1}$}}
\psfrag{z}[Bc][Bc]{\scalebox{.8}{$b_n$}}
\psfrag{e}[Bc][Bc]{\scalebox{.9}{$e$}}
\psfrag{B}[Br][Br]{\scalebox{.9}{$B$}}
\psfrag{S}[Br][Br]{\scalebox{.9}{$\suc(B)$}}
\rsdraw{.45}{.9}{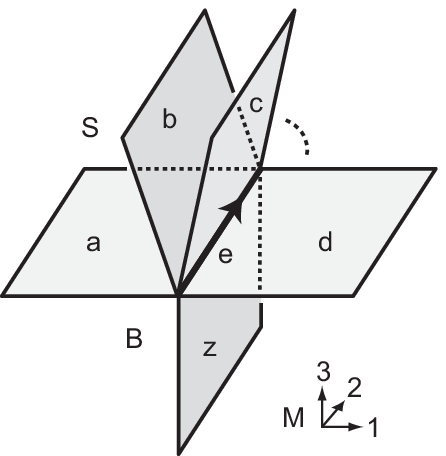} \; .
$$
Here and in the sequel, for any branch $b$ of $P$ at a vertex or an edge of $P$, we let $\colr(b)=\colr(r_b)$ where~$r_b$ is the unique region of $P$ containing $b$.
It follows from \eqref{precol2} that for any oriented edge $e$ of $P$, the map $B \in P^{\text{ball}}_e \mapsto \beta(e,B) \in E$ is fully determined by its value on one element of~$P^{\text{ball}}_e$. Also, if \eqref{precol1} holds for one $P$-ball branch at $e$, then it holds for all $P$-ball branches at $e$. Similarly, \eqref{precol3} implies that if \eqref{precol1} or \eqref{precol2} holds for $(e,B) \in \EB(P)$, then it holds for $(-e,B)$.

A pre-$\CM$-labeling $(\colr,\coleb)$ of $P$ turns the set $P_e$ of branches of $P$ at an oriented edge~$e$ of $P$ into a $\CM$-cyclic set $P_e=(P_e,\alpha_e,\beta_e,\varepsilon_e)$ as follows. For any branch $b \in P_e$, set $\alpha_e(b)=\colr(b) \in H$ and $\beta_e(b)=\coleb(e,B)\in E$, where $B$ is the $P$-ball branch at $e$ such that $b=b_B$ (in the above notation).

A pre-$\CM$-labeling $(\colr,\coleb)$ of $P$ turns the link graph $\Gamma_x$ of a vertex $x$ of $P$ into a $\CM$-graph in $\partial B_x \cong S^2$ as follows, where $B_x$ is a $P$-cone neighborhood  of~$x$ (see Section~\ref{sect-skeletons}).
Each  edge of $\Gamma_x$ is the intersection  of $\partial B_x$ with a branch $b$ of~$P$ at~$x$ and we endow  this edge  with the $H$-label $\alpha(b) \in H$. Each half-edge $h$ of~$\Gamma_x$ is uniquely determined by an oriented edge $e$ of $P$  incident to $x$ and a $P$-ball branch~$B$ at $e$. Indeed, the vertex $v$ of $\Gamma_x$ to which $h$ is incident is the first intersection point of~$\partial B_x$ with $e$ while traversing $e$ following its orientation. The orientation of $e$ and $M$ determine as above a
branch $b_B$ of $P$ at $e$. The half-edge $h$ is then the connected component containing $v$ of the intersection of $\partial B_x$ with $b_B$. We endow  the half-edge  $h$ with the $E$-label $\beta(e,B) \in E$. It follows from \eqref{precol1}-\eqref{precol3} that these labels of (half-)edges make $\Gamma_x$ a $\CM$-graph (see Section~\ref{sect-Xi-graphs}).

A \emph{$\CM$-labeling} of $P$ is a pre-$\CM$-labeling of $P$ such that, for any vertex $x$ of $P$, the $\CM$-graph $\Gamma_x$ in $B_x \cong S^2$ is 1-spherical (see Section~\ref{sect-Xi-graphs-spherical}).

Let us define the \emph{gauge group} of $\mathcal{G}_P$ of $P$ as follows. As a set,
$$
\mathcal{G}_P=\mathrm{Map}(\Ball(P),H) \times \mathrm{Map}(\Reg(P),E)
$$
where $\Ball(P)$ is the set of $P$-balls. For any $r \in \Reg(P)$, let $r_\pm$ be the $P$-balls (possibly coinciding) adjacent to the region~$r$ and  indexed so that the orientation of $r$ followed by a normal vector pointing from $r_-$ to $r_+$
yields the given orientation of $M$. The product in $\mathcal{G}_P$ is defined by
$
(\lambda',\mu') \cdot (\lambda,\mu)=(\lambda'',\mu'')
$
with
$$
\lambda''(B)=\lambda'(B) \lambda(B) \quad \text{and} \quad \mu''(r)= \mu(r) \left(\lact{\lambda(r_+)^{-1}}{\!\!\mu'(r)} \right)
$$
for any $P$-ball $B$ and any region $r$ of $P$.  The group $\mathcal{G}_P$ acts (on the left) on the set of $\CM$-labelings of~$P$. To define this action we need some notations. Pick  a point, called the  \emph{center}, in every $P$-ball. For each region $r$ of $P$, pick an oriented arc $\gamma_r$ in $M$ connecting the center of $r_-$ to the center of $r_+$ and intersecting $r$ transversely at a single point. We choose the arcs $\{\gamma_r\}_r$ so that they intersect with each other only at their endpoints. For any branch~$b$ of $P$, we set $\gamma_b=\gamma_{r_b}$ where~$r_b$ is the unique region of $P$ containing~$b$. Each $(e,B) \in \EB(P)$ determines a loop $\gamma_{(e,B)}$ in~$M$ based at the center of $B$ defined by the concatenation
$$
\gamma_{(e,B)}=\gamma_{b_1}^{\varepsilon_e(b_1)} \cdots \gamma_{b_n}^{\varepsilon_e(b_n)},
$$
where $b_1< \cdots <b_n$ are the elements of $P_e$ enumerated in the linear order determined by $B$. For example, we have $\gamma_{(e,B)}=\gamma_1 \gamma_2^{-1} \gamma_3$ for the following trivalent oriented edge~$e$ and the $P$-ball branch $B$:
$$
\psfrag{1}[Bc][Bc]{\scalebox{.7}{$1$}}
\psfrag{2}[Bc][Bc]{\scalebox{.7}{$2$}}
\psfrag{3}[Bc][Bc]{\scalebox{.7}{$3$}}
\psfrag{M}[Br][Br]{\scalebox{.9}{$M$}}
\psfrag{a}[Bc][Bc]{\scalebox{.9}{$b_2$}}
\psfrag{b}[Bc][Bc]{\scalebox{.9}{\textcolor{blue}{$\gamma_2$}}}
\psfrag{c}[Bc][Bc]{\scalebox{.9}{\textcolor{blue}{$\gamma_1$}}}
\psfrag{k}[Bc][Bc]{\scalebox{.9}{\textcolor{blue}{$\gamma_3$}}}
\psfrag{d}[Bc][Bc]{\scalebox{.9}{$b_1$}}
\psfrag{z}[Bc][Bc]{\scalebox{.9}{$b_3$}}
\psfrag{e}[Bc][Bc]{\scalebox{.9}{\textcolor{red}{$e$}}}
\psfrag{B}[Bl][Bl]{\scalebox{.9}{\textcolor{red}{$B$}}}
\psfrag{x}[Bc][Bc]{\scalebox{.9}{$c_e$}}
\psfrag{v}[Bc][Bc]{\scalebox{.9}{$\delta_e$}}
\rsdraw{.45}{.9}{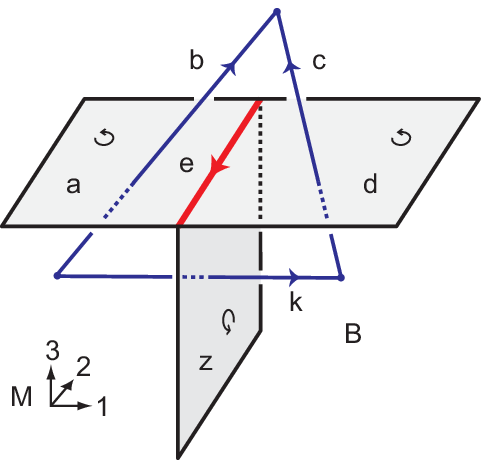} \; .
$$
The action of $\mathcal{G}_P$ on the set of $\CM$-labelings of~$P$ is then defined as follows.  For any  $(\lambda,\mu) \in \mathcal{G}_P$ and any $\CM$-labeling $(\colr,\coleb)$ of $P$, the $\CM$-labeling $(\colr',\coleb')=(\lambda,\mu) \cdot (\colr,\coleb)$ of $P$ is defined by
\begin{align}
\colr'(r) &= \lambda(r_-) \colr(r) \CM(\mu(r)) \lambda(r_+)^{-1},  \label{act-col1} \\
\coleb'(e,B) & =\lact{\lambda(B)}{\!{\big(\coleb(e,B) \mu_\alpha(\gamma_{(e,B)})\big)}}, \label{act-col2}
\end{align}
for all $r \in \Reg(P)$ and $(e,B) \in \EB(P)$.
Here, $\mu_\alpha$ extends $\mu$ to the groupoid freely generated by the morphisms  $\{\gamma_r\co r_- \to r_+\}_{r \in \Reg(P)}$ via the following rules:
$\mu_\alpha(\emptyset)=1$ and for any region $r$ of $P$ and any morphism $\omega$ in this groupoid,
$$
\mu_\alpha(\omega \ast \gamma_r)= \left(\lact{\alpha(r)^{-1}}{\!\!\mu_\alpha(\omega)}\right) \!\mu(r)  \quad \text{and} \quad
\mu_\alpha(\omega \ast \gamma_r^{-1})= \lact{\alpha(r)}{\!\left(\mu_\alpha(\omega)\mu(r)^{-1}\right)},
$$
where the symbol $\ast$ denotes the concatenation of the edges of a graph (which is opposite to the composition of morphisms in this groupoid). Two  $\CM$-labelings of $P$ are \emph{gauge equivalent} if they belong to the same orbit under the action of the gauge group $\mathcal{G}_P$.

\begin{lem}\label{lem-gauge-group-labelings}
There is an injective
map from the set of gauge equivalence classes of $\CM$-labelings of $P$ to the set of homotopy classes of maps $M \to B\CM$ which is bijective when the regions of $P$ are disks.
\end{lem}

We prove Lemma~\ref{lem-gauge-group-labelings} in Section~\ref{sect-proof-lem-gauge-group-labelings}. The next example illustrates the lemma for lens spaces.

\subsection{Example}\label{ex-lens-spaces-labelings}
Consider the lens space $L(p,q)$ where  $p,q$ are coprime integers with $1 \leq q < p$.
It is a  closed connected oriented   3-dimensional manifold  obtained by gluing two solid tori as follows.  Let   $D=\{z \in \CC\,\vert \, \vert z\vert \leq 1\}$ be the unit disk in~$\CC$    with   counterclockwise orientation. Let $S^1=\partial D$ be the unit circle    with     counterclockwise orientation.
Let $U,V$ be two copies of the solid torus $S^1 \times D$. The boundaries of~$U$ and~$V$ are copies of    $ S^1 \times S^1$. Then $L(p,q)=U\cup_\varphi V$  is obtained by  gluing~$U$ to~$V$ along a  homeomorphism $\varphi \co \partial U \to \partial V$ such that   $\varphi (1,t)=(t^p, t^q)$ for all~$t\in S^1$. We provide $L(p,q)$ with the  orientation extending the product orientation of $V=S^1 \times D$.

A skeleton~$P$ of $L(p,q)$ is  obtained from the disk $\{1\}\times D\subset U$ and the circle  $S^1\times \{0\}\subset V$ by pulling the boundary of this  disk towards   this  circle in $V$.  More precisely,
$$
P=(\{1\}\times D)  \cup_\varphi  \{(t^p, kt^q)\, \vert \, t\in S^1, k\in [0,1]\}  \subset U\cup_\varphi V=L(p,q).
$$
The polyhedron~$P$ contains the circle  $S^1\times \{0\}\subset V$ which, viewed as a loop  based at the point $(1,0)$, determines a stratification of~$P$.  The stratified polyhedron~$P$ has one vertex, one edge, and one disk region which we orient by extending the orientation of $D=\{1\}\times D \subset U$. In this way,  $P$ becomes a  skeleton of $L(p,q)$ with a single $P$-ball. Note that there are $p$ branches of~$P$ at the edge of $P$.

Any pre-$\CM$-labeling of $P$ is determined by a pair $(h,e) \in H \times E$ with $\CM(e)=h^p$. Here $h$ is the $H$-label of the disk region and $e$ is the $E$-label of the edge (oriented as $S^1=S^1\times \{0\}\subset V$ and with respect to some chosen adjacent $P$-ball branch). The $\CM$-graph $\Gamma$ in $S^2=\R^2\cup\{\infty\}$ associated with the vertex of $P$ is
$$
\psfrag{e}[Br][Br]{\scalebox{.9}{$e$}}
\psfrag{a}[Br][Br]{\scalebox{.9}{$e^{-1}$}}
\psfrag{h}[Bc][Bc]{\scalebox{.9}{$h$}}
\psfrag{q}[Bl][Bl]{\scalebox{.7}{$q$}}
\psfrag{p}[Bl][Bl]{\scalebox{.7}{$p-q$}}
\Gamma=\,\rsdraw{.45}{.9}{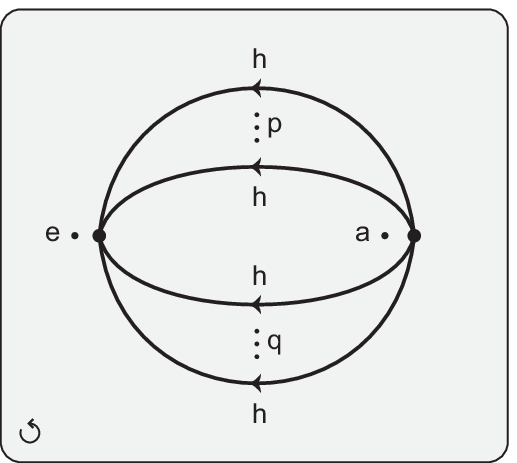}
$$
with $q$ edges below and $(p-q)$ edges above the horizontal axis. By Section~\ref{sect-Xi-graphs-degree}, the grade of $\Gamma$ is
$$
|\Gamma|=e\left(\elact{(h^{-q})}{(e^{-1})}\right).
$$
Then $\Gamma$ is 1-spherical (that is, $|\Gamma|=1$) if and only if $\lact{(h^q)}{e}=e$. Since $p,q$ are coprime and $\lact{(h^p)}{e}=\lact{\CM(e)}{e}=e e e^{-1}=e$, this condition is equivalent to $\lact{h}{e}=e$. Consequently any $\CM$-labeling of $P$ is fully determined by a pair $(h,e) \in H \times E$ such that
\begin{equation}\label{eq-Xi-labeling-lens}
\CM(e)=h^p \quad \text{and} \quad  \lact{h}{e}=e.
\end{equation}
Note that the gauge group of $P$ is $\mathcal{G}_P=H \times E$ with product
$$
(\lambda',\mu') \ast (\lambda,\mu)=\left(\lambda'\lambda, \, \mu \lact{(\lambda^{-1})}{\!\mu'}\right)
$$
and that it acts on $\CM$-labelings as:
$$
(\lambda,\mu) \cdot (h,e)=\left(\lambda h \CM(\mu) \lambda^{-1}, \,\lact{\lambda}{(e f_p)} \right) \quad \text{where} \quad
f_p= \left(\lact{(h^{-(p-1)})}{\mu}\right) \cdots \left(\lact{(h^{-1})}{\mu} \right) \mu.
$$

By Lemma~\ref{lem-gauge-group-labelings}, the $\CM$-labeling of $P$ induced by a pair $(h,e) \in H \times E$ satisfying~\eqref{eq-Xi-labeling-lens} determines a homotopy class $g_{h,e} \in [L(p,q), B\CM]$ which depends only on the $\mathcal{G}_P$-orbit of $(h,e)$. Moreover, since the region of $P$ is a disk, any $g \in [L(p,q), B\CM]$ is of the form $g=g_{h,e}$ for some $(h,e) \in H \times E$ satisfying \eqref{eq-Xi-labeling-lens}.

\subsection{Proof of Lemma~\ref{lem-gauge-group-labelings}}\label{sect-proof-lem-gauge-group-labelings}
For each (unoriented) edge~$e$ of~$P$, pick a $P$-ball branch~$B_e$ at $e$, an orientation for~$e$, and a disk $\delta_e$ embedded in~$M$ which bounds $\gamma_{(e,B_e)}$, intersects~$e$ transversely at a single point, and intersects $P$ transversely in arcs lying in the regions adjacent to~$e$. We endow $\delta_e$ with the basepoint given by the center~$c_e$ of $B_e$ and orient~$\delta_e$ so that its orientation followed by that of $e$ yields the given orientation of $M$. The condition $\delta_e$ bounds $\gamma_{(e,B_e)}$ means that $\partial \delta_e=\gamma_{(e,B_e)}$ as loops in $M$ based at $c_e$. (Here we endow $\partial \delta_e$ with the boundary orientation using the first outward pointing convention.)  For example, we have $\partial \delta_e=\gamma_1 \gamma_2^{-1} \gamma_3$ for the following trivalent edge~$e$ and $P$-ball branch $B_e$:
$$
\psfrag{1}[Bc][Bc]{\scalebox{.7}{$1$}}
\psfrag{2}[Bc][Bc]{\scalebox{.7}{$2$}}
\psfrag{3}[Bc][Bc]{\scalebox{.7}{$3$}}
\psfrag{M}[Br][Br]{\scalebox{.9}{$M$}}
\psfrag{a}[Bc][Bc]{\scalebox{.9}{$b_2$}}
\psfrag{b}[Bc][Bc]{\scalebox{.9}{$\gamma_2$}}
\psfrag{c}[Bc][Bc]{\scalebox{.9}{$\gamma_1$}}
\psfrag{k}[Bc][Bc]{\scalebox{.9}{$\gamma_3$}}
\psfrag{d}[Bc][Bc]{\scalebox{.9}{$b_1$}}
\psfrag{z}[Bc][Bc]{\scalebox{.9}{$b_3$}}
\psfrag{e}[Bc][Bc]{\scalebox{.9}{\textcolor{red}{$e$}}}
\psfrag{B}[Bl][Bl]{\scalebox{.9}{\textcolor{red}{$B_e$}}}
\psfrag{x}[Bl][Bl]{\scalebox{.9}{$c_e$}}
\psfrag{v}[Bc][Bc]{\scalebox{.9}{$\delta_e$}}
\rsdraw{.45}{.9}{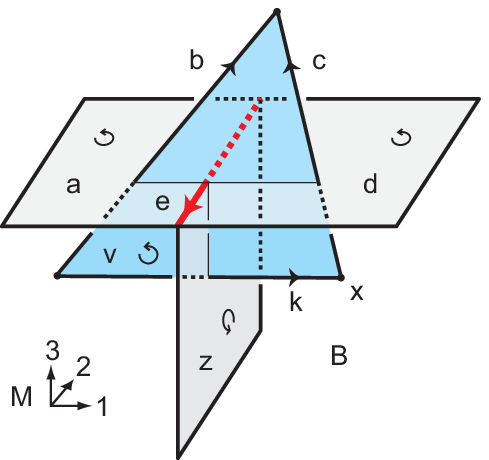} \; .
$$
Each region~$r$ of $P$ is a disk with holes and may be collapsed onto a wedge of circles $W_r \subset r \setminus \partial r$ based at the point $\alpha_r\cap r$. (Here, a circle corresponding to a hole $h$ is chosen to be isotopic to
$p_h (\partial h) p_h^{-1}$, where $p_h$ is a path from the point $\alpha_r\cap r$ to some vertex of $h$.)
Let $S_r$ be obtained from $W_r\times [-1,1]$ by contracting the sets ~$W_r\times \{-1\}$ and $W_r\times \{1\}$ into points which are called the summits of ~$S_r$. (If the two $P$-balls adjacent to~$r$ are equal, then we identify the two summits of $S_r$.)  We embed $S_r$ into $M$ as a union of two cones with base~$W_r$ and with summits
in the centers of the $P$-balls adjacent to $r$. We can assume that~$\gamma_r \subset S_r$ and that $S_r$ does not meet $S_{r'}$ for $r\neq r'$ except possibly  in the summits. Note that $S_r=\gamma_r$ when $r$ is a disk. Set $S=\cup_r S_r$. Then the pair $(M, S)$ has a relative CW-decomposition with only 2-cells and 3-cells given by the disks~$\{\delta_e\}_e$ and ball neighborhoods  in $M$ of the vertices of~$P$. Let $D=S\cup (\cup_e \delta_e)$ be the relative 2-skeleton of $(M, S)$ and  $C$ be the set of centers of $P$-balls.
Consider the following filtration of $M$:
$$
M_*=(C \subset S \subset D \subset M).
$$

By Section~\ref{sect-crossed-modules-classifying-spaces}, the classifying space $B\CM$ of $\CM$ has a canonical filtration $$(B\CM)_*=(\{x\}\subset BH \subset B\CM)$$
where $x$ is the unique 0-cell of $B\CM$ (serving as basepoint) and $BH$ is a subcomplex of $B\CM$ which is a classifying space of the group $H$, contains all 1-cells of $B\CM$, and whose associated boundary map $\partial \co \pi _{2}(B\CM,BH;x)\to \pi_{1}(BH,x)$ is the crossed module $\CM\co E \to H$.

We denote by $\FTop(M_*,(B\CM)_*)$ the set of filtered maps $M_*\to (B\CM)_*$, that is, of continuous maps $f \co M \to B\CM$ such that $f(C)=\{x\}$ and $f(S) \subset BH$.

\begin{claim}\label{claim-filtered-map-to-Xi-labelings}
Any $f \in \FTop(M_*,(B\CM)_*)$ defines a $\CM$-labeling $(\alpha_f,\beta_f)$ of $P$ such that:
\begin{itemize}
\item $\alpha_f(r)=f_*([\gamma_r]) \in \pi_1(BH,x)=H$ for every region $r$ of $P$,
\item $\beta_f(e,B_e)=f_*([\delta_e]) \in \pi_2(B\CM,BH;x)=E$ for every edge $e$ of $P$.
\end{itemize}
The induced map $f \in \FTop(M_*,(B\CM)_*)\mapsto (\alpha_f,\beta_f) \in \{\text{$\CM$-labelings of $P$}\}$ has a section $(\alpha,\beta) \mapsto f_{\alpha,\beta}$.
\end{claim}
\begin{proof}
For each (unoriented) edge~$e$ of~$P$, the above chosen orientation for~$e$ and $P$-ball branch $B_e$ at $e$ makes $(e,B_e)$ an element of $\EB(P)$. Denote by $\EB_0(P)$ the subset of $\EB(P)$ made of pairs $(e,B_e)$ as $e$ runs over all edges of $P$.
The definitions in the claim give well defined maps $\alpha_f \co \Reg(P) \to H$ and $\beta_f \co \EB_0(P) \to E$.
For any edge~$e$ of~$P$, recall that $\partial\delta_e=\gamma_{b_1}^{\varepsilon_e(b_1)} \cdots \gamma_{b_n}^{\varepsilon_e(b_n)}$
where $b_1< \cdots <b_n$ are the branches of $P$ at~$e$ enumerated in the linear order determined by $B_e$ (and the chosen orientation of $e$), and so
\begin{align*}
\CM(\beta_f(e,B_e))& =\CM(f_*([\delta_e]))=f_*([\partial\delta_e])
= \alpha_f(b_1)^{\varepsilon_e(b_1)} \cdots \alpha_f(b_n)^{\varepsilon_e(b_n)}.
\end{align*}
Thus \eqref{precol1} is satisfied for all elements in $\EB_0(P)$. Now, one easily verifies that any pair of maps $(\alpha\co \Reg(P) \to H,\beta \co \EB_0(P) \to E)$ satisfying \eqref{precol1} for all elements in $\EB_0(P)$ extends uniquely to a pre-$\CM$-labeling of $P$ (by extending $\beta$ to $\EB(P)$ using~\eqref{precol2} and \eqref{precol3}). Then $(\alpha_f,\beta_f)$ extends uniquely to a pre-$\CM$-labeling of $P$ also denoted by~$(\alpha_f,\beta_f)$. This pre-$\CM$-labeling is a $\CM$-labeling  since the restriction of $f$ to the boundary of any $P$-cone neighborhood of a vertex of $P$ (see Section~\ref{sect-skeletons}) is null-homotopic.

Let us prove the second assertion of the claim. For each region $r$ of $P$, the projection $W_r\times [-1,1]\to [-1,1]$ induces a retraction $S_r\to \gamma_r$. Composing with loops $\gamma_r \to BH$ representing the elements $\alpha(r)\in H=\pi_1(BH,x)$, we obtain   a map $S=\cup_r S_r\to BH \subset B\CM $ carrying the centers of the $P$-balls to the basepoint~$x$  of $B\CM$
and carrying  each $\gamma_r$ to a loop in $BH$ representing $\alpha(r)$.   Condition~\eqref{precol1} ensures that this map extends to the relative 2-skeleton $D=S\cup (\cup_e \delta_e)$ of $(M, S)$ by carrying each pointed disk $\delta_e$ to a disk in~$B\CM$ representing the element $\coleb(e, B_e)\in E=\pi _{2}(B\CM,BH;x)$.
The $1$-sphericality of  the $\CM$-graphs
associated with the vertices of $P$   implies that the restriction of this map to the boundary of each~$3$-cell is null-homotopic.    Therefore, there is a further extension to~$M$ producing a filtered map ~$f_{\alpha,\beta}\co  M_* \to (B\CM)_*$.  The conditions \eqref{precol2} and \eqref{precol3} ensure that the map
$f_{\alpha,\beta}$ does not depend on the choices of the adjacent $P$-ball branches and orientations for the edges of $P$. Finally, it follows directly from the definitions that $(\alpha_{f_{\alpha,\beta}},\beta_{f_{\alpha,\beta}})=(\alpha,\beta)$.
\end{proof}

Note that for any $\CM$-labeling $(\alpha,\beta)$ of $P$, the homotopy class  $[f_{\alpha,\beta}]\in [M,B\CM]$
of the map $f_{\alpha,\beta}$ defined in Claim~\ref{claim-filtered-map-to-Xi-labelings} depends only on $P$ and $(\colr,\coleb)$ since any two systems of arcs $\{\gamma_r\}_r$, disks $\{ \delta_e\}_e$, and suspensions $\{ S_r\}_r$ are  isotopic  in $M$.

\begin{claim}\label{claim-gauge-equiv-homotopic}
Two $\CM$-labelings $(\alpha,\beta)$ and $(\alpha',\beta')$ of $P$ are gauge equivalent if and only if their associated maps  $f_{\alpha,\beta}$  and $f_{\alpha',\beta'}$ are homotopic.
\end{claim}
\begin{proof}
Assume first that $f_{\alpha,\beta}$  and $f_{\alpha',\beta'}$ are homotopic. Set $I=[0,1]$ with its canonical filtration $I_*=(\{0,1\} \subset I)$. Let $G \co M \times I \to B\CM$ be a homotopy from $f_{\alpha',\beta'}$ to $f_{\alpha,\beta}$. Note that $G$ is filtered (where $M \times I$ is endowed with the product filtration) if and only if  and
$G(C \times I) \subset BH$. The surjectivity of the map $\pi_1(BH,x) \to \pi_1(B\CM,x)=\text{Coker}(\CM)$ implies that for each $c \in C$, the loop $G_{|\{c\} \times I}$ in $B\CM$ is homotopic relative $BH$ to a loop in $BH$.
Consequently, any homotopy among filtered maps $M_* \to (B\CM)_*$ can be realized by a filtered homotopy. Thus, without loss of generality, we can assume that $G$ is filtered.
For any $P$-ball $B$, let $c_B \in C$ be the center of $B$. The restriction of $\theta_B=G_{|\{c_B\}\times I}$ is a loop in $BH$ and represents a class $\lambda(B) \in \pi_1(BH,x)=H$. For any region~$r$ of~$P$, viewing the square $\gamma_r \times I$ (with product orientation) as a disk pointed at $(c_{r_+}, 1)$, the restriction $u_r=G_{|\gamma_{r}\times I}$ represents a class $\mu(r) \in \pi_2(B\CM, BH;x)=E$.  These assignments give rise to an element $(\lambda,\mu) \in \mathcal{G}_P$.
Let us check that $(\alpha',\beta')=(\lambda,\mu)\cdot (\alpha,\beta)$.
For any region $r$ of~$P$, we represent $u_r$ as
$$
\psfrag{e}[Bl][Bl]{\scalebox{.9}{$(c_{r_+}, 1)$}}
\psfrag{d}[Br][Br]{\scalebox{.9}{$\theta_{r_-}$}}
\psfrag{b}[Bl][Bl]{\scalebox{.9}{$\theta_{r_+}$}}
\psfrag{c}[Bc][Bc]{\scalebox{.9}{$a_r$}}
\psfrag{a}[Bc][Bc]{\scalebox{.9}{$a'_r$}}
\psfrag{x}[Bc][Bc]{\scalebox{.9}{\textcolor{mygreen}{$u_r$}}}
\rsdraw{.45}{.9}{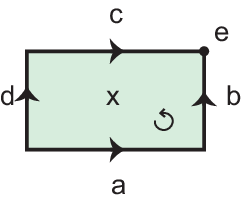}
$$
where $a_r=G_{|\gamma_r \times \{1\}}=(f_{\alpha,\beta})_{| \gamma_{r}}$ and $a'_r=G_{|\gamma_r \times \{0\}}=(f_{\alpha',\beta'})_{| \gamma_r}$ are loops in~$BH$.
By Claim~\ref{claim-filtered-map-to-Xi-labelings}, the homotopy classes of $a_r$ and $a'_r$ are $\alpha(r)$ and $\alpha'(r)$, respectively. Since the loop $\partial u_r=G_{|\partial(\gamma_r \times I)}$ is homotopic to the concatenation
$a_r^{-1}\theta_{r_-}^{-1}  a'_r \theta_{r_+}$, we obtain that
$\CM(\mu(r))=\alpha(r)^{-1} \lambda(r_-)^{-1} \alpha'(r) \lambda(r_+)$, and so \eqref{act-col1} is satisfied.
As explained in the proof of Claim~\ref{claim-filtered-map-to-Xi-labelings}, we only need to verify \eqref{act-col2} for $(e,B_e) \in EB(P)$ where $e$ is any edge of $P$. Recall that the disk $\delta_e$ is based at the center~$c_e$ of $B_e$ and that $\partial \delta_e=\gamma_{(e,B_e)}=\gamma_{r_1}^{\varepsilon_1} \cdots \gamma_{r_n}^{\varepsilon_n}$ with $r_i \in \Reg(P)$ and $\varepsilon_i\in \{+,-\}$. Consider the restrictions $b_e=G_{|\delta_e \times \{1\}}=(f_{\alpha,\beta})_{|\delta_e}$ and
$b'_e=G_{|\delta_e \times \{0\}}=(f_{\alpha',\beta'})_{|\delta_e}$.
Then $b_e$ is homotopic  (relative boundary) to
the restriction of $G$ to $ (\cup_{i=1}^n \gamma_{r_i} \times I)\cup(\delta_e \times \{0\})$ depicted as:
$$
\psfrag{e}[Bc][Bc]{\scalebox{.9}{$a_{r_2}$}}
\psfrag{z}[Bc][Bc]{\scalebox{.9}{\textcolor{mygreen}{$u_{r_2}$}}}
\psfrag{b}[Bc][Bc]{\scalebox{.9}{$a'_{r_2}$}}
\psfrag{x}[Bc][Bc]{\scalebox{.9}{\textcolor{blue}{$b'_e$}}}
\psfrag{n}[Bc][Bc]{\scalebox{.9}{\textcolor{mygreen}{$u_{r_1}$}}}
\psfrag{r}[Bc][Bc]{\scalebox{.9}{\textcolor{mygreen}{$u_{r_n}$}}}
\psfrag{t}[Bc][Bc]{\scalebox{.9}{$\theta_{B_e}$}}
\psfrag{a}[Br][Br]{\scalebox{.9}{$a'_{r_1}$}}
\psfrag{c}[Br][Br]{\scalebox{.9}{$a'_{r_n}$}}
\psfrag{d}[Bl][Bl]{\scalebox{.9}{$a_{r_1}$}}
\psfrag{s}[Bl][Bl]{\scalebox{.9}{$(c_e, 1)$}}
\psfrag{u}[Bl][Bl]{\scalebox{.9}{$a_{r_n}$}}
\rsdraw{.45}{.9}{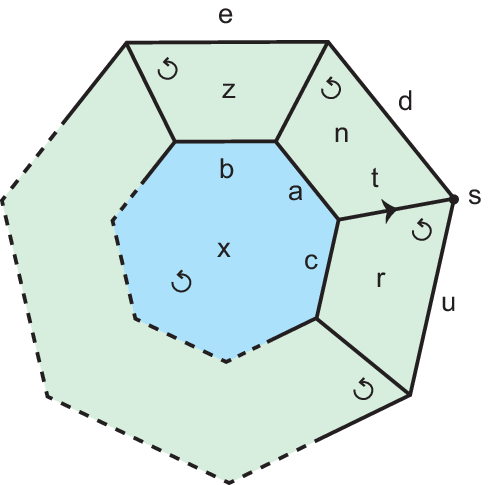} 
$$
Choosing $(c_e,1)$ as basepoint, they thus represent the same class in $\pi_2(B\CM,BH;x)$.
Now it follows from Claim~\ref{claim-filtered-map-to-Xi-labelings}, the definition of $\mu_\alpha$, and the definition of the action of $\pi_1(BH,x)$ on $\pi_2(B\CM,BH;x)$ that $b_e$ represents the class  $\beta(e,B_e)$ and the latter restriction of $G$ represents the class $\lact{\lambda(B_e)^{-1}}{\!\!\beta'(e,B_e)}\,\mu_\alpha(\gamma_{(e,B_e)})^{-1}$. Then the equality of these classes implies that \eqref{act-col2} is satisfied for $(e,B_e)$.
Hence $(\alpha',\beta')=(\lambda,\mu)\cdot (\alpha,\beta)$. In particular, $(\alpha,\beta)$ and $(\alpha',\beta')$  are gauge equivalent.

Conversely, assume that  $(\alpha,\beta)$ and $(\alpha',\beta')$ are gauge equivalent. Pick $(\lambda,\mu) \in \mathcal{G}_P$ such that $(\alpha',\beta')=(\lambda,\mu)\cdot (\alpha,\beta)$. For any $c \in C$, let $B_c$ be the $P$-ball whose center is~$c$. Pick a loop $\theta_c$ in $BH$ representing $\lambda(B_c) \in H=\pi_1(BH,x)$. For any region~$r$ of $P$, $a_r=(f_{\alpha,\beta})_{| \gamma_{r}}$ and $a'_r=(f_{\alpha',\beta'})_{| \gamma_r}$ are loops in~$BH$ whose homotopy classes are $\alpha(r)$ and $\alpha'(r)$, respectively (see Claim~\ref{claim-filtered-map-to-Xi-labelings}). Using \eqref{act-col1}, the homotopy class of the loop $a_r^{-1}(\theta_{c_{r_-}})^{-1}  a'_r \theta_{c_{r_+}}$ is
$$
\alpha(r)^{-1} \lambda(r_-)^{-1} \alpha'(r) \lambda(r_+)=\CM(\mu(r)).
$$
Thus, viewing the square $\gamma_r \times I$ as a disk pointed at $(c_{r_+}, 1)$, there is a map
$u_r \co \gamma_{r}\times I \to B\CM$ which represents $\mu(r) \in E=\pi_2(B\CM, BH;x)$ and such that
$$
(u_r)_{|\{c_{r_+}\}\times I}=\theta_{c_{r_+}}, \quad (u_r)_{|\{c_{r_-}\}\times I}=\theta_{c_{r_-}}, \quad (u_r)_{|\gamma_r\times \{1\}}=a_r, \quad (u_r)_{|\gamma_r\times \{0\}}=a'_r
$$
as depicted above. We now construct a (filtered) homotopy $f_{\alpha',\beta'}$ to $f_{\alpha,\beta}$ via successive extensions.
First, we define a map $G_1 \co (M \times \{0,1\}) \cup (S \times I) \to B\CM$ as follows: set  $(G_1)_{|M \times \{0\}}=f_{\alpha',\beta'}$, $(G_1)_{|M \times \{1\}}=f_{\alpha,\beta}$, and for any region $r$ of $P$,
let $(G_1)_{|S_r\times I}$ be the composition of the map $S_r \times I \to \gamma_r \times I$ (induced by the canonical retraction $S_r \to \gamma_r$ and $\id_I$) with the map $u_r$. For any edge $e$ of $P$, \eqref{act-col2} applied to $(e,B_e)$ implies that
the restriction of $G_1$ to the boundary of the 3-cell $\delta_e \times I$ is null-homotopic, and so $G_1$ can be extended to $\delta_e \times I$. Pick then an extension $G_2 \co (M \times \{0,1\}) \cup (D \times I) \to B\CM$ of $G_1$. Note that $G_2$ can further be extended to $M \times I$ since the target $B\CM$ is a $2$-type. Any such extension $G \co M \times I \to B\CM$ is then a (filtered) homotopy from $f_{\alpha',\beta'}$ to $f_{\alpha,\beta}$.
\end{proof}

Claims~\ref{claim-filtered-map-to-Xi-labelings} and \ref{claim-gauge-equiv-homotopic}
imply that the assignment $(\alpha,\beta) \mapsto f_{\alpha,\beta}$ induces a well-defined injective map $\{ \CM\text{-labelings of } P\}/\mathcal{G}_P \to [M , B\CM]$ which maps the gauge equivalence class of a  $\CM$-labeling $(\alpha,\beta)$ to the homotopy class of the map $f_{\alpha,\beta}$.

Finally, assume that the regions of $P$ are disks. Then the above filtration $M_*=(C \subset S \subset D \subset M)$ is the skeletal filtration of the CW-decomposition of~$M$ dual to~$P$ (since  $S=\cup_r \gamma_r$).
By the cellular approximation theorem, any map $f\co M \to B\CM$ is homotopic to a cellular map  $f'\co M \to B\CM$. Any such map $f'$ is then a filtered map $M_*\to (B\CM)_*$ since $f'(C)=\{x\}$ and $f'(S) \subset \{\text{1-cells of $B\CM$}\}\subset BH$. Thus any map $M \to B\CM$ is homotopic to a filtered one $M_*\to (B\CM)_*$. Also,  any filtered map $f \co M_*\to (B\CM)_*$ is homotopic to the map $f_{\alpha_f,\beta_f}$ (because the $\CM$-labeling $(\alpha_f,\beta_f)$ encodes the restriction of these maps on the 2-skeleton of $M$ and $B\CM$ is a 2-type). Consequently, the map $\{ \CM\text{-labelings of } P\}/\mathcal{G}_P \to [M , B\CM]$ is surjective.

\subsection{Remark}
As in the end of the above proof, assume that the regions of $P$ are disks. Let $\Pi(P)$ be the crossed complex associated with the skeletal filtration of the CW-decomposition of~$M$ dual to~$P$. Then $\CM$-labelings of $P$ correspond to crossed complex morphisms $\Pi(P) \to \CM$ (called formal maps in \cite{P}) and the last claim of Lemma~\ref{lem-gauge-group-labelings}  follows from the homotopy classification theorem \cite[Theorem~A]{BH1}.

\subsection{Closed $\CM$-manifolds and their $\CM$-skeletons}\label{sect-Xi-skeletons}
By a \emph{closed $\CM$-manifold}, we mean a pair $(M,g)$ where $M$ is a closed oriented 3-dimensional manifold and $g$ is  a homotopy class of maps $M\to B\CM$.

A \emph{$\CM$-skeleton}  of a closed $\CM$-manifold $(M,g)$
is a skeleton $P$ of $M$ endowed with a $\CM$-labeling of $P$ representing $g \in [M,B\CM]$.
For brevity, $\CM$-skeleton will be often denoted by the same letter $P$ as the underlying skeleton.  Lemma~\ref{lem-gauge-group-labelings} shows that any closed $\CM$\ti manifold has a $\CM$-skeleton (since any oriented closed 3-dimensional manifold always has a skeleton with disk regions, such as the 2-skeleton  of a triangulation).

\subsection{Moves on $\CM$-skeletons}\label{sect-Xi-moves}
Let $M$ be a closed oriented 3-dimensional manifold.  We consider four  local moves   $T_0, T_1, T_2, T_3$  on a skeleton~$P$   of~$M$ transforming~$P$ into   a new skeleton of~$M$ (see \cite{TVi5}).  Each of these moves  modifies~$P$ inside  a closed 3-ball in $M$ as   in Figure~\ref{fig-moves} where the shaded  horizontal  plane represents (a piece of)~$P$ and the upper half-space represents an adjacent $P$-ball branch $B$.  The edges shown  in the horizontal plane  before the moves    may be adjacent  to   regions of~$P$ lying in the lower half-space (and not shown in the picture).

The \emph{bubble  move} $T_0$  adds to~$P$ an embedded disk $D_+\subset M$ such that $\partial D_+=D_+\cap P \subset
P\setminus P^{(1)}$, the circle $\partial D_+$ bounds a disk~$D_-$
in $ P\setminus P^{(1)}$, and the 2-sphere $D_+\cup D_-$ bounds a
ball in $M$ meeting~$P$ precisely    along~$D_-$.   A  point of  the circle  $\partial D_+=\partial D_- $ is chosen as a vertex   of the   skeleton $T_0(P)=P\cup D_+$, and the circle itself is   an edge of $T_0(P)$. The disks $D_+$ and $D_-$ become  regions of   $T_0(P)$, and all other regions of $T_0(P)$ correspond bijectively to the regions of $P$.

The \emph{phantom edge  move}~$T_1$ keeps~$P$ as a polyhedron and  adds a new  2-valent   edge   meeting
$P^{(1)}$ solely at its endpoints which must be distinct vertices of~$P$.
The \emph{contraction move} $T_2$ collapses an edge of~$P$ with distinct endpoints into a point.
This move   is allowed only when at least one endpoint of the collapsing  edge is an endpoint of some other edge.   (The valence of the collapsing edge may be arbitrary.)
The \emph{percolation  move} $T_3$  pushes a branch  of~$P$
through a vertex~$x$ of~$P$. The branch  is pushed across a small
disk~$D$ lying in    another branch of~$P$ at~$x$  so that   these branches are adjacent to the same component of $M\setminus P$ and
$D\cap P^{(1)}= \partial D\cap P^{(1)}=\{x\}$.
The loop $\partial D$ based at~$x$ becomes an edge of the resulting 2-polyhedron.   The disk $D$ becomes a region of the resulting skeleton $T_3(P)$ and all other regions of $T_3(P)$ correspond bijectively to the regions of $P$ in the obvious way.

\begin{figure}[!ht]
\begin{center}
   \psfrag{T}[Bc][Bc]{\scalebox{.9}{$T_0$}}  \rsdraw{.45}{.9}{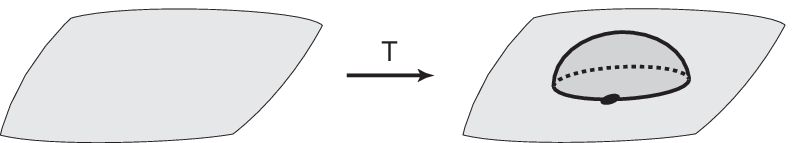}\\[1.6em]
   \psfrag{T}[Bc][Bc]{\scalebox{.9}{$T_1$}}  \rsdraw{.45}{.9}{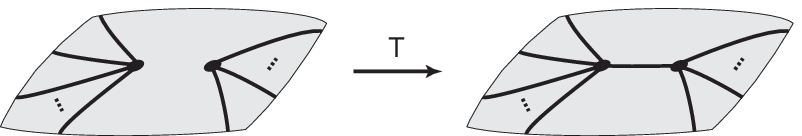} \\[1.6em]
   \psfrag{T}[Bc][Bc]{\scalebox{.9}{$T_2$}}  \rsdraw{.45}{.9}{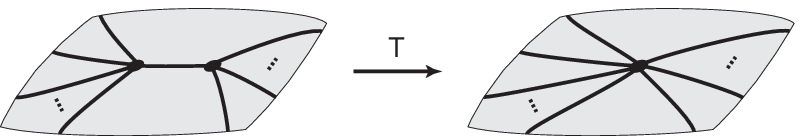} \\[1.6em]
   \psfrag{T}[Bc][Bc]{\scalebox{.9}{$T_3$}}  \rsdraw{.45}{.9}{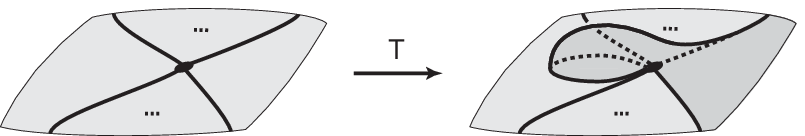}
\end{center}
\caption{Local moves on skeletons}
\label{fig-moves}
\end{figure}

In the  pictures of $T_0-T_3$, we
distinguish the \lq\lq small" regions and edges  entirely contained in the
3-ball where the move proceeds,  from  the \lq\lq big" regions and edges  not
entirely contained in  that  3-ball.  Recall that all regions of a skeleton are oriented. The     moves $T_0-T_3$ preserve orientation in the  big regions while orientation of the small disk regions created  by $T_0$ and $T_3$   may be arbitrary.

The  moves $T_0-T_3$  lift to $\CM$-labelings as follows. The $H$-labels of all big regions and $E$-labels of all big edges are preserved under the moves.
Under $T_0$, the $H$-label of $D_+$ and the $E$-label of the edge $\partial D_+$ (for some orientation of $\partial D_+$ and with respect to $B$) may be chosen arbitrarily, and then the $H$-label of $D_-$ is determined uniquely by \eqref{precol1}.
One easily verifies that different choices for the labels of $D_+$ and $\partial D_+$ give rise to gauge equivalent $\CM$-labelings.
Under $T_1$, the $E$-label of the added  2-valent edge should be  $1\in E$.
For $T_2$, there are no conditions on the $E$-label of the collapsed edge.
Under~$T_3$, the $E$-label of the edge $\partial D$ (for some orientation of $\partial D$ and with respect to $B$) may be chosen arbitrarily and the $H$-label of $D$ is then determined uniquely by \eqref{precol1}. Different choices for the $E$-label of $\partial D$ give rise to gauge equivalent $\CM$-labelings.
Note that the link graphs of the vertices created by these moves are indeed 1-spherical $\CM$-graphs. These moves lifted to $\CM$-labelings are called \emph{$\CM$-moves}.

The $\CM$-moves $T_0-T_3$ have inverses. The $\CM$-move  $T_0^{-1}$ is obvious from the picture of  $T_0$. The $\CM$-move $T_1^{-1}$ deletes a 2-valent edge with distinct endpoints which are also endpoints of some other edges. This move   is allowed only when the $E$-label (with respect to $B$) of the edge is $1$ and the orientations of the  two   regions adjacent to that edge are compatible.  The $\CM$-move $T_2^{-1}$ stretches a vertex into  an edge and the $E$-label (with respect to $B$) of this new edge is uniquely determined by the 1-sphericality of the $\CM$-graph associated to a created vertex.
The $\CM$-move $T_3^{-1}$ collapses an embedded disk region whose boundary is formed by a single 3-valent edge with coinciding endpoints.

The $\CM$-moves $T_0^{\pm 1}, T_1^{\pm 1}, T_2^{\pm 1}, T_3^{\pm 1}$ together with the label-preserving
ambient isotopies of $\CM$-skeletons in $M$ are called {\it primary $\CM$-moves}.

\begin{lem}\label{lem-Xi-moves}
Any primary $\CM$-move transforms a $\CM$-skeleton of a closed $\CM$-manifold $(M,g)$ into a $\CM$-skeleton of $(M,g)$.
Moreover, any two $\CM$-skeletons  of $(M,g)$ can be related by a finite sequence of primary $\CM$-moves.
\end{lem}

We prove Lemma~\ref{lem-Xi-moves} in the next section.

\subsection{Proof of Lemma~\ref{lem-Xi-moves}}\label{sect-proof-lem-Xi-moves}
If a skeleton $P$ of $M$ is transformed into a skeleton~$P'$ of $M$ via a move $T_i$ with $i\in\{0,2,3\}$, then any lift of this move to a $\CM$-move induces a bijection
$$
\{\text{$\CM$-labelings of $P$}\}_{/{\mathcal{G}_P}} \xrightarrow{{\; \cong \;}} \{\text{$\CM$-labelings of $P'$}\}_{/{\mathcal{G}_{P'}}}.
$$
Also, if a skeleton $P$ of $M$ is transformed into a skeleton $P'$ of $M$ via the move $T_1$ by adding an edge $e$ inside a region $r$ of $P$, then the lift of this move to a $\CM$-move (by labeling $e$ with $1\in E$) induces a bijection
$$
\{\text{$\CM$-labelings of $P$}\}_{/{\mathcal{G}_P}} \xrightarrow{{\; \cong \;}} \{\text{$\CM$-labelings of $P'$ with $e$ labeled by $1$}\}_{/{\mathcal{G}_{P'}}}.
$$
It is worth to notice that if adding the edge $e$ transforms $r$ into two distinct regions, then
$$
\{\text{$\CM$-labelings of $P'$ with $e$ labeled by $1$}\}_{/{\mathcal{G}_{P'}}} = \{\text{$\CM$-labelings of $P'$}\}_{/{\mathcal{G}_{P'}}}.
$$
The existence and injectivity of these maps follow from the fact that two $\CM$-labelings of $P$ give rise to gauge equivalent $\CM$-labelings of $P'$ if and only if they are gauge equivalent. The surjectivity of these maps follow from the existence of inverse $\CM$-moves.  Moreover, these bijections commutes with the maps to the set of homotopy classes of maps from $M$ to $B\CM$ given by Lemma~\ref{lem-gauge-group-labelings}. This implies the first assertion of the lemma.

To prove the second assertion of the lemma, we need further $\CM$-moves on $\CM$-skeletons.
The lune $\CM$-move  $\mathcal{L}$ is defined as:
$$
\psfrag{T}[Bc][Bc]{\scalebox{.9}{$\mathcal{L}$}}
\psfrag{L}[Bc][Bc]{\scalebox{.9}{$\mathcal{L}^{-1}$}}
 \rsdraw{.45}{.9}{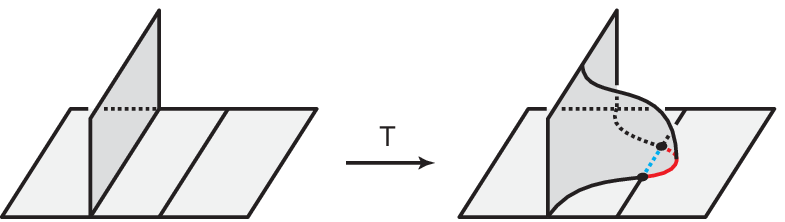}
$$
and, for any non-negative  integers $m, n$ with $m+n\geq 1$, the $\CM$-move $T^{m,n}$ is defined as:
$$
\psfrag{T}[Bc][Bc]{\scalebox{1}{$T^{m,n}$}}
\psfrag{E}[cr][cr]{\scalebox{1.5}[3.1]{$\{$}}
\psfrag{L}[cl][cl]{\scalebox{1.5}[4.4]{$\}$}}
\psfrag{u}[cc][cc]{\scalebox{.9}{$m$}}
\psfrag{v}[cc][cc]{\scalebox{.9}{$n$}}
\rsdraw{.175}{.9}{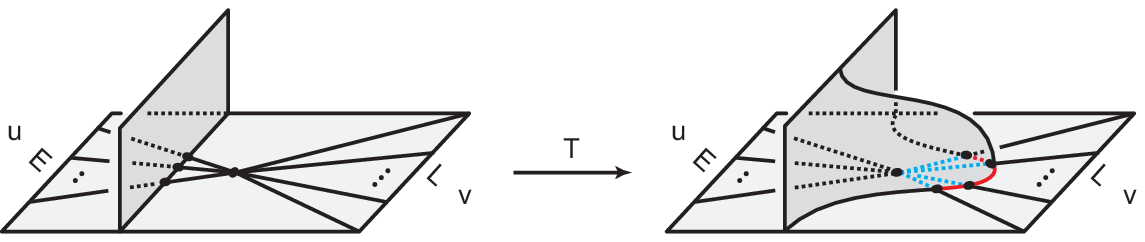}  \;.
$$
The $\CM$-move $\mathcal{L}$ creates a small disk region. The $\CM$-move $T^{m,n}$  destroys $\max(m-1, 0)$ small disk regions and creates $\max(n-1, 0)$ small disk regions. These $\CM$-moves preserve orientation in the  big regions. The orientation of a small
created disk is induced by the orientation of the corresponding big region before the move. As above, the $H$-labels of all big regions and $E$-labels of all big edges are preserved under the move. The  $E$-labels of the created edges depicted in red are all equal to $1\in E$. This uniquely determines the $H$-labels of the small created disk regions (using \eqref{precol1}) as well as the $E$-labels of the created edges depicted in blue (using the 1-sphericity of the $\CM$-graphs associated to the created vertices).
The $\CM$-move  $\mathcal{L}$ has an obvious inverse~$\mathcal{L}^{-1}$. It is allowed only
when the orientations and $H$-labels of two regions united under this move as well as the $E$-labels of two edges united under this move are compatible. The $\CM$-moves $\mathcal{L}^{\pm 1}$ and $T^{m,n}$ transform a $\CM$-skeleton of $(M,g)$ into a $\CM$-skeleton of $(M,g)$.  Indeed, they are compositions of the primary $\CM$-moves $T_1^{\pm 1}$, $T_2^{\pm 1}$, and $T_3$ for which the $E$-label of the created edge is $1\in E$.  (These compositions lift the decomposition of the corresponding moves on skeletons given in \cite[Section 11.3.2]{TVi5}.)

We also need the $\CM$-move $T'_0$ defined as:
$$
\psfrag{T}[Bc][Bc]{\scalebox{.9}{$T_0'$}}  \rsdraw{.45}{.9}{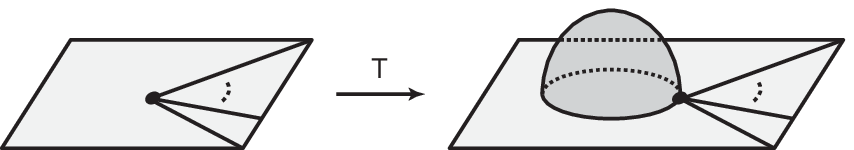}
$$
and viewed as the composition of the $\CM$-moves $T_0$, $T_1$, and then $T_2$ (first apply a bubble move to the left region, then connect the two vertices with an edge, and lastly contract this edge).  The inverse of $T'_0$ always exists. This follows from the fact that the 1-sphericality of the $\CM$-graph associated to the vertex of the right-hand side of the picture always implies the 1-sphericality of the $\CM$-graph associated to the vertex of the left-hand side.
As above, we will denote the moves on skeletons underlying the $\CM$-moves $T^{m,n}$ and $(T'_0)^{\pm 1}$ by the same letters.
\begin{claim}\label{claim-disk-region-skeletons}
Two skeletons of $M$ with disk regions 
can be related by a finite sequence of the moves $(T_0')^{\pm 1}$, T$_2^{\pm 1}$, and  $T_3^{\pm 1}$.
\end{claim}
\begin{proof}
Recall from \cite[Section 11.1.5]{TVi5} that an $s$-skeleton of $M$ is a skeleton $P$ of~$M$ such that any point of $P$ has an open neighborhood homeomorphic to an open subset of the set
$$
\{(x_1, x_2, x_3) \in \RR^3 \, \vert\, \,  x_3 = 0, \, {\text {or}}\,\,  x_1 = 0 \, \,{\text {and}}\,\,
   x_3 > 0, \, \,{\text {or}}\, \,x_2 =
0 \, \, {\text {and}}\, \, x_3 < 0\}
$$
depicted as
\begin{center}
\psfrag{a}[Bl][Bl]{\scalebox{1}{$x_1$}}
\psfrag{x}[Bc][Bc]{\scalebox{1}{$x_2$}}
\psfrag{e}[Bc][Bc]{\scalebox{1}{$x_3$}}
\rsdraw{.34}{.9}{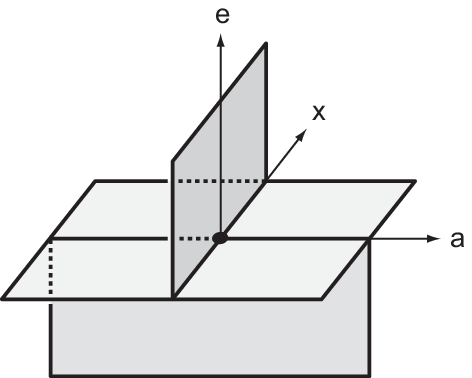} \quad\,.
\end{center}
The proof of \cite[Lemma 11.4]{TVi5} implies that any skeleton of~$M$ with disk regions can be turned into an $s$-skeleton of $M$ with disk regions by finitely many $(T_0')^{\pm 1}, T_2^{\pm 1}, T_3^{\pm 1}$-moves.
Next, by \cite[Theorem 1.2.30]{Mat1}, any two $s$-skeletons of~$M$ with disk regions can be related by finitely many $T^{1,2}$ and $T^{2,1}$-moves which can be expressed using $T_2^{\pm 1}$- and $T_3^{\pm 1}$-moves.
\end{proof}

\begin{claim}\label{claim-gauge-Xi-moves}
Two $\CM$-skeletons of $(M,g)$ with the same underlying skeleton of $M$ are related by a finite sequence of primary $\CM$-moves.
\end{claim}

\begin{proof}
Let  $(P,(\colr,\coleb))$ and $(P,(\colr',\coleb'))$ be $\CM$-skeletons of $(M,g)$ with the same underlying skeleton $P$ of $M$. By Lemma~\ref{lem-gauge-group-labelings}, the $\CM$-labelings $(\colr,\coleb)$ and $(\colr',\coleb')$ are gauge equivalent (since they represent the same homotopy class $g$).
Moreover, the gauge group $\mathcal{G}_P$ is generated by the set
$$
\bigl\{(\lambda_{B,h},\mu_{r,e})\, \big | \, \text{B is a $P$-ball, $r$ is a region adjacent to $B$, $h \in H$, $e \in E$}\bigr\},
$$
where $\lambda_{B,h} \co \pi_0(M\setminus P) \to H$ carries $B$ to $h$ and carries all other $P$-balls to $1$, while
$\mu_{r,e} \co \Reg(P) \to E$ carries $r$ to $e$ and carries all other regions to $1$. Thus, it is enough to show that there is a finite sequence of primary $\CM$-moves between $(P,(\colr,\coleb))$ and $(P,(\colr',\coleb'))$
when $(\colr',\coleb')= (\lambda_{B,h},\mu_{r,e}) \cdot (\colr,\coleb)$. To prove this, we
first apply a bubble move which adds to $P$ a 2-disk $D_+ \subset \overline B$  such that the circle  $\partial D_+$ lies in~$r$   and bounds a 2-disk $D_-\subset r$.
We endow $D_-$ with the orientation induced by that of $r$. We orient $D_+$ so that the orientation of $D_+$ followed by a normal vector pointing towards $B$ yields the given orientation of $M$. We endow the small region~$D_+$ with the $H$-label $h$. We endow the oriented edge $\partial D_-$ (relative to~$B$) with $E$-label $e$ if $r_+=B$ and with the $E$-label $\elact{\alpha(r)}{e}$ otherwise. The $H$-label of $D_-$, uniquely determined by \eqref{precol1}, is then equal to $\alpha'(r)$.
Next,  we isotop  the disk~$D_+$ in $B$ so that it  sweeps $B$ almost entirely   while its   boundary  slides along $\partial \overline B$. We arrange that in the terminal position $D'_{+}$ of the moving disk,  its boundary circle lies   in   $r \setminus
D_-$  and bounds there a   2-disk. This isotopy of $D_+$ can be performed using a finite sequence of the $\CM$-moves $\mathcal{L}^{\pm 1}$  and $T^{m,n}$. Next we remove the disk $D'_{+}$ using the inverse bubble $\CM$-move. This sequence transforms  the $\CM$-skeleton  $(P,(\colr,\coleb))$ into a $\CM$-skeleton $(P,(\colr'',\coleb''))$ of $(M,g)$. Finally, it follows from the choices of labelings of the first bubble $\CM$-move that $(\colr'',\coleb'')=(\lambda_{B,h},\mu_{r,e}) \cdot (\colr,\coleb)=(\colr',\coleb')$.
\end{proof}

Let us prove the second assertion of the lemma. Let $(P,(\colr,\coleb))$ and $(P',(\colr',\coleb'))$ be two $\CM$-skeletons of $(M,g)$.
If necessary, apply finitely many $T_1$-moves to transform non-disk regions of $P$ and $P'$ into disk regions, obtaining then skeletons~$Q$ and~$Q'$ of $M$ with disk regions. These $T_1$-moves lift to $\CM$-moves by labeling added edges with $1 \in E$ and we obtain $\CM$-skeletons $(Q,(\gamma,\delta))$ and $(Q',(\gamma',\delta'))$ of $(M,g)$. Note that these $\CM$-moves have inverses. In particular, there are  finite sequences of primary $\CM$-moves from $(P,(\colr,\coleb))$ to $(Q,(\gamma,\delta))$ and from $(Q',(\gamma',\delta'))$ to $(P',(\colr',\coleb'))$.
By Claim \ref{claim-disk-region-skeletons}, there is a finite sequence of moves $(T_0')^{\pm 1}$, $T_2^{\pm 1}$, and $T_3^{\pm 1}$ transforming $Q$ into $Q'$. Since all these moves can always be lifted to $\CM$-moves (because they transform a skeleton with disk regions into a skeleton with disk regions), the $\CM$-labeling $(\gamma,\delta)$ of $Q$ can be transported under this sequence (by lifting successively the moves to $\CM$-moves) to a $\CM$-labeling $(\gamma'',\delta'')$ of $Q'$. In particular, there is a sequence of primary $\CM$-moves between $(Q,(\gamma,\delta))$ and $(Q',(\gamma'',\delta''))$.
By the first assertion of the lemma, $(Q',(\gamma'',\delta''))$ and $(Q',(\gamma',\delta'))$ are $\CM$-skeletons of $(M,g)$.
Then Claim~\ref{claim-gauge-Xi-moves} implies that there is a finite sequence of primary $\CM$-moves between $(Q',(\gamma'',\delta''))$ and $(Q',(\gamma',\delta'))$. Hence, by composing the above sequences, we obtain a finite sequence of primary $\CM$-moves between $(P,(\colr,\coleb))$ and $(P',(\colr',\coleb'))$.

\section{State-sum invariants of  closed \texorpdfstring{$\CM$}{X}-manifolds}\label{sect-state-sum-invariants-closed}

We fix throughout this section a crossed module $\CM \co E \to H$, a  spherical $\CM$-fusion category $\cc$ (over~$\kk$), and a $\CM$-representative set $I=\amalg_{h\in H}\, I_h$ for $\cc$.
We   assume that the dimension $\dim(\cc_1^1)$ of the neutral component $\cc_1^1$ of the 1-subcategory $\cc^1$ of $\cc$ is invertible in the ground ring $\kk$ (see Section~\ref{sect-crossed-module-graded-fusion}). We derive from  this data a scalar topological invariant of closed $\CM$-manifolds (see Section~\ref{sect-Xi-skeletons}) defined as a state sum on $\CM$-skeletons.

\subsection{The state sum invariant}\label{sec-computat}
Let $(M,g)$ be   a   closed $\CM$-manifold.
Pick a $\CM$-skeleton~$P$ of~$M$ (see Section~\ref{sect-Xi-skeletons}). Denote by $(\colr,\coleb)$ the $\CM$-labeling of $P$ and recall from Section~\ref{sect-stratified-polyhedron} that $\Reg(P)$ is the (finite) set of regions of~$P$.

A \emph{coloring} of $P$ is a map $c \co \Reg(P) \to I$ such that $c(r)\in I_{\colr (r)}$ for all $r \in \Reg(P)$.
The object $c(r) \in I$ assigned to a region $r$ of $P$ is called the \emph{$c$-color} of $r$.
Note that there are finitely many colorings of $P$ (because the sets $I_h$ with $h \in H$ and $\Reg(P)$ are finite).
For a coloring $c$ of $P$, we set
\begin{equation}\label{dimcoloring}
\dim(c)=   \prod_{r \in \Reg(P)} (\dim c(r))^{\Euler(r)}      \in \kk,
\end{equation}
where    $\Euler(r) $ is the Euler characteristic of the region $r$. Note that $\dim(c)$ is well-defined since the dimension of a simple object in the 1-subcategory $\cc^1$ of $\cc$ is invertible (see Section~\ref{sect-crossed-module-graded-fusion}).

Next, we define    a    scalar   $\vert c \vert \in \kk$  for  any coloring $c$ of $P$.
For each oriented edge~$e$ of $P$, the coloring $c$ of $P$ turns the $\CM$-cyclic set $P_e$ of branches of $P$ at $e$ (see Section~\ref{sect-Xi-labelings}) into a $\CM$-cyclic $\cc$-set by assigning to each branch $b \in P_e$ the $c$-color of the region of~$P$ containing  $b$.  Let $H_c(e)=H (P_e)$  be its multiplicity   module (see Section~\ref{sect-muliplicity-modules}). It is a free \kt module of finite rank (by Lemma~\ref{lem-E-semisimple}).
Let
$$
H_c=   \underset{e}{\bigotimes} \, H_c(e)
$$
be the unordered tensor product of the \kt  modules $H_c(e)$ over all oriented edges  $e$  of~$P$.

Each unoriented edge  $e$  of  $P$ gives rise to two opposite oriented edges $e_1$, $e_2$ of~$P$. The  $\CM$-cyclic $\cc$-set $P_{e_1}$ and $P_{e_2}$ are dual: $P_{e_2}=(P_{e_1})^\opp$. This   yields a vector
\begin{equation*}
\ast_e =\ast_{P_{e_1}} \in  H_c(e_1)  \otimes H_c(e_2)
\end{equation*}
independent of the numeration of  $e_1$, $e_2$ (see Section~\ref{sect-muliplicity-modules}).
Set
\begin{equation*}
\ast_c = \otimes_e \, \ast_e\in H_c
\end{equation*}
where  $\otimes_e$ is the unordered tensor product   over all unoriented edges $e$ of $P$.

For a  vertex $x$ of $P$, consider the link graph $\Gamma_x \subset \partial B_x$ where $B_x\subset M$ is a $P$\ti cone neighborhood of $x$ (see Section~\ref{sect-skeletons}). Recall from Section~\ref{sect-Xi-labelings} that the $\CM$-labeling  of $P$ turns $\Gamma_x$ into a 1-spherical $\CM$-graph in $\partial B_x \cong S^2$. The coloring $c$ of~$P$ makes it $\cc$-colored: every edge of $\Gamma_x$ lies   in a region~$r$ of~$P$ and is colored with   $c(r) \in I$.
In this way, $\Gamma_x$  becomes  a 1-spherical $\cc$-colored $\CM$-graph  in~$\partial B_x$. It is    denoted $\Gamma_x^c$.
Since $\End_\cc^1(\un)=\kk$, Section~\ref{sect-Xi-graphs-spherical} yields a vector
\begin{equation*}
\inv_\cc (\Gamma^c_x)\in H(\Gamma_x^c)^\star
\end{equation*}
where we denote by $U^\star=\Hom_{\kk}(U, \kk)$ the dual of a \kt module $U$.
Note that the $\CM$-cyclic $\cc$-set associated with any vertex $v$ of $\Gamma_x^c$ is canonically isomorphic to the  $\CM$-cyclic $\cc$-set $P_{e}$, where $e$ is the edge of $P$ containing $v$ and oriented away from $x$. Therefore,  there are canonical isomorphisms
\begin{equation*}
H(\Gamma_x^c)   \simeq  \otimes_{e_x} \, H_c(e_x) \qquad \text{and} \qquad
H(\Gamma_x^c)^\star  \simeq  \otimes_{e_x} \, H_c(e_x)^\star,
\end{equation*}
where $e_x$ runs over all edges of $P$ incident to $x$ and oriented away from $x$. (An edge   with both endpoints in $x$ appears in each of these tensor products twice  with opposite orientations.)  The tensor product of the latter isomorphisms over all vertices $x$ of~$P$ yields a \kt linear isomorphism
$$
\bigotimes_x  H(\Gamma_x^c)^\star \simeq \bigotimes_x \bigotimes_{e_x}   H_c(e_x)^\star \simeq \bigotimes_e  H_c(e)^\star \simeq H_c^\star
$$
where $e$ runs over all oriented edges of $P$.
The image  under this isomorphism  of the unordered tensor product $\bigotimes_x \inv_\cc (\Gamma^c_x)$, where $x$ runs over all
vertices of~$P$, is a vector  $V_c \in H_c^\star$.
We evaluate  $V_c$ on   $\ast_c$ and set
$$
\vert c\vert=  V_c(\ast_c) \in \kk.
$$
Finally,   set
\begin{equation}\label{eq-simplstatesum+}
|M,g|_\cc =\dim (\cc_1^1)^{-\vert M\setminus P \vert} \sum_{c} \,\,  \dim (c) \,   \vert c\vert  \in \kk,
\end{equation}
where    $\vert M\setminus P \vert$  is the number of components of $M\setminus P$ and  $c$ runs over all colorings of~$P$.    The right-hand side of~\eqref{eq-simplstatesum+} is well-defined because there are finitely many colorings of $P$.

\begin{thm}\label{thm-state-3man-Xi}
$|M,g|_\cc$ is a topological invariant of the closed $\CM$-manifold $(M,g)$ independent of the choice    of~$P$ and~$I$.
\end{thm}

We prove Theorem~\ref{thm-state-3man-Xi} in Section~\ref{sect-proof-thm-state-3man-Xi} using moves on $\CM$-skeletons introduced in Section~\ref{sect-Xi-moves}. The invariant of Theorem~\ref{thm-state-3man-Xi} is nontrivial and may even distinguish homotopy classes of phantom maps (i.e., of maps inducing trivial homomorphisms on homotopy groups), see for instance Example~\ref{ex-nontrivial-inv} below.

Note that if $g$ is the homotopy class of a constant map, then $|M,g|_\cc= |M|_{\cc^1_1}$ where the right hand side is the Turaev-Viro-Barrett-Westbury state sum invariant of the oriented closed 3-dimensional manifold $M$ derived from the spherical fusion category $\cc_1^1$ (see \cite{BW,TVi5}). Moreover Theorem~\ref{thm-state-3man-Xi} generalizes  the construction of \cite{TVi1} which produces an invariant
of closed oriented 3-dimensional manifolds endowed with a homotopy class of maps to a $K(G,1)$ space, where $G$ is a group, from any spherical $G$-fusion category (over $\kk$) whose neutral component has invertible dimension in $\kk$. Indeed, the classifying space $B\CM$ is aspherical if and only if~$\CM$ is injective. And if $\CM\co E \to H$ is injective (for example $E=1$), then $B\CM$ is a $K(G,1)$ space with $G=\Coker(\CM)=H/\Ima(\CM)$, the category $\cc$ is a $G$-fusion category (in the sense of \cite{TVi1}) whose neutral component has invertible dimension $\dim(\cc_1^1)$ (see Remark~\ref{rk-case-injective-G-fusion}), and the invariant of closed $\CM$-manifolds from \cite{TVi1} defined using this $G$-fusion category is equal to the invariant $|\cdot |_\cc$ of Theorem~\ref{thm-state-3man-Xi}.

We illustrate Theorem~\ref{thm-state-3man-Xi} with the following two examples.

\subsection{Example}
Consider the closed 3-dimensional manifold $S^1 \times S^2$ with the product orientation. Pick   a  point  $p\in
S^1$ and an embedded loop $\gamma\subset S^2$. The 2\ti polyhedron $P=(\{p\} \times S^2 ) \cup (S^1\times \gamma)$ has two disk regions and one annulus region. Stratified by the based loop $\{p\}\times \gamma$, the polyhedron $P$ is a skeleton of $S^1\times S^2$ with one vertex, one edge, and two $P$-balls.
Any $\CM$-labeling of~$P$ is  fully determined by a tuple $\ell=(x,y,z,e) \in H^3 \times E$ such that $\CM(e)=z^{-1}xzy$.
Here $x,y$ are the $H$-labels of the disk regions,  $z$ is the $H$-label of the annulus region, and~$e$ is the $E$-label of the edge of $P$ with respect to the front $P$-ball branch and for the choices of orientations as in the next picture which represents a neighborhood of the  vertex of~$P$:
$$
\psfrag{1}[Bc][Bc]{\scalebox{.7}{$1$}}
\psfrag{2}[Bc][Bc]{\scalebox{.7}{$2$}}
\psfrag{3}[Bc][Bc]{\scalebox{.7}{$3$}}
\psfrag{x}[Bc][Bc]{\scalebox{.9}{$x$}}
\psfrag{y}[Bc][Bc]{\scalebox{.9}{$y$}}
\psfrag{z}[Bc][Bc]{\scalebox{.9}{$z$}}
\psfrag{e}[Bc][Bc]{\scalebox{.9}{$e$}}
\rsdraw{.45}{.9}{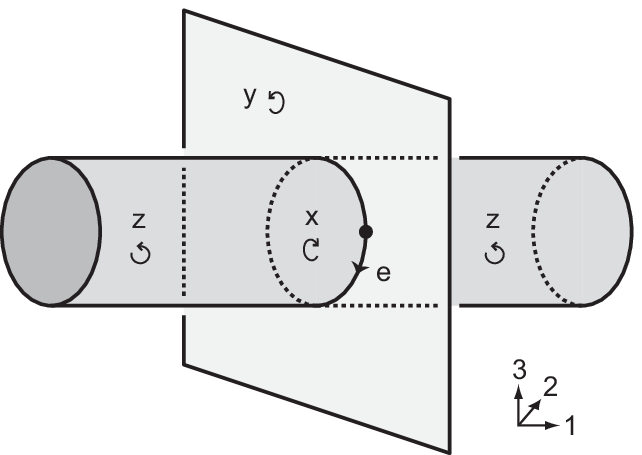} \; \;.
$$
By Lemma~\ref{lem-gauge-group-labelings}, this
$\CM$-labeling determines a homotopy class $g_\ell\in [S^1 \times S^2,B\CM]$. Then
\begin{gather*}
|S^1 \times S^2,g_\ell|_\cc \overset{(i)}{=}\dim (\cc_1^1)^{-2} \hspace*{-1em} \sum_{i\in I_x,\, j \in I_y, \,k \in I_z}   \hspace*{-1em} \dim (i) \dim(j) N^{\un,e}_{k^*\otimes i \otimes k \otimes j} \\
\overset{(ii)}{=}\dim (\cc_1^1)^{-1}  \sum_{k \in I_z} \dim(k)^2\overset{(iii)}{=}1_\kk.
\end{gather*}
Here $(i)$ follows from \eqref{eq-simplstatesum+}, \eqref{eq-compute-inverse-s}, and \eqref{eq-Nxi},   $(ii)$  from Lemma~\ref{lem-bubbleidentity} (with $g=x$, $h=y$, $X=k^*$, $Y=k$, $m=i$, $n=j$), and $(iii)$ from \eqref{eq-computation-dimension-neutral-component}.

\subsection{Example}
Consider the lens space $L(p,q)$ where  $p,q$ are coprime integers with $1 \leq q < p$. Recall from Example~\ref{ex-lens-spaces-labelings} that any $g\in[L(p,q),B\CM]$ is of the form $g=g_{h,e}$ for some $(h,e) \in H \times E$ satisfying $\CM(e)=h^p$ and $\lact{h}{e}=e$. By applying Formula~\eqref{eq-simplstatesum+} to the $\CM$-skeleton $P$ of $(L(p,q),g)$ induced by $(h,e)$ as in Example~\ref{ex-lens-spaces-labelings}, we obtain
$$
|L(p,q),g|_\cc=\dim (\cc_1^1)^{-1} \sum_{i\in I_h} \,\,   \dim (i)  \,
   \Tr \bigl((\sigma_{i})^q\bigr)\in \kk,
$$
where $\sigma_{i}$ is the \kt linear endomorphism of $\Hom_\cc^e(\un, i^{\otimes p})$ defined by
$$
\psfrag{i}[Bl][Bl]{\scalebox{.8}{$i$}}
\psfrag{u}[Bc][Bc]{\scalebox{1}{$u$}}
\psfrag{e}[Br][Br]{\scalebox{.8}{$i^{\otimes (p-1)}$}}
u \mapsto \sigma_{i}(u) \;=\;\rsdraw{.45}{.9}{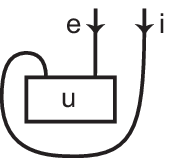}
$$
and $\Tr$ is the standard trace of \kt linear endomorphisms. This follows from the description of the $\CM$-graph $\Gamma$ associated with the single vertex of $P$ given in Example~\ref{ex-lens-spaces-labelings}, the definition of the invariant $\inv_\cc(\Gamma)$, and Formula~\eqref{eq-compute-inverse-s} which provides a computation of the contraction vector associated with the single edge of $P$.

\subsection{A conjecture}
Consider the spherical $\CM$-fusion category $\kk\tg_\CM$ of Example~\ref{ex-linearized-crossed-module-fusion}. It follows directly from the definitions that for any closed $\CM$-manifold $(M,g)$,
\begin{equation}\label{eq-case-G-Xi}
  |M,g|_{\kk\tg_\CM}=1_\kk.
\end{equation}
More generally, consider the spherical $\CM$-fusion category $\kk\tg_\CM^\omega$ obtained by twisting the $\CM$-category $\kk\tg_\CM$ with a normalized 3-cocycle $\omega$ for $\CM$ with values in $\kk^*$ (see Section~\ref{sect-twisting-fusion-cocycles}). Let $[\omega] \in H^3(B\CM,\kk^*)$ be the cohomology class induced by $\omega$. We conjecture that for any  closed $\CM$-manifold $(M,g)$,
\begin{equation}\label{conj-case-G-Xi-twisted}
  |M,g|_{\kk\tg_\CM^\omega}=\langle g^*([\omega]), [M] \rangle,
\end{equation}
where $g^*([\omega])  \in H^3(M,\kk^*)$ is the pullback of $[\omega]$ along $g$, $[M] \in H_3(M,\ZZ)$ is the fundamental class of $M$, and $\langle \,,\, \rangle \co  H^3(M,\kk^*) \times H_3(M,\ZZ) \to \kk$ is the Kronecker pairing. Note that the computations in \cite[Appendix H]{TVi5} imply that this conjecture is true when $E=1$ (and so $\omega$ is a 3-cocycle for $H$), that is, for the crossed modules of the form $1 \to H$ with $H$ a group.

\subsection{Invariants of 3-manifolds}
Let $M$ be a closed oriented 3-dimensional manifold. It follows from Theorem~\ref{thm-state-3man-Xi} that if $\mathfrak{a}=(\mathfrak{a}_g)_{g \in [M, B\CM]}$ is a family of elements of~$\kk$ such that $\mathfrak{a}_g=0$ for all but finitely many $g$, then the linear combination
$$
I_{\cc,\mathfrak{a}}(M)= \!\!\sum_{g \in [M, B\CM]} \!\!\! \mathfrak{a}_g \, |M,g|_\cc \in \kk
$$
is a topological invariant of $M$.

Examples of such families of coefficients can for instance be derived from the mapping space $\TOP(M,B\CM)$  of continuous maps from $M$ to $B\CM$ (endowed with the compact-open topology). In particular, assume that the crossed module $\CM\co E \to H$ is finite (meaning that both $E$ and $H$ are finite), so that the set $[M, B\CM]=\pi_0( \TOP(M,B\CM))$ and the homotopy groups of $\TOP(M,B\CM)$ are finite. Then, assuming also that $\kk$ is a field of characteristic zero, we can define the family of coefficients  $\mathfrak{y}=(\mathfrak{y}_g)_{g \in [M, B\CM]}$  by setting
$$
\mathfrak{y}_g=\dfrac{\card\bigl(\pi_2(\TOP(M,B\CM),g)\bigr)}{\card\bigl(\pi_1(\TOP(M,B\CM),g)\bigr)} \in \kk.
$$
For the spherical $\CM$-fusion category $\kk\tg_\CM$ of Example~\ref{ex-linearized-crossed-module-fusion}, we obtain
$$
I_{\kk\tg_\CM,\mathfrak{y}}(M)=\Yet_\CM(M),
$$
where $\Yet_\CM(M)$ is the Yetter invariant of $M$ associated with the finite crossed module $\CM$ (see \cite{Yet}). This follows from \eqref{eq-case-G-Xi} and the interpretation of the Yetter invariant given in \cite[Theorem 2.25]{FP}.
More generally, for the spherical $\CM$-fusion category $\kk\tg_\CM^\omega$  obtained by twisting the $\CM$-category $\kk\tg_\CM$ with a normalized 3-cocycle~$\omega$ for $\CM$ with values in $\kk^*$ (see Section~\ref{sect-twisting-fusion-cocycles}) and provided Conjecture \eqref{conj-case-G-Xi-twisted} is true, we obtain
$$
I_{\kk\tg_\CM^\omega, \mathfrak{y}}(M)=\Yet_{\CM,\omega}(M),
$$
where $\Yet_{\CM,\omega}(M)$ is the $\omega$-twisted Yetter invariant of $M$ defined by Faria Martins and Porter in \cite{FP}.

\subsection{Proof of   Theorem~\ref{thm-state-3man-Xi}}\label{sect-proof-thm-state-3man-Xi}
The state sum $|M|_\cc$ does not depend on the choice of the $\CM$-representative set~$I$   by the naturality of
$\inv_{\cc}$ and of the contraction vectors.
Denote the right hand side of \eqref{eq-simplstatesum+} by~$|P|_\cc$. Any $\CM$-homeomorphism of $\CM$-manifolds $M \to M'$ carries a $\CM$-skeleton $P $ of~$M$ to a $\CM$-skeleton~$P'$  of~$M'$ and,  clearly,   $|P|_\cc=|P'|_\cc$.   Thus, by Lemma~\ref{lem-Xi-moves},   we need only to prove the invariance of $|P|_\cc$ under the primary $\CM$-moves $P \mapsto T_s(P)$ for all $s \in \{0,1,2,3\}$.  This follows from the \lq\lq local invariance" which says that the contribution of any coloring $c$ of $P$ to the state sum is equal to the sum of the contributions of all colorings of $T_s(P)$ which are equal to~$c$ on all big regions.

Let us verify the local invariance for the bubble $\CM$-move~$T_0$. Denote by $x$  the created vertex, by
$r$ the region of $P$ where the bubble is attached, by $D_\pm$ the disks created by the move:
$$
\psfrag{1}[Bc][Bc]{\scalebox{.7}{$1$}}
\psfrag{2}[Bc][Bc]{\scalebox{.7}{$2$}}
\psfrag{3}[Bc][Bc]{\scalebox{.7}{$3$}}
\psfrag{M}[Br][Br]{\scalebox{.9}{$M$}}
\psfrag{i}[Bc][Bc]{\scalebox{.9}{$r$}}
\psfrag{x}[Bc][Bc]{\scalebox{.9}{$x$}}
\psfrag{m}[Bc][Bc]{\scalebox{.85}{$D_+$}}
\psfrag{n}[Bc][Bc]{\scalebox{.85}{$D_-$}}
\rsdraw{.35}{.9}{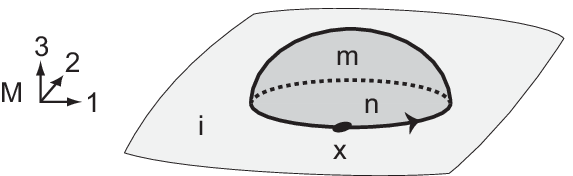}.
$$
Let $\alpha,\beta,\gamma \in H$ be the $H$-labels and $i,k,\ell \in I$ be the colors of $r,D_+,D_-$, respectively. Let $e \in E$ be the $E$-label of the created edge endowed with the depicted orientation and with respect to the (small) $P$-ball bounded by $D_+\cup D_-$.
Condition~\eqref{precol1} becomes $$\CM(e)=\beta^{\nu} \alpha^\varepsilon \gamma^{\mu},$$ where $\varepsilon,\nu,\mu$ are signs depending on the orientations of the regions $r,D_+,D_-$, respectively. The $\cc$-colored $\CM$-graph in $S^2$ associated with $x$ is
$$
\psfrag{e}[Br][Br]{\scalebox{.9}{$e$}}
\psfrag{a}[Bl][Bl]{\scalebox{.9}{$e^{-1}$}}
\psfrag{i}[Bc][Bc]{\scalebox{.9}{$i$}}
\psfrag{k}[Bc][Bc]{\scalebox{.9}{$k$}}
\psfrag{l}[Bc][Bc]{\scalebox{.9}{$\ell$}}
\Gamma=\,\rsdraw{.45}{.9}{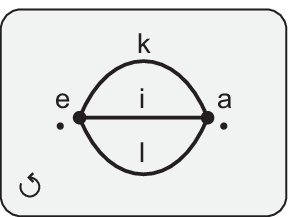} \,\subset \R^2 \cup \{\infty \}=S^2,
$$
where the edge colored by $k$  (respectively, $i,\ell$) is oriented towards the left vertex if the sign $\nu$ (respectively, $\varepsilon,\mu$) is positive and towards the right vertex otherwise.
The contraction vector associated to the edge created by the move is the contraction vector $\ast$ determined   by the duality between the two vertices of $\Gamma$  induced   by the symmetry   with respect to  a vertical great circle. Now, using \eqref{eq-compute-inverse-s}  and \eqref{eq-Nxi}, we obtain:
\begin{gather*}
\psfrag{e}[Bc][Bc]{\scalebox{.9}{$e$}}
\psfrag{i}[Bl][Bl]{\scalebox{.9}{$i$}}
\psfrag{k}[Bl][Bl]{\scalebox{.9}{$k$}}
\psfrag{l}[Bl][Bl]{\scalebox{.9}{$\ell$}}
\inv_\cc(\Gamma) (\ast) = \;\rsdraw{.45}{.9}{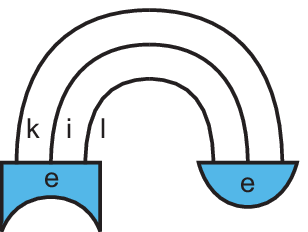} \;
=\;\rsdraw{.45}{.9}{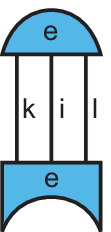} \,\;
= N^{\un,e}_{k^{\nu} \otimes i^\varepsilon \otimes \ell^{\mu}}.
\end{gather*}
Then   the contribution of~$x$ to the state sum is
\begin{gather*}
\sum_{\substack{k \in I_{\beta}\\ \ell \in I_{\gamma}}}\dim(k) \dim(\ell) \,  \inv_\cc(\Gamma) (\ast)
\overset{(i)}{=} \sum_{\substack{k \in I_{\beta}\\ \ell \in I_{\gamma}}}\dim(k) \dim(\ell) \, N^{\un,e}_{k^{\nu} \otimes i^\varepsilon \otimes \ell^{\mu}}\\
\overset{(ii)}{=} \sum_{\substack{m \in I_{\beta^{\nu}} \\ n \in I_{\gamma^{\mu}}}}\dim(m) \dim(n) \, N^{\un,e}_{m \otimes i^\varepsilon \otimes n}
\overset{(iii)}{=} \dim(i^\varepsilon)\dim(\cc^1_1) \overset{(iv)}{=} \dim(i) \dim(\cc^1_1).
\end{gather*}
Here  $(i)$ follows from the previous computation,
$(ii)$ and $(iv)$ from the equality $\dim(X^*)=\dim(X)$ for any $X \in \cc$, and $(iii)$ from Lemma~\ref{lem-bubbleidentity} applied with $X=\un$ and $Y=i^\varepsilon$. The   factor  $\dim (i) \dim (\cc_1^1) $ is compensated by the change  in the Euler characteristic of the big region and in
the number of components of $M\setminus P$.

Let us verify the local invariance for the phantom edge $\CM$-move~$T_1$.
This move adds a new edge $e$ connecting  two distinct vertices  of~$P$. This modifies the  link graphs of these vertices  by adding a new vertex $u$ (respectively,~$v$) inside an edge.   The $E$-label of $e$ is $1 \in E$ and so, by \eqref{precol1}, the (possibly coinciding) regions $r$ and~$r'$ of~$T_1(P)$ lying on  the  two sides of $e$ have the same $H$-label $h \in H$. The colorings of~$T_1(P)$ assigning different colors to the
regions $r$ and $r'$ contribute zero to the state sum  (because $\Hom_\cc^1(i,j)=0$ for distinct $i,j \in I$). The colorings of~$T_1(P)$ assigning the same color $i\in I$ to the
regions $r$ and $r'$   contribute the same as the colorings of~$P$ assigning~$i$ to the region of~$P$
containing~$e$. Indeed,
\begin{gather*}
 \psfrag{i}[Bc][Bc]{\scalebox{.9}{$i$}}
\psfrag{e}[Br][Br]{\scalebox{.9}{$1$}}
\psfrag{a}[Bl][Bl]{\scalebox{.9}{$1$}}
\inv_\cc \!\left (\, \rsdraw{.45}{.9}{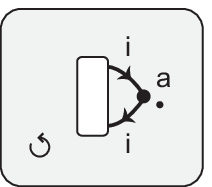}\,\right )
 \otimes \,
 \inv_\cc \!\left (\, \rsdraw{.45}{.9}{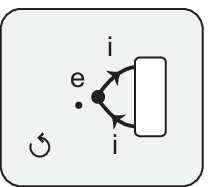}\,\right ) \, (*_e) \\[.4em]
= \dim(i)^{-1} \;
\psfrag{i}[Bl][Bl]{\scalebox{.9}{$i$}}
\inv_\cc \!\left (\, \rsdraw{.45}{.9}{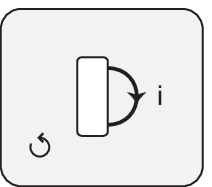}\,\right )
\otimes\,
\psfrag{i}[Br][Br]{\scalebox{.9}{$i$}}
\inv_\cc \!\left (\, \rsdraw{.45}{.9}{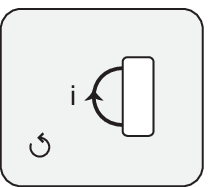}\,\right )
\end{gather*}
and the factor $\dim(i)^{-1}$ is compensated by the change in the Euler  characteristics of the regions.
The latter equality follows from the fact that $\ast_e$ is induced by the inverse $\Omega$ of the pairing
$$
\omega \co H_i \otimes_\kk H_i \to \kk, \quad u \otimes_\kk v \mapsto
  \psfrag{u}[Bc][Bc]{\scalebox{.9}{$u$}}
  \psfrag{v}[Bc][Bc]{\scalebox{.9}{$v$}}
  \psfrag{i}[Bc][Bc]{\scalebox{.9}{$i$}}
  \rsdraw{.45}{.9}{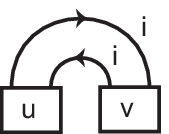}
$$
where $H_i=\Hom_\cc^1(\un,i^* \otimes i)$, and that $\Omega\co \kk \to H_i \otimes_\kk H_i$ is computed by
$$
\Omega(1_\kk)= \dim(i)^{-1} \, \rcoev_i \otimes_\kk \rcoev_i
$$
because  $H_i \simeq \End_\cc^1(i) $ is free \kt module of rank 1 and
\begin{equation*}
\omega (\rcoev_i \otimes_\kk \rcoev_i)= \,
\psfrag{i}[Br][Br]{\scalebox{.8}{$i$}} \rsdraw{.4}{.9}{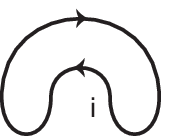}
\,=\, \psfrag{u}[Bl][Bl]{\scalebox{.8}{$i$}} \rsdraw{.4}{.9}{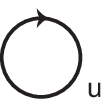}
\!=\dim(i).
\end{equation*}

Let us verify the local invariance for the contraction $\CM$-move~$T_2$.
This move collapses an edge of~$P$ between distinct vertices $x,x'$ of $P$ into a vertex $y$ of $T_2(P)$. We can assume that $x$  is an endpoint of some other edge. Let $e \in E$ be the $E$-label of the collapsed edge oriented from $x$ to $x'$ and with respect to the $P$-ball branch represented by the upper half-space. Note that the colorings of $P$ and $T_2(P)$ are the same. Given such a coloring, the $\cc$-colored $\CM$-graphs $\Gamma_x,\Gamma_{x'}$ in $S^2$ associated with $x,x'$ are of the following form:
$$
\psfrag{u}[Bl][Bl]{\scalebox{.9}{$e$}}
\psfrag{v}[Bl][Bl]{\scalebox{.9}{$e^{-1}$}}
\Gamma_x=\, \rsdraw{.45}{.9}{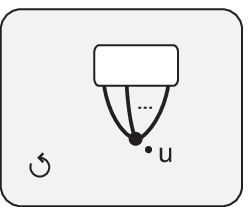} \quad \text{and} \quad \Gamma_{x'}=\, \rsdraw{.45}{.9}{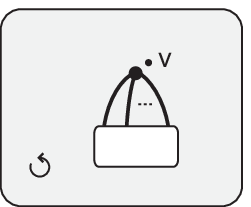}
$$
where  white boxes  stand  for  a    piece   of a $\cc$-colored $\CM$-graph and the depicted vertices are in duality (induced by the collapsed edge). Denote by $\ast$ the contraction vector determined  by this duality.
The $\cc$-colored $\CM$-graph $\Gamma_y$ in $S^2$ associated with $y$ is obtained from $\Gamma_x$ and $\Gamma_{x'}$ by joining their dual vertices:
$$
\psfrag{u}[Bl][Bl]{\scalebox{.9}{$e$}}
\psfrag{v}[Bl][Bl]{\scalebox{.9}{$e^{-1}$}}
\Gamma_y=\, \rsdraw{.45}{.9}{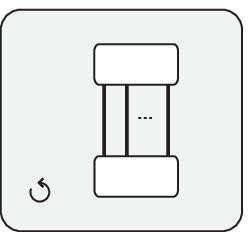}\,\;.
$$
Then the contribution of $x,x'$ to the state sum $|P|_\cc$ is
\begin{gather*}
\left (\inv_\cc(\Gamma_x) \otimes \inv_\cc(\Gamma_{x'}) \right) (\ast)
\overset{(i)}{=}
\psfrag{u}[Bl][Bl]{\scalebox{.9}{$e$}}
\psfrag{v}[Bl][Bl]{\scalebox{.9}{$e^{-1}$}}
\inv_\cc \left (\,
\rsdraw{.45}{.9}{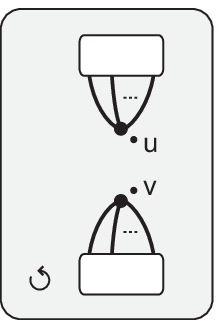}
\,\right )  (\ast ) \\[.4em]
\overset{(ii)}{=}
\sum_{i \in I_1} \dim(i) \;
\psfrag{i}[Br][Br]{\scalebox{.9}{$i$}}
\psfrag{u}[Bl][Bl]{\scalebox{.9}{$e$}}
\psfrag{v}[Bl][Bl]{\scalebox{.9}{$e^{-1}$}}
\inv_\cc \left (\,
\rsdraw{.45}{.9}{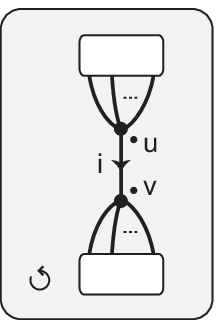}
\,\right )  (\ast )
\overset{(iii)}{=}
 \inv_\cc(\Gamma_y)
\end{gather*}
and thus is equal to the contribution of $y$ to the state sum $|T_3(P)|_\cc$.
Here  $(i)$ follows from the $\otimes$-multiplicativity of $\inv_\cc$ (see Section~\ref{sect-inv-Xi-graphs}),
$(ii)$ from the facts that the $\CM$-graph $\Gamma_x$ is 1-spherical, $\Hom_\cc^1(i,\un)=0$
for all $i \in I_1$ distinct from~$\un$, and $\dim(\un)=1_\kk$, and $(iii)$ from Lemma~\ref{lem-calc-diag} since the element $h$ in that lemma is equal to $\CM(e)$ by~\eqref{precol1}.

Finally,  let us verify the local invariance for the percolation $\CM$-move~$T_3$.
This move  pushes a branch $b$ of~$P$ through a vertex~$x$ of $P$ and  across a small disk~$D$ lying in    another branch $b'$ of~$P$ at~$x$. Denote by $y$ the vertex of $T_3(P)$ corresponding to~$x$:
$$
\psfrag{1}[Bc][Bc]{\scalebox{.7}{$1$}}
\psfrag{2}[Bc][Bc]{\scalebox{.7}{$2$}}
\psfrag{3}[Bc][Bc]{\scalebox{.7}{$3$}}
\psfrag{M}[Br][Br]{\scalebox{.9}{$M$}}
\psfrag{b}[Bc][Bc]{\scalebox{.9}{$b$}}
\psfrag{d}[Bc][Bc]{\scalebox{.9}{$b'$}}
\psfrag{x}[Bc][Bc]{\scalebox{.9}{$y$}}
\psfrag{D}[Bc][Bc]{\scalebox{.8}{$D$}}
\rsdraw{.35}{.9}{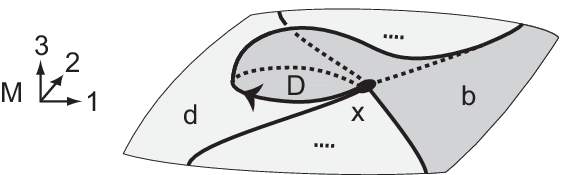}.
$$
Let $\alpha,\beta,\gamma \in H$ be the $H$-labels and $m,k,\ell \in I$ be the colors of $D,b,b'$, respectively. Let $e \in E$ be the $E$-label of $\partial D$ endowed with the depicted orientation and with respect to the $P$-ball branch represented by the upper half-space above $D$. Condition~\eqref{precol1} becomes $$\CM(e)=\alpha^{-\varepsilon} \gamma^{\mu} \beta^{\nu} ,$$ where $\varepsilon,\nu,\mu$ are signs depending on the orientations of the regions $D,b,b'$, respectively. The $\cc$-colored $\CM$-graph in $S^2$ associated with $y$ is obtained from $\cc$-colored $\CM$-graph in $S^2$ associated with $x$ by locally replacing
$$
\psfrag{X}[Br][Br]{\scalebox{.9}{$\ell$}}
\psfrag{Z}[Bl][Bl]{\scalebox{.9}{$k$}}
\rsdraw{.45}{.9}{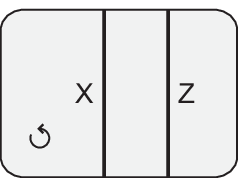}
\quad \text{with} \quad
\psfrag{i}[Br][Br]{\scalebox{.9}{$m$}}
\psfrag{u}[Bl][Bl]{\scalebox{.9}{$e$}}
\psfrag{v}[Bl][Bl]{\scalebox{.9}{$e^{-1}$}}
\psfrag{X}[Br][Br]{\scalebox{.8}{$\ell$}}
\psfrag{Z}[Bl][Bl]{\scalebox{.8}{$k$}}
\rsdraw{.45}{.9}{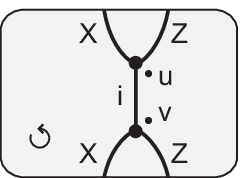}
$$
where the edge colored by $k$  (respectively, $\ell,m$) is oriented downwards if the sign~$\nu$ (respectively, $\mu,\varepsilon$) is positive and upwards otherwise.
The contraction vector associated to the edge $\partial D$ is the contraction vector $\ast$ determined   by the duality between the two vertices (in the above right picture) induced   by the symmetry   with respect to  a horizontal great circle.
Then the contribution of $y$ to the state sum $|T_3(P)|_\cc$ is
\begin{gather*}
\psfrag{i}[Br][Br]{\scalebox{.9}{$m^\varepsilon$}}
\psfrag{u}[Bl][Bl]{\scalebox{.9}{$e$}}
\psfrag{v}[Bl][Bl]{\scalebox{.9}{$e^{-1}$}}
\psfrag{X}[Br][Br]{\scalebox{.8}{$\ell$}}
\psfrag{Z}[Bl][Bl]{\scalebox{.8}{$k$}}
\sum_{m \in I_\alpha} \dim(m) \;
\inv_\cc \left (\,
\rsdraw{.45}{.9}{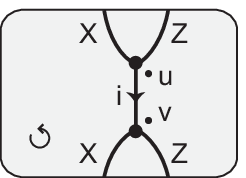}
\,\right )  (\ast ) \\[.4em]
\overset{(i)}{=}
\psfrag{i}[Br][Br]{\scalebox{.9}{$i$}}
\psfrag{u}[Bl][Bl]{\scalebox{.9}{$e$}}
\psfrag{v}[Bl][Bl]{\scalebox{.9}{$e^{-1}$}}
\psfrag{X}[Br][Br]{\scalebox{.8}{$\ell$}}
\psfrag{Z}[Bl][Bl]{\scalebox{.8}{$k$}}
\sum_{i \in I_{\alpha^\varepsilon}} \dim(i) \;
\inv_\cc \left (\,
\rsdraw{.45}{.9}{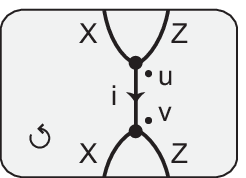}
\,\right )  (\ast )
\; \overset{(ii)}{=} \;
\psfrag{X}[Br][Br]{\scalebox{.8}{$\ell$}}
\psfrag{Z}[Bl][Bl]{\scalebox{.8}{$k$}}
\inv_\cc \left (\, \rsdraw{.45}{.9}{proof-Xi-M-inv-3d}
\,\right )
\end{gather*}
and thus is equal to the contribution of $x$ to the state sum $|P|_\cc$.
Here  $(i)$ follows from the equality $\dim(X^*)=\dim(X)$ for any $X \in \cc$ and $(ii)$ from Lemma~\ref{lem-calc-diag} since $\alpha^\varepsilon= \gamma^{\mu} \beta^{\nu} \CM(e^{-1})$.

\subsection{Remark}\label{rem-onemoreinv}
The right-hand side of Formula~\eqref{eq-simplstatesum+}  is the product of $\dim (\cc_1^1)^{-\vert M\setminus P \vert}$ and a certain sum which we denote
$\Sigma_\cc(P)$. The definition of $\Sigma_\cc(P)\in \kk$ does not use the assumption that $ \dim(\cc^1_1) $ is invertible in~$\kk$ and applies to an arbitrary spherical $\CM$-fusion category~$\cc$. This allows us to generalize the invariant $ \vert
 M,g \vert_\cc$ of a closed $\CM$-manifold  $(M,g)$  to  any such~$\cc$.  We use the theory of spines, see \cite{Mat1}. By a spine of $M$, we mean an oriented stratified 2-polyhedron $P\subset M$ such that $P$ has at least 2 vertices,
$P$ is locally  homeomorphic to the cone over the 1-skeleton of a tetrahedron,  and  $M\setminus P$ is an open ball.   By \cite{Mat1}, $M$  has a   spine~$P$ and any two  spines of~$M$ can be related by  the   moves $T^{1,2}$, $T^{2,1}$ in the class of  spines.  These moves lift to $\CM$-moves and the arguments of Section~\ref{sect-proof-thm-state-3man-Xi} imply that   $\Sigma_\cc(P)$ is preserved under the  $\CM$-moves $T^{1,2}$, $T^{2,1}$. Therefore  $\vert \vert M,g \vert\vert_\cc= \Sigma_\cc(P)$  is a topological invariant of $(M,g)$. If $\dim(\cc_1^1)$ is invertible, then $   \vert \vert M,g\vert\vert_\cc= \dim(\cc_1^1)\,  \vert M ,g\vert_\cc$.

\subsection{The case of pointed $\CM$-manifolds}\label{sect-basepoint-case}
For pointed topological spaces $X,Y$, we denote by $[X,Y]_*$ the set of pointed homotopy classes of pointed maps $X \to Y$.
Note that if $X$ is a connected CW-complex with basepoint a 0-cell  and if $Y$ is path-connected, then $\pi_1(Y)$ acts (on the right) on~$[X,Y]_*$ and the canonical map $[X,Y]_* \to [X,Y]$ induces a bijection
$$
[X,Y]_*/\pi_1(Y) \simeq [X,Y].
$$

By a \emph{pointed closed $\CM$-manifold}, we mean a pair $(M,g)$ where $M$ is a pointed closed oriented 3-dimensional manifold and $g\in [M,B\CM]_*$. Here $B\CM$ is pointed by its unique 0-cell (see Section~\ref{sect-crossed-modules-classifying-spaces}).
Then the invariant $\vert \cdot \vert_{\cc}$ of closed $\CM$-manifolds from Theorem~\ref{thm-state-3man-Xi} extends trivially to pointed closed $\CM$-manifolds by setting
$$
\vert M,g \vert_{\cc}=\vert M,\upsilon(g) \vert_{\cc}
$$
for any pointed closed $\CM$-manifold $(M,g)$, where $\upsilon(g)$ is the image of $g$ under the canonical map $[M,B\CM]_* \to [M,B\CM]$.

\subsection{The case of push-forwards}\label{sect-push-3man-Xi}
Let $\phi=(\psi \co E \to E',\varphi \co H \to H')$ be a morphism from a crossed module $\CM \co E \to H$ to a crossed module $\CM' \co E' \to H'$ such that $\psi$ and $\varphi$ are surjective, $\Ker(\psi)\cap \Ker(\CM)=1$, and $\Ker(\varphi)$ is finite. Let $\cc$ be a spherical $\CM$-fusion category. By Section~\ref{sect-push-forward-fusion}, the push-forward $\phi_*(\cc)$ of $\cc$ is a $\CM'$-fusion category such that
$$
\dim\bigr(\phi_*(\cc)^1_1\bigl)=d_\phi \dim(\cc^1_1) \quad \text{with} \quad
d_\phi=\frac{\card\bigr(\Ker(\varphi)\bigl)}{\card\bigr(\Ker(\psi)\bigl)} \in \ZZ_+.
$$
Assume that $\Ker(\CM')$ is finite and that $\card\bigr(\Ker(\varphi)\bigl)$, $\card\bigr(\Ker(\CM')\bigl)$, and  $\dim(\cc^1_1)$ are invertible in $\kk$. Then $d_\phi$ and $\dim(\phi_*(\cc)^1_1)$ are also invertible. Moreover, since $\psi$ induces an injective group homomorphism  $\Ker(\CM) \to \Ker(\CM')$ (because $\CM'\psi=\varphi\CM$ and $\Ker(\psi)\cap \Ker(\CM)=1$), we get that $\card\bigr(\Ker(\CM)\bigl)$ divides $\card\bigr(\Ker(\CM')\bigl)$ and so is finite and invertible in $\kk$.

By Section~\ref{sect-basepoint-case}, the category $\phi_*(\cc)$ defines the invariant  $\vert \cdot \vert_{\phi_*(\cc)}$ of pointed closed $\CM'$-manifolds and the category $\cc$ defines the invariant  $\vert \cdot \vert_\cc$ of pointed closed $\CM$-manifolds. The next theorem computes the former from the latter. A (pointed) closed $\CM$-manifold is \emph{connected} if the underlying manifold is connected.

\begin{thm}\label{thm-pushforward-3man-Xi}
Let $(M,g')$ be a connected pointed closed $\CM'$-manifold.
Then
$$
\vert M,g'\vert_{\phi_*(\cc)} =d_\phi^{-1} \sum_{ \substack{g \in [M,B\CM]_*\\ B\phi\circ g=g'}}
\dfrac{\card\bigl(\pi_1(\TOP_*(M,B\CM'),g')\bigr)}{\card\bigl(\pi_1(\TOP_*(M,B\CM),g)\bigr)} \,
 \vert M,g\vert_{ \cc },
$$
where $\TOP_*(X,Y)$ denotes the space of pointed continuous maps from $X$ to $Y$ (endowed with the compact-open topology) and $B\phi \co B \CM \to B\CM'$ is the (pointed) map induced by $\phi$.
\end{thm}
Note that the finiteness of the above sum follows from the fact that $\Ker(\varphi)$ is finite. Also, the involved fundamental groups are finite with invertible cardinal in~$\kk$ (since $M$ is compact and $\Ker(\CM)=\pi_2(B\CM)$ and $\Ker(\CM')=\pi_2(B\CM')$ are finite with invertible cardinals, see Claim~\ref{claim-sss-pi} below). We prove  Theorem~\ref{thm-pushforward-3man-Xi} in Section~\ref{sect-proof-thm-pushforward-3man-Xi}.

In the case of groups, that is, when $E$ and $E'$ are trivial, the fundamental groups
$\pi_1(\TOP(M,B\CM),g)$ and $\pi_1(\TOP(M,B\CM'),g')$ are trivial, $d_\phi=\card\bigr(\Ker(\varphi)\bigl)$, and Theorem~\ref{thm-pushforward-3man-Xi} gives back the relationship given in the appendix of \cite{TVi1}.

If $\phi$ is an equivalence of crossed modules (meaning that $\psi$ and $\varphi$ respectively induce isomorphisms $\Ker(\CM)\cong\Ker(\CM')$ and $\Coker(\CM)\cong\Coker(\CM')$), then $B\phi$ is a homotopy equivalence (by the Whitehead theorem),  $d_\phi=1$,  and Theorem~\ref{thm-pushforward-3man-Xi} gives
$$
\vert M,g'\vert_{\phi_*(\cc)}=\vert M,(B\phi)^{-1} \circ g'\vert_{ \cc }
$$
where $(B\phi)^{-1}$ is a homotopy inverse of $B\phi$.

\subsection{Example}\label{ex-nontrivial-inv}
Consider the spherical $\CM$-fusion category $\cc=\kk\tg_\CM$ of Example~\ref{ex-linearized-crossed-module-fusion}.
Then Theorem~\ref{thm-pushforward-3man-Xi} together with~\eqref{eq-case-G-Xi}  implies that for any connected pointed closed $\CM'$-manifold $(M,g')$,
\begin{equation}\label{eq-card-push}
\vert M,g'\vert_{\phi_*(\kk\tg_\CM)} =d_\phi^{-1} \sum_{ \substack{g \in [M,B\CM]_*\\ B\phi\circ g=g'}}
\dfrac{\card\bigl(\pi_1(\TOP_*(M,B\CM'),g')\bigr)}{\card\bigl(\pi_1(\TOP_*(M,B\CM),g)\bigr)} 1_\kk.
\end{equation}
Let us illustrate this with the following explicit morphism of crossed modules: consider  finite abelian groups $A,G$ with invertible cardinal in $\kk$ and such that $H^2(G,A)=0$. Let
$$
\phi=\left(
\vcenter{\xymatrix@R=.7cm @C=.7cm{ A \ar[r]^-{\CM}  \ar[d]_-{\psi} & A \times G \ar[d]^-{\varphi} \\
    A \ar[r]_-{\CM'} & 1}}
\right)
$$
where $\CM(a)=(a,1)$ for all $a \in A$, the action of $A \times G$  on $A$ is trivial, and $\psi=\id_A$.
Note that  $\Ker(\CM)=1$ and so $\Ker(\psi)\cap \Ker(\CM)=1$. Also $\Ker(\varphi)=A \times G$ and $\Ker(\CM')=A$ are finite with invertible cardinal in $\kk$. Using (i) of Section~\ref{sect-crossed-modules-classifying-spaces}, we obtain that the classifying spaces $B\CM$ and $B\CM'$ are Eilenberg-MacLane spaces  of type $K(G,1)$ and $K(A,2)$, respectively. Recall that $[X,K(B,n)]_* \simeq H^n(X,B)$ for any CW-complex $X$, abelian group $B$, and $n \geq 1$.
Then the induced map $B\phi \co B \CM \to B\CM'$ is (pointed) null-homotopic because
$$
[B\CM,B\CM']_* \simeq H^2(B\CM, A)\simeq H^2(G,A)=0.
$$
Let $M$ be a connected pointed closed oriented 3-dimensional manifold. For any $g' \in [X,B\CM']_*$ which is not null-homotopic, there are no $g \in [M,B\CM]_*$ such that $B\phi\circ g=g'$ (since~$B\phi$ is null-homotopic) and so Formula~\eqref{eq-card-push} gives that
$$
\vert M,g'\vert_{\phi_*(\kk\tg_\CM)} =0.
$$
Let $c' \in [M,B\CM']_*$ be the null-homotopy class. For any $g\in [M,B\CM]_*$, we have $B\phi\circ g=c'$ (because $B\phi$ is null-homotopic) and $\pi_1(\TOP_*(M,B\CM),g)=1$ (because $B\CM$ is a $K(G,1)$ space). Since $d_\phi=\card(A)\card(G)$, $[M,B\CM]_* \simeq H^1(M,G)$, and
\begin{gather*}
\pi_1(\TOP_*(M,B\CM'),c')\simeq [S^1,(\TOP_*(M,B\CM'),c')]_* \\ \simeq [M\wedge S^1,B\CM']_* \simeq H^2(M\wedge S^1,A)  \simeq H^1(M,A),
\end{gather*}
Formula~\eqref{eq-card-push} gives that
$$
\vert M,c'\vert_{\phi_*(\kk\tg_\CM)} =\dfrac{\card\bigl(H^1(M,A))\bigr)}{\card(A)} \,
\dfrac{\card\bigl(H^1(M,G)\bigr)}{\card(G)} 1_\kk.
$$
For example, assume that $\kk=\Q$, $A=\Z/2\Z$, $G=\Z/3\Z$, and $M=\R P^3$. (Note that we do have $H^2(G,A)=H^2(\Z/3\Z,\Z/2\Z)=0$.) Since
$$
[\R P^3,B\CM']_*=H^2(\R P^3,\Z/2\Z) \simeq \Z/2\Z,
$$
there are two distinct pointed homotopy classes from $\R P^3$ to $B\CM'$, namely the null-homotopy class $c'$ and another class $g'\neq c'$. Using that $H^1(\R P^3,\Z/2\Z) \simeq \Z/2\Z$ and $H^1(\R P^3,\Z/3\Z)=0$, the above computations give:
$$
\vert \R P^3,c'\vert_{\phi_*(\Q\tg_\CM)} = \frac{1}{3} \quad \text{and} \quad \vert \R P^3,g'\vert_{\phi_*(\Q\tg_\CM)} =0.
$$
In particular, $\vert \R P^3,c'\vert_{\phi_*(\Q\tg_\CM)} \neq \vert \R P^3,g'\vert_{\phi_*(\Q\tg_\CM)}$. Since $\pi_1(B\CM')=1$, we have:
$$
[\R P^3,B\CM']\simeq  [\R P^3,B\CM']_*.
$$
Also, all the maps $\R P^3 \to B\CM'$ are phantom maps (meaning that they induce trivial homomorphisms on homotopy groups) because $\pi_2(\R P^3)=1$ and $B\CM'$ is a $K(\Z/2\Z,2)$ space. This example shows that the invariant~$| \cdot |_\cc$ of Theorem~\ref{thm-state-3man-Xi} is nontrivial and may distinguish distinct homotopy classes of phantom maps.

\subsection{Example}
Let $\CM \co E \to H$ be a crossed module such that $E$ is finite and assume that $\kk$ is a field of characteristic zero.  Then the category $\CM \ti \vect_\kk$ of Example~\ref{ex-crossed-module-vector-spaces} is a spherical $\CM$-fusion category with $\dim\bigr((\CM \ti \vect_\kk)^1_1\bigl)=\card(E) 1_\kk \neq 0$. Recall from Section~\ref{sect-crossed-modules-classifying-spaces} the canonical inclusion $BH \hookrightarrow B \CM$. Then, for any connected pointed closed $\CM$-manifold $(M,g)$, we have:
\begin{equation}\label{eq-push-xi-vect}
\vert M,g\vert_{\CM \ti \vect_\kk}\neq 0 \; \Leftrightarrow \; \text{$g$ factorizes through $BH \hookrightarrow B \CM$.}
\end{equation}
Indeed, recall that $\CM \ti \vect_\kk=\phi_*((E,H\ltimes E)\mti\vect_\kk)$ where
$$
\phi=\left(
\vcenter{\xymatrix@R=.7cm @C=.7cm{ E \ar[r]^-{\iota}  \ar[d]_-{\id_E} & H\ltimes E \ar[d]^-{\varphi} \\
    E \ar[r]_-{\CM} & H}}
\right)
$$
with $\iota(e)=(1,e)$ and $ \varphi(h,e)=\CM(e)h$. Note that $B\iota$ is a $K(H,1)$ space (since $\iota$ is injective and $\Coker(\iota) \simeq H$). For any $\tilde{g}\in [M,B\iota]_*$, $\pi_1(\TOP_*(M,B\iota), \tilde{g})=1$ (because $B\iota$ is a $K(H,1)$ space). Also,
$$
\vert M,\tilde{g}\vert_{(E,H\ltimes E)\mti\vect_\kk}\overset{(i)}{=} \vert M,\tilde{g}\vert_{H\mti\vect_\kk}\overset{(ii)}{=} 1_\kk.
$$
Here $(i)$ follows from the remark after Theorem~\ref{thm-state-3man-Xi} relating the invariants of Theorem~\ref{thm-state-3man-Xi} and \cite{TVi1} (using the injectivity of $\iota$) and the fact that $(E,H\ltimes E)\mti\vect_\kk$ is isomorphic (as a $H$-fusion category) to  the category of finite-dimensional $H$-graded $\kk$-vectors spaces, and $(ii)$ follows from \cite[Example 8.6]{TVi1} (with $\theta=1$). Then, since $d_\phi=\card(E)$, Theorem~\ref{thm-pushforward-3man-Xi} implies that
$$
\vert M,g\vert_{\CM \ti \vect_\kk} =\dfrac{\card\bigl(\pi_1(\TOP_*(M,B\CM),g)\bigr)}{\card(E)}\, \card\bigl(\{ \tilde{g} \in [M,B\iota]_* \, | \, B\phi \circ \tilde{g}=g\} \bigr)\, 1_\kk.
$$
Since $\kk$ is a field of characteristic zero, we obtain that $\vert M,g\vert_{\CM \ti \vect_\kk}  \neq 0$ if and only if there is
$\tilde{g} \in [M,B\iota]_*$ such that $B\phi \circ \tilde{g}=g$. This proves \eqref{eq-push-xi-vect} using that $B\phi$ is pointed homotopic to the canonical inclusion $BH \hookrightarrow B \CM$ (up to a pointed homotopy equivalence $B\iota \simeq BH$).

\subsection{Proof of Theorem~\ref{thm-pushforward-3man-Xi}}\label{sect-proof-thm-pushforward-3man-Xi}
For $X \in \cc_{\mathrm{hom}}= \phi_*(\cc)_{\mathrm{hom}}$, the degree of $X$ in $\cc$ is denoted by $|X|$ and the degree of $X$ in $\phi_*(\cc)$ is then $\varphi(|X|)$.
Let $I=\amalg_{h\in H}\, I_h$ be a $\CM$-representative set  for~$\cc$.  Consider the left action $E \times I \to I$, $(e,i) \mapsto e \cdot i$  given by Lemma~\ref{lem-E-action-on-representatives}(iii). Recall that $j=e \cdot i$  if and only if
$i \cong_e j$, and so $|e \cdot i|=\CM(e) |i|$. This together with the equality $\varphi \CM= \CM'\psi$ implies that for all $e \in \Ker(\psi)$ and $i \in I$,
\begin{equation}\label{eq-varphi-dot-deg}
\varphi(|e \cdot i|)=\varphi(|i|).
\end{equation}

\begin{claim}\label{claim-free-Ker}
The subgroup $\Ker(\psi)$ of $E$ acts freely on $I$.
\end{claim}
\begin{proof}
Let $e \in \Ker(\psi)$ and $i \in I$ such that $i=e \cdot i$. Then $|i|=|e \cdot i|=\CM(e) |i|$ and so $\CM(e)=1$. Therefore $e \in \Ker(\psi) \cap \Ker(\CM)=1$ and $e=1$.
\end{proof}

Pick a complete set $J \subset I$ of representatives of $I/\Ker(\psi)$ such that $\un \in J$.

\begin{claim}\label{claim-repres-J}
We have:
\begin{enumerate}
\labeli
\item For all $i \in I$, there is a unique $(e,j) \in \Ker(\psi) \times J$ such that $i=e \cdot j$.
\item
$J$ is a $\CM'$-representative set for $\phi_*(\cc)$ splitting as $J=\amalg_{h'\in H'}\, J_{h'}$ with
$$
J_{h'}=  \amalg_{h \in \varphi^{-1}(h')}   \, J\cap I_h .
$$

\item For any $h' \in H'$, the map
$$
\left \{ \begin{array}{ccl} \Ker(\psi) \times J_{h'} & \to & \amalg_{h \in \varphi^{-1}(h')} I_h \\
(e,j) & \mapsto & e \cdot j \end{array} \right.
$$
is a bijection.
\end{enumerate}
\end{claim}
\begin{proof}
(i) follows directly from the definition of $J$ and Claim~\ref{claim-free-Ker}. Let us prove~(ii). By definition, $J \subset I \subset \cc_{\mathrm{hom}}= \phi_*(\cc)_{\mathrm{hom}}$ and $\un \in J$.  Recall from Section~\ref{sect-push-forward-fusion} that an object is simple in $\phi_*(\cc)^1$ if and only if it is simple in $\cc^1$. In particular $J \subset I$ is made of simple objects of $\phi_*(\cc)^1$. Next, let $X$ be a simple object of $\phi_*(\cc)^1$. By definition of~$I$ and since $X$ is simple in $\cc^1$, there is $i \in I$ such that $X$ is $1$-isomorphic in~$\cc$ to $i$. By~(i), there is  $(e,j) \in \Ker(\psi) \times J$ such that $i=e \cdot j$, that is,~$i$ is $e$-isomorphic in $\cc$ to $j$. Then, by \eqref{eq-cong-E}, $X$ is $e$-isomorphic in $\cc$ to $j$. Since $\psi(e)=1$, we deduce that $X$ is $1$-isomorphic in $\phi_*(\cc)$ to $j$. To prove the uniqueness of such a $j\in J$, consider another $k \in J$ such that $X$ is $1$-isomorphic in $\phi_*(\cc)$ to~$k$. Then $k$ is $1$-isomorphic in $\phi_*(\cc)$ to $j$ and so, by Lemma~\ref{lem-E-semisimple}(ii), the  \kt module $\Hom_{\phi_*(\cc)}^1(k,j)$ is free  of rank 1. Since
$$
\Hom_{\phi_*(\cc)}^1(k,j)=\bigoplus_{a \in \Ker(\psi)} \Hom_\cc^a(k,j),
$$
there is $a \in \Ker(\psi)$ such that $\Hom_\cc^a(k,j) \neq 0$. Lemma~\ref{lem-E-semisimple}(ii) implies that $k$ is $a$-isomorphic in~$\cc$ to $j$, and so $j=a \cdot k$. Then $j$ and $k$ belong to the same class in~$I/\Ker(\psi)$ and so are equal (by definition of $J$). Then $X$ is $1$-isomorphic in $\phi_*(\cc)$ to a unique element of $J$. Hence $J$ is a $\CM'$-representative set for $\phi_*(\cc)$. The splitting of $J$ is derived directly from the splitting of $I$ and the fact that the degree in $\phi_*(\cc)$ of $j \in J$ is  $\varphi(|j|)$.

Let us prove (iii). The map is well-defined since for any $(e,j) \in \Ker(\psi) \times J_{h'}$, \eqref{eq-varphi-dot-deg} gives that $\varphi(|e \cdot j|)=\varphi(|j|)=h'$ and so $|e \cdot j| \in \varphi^{-1}(h')$. Let $i \in I_h$ with $h \in \varphi^{-1}(h')$. By~(i), there is a unique $(e,j) \in \Ker(\psi) \times J$ such that $i=e \cdot j$. Now $j \in J_{h'}$ since
$j=e^{-1} \cdot i$ and $\varphi(|e^{-1} \cdot i|)=\varphi(|i|)=\varphi(h)=h'$ by \eqref{eq-varphi-dot-deg}.
Thus there is a unique  $(e,j) \in \Ker(\psi) \times J_{h'}$ such that $i=e \cdot j$. Hence the map is bijective.
\end{proof}

\begin{claim}\label{claim-dim-push}
$\dim\bigr(\phi_*(\cc)^1_1\bigl)=d_\phi \dim(\cc^1_1)$.
\end{claim}
\begin{proof}
The bijection of Claim~\ref{claim-repres-J}(iii) for $h'=1$ gives that
$$
\sum_{(e,j)\in \Ker(\psi) \times J_1}\hspace{-1.7em} \dim(e \cdot j)^2 \; =\hspace{-.5em}\sum_{i \in \amalg_{h \in \Ker(\varphi)} I_h} \hspace{-1.5em} \dim(i)^2.
$$
Note that $\dim(e \cdot j)=\dim(j)$ because $e \cdot j$ and $j$ are isomorphic in $\cc$, the left-hand side is equal to
$\card\bigr(\Ker(\psi)\bigl)  \dim\bigr(\phi_*(\cc)^1_1\bigl)$. By \eqref{eq-computation-dimension-neutral-component},
the right-hand side is equal to $\card\bigr(\Ker(\varphi)\bigl)  \dim(\cc^1_1)$. We conclude by using the definition of $d_\phi$.
\end{proof}

Let $M$ be a connected pointed oriented closed 3-dimensional manifold. Since $M$ is connected, there is a skeleton $P$ of $M$ with disk regions and with a single $P$-ball containing the basepoint of $M$. (An example of such a skeleton is provided by a spine of $M \setminus B$ where $B$ is an open 3-ball in~$M$ containing the basepoint of $M$.) Denote by $\mathcal{G}_{P,\CM}$  the gauge group of $\CM$-colorings of~$P$ (see Section~\ref{sect-Xi-labelings}) and consider its subgroup
$$
\mathcal{G}_{P,\CM}^*=\{1\} \times \mathrm{Map}(\Reg(P),E),
$$
where~$1$ is the trivial map (sending the unique $P$-ball to $1\in H$). Two $\CM$-labelings of $P$ are \emph{pointed gauge equivalent} if they belong to the same $\mathcal{G}_{P,\CM}^*$-orbit. The next claim is a version of Lemma~\ref{lem-gauge-group-labelings} for pointed homotopy classes.

\begin{claim}\label{claim-pointed-labelings}
There is a bijection from the set of pointed gauge equivalence classes of $\CM$-labelings of $P$ to the set of pointed homotopy classes of pointed maps $M \to B\CM$.
\end{claim}
\begin{proof}
The proof is similar to that of Lemma~\ref{lem-gauge-group-labelings} by picking the basepoint of~$M$ as the center of the unique $P$-ball and picking arcs $\{\gamma_r\}_r$ and disks~$\{\delta_e\}_e$ as in Section~\ref{sect-proof-lem-gauge-group-labelings}.
For any $\CM$-labeling $(\alpha,\beta)$ of $P$, the pointed homotopy class  of the map~$f_{\alpha,\beta}$ defined in Claim~\ref{claim-filtered-map-to-Xi-labelings} depends only on $P$ and $(\colr,\coleb)$ since any two systems of arcs $\{\gamma_r\}_r$ and disks $\{ \delta_e\}_e$ are  isotopic  in $M$ via an isotopy preserving the basepoint  of $M$.
Also, the fact that two $\CM$-labelings $(\alpha,\beta)$ and $(\alpha',\beta')$ of $P$ are pointed gauge equivalent if and only if their associated maps  $f_{\alpha,\beta}$  and $f_{\alpha',\beta'}$ are pointed homotopic is shown similarly to Claim~\ref{claim-gauge-equiv-homotopic}. Thus the assignment $(\alpha,\beta) \mapsto f_{\alpha,\beta}$ induces a well-defined injective map $\{ \CM\text{-labelings of } P\}/\mathcal{G}_{P,\CM}^* \to [M , B\CM]_*$ which maps the pointed gauge equivalence class of a  $\CM$-labeling $(\alpha,\beta)$ to the pointed homotopy class of the map $f_{\alpha,\beta}$. Since regions of $P$ are disks, this map is surjective by the same arguments of the end of Section~\ref{sect-proof-lem-gauge-group-labelings}.
\end{proof}

Let $g' \in [M,B\CM']_*$. By Claim~\ref{claim-pointed-labelings}, the pointed homotopy class $g'$ is encoded by a $\CM'$-labeling $(\alpha',\beta')$ of $P$. Note that it follows from the definition of a crossed module morphism that if~$(\alpha,\beta)$ is a $\CM$-labeling of $P$, then $(\varphi\alpha, \psi\beta)$ is a  $\CM'$-labeling of $P$. Consider the set
$$
\lhs=\bigl\{ (\alpha,\beta) \text{ $\CM$-labeling of $P$} \, \big | \, (\varphi\alpha, \psi\beta)=(\alpha',\beta') \bigr \}.
$$
For a $\CM$-labeling $(\alpha,\beta)$ of $P$, denote by $g_{\alpha,\beta} \in [M,B\CM]_*$ the pointed homotopy class determined by the pointed gauge equivalence class of $(\alpha,\beta)$ as in Claim~\ref{claim-pointed-labelings}.

\begin{claim}\label{claim-push-repres-lifts}
For any $(\alpha,\beta) \in \lhs$, we have $B\phi\circ g_{\alpha,\beta}=g'$. Conversely, for any $g \in [M,B\CM]_*$ such that $B\phi\circ g=g'$, there is $(\alpha,\beta) \in \lhs$ such that $g=g_{\alpha,\beta}$.
\end{claim}
\begin{proof}
Notice first that if $(\alpha,\beta)$ is a $\CM$-labeling of $P$, then $B\phi \circ g_{\alpha,\beta}=g_{\varphi \alpha,\psi \beta}$. In particular, for any $(\alpha,\beta) \in \lhs$, we have:
$B\phi \circ g_{\alpha,\beta}=g_{\varphi \alpha,\psi \beta}=g_{\alpha',\beta'}=g'$.
Conversely, let $g \in [M,B\CM]_*$ such that $B\phi\circ g=g'$. By  Claim~\ref{claim-pointed-labelings}, there is a $\CM$-labeling $(\alpha_0,\beta_0)$ of $P$ such that $g=g_{\alpha_0,\beta_0}$. Since
$$
g_{\varphi \alpha_0,\psi \beta_0}= B\phi \circ g_{\alpha_0,\beta_0}= B\phi \circ g=g'=g_{\alpha',\beta'},
$$
Claim~\ref{claim-pointed-labelings} implies that $(\varphi\alpha_0,\psi\beta_0)$ is pointed gauge equivalent to $(\alpha',\beta')$, and so there is a  map $\mu' \co \Reg(P) \to E'$ such that
$(\alpha',\beta')=(1,\mu')\cdot (\varphi\alpha_0,\psi\beta_0)$. Since~$\psi$ is surjective, there is a map
$\mu \co \Reg(P) \to E$ such that $\mu'=\psi\mu$.
Consider the $\CM$-labeling  $(\alpha,\beta)=(1,\mu)\cdot(\alpha_0,\beta_0)$ of $P$. By Claim~\ref{claim-pointed-labelings}, $g_{\alpha,\beta}=g_{\alpha_0,\beta_0}=g$. Also, for any $r \in \Reg(P)$,
$$
 \varphi\alpha(r) \overset{(i)}{=} \varphi\bigl( \alpha_0(r) \CM(\mu(r))\bigr) \overset{(ii)}{=} \varphi\alpha_0(r) \, \varphi\CM(\mu(r))
 \overset{(iii)}{=}  \varphi\alpha_0(r) \, \CM'(\mu'(r)) \,  \overset{(iv)}{=}\alpha'(r).
$$
Here $(i)$ and $(iv)$ follow from \eqref{act-col1},  $(ii)$ from the multiplicativity of $\varphi$, and $(iii)$ from the facts that $\varphi\CM= \CM'\psi$ and  $\psi\mu=\mu'$.
Moreover, for any $(e,B) \in \EB(P)$,
\begin{gather*}
 \psi\beta(e,B) \overset{(i)}{=} \psi\big(\beta_0(e,B) \mu_{\alpha_0}(\gamma_{(e,B)})\big)
 \overset{(ii)}{=}\psi\beta_0(e,B) \,\psi(\mu_{\alpha_0}(\gamma_{(e,B)}))\\
 \overset{(iii)}{=}\psi\beta_0(e,B)\, \mu'_{\varphi\alpha_0}(\gamma_{(e,B)})\overset{(iv)}{=}\beta'(e,B).
\end{gather*}
Here $(i)$ and $(iv)$ follow from \eqref{act-col2},  $(ii)$ from the multiplicativity of $\psi$, and $(iii)$ from the facts that $\psi(\mu_{\alpha_0})=(\psi\mu)_{\varphi\alpha_0}$ (see Section~\ref{sect-Xi-labelings}) and $\psi\mu=\mu'$. Thus we get that $(\varphi\alpha,\psi\beta)=(\alpha',\beta')$. Hence $(\alpha,\beta) \in \lhs$.
\end{proof}

For $\mathcal{A} \subset \phi_*(\cc)_{\mathrm{hom}}=\cc_{\mathrm{hom}}$, let
$$
  \col_{\mathcal{A}} = \{ c \co \Reg(P)\to \mathcal{A} \, | \, \varphi(|c(r)|)= \alpha'(r) \text{ for all } r \in \Reg(P) \}.
$$
For $c \in \col_{\mathcal{A}}$, we define the \kt module $H_{c,\phi_*(\cc)}=H_c$ and the scalars $\dim(c)$ and $|c|_{\phi_*(\cc)}=  V_c(\ast_c)$ as in Section~\ref{sec-computat} but using the $\CM'$-labeling $(\alpha',\beta')$ and the  spherical $\CM'$-fusion category~$\phi_*(\cc)$. Clearly, $|c|_{\phi_*(\cc)}=0$ whenever $H_{c,\phi_*(\cc)}=0$. Set
$$
  \col^*_{\mathcal{A}} = \{ c \in \col_{\mathcal{A}} \, | \, H_{c,\phi_*(\cc)}\neq 0 \}.
$$
To maps $\eta \co \Reg(P) \to \Ker(\psi)$ and $d \co \Reg(P) \to  J$, we associate the map
$$
 \eta\cdot d \co \left \{ \begin{array}{ccl} \Reg(P) &\to & I \\
r &\mapsto & \eta(r)\cdot d(r). \end{array} \right.
$$
Denote by $\mm$ the set of maps from $\Reg(P)$ to $\Ker(\psi)$.

\begin{claim}\label{claim-comparison-colorings-g}
The map $\mm \times \col_J  \to \col_I$, defined by $(\eta,d)  \mapsto  \eta \cdot d$, is bijective.
\end{claim}
\begin{proof}
Note that $\eta \cdot d \in \col_I$ for any $(\eta,d) \in \mm \times \col_J$ since, by using \eqref{eq-varphi-dot-deg},
$$
 \varphi(|(\eta \cdot d)(r)|)=  \varphi(|\eta(r) \cdot d(r)|) = \varphi(|d(r)|) =\alpha'(r)
$$
for all $r \in \Reg(P)$. Hence the map is well-defined. Let $c \in \col_I$. Claim~\ref{claim-repres-J}(i) implies that there is a unique pair of maps $(\eta \co \Reg(P) \to \Ker(\psi), d \co \Reg(P) \to J)$ such that $c(r)=\eta(r) \cdot d(r)$ for all $r \in \Reg(P)$. Note that $d \in \col_J$ since, by using \eqref{eq-varphi-dot-deg},
$$
 \varphi(|d(r)|)=  \varphi(|\eta(r)^{-1} \cdot c(r)|) = \varphi(|c(r)|) =\alpha'(r)
$$
for all $r \in \Reg(P)$. Thus there is a unique $(\eta,d) \in \mm \times \col_J$  such that $c=\eta \cdot d$. Hence the map is bijective.
\end{proof}

\begin{claim}\label{claim-comparison-of-colorings-comput}
For all $\eta \in \mm$ and $d \in \col_J$,
$$
\dim (\eta \cdot d)=\dim(d) \quad \text{and} \quad  |\eta \cdot d|_{\phi_*(\cc)}=|d|_{\phi_*(\cc)}.
$$
\end{claim}
\begin{proof}
For any $r \in \Reg(P)$, the objects $\eta(r) \cdot d(r)$ and $d(r)$ are isomorphic in $\cc$ and so have same dimension. Then
$\dim (\eta \cdot d)=\dim(d)$.

For each $j\in J$ and $e \in \Ker(\psi)$, pick an $e$-isomorphism $\theta_{e,j} \co j \to e \cdot j$ in $\cc$. Note that different choices of such an isomorphism differ by multiplication of an invertible scalar since $\Hom_\cc^e(j,e \cdot j)$ is free  of rank 1 (see Lemma~\ref{lem-E-semisimple}(ii)). For any $r \in \Reg(P)$, define the isomorphism $ \theta_{r,\pm} \co d(r)^\pm \to (\eta(r)\cdot d(r))^\pm$ by
$$
\theta_{r,+}= \theta_{\eta(r),d(r)}  \quad \text{and} \quad \theta_{r,-}=\bigl(\theta^*_{\eta(r),d(r)}\bigr)^{-1}.
$$
Note that  $ \theta_{r,+}$ has degree $\psi(\eta(r))=1$ in $\phi_*(\cc)$, as well as $\theta_{r,-}=(\theta^*_{r,+})^{-1}$.
For any oriented edge $e$ of $P$, the above isomorphisms induce a \kt linear isomorphism
$$
\theta_e \co H_d(e) \xrightarrow{\simeq}  H_{\eta \cdot d}(e)
$$
such that for any $P$-ball branch $B$ at $e$, the following diagram commutes:
$$
\xymatrix@R=.7cm{ H_d(e) \ar[r]^-{\simeq}  \ar[d]_-{\theta_e} & \Hom_{\phi_*(\cc)}^{\beta'(e,B)}\bigl(\un, d(b_1)^{\varepsilon_1} \otimes \cdots \otimes d(b_n)^{\varepsilon_n}\bigr) \ar[d]^-{\theta_{e,B}} \\
    H_{\eta \cdot d}(e) \ar[r]^-{\simeq} & \Hom_{\phi_*(\cc)}^{\beta'(e,B)}\left(\un, (\eta(b_1) \cdot d(b_1))^{\varepsilon_1} \otimes \cdots \otimes (\eta(b_n) \cdot d(b_n))^{\varepsilon_n}\right).
    }
$$
Here, the horizontal maps are the cone isomorphisms, $b_1< \cdots <b_n$  are the branches of $P$ at $e$ enumerated in the linear order determined by $B$,  $\varepsilon_i=\varepsilon_e(b_i) \in \{+,-\}$, and $\theta_{e,B}$ is defined by
$$
\theta_{e,B}(\alpha)=(\theta_{b_1,\varepsilon_1} \otimes \cdots \otimes \theta_{b_n,\varepsilon_n})\circ \alpha.
$$
For any vertex $x$ of $P$,  the unordered tensor product $\otimes_{e_x} \, \theta_{e_x}$, where $e_x$ runs over all edges of $P$ incident to $x$ and oriented away from $x$, induces a \kt linear isomorphism $\theta_x \co H(\Gamma^d_x) \to  H(\Gamma^{\eta \cdot d}_x)$ satisfying
\begin{equation}\label{eq-phi-nat1}
\inv_{\!\phi_*(\cc)}\bigl(\Gamma^d_x\bigr)=  \inv_{\!\phi_*(\cc)}\bigl(\Gamma^{\eta \cdot d}_x\bigr) \circ \theta_x.
\end{equation}
This follows from the naturality of $\inv_{\!\phi_*(\cc)}$ (see Section~\ref{sect-inv-Xi-graphs}). Also, as in Section~\ref{sec-computat}, each unoriented edge  $e$  of  $P$ gives rise to two opposite oriented edges $e_1$, $e_2$ of~$P$ and to  the vectors $\ast^d_e \in H_d(e_1) \otimes H_d(e_2)$ and $\ast^{\eta \cdot d}_e \in H_{\eta \cdot d}(e_1)\otimes H_{\eta \cdot d}(e_2)$ satisfying
\begin{equation}\label{eq-phi-nat2}
(\theta_{e_1} \otimes \theta_{e_2}) (\ast^d_e)= \ast_e^{\eta \cdot d}.
\end{equation}
Then the equality $|\eta \cdot d|_{\phi_*(\cc)}=|d|_{\phi_*(\cc)}$ follows directly from \eqref{eq-phi-nat1} and \eqref{eq-phi-nat2}.
\end{proof}

For a $\CM$-labeling $(\alpha,\beta)$ of $P$, let
$$
  \col(\alpha,\beta) = \bigl\{ c \co \Reg(P)\to I \, \big | \, |c(r)|= \alpha(r) \text{ for all } r \in \Reg(P) \bigr \}.
$$
For $c \in \col(\alpha,\beta)$, we define the \kt module $H_{c,\cc}=H_c$ and the scalar $|c|_\cc=  V_c(\ast_c)$ as in Section~\ref{sec-computat} by using the $\CM$-labeling $(\alpha,\beta)$ and the  spherical $\CM$-fusion category~$\cc$.
Clearly, $|c|_\cc=0$ whenever $H_{c,\cc}=0$. Set
$$
  \col^*(\alpha,\beta) = \{ c \in \col(\alpha,\beta) \,  | \, H_{c,\cc}\neq 0 \}.
$$
The next claim is instrumental to define a map $\col^*_I \to \lhs$ in Claim~\ref{claim-coloring-to-labeling}.

\begin{claim}\label{claim-Hom-XYe}
Let $X$ be a $1$-direct sum in $\cc$ of homogeneous objects of degree $h \in H$. Let $e' \in E'$ with $\Hom_{\phi_*(\cc)}^{e'}(\un,X)\neq 0$. Then there is a unique $e \in \psi^{-1}(e')$ such that
$$
\Hom_{\phi_*(\cc)}^{e'}(\un,X)= \Hom^e_{\cc}(\un,X).
$$
Moreover $\CM(e)=h$.
\end{claim}
\begin{proof}
Assume that $X$ is a 1-direct sum of a finite family $(X_\alpha)_{\alpha \in A}$ of homogeneous objects of $\cc$ of degree $h$. Since
$$
\bigoplus_{e\in \psi^{-1}(e')} \Hom_\cc^e(\un,X)=\Hom_{\phi_*(\cc)}^{e'}(\un,X)\neq 0,
$$
there is $e\in \psi^{-1}(e')$ with $\Hom_\cc^e(\un,X) \neq 0$. Then $\CM(e)=h$. Indeed, since by~\eqref{eq-Hom-direct-sum}
$$
\bigoplus_{\alpha \in A} \Hom_\cc^{e}( \un,X_\alpha) \simeq \Hom_\cc^e(\un,X) \neq 0,
$$
there is $\alpha \in A$ such that $\Hom_\cc^{e}( \un,X_\alpha) \neq 0$, and so $\CM(e)=|X_\alpha|=h$. If $f \in \psi^{-1}(e')$ is such that $\Hom_\cc^f(\un,X) \neq 0$, then $\CM(f)=h$ (as above) and so $f=e$ because $f e^{-1} \in \Ker(\psi) \cap \Ker(\CM)=1$. Thus  $\Hom_\cc^f(\un,X) = 0$ for all $f \in \psi^{-1}(e')$ with $f \neq e$. Hence $\Hom_{\phi_*(\cc)}^{e'}(\un,X)= \bigoplus_{f\in \psi^{-1}(e')} \Hom_\cc^f(\un,X)= \Hom^e_{\cc}(\un,X)$.
\end{proof}

\begin{claim}\label{claim-coloring-to-labeling}
Let $c \in \col_I^*$. Then  there is a unique $(\alpha_c,\beta_c) \in \lhs$ with  $c \in \col^*(\alpha_c,\beta_c)$. Moreover  $|c|_{\phi_*(\cc)}=|c|_\cc$.
\end{claim}
\begin{proof}
The condition $c \in \col^*(\alpha_c,\beta_c)$ implies that  $\alpha_c(r)= |c(r)|$ for all $r \in \Reg(P)$. Let $(e,B) \in \EB(P)$. The cone isomorphism gives a \kt linear isomorphism
$$
H_{c,\phi_*(\cc)}(e)  \simeq  \Hom_{\phi_*(\cc)}^{\beta'(e,B)}\bigl(\un, c(b_1)^{\varepsilon_1} \otimes \cdots \otimes c(b_n)^{\varepsilon_n}\bigr),
$$
where $b_1< \cdots <b_n$  are the branches of $P$ at $e$ enumerated in the linear order determined by $B$ and $\varepsilon_i=\varepsilon_e(b_i) \in \{+,-\}$.
Since $H_{c,\phi_*(\cc)}(e)\neq 0$
and the object $c(b_1)^{\varepsilon_1} \otimes \cdots \otimes c(b_n)^{\varepsilon_n}$ is a $1$-direct sum in $\cc$ of homogeneous objects of degree $|c(b_1)|^{\varepsilon_1} \cdots |c(b_n)|^{\varepsilon_n}$ by (c) of Section~\ref{sect-crossed-module-graded-categories},
Claim~\ref{claim-Hom-XYe} gives that there is a unique
$\beta_c(e,B) \in \psi^{-1}(\beta'(e,B))$ such that
\begin{equation}\label{eq-proof-lift}
\Hom_{\phi_*(\cc)}^{\beta'(e,B)}\bigl(\un, c(b_1)^{\varepsilon_1} \otimes \cdots \otimes c(b_n)^{\varepsilon_n}\bigr)=\Hom_\cc^{\beta_c(e,B)}\bigl(\un, c(b_1)^{\varepsilon_1} \otimes \cdots \otimes c(b_n)^{\varepsilon_n}\bigr),
\end{equation}
and moreover
$$
\CM(\beta_c(e,B))=|c(b_1)|^{\varepsilon_1} \cdots |c(b_n)|^{\varepsilon_n}=\alpha_c(b_1)^{\varepsilon_1}\cdots \alpha_c(b_n)^{\varepsilon_n}.
$$
Let $\suc(B)$ be as in Section~\ref{sect-Xi-labelings}.  Since the cone isomorphisms are compatible with the map $p_{b_1,b_2}$ as in~\eqref{eq-def-permutation}, we obtain that
\begin{gather*}
 \Hom_{\phi_*(\cc)}^{\beta'(e,\suc(B))}\bigl(\un, c(b_2)^{\varepsilon_2} \otimes \cdots \otimes c(b_n)^{\varepsilon_n} \otimes c(b_1)^{\varepsilon_1}\bigr)\\
=\Hom_{\cc}^{f}\bigl(\un, c(b_2)^{\varepsilon_2} \otimes \cdots \otimes c(b_n)^{\varepsilon_n} \otimes c(b_1)^{\varepsilon_1}\bigr)
\end{gather*}
with $f=\lact{\left(|c(b_1)|^{-\varepsilon_1}\right)}{\!\beta_c(e,B)}$ by (e) of Section~\ref{sect-crossed-module-graded-categories}. The uniqueness in Claim~\ref{claim-Hom-XYe} implies that
$$
\beta_c(e,\suc(B))=f=\lact{\left(\alpha_c(b_1)^{-\varepsilon_1}\right)}{\!\beta_c(e,B)}.
$$
Therefore $(\alpha_c,\beta_c)$ is a $\CM$-labeling of $P$. By construction, $(\varphi\alpha_c, \psi\beta_c)=(\alpha',\beta')$, so that  $(\alpha_c,\beta_c)\in \lhs$, and  $c \in \col(\alpha_c,\beta_c)$. It follows from \eqref{eq-proof-lift} that there is a canonical \kt linear isomorphism $H_{c,\phi_*(\cc)}(e) \simeq H_{c,\cc}(e)$ for each oriented edge $e$ of $P$, and so $H_{c,\cc}\simeq H_{c,\phi_*(\cc)}\neq 0$. Hence $c \in \col^*(\alpha_c,\beta_c)$.
The uniqueness of $(\alpha_c,\beta_c)$ follows from the fact that~$\alpha_c$ is fully determined by the condition $c \in \col^*(\alpha_c,\beta_c)$  and from the uniqueness in Claim~\ref{claim-Hom-XYe}.

The fact that $\phi_*(\cc)=\cc$ as pivotal categories implies that, using $\phi_*(\cc)$ and $c \in \col_I^*$ on the one hand and using $\cc$ and $c \in \col^*(\alpha_c,\beta_c)$ on the other,  the contraction vectors associated with the unoriented edges of $P$ are the same and that the vectors associated with the vertices of $P$ are the same (up to the canonical isomorphism $H_{c,\phi_*(\cc)}\simeq H_{c,\cc}$). Consequently, $|c|_{\phi_*(\cc)}=|c|_\cc$.
\end{proof}

\begin{claim}\label{claim-comparison-all-colorings}
The map
$$
\col^*_I \to \hspace{-.3em} \coprod_{(\alpha,\beta) \in \lhs} \hspace{-.6em}  \col^*(\alpha,\beta), \quad c \mapsto c \in \col^*(\alpha_c,\beta_c)
$$
is bijective.
\end{claim}
\begin{proof}
The uniqueness in Claim~\ref{claim-coloring-to-labeling} implies  that if $(\alpha,\beta) \in \lhs $ and $c \in \col^*(\alpha,\beta)$, then $(\alpha_c,\beta_c)=(\alpha,\beta)$. In particular, this gives that $\bigcup_{(\alpha,\beta) \in \lhs} \col^*(\alpha,\beta)$ is a disjoint union. Then the map is well-defined and injective. Let us prove that it is surjective. Let $(\alpha,\beta) \in \lhs $ and $c \in \col^*(\alpha,\beta)$. We need to prove that $c \in \col^*_I$.
Since $\alpha'(r)=\varphi(\alpha(r))=\varphi(|c(r)|)$ for all $r \in \Reg(P)$, we get that $c \in \col_I$. Let $e$ be an oriented edge of $P$. Pick a  $P$-ball branch $B$ at $e$ and set
$
X=c(b_1)^{\varepsilon_1} \otimes \cdots \otimes c(b_n)^{\varepsilon_n},
$
where $b_1< \cdots <b_n$  are the branches of $P$ at $e$ enumerated in the linear order determined by $B$ and $\varepsilon_i=\varepsilon_e(b_i) \in \{+,-\}$. Using that $\beta(e,B) \in\psi^{-1}(\beta'(e,B))$ and $c \in \col^*(\alpha,\beta)$, we obtain
$$
H_{c,\phi_*(\cc)}(e)  \simeq  \Hom_{\phi_*(\cc)}^{\beta'(e,B)}(\un,X)\supset \Hom_{\cc}^{\beta(e,B)}(\un,X) \simeq H_{c,\cc}(e) \neq 0.
$$
Thus $H_{c,\phi_*(\cc)}=\bigotimes_e H_{c,\phi_*(\cc)}(e) \neq 0$ and so $c \in \col^*_I$.
\end{proof}

By definition, since $|M \setminus P|=1$,
$$
  |M,g'|_{\phi_*(\cc)} = \dim\bigr(\phi_*(\cc)^1_1\bigl)^{-1}   \sum_{d \in \col_{J}}  \hspace{-.3em} \dim(d) \, |d|_{\phi_*(\cc)}.
$$
Now, we have:
\begin{gather*}
\sum_{\substack{(\alpha,\beta) \in \lhs \\ c \in \col(\alpha,\beta)}} \hspace{-.9em} \dim(c)\, |c|_\cc \,
\overset{(i)}{=}
\hspace{-1em} \sum_{\substack{(\alpha,\beta) \in \lhs \\ c \in \col^*(\alpha,\beta)}} \hspace{-1.1em} \dim(c)\, |c|_\cc \,
\overset{(ii)}{=}
\sum_{c \in \col_I^*} \hspace{-.3em} \dim(c)\, |c|_{\phi_*(\cc)} \\
\overset{(iii)}{=}\sum_{c \in \col_I} \hspace{-.3em} \dim(c)\, |c|_{\phi_*(\cc)}
\overset{(iv)}{=} \hspace{-1.8em} \sum_{(\eta,d) \in \mm \times \col_J} \hspace{-1.8em} \dim(\eta \cdot d) \, | \eta \cdot d|_{\phi_*(\cc)} \\
\overset{(v)}{=} \card(\mm) \hspace{-.3em} \sum_{d \in  \col_J}  \hspace{-.3em} \dim(d) \, |d|_{\phi_*(\cc)}.
\end{gather*}
Here $(i)$ follows from the fact that $|c|_\cc=0$ for all $c \in \col(\alpha,\beta)$ with $H_{c,\cc}=0$,
$(ii)$ from the bijection of Claim~\ref{claim-comparison-all-colorings} and the last assertion of Claim~\ref{claim-coloring-to-labeling}, $(iii)$ from the fact that $|c|_{\phi_*(\cc)}=0$ for all $c \in \col_I$ with $H_{c,\phi_*(\cc)}=0$, $(iv)$ from the bijection in Claim~\ref{claim-comparison-colorings-g}, and $(v)$ from
Claim~\ref{claim-comparison-of-colorings-comput}. Therefore
$$
  |M,g'|_{\phi_*(\cc)}  = \dim\bigr(\phi_*(\cc)^1_1\bigl)^{-1}   \card(\mm)^{-1} \hspace{-.9em}  \sum_{\substack{(\alpha,\beta) \in \lhs \\ c \in \col(\alpha,\beta)}} \hspace{-.9em} \dim(c)\, |c|_\cc
$$
and so, using Claim~\ref{claim-dim-push},  we obtain
$$
  |M,g'|_{\phi_*(\cc)}  = \sum_{(\alpha,\beta) \in \lhs} d_\phi^{-1}  \card(\mm)^{-1} \dim(\cc^1_1)^{-1} \hspace{-.9em}  \sum_{c \in \col(\alpha,\beta)} \hspace{-.9em} \dim(c)\, |c|_\cc.
$$
Then, it follows from the definition of $|\cdot |_\cc$ that
$$
  |M,g'|_{\phi_*(\cc)} = \sum_{(\alpha,\beta) \in \lhs} d_\phi^{-1}  \card(\mm)^{-1} |M,g_{\alpha,\beta}|_\cc.
$$
Now the first assertion of Claim~\ref{claim-push-repres-lifts} gives that for any $(\alpha,\beta) \in \lhs$,  the pointed homotopy class $g_{\alpha,\beta}\in [M,B\CM]_*$ satisfies $B\phi\circ g_{\alpha,\beta}=g'$. Thus
\begin{equation}\label{eq-final-push-lambda-g}
  |M,g'|_{\phi_*(\cc)}  = d_\phi^{-1} \sum_{ \substack{g \in [M,B\CM]_*\\ B\phi\circ g=g'}} \hspace{-.5em}  \lambda_g \, |M,g|_\cc \quad \text{where} \quad
  \lambda_g=  \hspace{-.5em}\sum_{\substack{(\alpha,\beta) \in \lhs \\ g_{\alpha,\beta}=g}} \hspace{-.5em} \card(\mm)^{-1}.
\end{equation}
Let $g \in [M,B\CM]_*$ such that $B\phi\circ g=g'$. By the second assertion of Claim~\ref{claim-push-repres-lifts}, there is $(\alpha_0,\beta_0) \in \lhs$ such that $g=g_{\alpha_0,\beta_0}$. Set
$$
\lhs_{(\alpha_0,\beta_0)}=\bigl\{ (\alpha,\beta)  \in \lhs \, \big | \, \text{$(\alpha,\beta)$ is pointed gauge  equivalent to $(\alpha_0,\beta_0)$} \bigr \}.
$$
For any $(\alpha,\beta) \in \lhs_{(\alpha_0,\beta_0)}$, Claim~\ref{claim-pointed-labelings} gives that $g_{\alpha,\beta}=g$ if and only if $(\alpha,\beta) \in \lhs_{(\alpha_0,\beta_0)}$. Thus
\begin{equation}\label{eq-lambda-g-1}
\lambda_g = \dfrac{\card(\lhs_{(\alpha_0,\beta_0)})}{\card(\mm)}.
\end{equation}
Consider the sets
$$
\ee=\left\{\mu \co \Reg(P) \to E \; \left| \begin{array}{l} \Ima(\mu) \subset \Ker(\varphi \CM), \\
  (\psi \mu)_{\alpha'}(\gamma_{(e,B)})=1_{E'} \text{ for all $(e,B) \in \EB(P)$}\end{array} \!\!\right \}\right.
$$
and
$$
\sss=\left\{\mu \co \Reg(P) \to E \; \left| \begin{array}{l} \Ima(\mu) \subset \Ker(\CM), \\
  \mu_{\alpha_0}(\gamma_{(e,B)})=1_E \text{ for all $(e,B) \in \EB(P)$}\end{array} \!\!\right \}.\right.
$$

\begin{claim}\label{claim-ee-def}
Let $\mu \co \Reg(P) \to E$ be a map and let $(\alpha,\beta) \in \lhs$. Then $\mu \in \ee$ if and only if $(1,\mu)\cdot (\alpha,\beta) \in \lhs$.
\end{claim}
\begin{proof}
Consider the $\CM$-labeling  $(\delta,\varepsilon)=(1,\mu)\cdot (\alpha,\beta)$ of $P$. As in the proof of Claim~\ref{claim-push-repres-lifts}, we get that for any $r \in \Reg(P)$ and $(e,B) \in \EB(P)$,
$$
\varphi\delta(r)=\alpha'(r) \varphi \CM(\mu(r)) \quad \text{and} \quad
\psi\varepsilon(e,B) =\beta'(e,B) \,(\psi\mu)_{\alpha'}(\gamma_{(e,B)}).
$$
Then $(\delta,\varepsilon) \in \lhs \Leftrightarrow (\varphi\delta,\psi\varepsilon)=(\alpha',\beta')  \Leftrightarrow \mu \in \ee$.
\end{proof}

\begin{claim}\label{claim-ee-action}
The set $\ee$ is a group (with opposite pointwise product) left acting transitively on $\lhs_{(\alpha_0,\beta_0)}$ and $\sss$ is the stabilizer of $(\alpha_0,\beta_0)$ under this action.
\end{claim}
\begin{proof}
Let $\eta,\mu \in \ee$. By Claim~\ref{claim-ee-def}, $(\alpha,\beta)=(1,\mu)\cdot (\alpha_0,\beta_0) \in \lhs$ and  $(1,\eta) \cdot (\alpha,\beta) \in \lhs$. Using the definition of the action of $\mathcal{G}_{P,\CM}^*\subset \mathcal{G}_{P,\CM}$ on $\CM$-labelings, we obtain:
$$
(1,\mu\eta)\cdot (\alpha_0,\beta_0)= \bigl((1,\eta) (1,\mu) \bigr) \cdot (\alpha_0,\beta_0) = (1,\eta) \cdot  (\alpha,\beta) \in \lhs.
$$
Thus $\mu\eta \in \ee$ by Claim~\ref{claim-ee-def}. We deduce from this that $\ee$ is a group.

Next, let $(\alpha,\beta)  \in \lhs_{(\alpha_0,\beta_0)}$. Since $(\alpha,\beta)$ is  pointed gauge  equivalent to $(\alpha_0,\beta_0)$, there is a map $\mu \co \Reg(P) \to E$ such that $(\alpha,\beta)=(1,\mu)\cdot (\alpha_0,\beta_0)$. Since $(\alpha_0,\beta_0)$ and $(\alpha,\beta)$ are both in $\lhs$, Claim~\ref{claim-ee-def} implies that $\mu \in \ee$. Thus $\ee$ acts transitively on~$\lhs_{(\alpha_0,\beta_0)}$.
Finally, let  $\mu \in \ee$. Set $(\alpha,\beta)=(1,\mu)\cdot (\alpha_0,\beta_0)$. For any $r \in \Reg(P)$ and $(e,B) \in \EB(P)$,
$$
\alpha(r)=\alpha_0(r) \CM(\mu(r)) \quad \text{and} \quad
\beta(e,B) =\beta_0(e,B) \,\mu_{\alpha_0}(\gamma_{(e,B)}).
$$
Then $(\alpha,\beta)=(\alpha_0,\beta_0) \Leftrightarrow \mu \in \sss$. Hence $\sss$ is the stabilizer of $(\alpha_0,\beta_0)$.
\end{proof}

Note that Claim~\ref{claim-ee-action} implies in particular that $\sss$ is a group (for the opposite pointwise product). Similarly, the set
$$
\sss'=\left\{\mu' \co \Reg(P) \to E' \; \left| \begin{array}{l} \Ima(\mu') \subset \Ker(\CM'), \\
  \mu'_{\alpha'}(\gamma_{(e,B)})=1_{E'} \text{ for all $(e,B) \in \EB(P)$}\end{array} \!\!\right \}\right.
$$
is a group (with opposite pointwise product).

\begin{claim}\label{claim-sss-prime}
The map $\ee \to \sss'$, $\mu \mapsto \psi \mu$, is a surjective group homomorphism with kernel $\mm$.
\end{claim}
\begin{proof}
The map is well defined since for any $\mu \in \ee$,
$$
\CM'(\Ima(\psi\mu))=\CM'\psi(\Ima(\mu))=\varphi\CM(\Ima(\mu))=\{1_{H'}\}.
$$
The fact that it is a group  homomorphism follows from the multiplicativity of $\psi$. Its surjectivity follows from the surjectivity of $\psi$. Using that a map $\mu \co \Reg(P) \to E$ verifies $\psi\mu=1$ if and only if $\Ima(\mu) \subset \Ker(\psi)$, we obtain that the Kernel of the map is $\mm$ (identified with the set of maps $\Reg(P) \to E$ with image in $\Ker(\psi)$).
\end{proof}

\begin{claim}\label{claim-sss-pi}
The group $\sss$ is isomorphic to $\pi_1(\TOP_*(M,B\CM), g)$. Similarly, the group $\sss'$ is isomorphic to $\pi_1(\TOP_*(M,B\CM'),g')$.
\end{claim}
\begin{proof}
Using the notation of the proof of Claim~\ref{claim-pointed-labelings}, consider the pointed map $f=f_{\alpha_0,\beta_0}\co M \to B\CM$ whose pointed homotopy class is $g_{\alpha_0,\beta_0}=g$. As in the proof of Claim~\ref{claim-gauge-equiv-homotopic}, any $\mu \in \sss$ defines a pointed homotopy $G_\mu$ from  $f$ to itself and so a loop in $\TOP_*(M,B\CM)$ based at $f$ whose homotopy class $\delta_\mu$ depends only on $\mu$. The assignment $\mu \mapsto \delta_\mu$ is a group homomorphism (by definition of the product of $\tg_{P,\CM}$). The injectivity of this assignment follows from the facts that~$B\CM$ is a 2-type and the filtration of $M$ used in the construction of $G_\mu$ has a unique 0\ti cell.  The surjectivity of the map follows from the fact that any pointed homotopy from~$f$ to itself is homotopic to $G_\mu$ for some $\mu \in \sss$ (as in the proof of Claim~\ref{claim-gauge-equiv-homotopic}). Therefore $\sss$ is isomorphic to $\pi_1(\TOP_*(M,B\CM),f)$. Note that the
isomorphism type of $\pi_1(\TOP_*(M,B\CM),f)$ only depend on the path component of  $f$ in $\TOP_*(M,B\CM)$, which is $g$, and is denoted $\pi_1(\TOP_*(M,B\CM),g)$.

The proof of the second statement of the claim is similar (by exchanging $\CM$ with~$\CM'$ and $(\alpha_0,\beta_0)$ with $(\alpha',\beta')$).
\end{proof}

We can now compute:
\begin{gather*}
\lambda_g\;\overset{(i)}{=} \dfrac{\card(\lhs_{(\alpha_0,\beta_0)})}{\card(\mm)}
 \overset{(ii)}{=} \dfrac{\card(\ee)}{\card(\sss)\,\card(\mm)}
 \overset{(iii)}{=} \dfrac{\card(\sss')}{\card(\sss)}\\
 \overset{(iv)}{=} \dfrac{\card\bigl(\pi_1(\TOP_*(M,B\CM'),g')\bigr)}{\card\bigl(\pi_1(\TOP_*(M,B\CM),g)\bigr)}.
\end{gather*}
Here $(i)$ follows from \eqref{eq-lambda-g-1}, $(ii)$ from Claim~\ref{claim-ee-action}, $(iii)$ from Claim~\ref{claim-sss-prime}, and $(iv)$ from Claim~\ref{claim-sss-pi}. This together with \eqref{eq-final-push-lambda-g} concludes the proof of Theorem~\ref{thm-pushforward-3man-Xi}.

\section{3-dimensional HQFTs with target \texorpdfstring{$B\CM$}{BX}}\label{sect-Xi-HQFTs}
In this section, given a crossed module $\CM \co E \to H$, we define $3$-dimensional HQFTs with target the classifying space~$B\CM$. We adapt here the definition given in~\cite{Tu1,TVi1} by replacing  pointed homotopies (which are used for aspherical targets)  by  $\sigma$-homotopies where $\sigma$ is a graph.

Recall from Section~\ref{sect-crossed-modules-classifying-spaces} that  the space $B\CM$ has a canonical filtration
$$
(B\CM)_*=(\{x\}\subset BH \subset B\CM)
$$
where $x$ is the unique 0-cell of $B\CM$ (serving as a basepoint) and $BH$ is a subcomplex of $B\CM$ which is a classifying space of the group $H$ and whose associated boundary crossed module $\pi _{2}(B\CM,BH;x)\to \pi_{1}(BH,x)$ is $\CM\co E \to H$.

\subsection{$\sigma$-homotopies}\label{sect-sigma-homotopies}
Any graph $\sigma$ embedded in a topological space $X$ yields a filtration $X_\sigma=(\sigma_0 \subset \sigma \subset X)$ of $X$, where  $\sigma_0$ is the set of vertices of~$\sigma$.
Recall that a filtered map $u\co X_\sigma \to (B\CM)_{\ast}$ is a continuous map $u \co X \to B\CM$ such~that
$$
u(\sigma_0)=\{x\} \quad \text{and} \quad u(\sigma)\subset BH.
$$
By a \emph{$\sigma$-homotopy} between filtered maps $u,v\co X_\sigma \to (B\CM)_{\ast}$, we mean a homotopy from $u$ to $v$ which is filtered at each time, i.e., a continuous map $H\co X \times [0,1] \to B\CM$ such that $H(\cdot,0)=u$, $H(\cdot,1)=v$, and for all $t \in [0,1]$,
$$
H(\sigma_0,t)=\{x\} \quad \text{and} \quad H(\sigma,t)\subset BH.
$$
We denote by $[X,B\CM]_\sigma$ the set of $\sigma$-homotopy classes of filtered maps $X_\sigma \to (B\CM)_{\ast}$.  If $A$ is a subspace of $X$ containing $\sigma$, then the restriction to $A$  induces a canonical map $[X,B\CM]_\sigma \to [A,B\CM]_\sigma$.

Note that if $\sigma=\emptyset$, then  $\sigma$-homotopy classes of filtered maps $X_\sigma \to (B\CM)_{\ast}$ correspond to (standard) homotopy classes of maps $X \to B\CM$. Also, if $\sigma$ has no edges (that is, $\sigma=\sigma_0$ is a set of points of~$X$), then  $\sigma$-homotopy classes of filtered maps $X_\sigma \to (B\CM)_{\ast}$ correspond to pointed homotopy classes of pointed maps $(X,\sigma_0) \to (B\CM,x)$.

\subsection{Skeletons of surfaces and their duals}\label{sect-skel-surfaces-duals}
Let  $\Sigma$ be a closed oriented surface. A \emph{skeleton} of $\Sigma$ is an oriented graph $A \subset \Sigma$ such that all components of~$\Sigma \setminus A$ are open disks and all vertices of $A$ have valence $\geq 2$.

An oriented graph $\sigma$ in $\Sigma$ is \emph{dual} to a skeleton $A$ of $\Sigma$ if
\begin{enumerate}
\labela
\item the vertices of $\sigma$ lie in $\Sigma \setminus A$ and each connected component of $\Sigma \setminus A$ contains exactly one vertex of $\sigma$;
\item each edge $s$ of $\sigma$ intersects $A$ transversely at a single point in the interior of a (unique) edge $s^*$ of $A$ and the map $s \mapsto s^*$ induces a bijection between the edges of $\sigma$ and those of $A$;
\item for each edge $s$ of $\sigma$, the orientation of $s^*$ followed by that of $s$ yields the orientation of $\Sigma$.
\end{enumerate}
Note that the skeleton $A$ of $\Sigma$ gives a cellular decomposition of $\Sigma$ (with 0-cells the vertices of $A$, 1-cells the edges of $A$, and 2-cells corresponding to the connected components of $\Sigma \setminus A$) and that $\sigma$ is then the 1-skeleton of a cellular decomposition dual to the one given by $A$. In particular the number of connected components of~$\Sigma \setminus A$ is equal to the number of vertices of $\sigma$.

For example, the 1-skeleton of a triangulation of $\Sigma$ is a graph in $\Sigma$ dual to a skeleton of $\Sigma$.
Two dual graphs of a skeleton of $\Sigma$ are isotopic, and two skeletons of $\Sigma$ having isotopic duals are isotopic.

\subsection{Preliminaries on $\CM$-surfaces and $\CM$-manifolds}\label{sect-prelim-X-manifolds}
By a \emph{$\CM$-surface}, we mean a triple $(\Sigma, \sigma, f)$ consisting of a closed oriented surface~$\Sigma$, an oriented graph $\sigma$ in~$\Sigma$  which is dual to some skeleton of $\Sigma$, and  $f\in [\Sigma,B\CM]_\sigma$, see Section~\ref{sect-sigma-homotopies}.
(The graph $\sigma$ will play the role of basepoints in~\cite{Tu1,TVi1}.) The \emph{opposite} of a $\CM$-surface  $(\Sigma, \sigma, f)$ is the $\CM$-surface $(-\Sigma, \sigma, f)$ where $-\Sigma$ is $\Sigma$ with the opposite orientation.

A \emph{$\CM$-manifold} is a triple $(M,\sigma,g)$ where $M$ is a compact oriented $3$-dimensional manifold, $\sigma$ is an oriented graph in the boundary $\partial M$ of $M$  which is dual to a skeleton of $\partial M$, and $g\in [M,B\CM]_\sigma$.
The \emph{boundary} of such a $\CM$-manifold is the $\CM$-surface $\partial M=(\partial M,\sigma,\partial g)$, where $\partial g \in [\partial M,B\CM]_\sigma$ is the restriction of $g$ to~$\partial M$ (see Section~\ref{sect-sigma-homotopies}).
Here we use the `outward vector first' convention for the induced orientation of the boundary: at any point of $\partial M$, the given orientation of $M$ is determined by the tuple (a tangent vector directed outward, a basis in the tangent space of $\partial M$ positive with respect to the induced orientation). A $\CM$-manifold $(M,\sigma,g)$ is {\it closed} if $\partial M=\emptyset$. In this case $\sigma=\emptyset$, $g \in [M,B\CM]_\emptyset=[M,B\CM]$, and we denote $(M,\sigma,g)=(M,\emptyset,g)$ by $(M,g)$ as in Section~\ref{sect-Xi-skeletons}.

Any $\CM$-surface $(\Sigma,\sigma,f)$ determines the \emph{cylinder $\CM$-manifold}
$$
(\Sigma \times [0,1], \bar{\sigma}= (\sigma \times \{0\})  \sqcup (\sigma\times \{1\}), \bar{f})
$$
where $\Sigma \times [0,1]$ is endowed with the product orientation and $\bar{f}$ is the image of~$f$ under the map
$[\Sigma,B\CM]_\sigma \to [\Sigma \times [0,1],B\CM]_{\bar{\sigma}}$ induced by pre-composing with the projection $\Sigma \times [0,1] \to \Sigma$.

Disjoint union of $\CM$-surfaces ($\CM$-manifolds) are $\CM$-surfaces ($\CM$-manifolds) in the obvious way. A \emph{$\CM$-homeomorphism} $(\Sigma,\sigma,f)\to(\Sigma',\sigma', f')$ between $\CM$-surfaces is an orientation preserving diffeomorphism $\phi\co \Sigma \to \Sigma'$ such that $\phi(\sigma)=\sigma'$ (as oriented graphs) and $f= f' \circ \phi$ (in $[\Sigma,B\CM]_\sigma$). A \emph{$\CM$-homeomorphism} $(M,\sigma,g) \to (M',\sigma',g')$ between $\CM$-manifolds is an orientation preserving diffeomorphism $\psi\co M\to M'$ such that $\psi(\sigma)=\sigma'$ (as oriented graphs) and $g= g'\circ \psi$ (in $[M,B\CM]_\sigma$). In particular, observe that any $\CM$-homeomorphism $(M,\sigma,g) \to (M',\sigma',g')$ between $\CM$-manifolds restricts to a $\CM$-homeomorphism $(\partial M, \sigma,\partial g) \to (\partial M',\sigma',\partial g')$ between the boundary $\CM$-surfaces.

For brevity, we shall sometimes omit the dual graphs $\sigma$ and the maps to $B\CM$ from the notation for $\CM$-surfaces and $\CM$-manifolds.

\subsection{The category of $\CM$-cobordisms}\label{sect-CM-cobordisms}
A \emph{$\CM$-cobordism} from a $\CM$-surface $\Sigma_0$ to a $\CM$-surface $\Sigma_1$ is a pair $(M,h)$ where $M$ is a $\CM$-manifold and $h \co (-\Sigma_0) \sqcup \Sigma_1 \to \partial M$ is a $\CM$-homeomorphism. Here,  $-\Sigma_0$ is the opposite of $\Sigma_0$ (see section~\ref{sect-prelim-X-manifolds}). Two $\CM$-cobordisms $(M,h)$ and $(M',h')$ from $\Sigma_0$ to $\Sigma_1$ are \emph{equivalent} if there is a $\CM$\ti homeomorphism $G \co M \to M'$ such that $h'=G \circ h$.

We denote by $\mathrm{Cob}^{\CM}_3$ the category of $3$-dimensional $\CM$-cobordisms. Objects of $\mathrm{Cob}^{\CM}_3$ are $\CM$-surfaces. A morphism   $\Sigma_0 \to \Sigma_1$ in $\mathrm{Cob}^{\CM}_3$ is an equivalence class of $\CM$-cobordisms
from $\Sigma_0$ to $\Sigma_1$.
The identity morphism of an $\CM$-surface $\Sigma$ is represented by the cylinder $\CM$-manifold $ \Sigma \times [0,1]$  with
tautological identification of the boundary with $(-\Sigma) \sqcup\Sigma$.
Composition of morphisms in $\mathrm{Cob}^{\CM}_3$ is the gluing of $\CM$-cobordisms.  More precisely, the
composition of two morphisms represented by the $\CM$-cobordisms $(M_0=(M_0,\sigma_0,g_0), h_0) \co \Sigma_0 \to \Sigma_1$
and $(M_1=(M_1,\sigma_1,g_1), h_1) \co \Sigma_1 \to \Sigma_2$  is    represented by
the $\CM$-cobordism  $(M=(M,\sigma,g),h)$ defined as follows. The 3-dimensional manifold $M$ is obtained
by gluing   $M_0$ and $M_1$ along $h_1 h_0^{-1} \co h_0(\Sigma_1)   \to h_1(\Sigma_1)$. The graph $\sigma$ in $\partial M=h_0(\Sigma_0)  \cup  h_1(\Sigma_2)$ is
$$
\sigma=\bigl (\sigma_0 \cap h_0(\Sigma_0) \bigr) \cup \bigl (\sigma_1 \cap h_1(\Sigma_2) \bigr).
$$
Pick representatives $(M_0)_{\sigma_0} \to (B\CM)_*$ of $g_0$  and $(M_1)_{\sigma_1} \to (B\CM)_*$ of $g_1$ which agree on $h_0(\Sigma_1)   \simeq h_1(\Sigma_1)$. They  define thus a filtered map $M_\sigma \to (B\CM)_*$. The fact that~$B\CM$ is a 2-type ensures that the $\sigma$-homotopy class $g \in [M,B\CM]_\sigma$ of this filtered map is well defined.
Finally,
$$
h=h_0 \vert_{\Sigma_0} \sqcup h_1 \vert_{\Sigma_2} \co (-\Sigma_0) \sqcup\Sigma_2 \simeq  \partial M.
$$

The disjoint union operation on $\CM$-surfaces and $\CM$-manifolds turns the category $\mathrm{Cob}^{\CM}_3$ into a symmetric monoidal category whose unit object is the empty $\CM$-surface.

\subsection{Homotopy quantum field theories}\label{sect-def-HQFT-target-X}
Let $\Mod_\kk$ be the category of $\kk$-modules and
$\kk$-linear homomorphisms. It is a symmetric monoidal category with the standard tensor product and the unit object $\kk$.

A 3-dimensional \emph{homotopy quantum field theory (HQFT) with target $B\CM$} is a symmetric strong monoidal functor $Z \co \mathrm{Cob}^{\CM}_3 \to \Mod_\kk$. In particular,   $Z( \Sigma  \sqcup \Sigma') \simeq Z(
\Sigma ) \otimes Z( \Sigma') $ for any  $\CM$-surfaces $\Sigma,
\Sigma'$, and similarly    for  morphisms. Also, $Z(\emptyset)\simeq\kk$. We refer to \cite{ML1} for a detailed definition of a strong monoidal functor.

\subsection{Representations of mapping class groups of $\CM$-surfaces}
The category  $\mathrm{Cob}^{\CM}_3$ of $\CM$-cobordisms includes as
a subcategory  the category $\mathrm{Homeo}^\CM$ of $\CM$-surfaces and
their $\CM$-homeomorphisms considered up to isotopy (in the class of $\CM$-homeo\-mor\-phisms).
Indeed, a  $\CM$-homeomorphism of $\CM$-surfaces  $\phi \co \Sigma \to \Sigma'$
determines  a morphism  $  \Sigma \to \Sigma' $ in $\mathrm{Cob}^\CM_3$
represented by  $(C,h \co    (-\Sigma) \sqcup\Sigma'  \to \partial C)$, where $C$ is the cylinder $\CM$\ti manifold associated with $\Sigma'$ and
$$
h(x)=\left\{\begin{array}{ll} (\phi(x), 0) &\text{if $x\in \Sigma$,} \\ (x, 0) &\text{if $x\in \Sigma'$.}  \end{array} \right.
$$
Isotopic $\CM$-homeomorphisms   give rise   to the same morphism in $\mathrm{Cob}^\CM$.   The category $\mathrm{Homeo}^\CM$ inherits a structure of a symmetric monoidal category from that of $\mathrm{Cob}^\CM_3$.  Restricting a 3-dimensional HQFT $ Z\co \mathrm{Cob}^\CM\to \Mod_\kk$ to $\mathrm{Homeo}^\CM$, we obtain a symmetric monoidal functor  $ \mathrm{Homeo}^\CM \to \Mod_\kk$.
In particular,  $Z$ induces a $\kk$-linear representation of the \emph{mapping class group of a $\CM$-surface~$\Sigma$} defined as   the group of isotopy classes of $\CM$-homeomorphisms   $\Sigma\to \Sigma$.

\subsection{Remark}\label{sect-remark-sigma}
Consider two $\CM$-surfaces $(\Sigma, \sigma_0,f_0)$ and $(\Sigma,\sigma_1,f_1)$ having the same underlying surface $\Sigma$. The oriented graph $\bar{\sigma}= (\sigma_0 \times \{0\})  \sqcup (\sigma_1\times \{1\})$ in $\partial(\Sigma \times [0,1])$ is dual to a skeleton of $\partial(\Sigma \times [0,1])$. Then any $g \in [\Sigma \times [0,1],B\CM]_{\bar{\sigma}}$ restricting to~$f_i$ on $\Sigma \times \{i\}$ for $i \in \{0,1\}$  determines a $\CM$-manifold $(\Sigma \times [0,1],  \bar{\sigma},g)$
which, endowed with tautological identification of its boundary with $(-\Sigma) \sqcup\Sigma$,  represents a $\CM$\ti cobordism $C_g \co (\Sigma, \sigma_0,f_0) \to (\Sigma,\sigma_1,f_1)$. This $\CM$-cobordism is an isomorphism in the category $\mathrm{Cob}^{\CM}_3$. Consequently, the image of $C_g$ under a 3-dimensional HQFT $Z\co \mathrm{Cob}^{\CM}_3 \to \Mod_\kk$ induces a \kt linear isomorphism $Z(\Sigma, \sigma_0,f_0) \simeq Z(\Sigma,\sigma_1,f_1)$.

\subsection{The case of aspherical targets}\label{sect-case-aspherical}
Let $X$ be  an aspherical pointed connected CW-complex. Recall that $X$ is homotopy equivalent to $BG$ with $G=\pi_1(X)$. The trivial group homomorphism $1 \to G$ is a crossed module such that
$$
B(1\to G)=BG \quad \text{and} \quad (B(1\to G))_*=(\{x\} \subset BG \subset BG).
$$
Let $\Sigma$ be a closed oriented surface. Assume  that $\Sigma$ is pointed  (meaning that every  connected   component  of $\Sigma$    is endowed with a basepoint). Pick an oriented graph~$\sigma$ in~$\Sigma$  which is dual to a skeleton of $\Sigma$ and whose set of vertices is the set of basepoints of $\Sigma$. Note that
$
[\Sigma,B(1\to G)]_\sigma=[\Sigma,BG]_*,
$
where the right-hand side denotes the set of pointed homotopy classes of pointed maps $\Sigma \to BG$.
Then any homotopy equivalence $X \approx BG$  induces a bijection
$$
[\Sigma,B(1\to G)]_\sigma\simeq [\Sigma,X]_*.
$$
This implies that the category $\mathrm{Cob}_3^{G}$ defined in \cite{TVi1} (whose objects are pointed closed oriented surfaces with a pointed homotopy class of pointed maps to $X$) is equivalent to the category $\mathrm{Cob}^{(1\to G)}_3$ defined in Section~\ref{sect-def-HQFT-target-X}. Consequently, HQFTs with target $X$ in the sense of \cite{TVi1} correspond to HQFTs with target $B(1\to G)$ in the sense  of Section~\ref{sect-def-HQFT-target-X}.

\section{The state-sum HQFT with target \texorpdfstring{$B\CM$}{BX}}\label{sect-statesum-HQFT}
We fix throughout this section a crossed module $\CM \co E \to H$, a  spherical $\CM$-fusion category $\cc$ (over~$\kk$) such that $\dim(\cc_1^1)\in \kk$ is invertible, and a $\CM$-representative set $I=\amalg_{h\in H}\, I_h$ for $\cc$. We extend the numerical state-sum invariant $|\cdot |_\cc$ of closed $\CM$-manifolds (defined in Section~\ref{sect-state-sum-invariants-closed}) to a $3$-dimensional HQFT with target $B\CM$ (in the sense of Section~\ref{sect-Xi-HQFTs}).

\subsection{Colored $\CM$-surfaces}\label{sect-colored-Xi-surfaces}
Let $\Sigma=(\Sigma,\sigma,f)$ be a $\CM$-surface  (see Section~\ref{sect-prelim-X-manifolds}). The \emph{$H$-label} of an edge $s$ of $\sigma$ is the element of $H=\pi_1(BH,x)$ represented by the loop $u(s)$ in $BH$, where $u\co \Sigma_\sigma \to (B\CM)_*$ is any representative of $f \in [\Sigma,B\CM]_\sigma$.

Any skeleton $A$ of $\Sigma$ dual to $\sigma$ (see Section~\ref{sect-skel-surfaces-duals}) becomes a $\CM$-graph in $\Sigma$ as follows. The $H$-label of an edge $a$ of $A$ is the $H$-label of the edge of $\sigma$ dual to $a$. Let $\ell$ be a half-edge of $A$. The vertex of $A$ adjacent to $\ell$ belongs to a connected component of $\Sigma \setminus \sigma$ which is the interior of a 2-cell $D$ attached to $\sigma$ along a map $\rho\co\partial D \to \sigma$.  We orient $D$ so that the induced orientation of its interior is opposite to that of $\Sigma$, and we orient  $\partial D$ using the `outward vector first' convention.  Let $S=\rho^{-1}(\sigma_0) \subset \partial D$, where $\sigma_0$ is the set of vertices of $\sigma$.
Denote by $\iota \co D \to \Sigma$ the canonical map obtained by attaching $D$ on $\sigma$ along~$\rho$. By traversing (from its adjacent vertex) the inverse image of $\ell$ under $\iota$, one reaches a point in $\partial D \setminus S$ whose predecessor in $S \subset \partial D$ (with respect to the orientation of $\partial D$) is denoted  $p_\ell$. (For an example see Figure~\ref{fig-coloring-Xi-surfaces} where $x,y \in H$ are $H$-labels of edges of $\sigma$ and blue dots are elements of~$S$.)
\begin{figure}
\begin{center}
 \psfrag{G}[Br][Br]{$D=$}
 \psfrag{F}[Br][Br]{$\cong$}
 \psfrag{E}[Br][Br]{$\xrightarrow{\iota}$}
 \psfrag{i}[Bc][Bc]{\scalebox{.9}{$\iota$}}
 \psfrag{D}[Br][Br]{\scalebox{.9}{\color{myblue}{$\partial D$}}}
 \psfrag{v}[Bc][Bc]{\scalebox{.9}{\color{myblue}{$\ell$}}}
 \psfrag{z}[Bc][Bc]{\scalebox{.8}{\color{myblue}{$y$}}}
 \psfrag{c}[Bl][Bl]{\scalebox{.8}{\color{myblue}{$x^{-1}$}}}
 \psfrag{e}[Bl][Bl]{\scalebox{.8}{\color{myblue}{$x$}}}
 \psfrag{p}[Br][Br]{\scalebox{.9}{\color{myblue}{$p_\ell$}}}
 \psfrag{q}[Bc][Bc]{\scalebox{.8}{\color{myblue}{$p_\ell$}}}
 \psfrag{A}[Bc][Bc]{\scalebox{.9}{$A$}}
 \psfrag{v}[Bc][Bc]{\scalebox{.9}{$\ell$}}
 \psfrag{s}[Bc][Bc]{\scalebox{.9}{$\textcolor{red}{\sigma}$}}
 \psfrag{x}[Bc][Bc]{\scalebox{.8}{$\textcolor{red}{x}$}}
 \psfrag{y}[Bl][Bl]{\scalebox{.8}{$\textcolor{red}{y}$}}
 \rsdraw{.45}{.9}{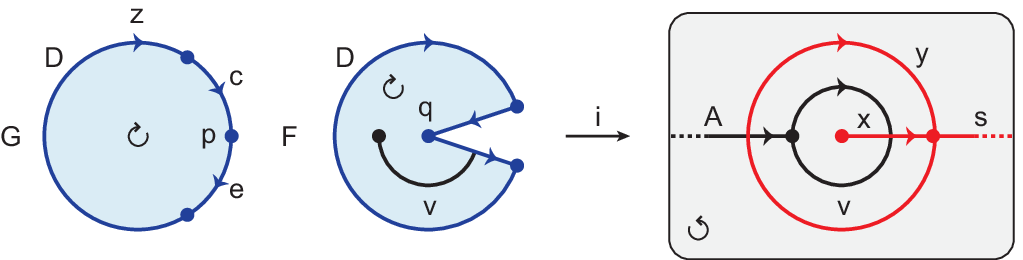}\;.
\end{center}
\captionsetup{justification=centering}
\caption{}
\label{fig-coloring-Xi-surfaces}
\end{figure}
The $E$-label of $\ell$ is then the element of $E=\pi_2(B\CM,BH;x)$ represented by the map $u\circ \iota \co (D,\partial D,p_\ell) \to (B\CM,BH,x)$, where $u\co \Sigma_\sigma \to (B\CM)_*$ is any representative of $f \in [\Sigma,B\CM]_\sigma$. With these labels, $A$ becomes a $\CM$-graph in $\Sigma$. This follows from the fact that the
boundary crossed module $\partial \co \pi_2(B\CM,BH;x) \to \pi_1(BH,x)$ is $\CM \co E \to H$. In particular, \eqref{eq-xi-graph-1} follows from the definition of the boundary map~$\partial$ and \eqref{eq-xi-graph-2} from the definition of the action of $\pi_1(BH,x)$ on $\pi_2(B\CM,BH;x)$. For instance, in the above example, the $E$-label $e$ of $\ell$ does satisfy the expected relation $\CM(e)=xyx^{-1}$ since both sides of this relation are equal to $\partial([f\circ i])$.

A \emph{$\cc$-coloring} of a $\CM$-surface $\Sigma=(\Sigma,\sigma,f)$ is a map $d\co \{\text{edges of $\sigma$}\} \to \Ob(\cc)$ such that for each edge $s$ of $\sigma$, the object $d(s)$ is homogeneous of degree the $H$-label of~$s$.
A $\cc$-coloring $d$ of $\Sigma$  turns the $\CM$-graph in $\Sigma$ associated with a skeleton $A$ of~$\Sigma$ dual to $\sigma$ (as above) into a $\cc$-colored $\CM$-graph denoted $A_d$ by assigning to each edge~$a$ of $A$ the object $d(a^*)$, where $a^*$ is the edge of $\sigma$ dual to $a$. Section~\ref{sect-colored-Xi-graphs} yields then the \kt module $H(A,d)$. If $A$ and $B$ are two skeletons of $\Sigma$ dual to $\sigma$, then the canonical bijection between their half-edges (coming from the fact they are both dual to $\sigma$) induces a \kt linear isomorphism $\rho_{A,B}\co H(A_d) \to H(B_d)$. The family $(\{H(A_d)\}_A,\{\rho_{A,B}\}_{A,B})$,  where $A,B$ run over all skeletons of $\Sigma$ dual to $\sigma$, is a projective system of \kt modules and \kt linear isomorphisms. Consider the projective limit of this system:
$$
H(\Sigma,d)=\underleftarrow{\lim} \, H(A_d) \, .
$$
The module $H(\Sigma,d)$   depends only on $(\Sigma,d)=(\Sigma,\sigma,f,d)$ and comes equipped with  a family $\{H(\Sigma,d) \simeq H(A_d)\}_{A}$ of    \kt linear isomorphisms, called the \emph{cone isomorphisms}, which commute with the $\rho_{A,B}$.

Let $-\Sigma=(-\Sigma,\sigma,f)$ be the opposite of $\Sigma$ (see Section~\ref{sect-prelim-X-manifolds}). Any $\cc$-coloring $d$ of $\Sigma$ is a $\cc$-coloring  of $-\Sigma$ in the obvious way. For any skeleton $A$ of $\Sigma$ dual to~$\sigma$, the oriented graph $A^\opp$ (which is $A$ with the orientation of all edges reversed) is a skeleton of $-\Sigma$, the $\cc$-colored $\CM$-graph $A^\opp_d=(A^\opp)_d\subset -\Sigma$ is opposite to $A_d\subset \Sigma$, and the canonical isomorphism $H(A_d^\opp)\simeq H(A_d)^\star$ (see Section~\ref{sect-colored-Xi-graphs}) yields a canonical isomorphism
$$
H(-\Sigma,d) \simeq H(\Sigma,d)^\star.
$$

\subsection{Skeletons of 3-manifolds}\label{sect-skeletons-boundary}
By a \emph{skeleton} of an oriented compact 3-dimensional manifold~$M$, we mean an oriented stratified 2-polyhedron $P\subset M$ such that $P\cap \partial M= \partial P$ and
\begin{enumerate}
  \labeli
    \item  the oriented graph $\partial P$ is a skeleton of $\partial M$;
\item for every vertex $v$ of $\partial P$, there is a unique edge
$e_v$ of $P$   such that $v$  is an endpoint of $e_v$ and $e_v
\nsubseteq \partial M$; the edge $e_v$ is not a loop  and
$e_v\cap  \partial M =\{v\}$;

\item for every edge $a$ of $\partial P$, the only region~$r_a$ of~$P$ adjacent to~$a$ is   a closed 2-disk
meeting $ \partial M$ precisely along $ a $;

\item all connected components of $M\setminus P$ are
open or half-open 3-balls.
\end{enumerate}

Conditions (i)--(iii) imply that the
intersection of $P$ with a tubular neighborhood of~$\partial M$ in
$M$ is  homeomorphic to   $\partial P\times [0,1]$. Also, an edge of $P$ with both endpoints in $\partial M$ is an edge of $\partial P$.

Components of $M\setminus P$ are called \emph{$P$-balls}.
Note that  the boundary disks of the half-open
$P$-balls are precisely  the components of  $
\partial M \setminus \partial P$. A \emph{$P$-ball branch}  at a vertex or an edge of~$P$ is a germ  (at the vertex or edge) of an adjacent $P$-ball.

An edge or vertex of $P$ is \emph{interior} if it is not included in $\partial P$.
The \emph{link graph} of an interior vertex of $P$ is defined as in Section~\ref{sect-skeletons}.

\subsection{$\CM$-labelings}\label{sect-Xi-labelings-boundary}
We extend the notion of a $\CM$-labeling of a skeleton of an oriented closed 3-dimensional manifold given in Section~\ref{sect-Xi-labelings} to the case of oriented compact 3-dimensional manifolds.

Let $P$ be a skeleton of an oriented compact 3-dimensional manifold~$M$.
As in Section~\ref{sect-Xi-labelings}, for each oriented interior edge $e$ of $P$, the set $P_e$ of branches of~$P$ at~$e$ inherits a cyclic order from the orientations of~$e$ and~$M$, and any $P$-ball branch~$B$ at~$e$ turns the cyclic order on $P_e$ into a linear order (so that the  first element is the first branch $b_B$ of~$P$ at~$e$ encountered while traversing a small loop in~$M$ positively encircling~$e$ starting from $B$). We define the map $\varepsilon_e \co P_e \to \{+, -\}$ as in Section~\ref{sect-Xi-labelings}. Note that when the orientation  of $e$ is reversed, the cyclic order  on $P_e$ is reversed and $ \varepsilon_e$ is multiplied by $-$.

Recall that  $\Reg(P)$ denotes the (finite) set of regions of~$P$ (see Section~\ref{sect-stratified-polyhedron}).
Let $\EB(P)$ be the set of pairs $(e,B)$ where $e$ is an interior oriented edge of $P$ and $B$ is a $P$-ball branch at $e$.
A \emph{pre-$\CM$-labeling} of $P$ is a pair
$$
\bigl(\colr \co \Reg (P) \to H,\coleb \co \EB(P) \to E\bigr)
$$
of maps verifying the conditions \eqref{precol1}, \eqref{precol2}, \eqref{precol3} of Section~\ref{sect-Xi-labelings}.

A pre-$\CM$-labeling $(\colr,\coleb)$ of $P$ turns the set $P_e$ of branches of $P$ at an oriented interior edge~$e$ of $P$ into a $\CM$-cyclic set $P_e=(P_e,\alpha_e,\beta_e,\varepsilon_e)$ as follows: for any branch $b \in P_e$, set $\alpha_e(b)=\colr(b) \in H$ and $\beta_e(b)=\coleb(e,B)\in E$, where $B$ is the $P$-ball branch at $e$ such that $b=b_B$ (in the above notation).

Also, a pre-$\CM$-labeling $(\colr,\coleb)$ of $P$ turns the oriented graph $\partial P$ into a $\CM$-graph in~$\partial M$ as follows. The $H$-label of an edge $a$ of $\partial P$ is $\alpha(r_a) \in H$, where $r_a$ is the region of $P$ adjacent to $a$.
Let $\ell$ be a half-edge of $\partial P$. It is adjacent to a vertex~$v$ of~$\partial P$. The vertex $v$ is adjacent to a (unique) interior edge $e$ of $P$. Orient $e$ so that it originates at $v$. Let $B$ be the $P$-ball branch at $e$ so that $\ell$ is the
the first half-edge of $\partial P$ encountered while traversing a small loop negatively encircling $v$ (with respect to the orientation of $\partial M$)  from any point of $B\cap \partial M$. Then the  $E$-label of~$\ell$ is $\beta(e,B) \in E$.

As in Section~\ref{sect-Xi-labelings}, a pre-$\CM$-labeling $(\colr,\coleb)$ of $P$ turns the link graph $\Gamma_v$ of an interior vertex $v$ of $P$ into a $\CM$-graph in $\partial B_v \cong S^2$, where $B_v$ is a $P$-cone neighborhood  of~$v$ (see Section~\ref{sect-skeletons}). A \emph{$\CM$-labeling} of $P$ is a pre-$\CM$-labeling of $P$ such that, for any interior vertex $v$ of $P$, the $\CM$-graph $\Gamma_v$ in $B_v \cong S^2$ is 1-spherical (see Section~\ref{sect-Xi-graphs-spherical}).

\begin{lem}\label{lem-map-col-boundary}
Let $M$ be an oriented compact 3-dimensional manifold such that its boundary is endowed with a structure of a $\CM$-surface $(\partial M,\sigma,f)$. Let $P$ be a skeleton of $M$ such that $\partial P$ is dual to $\sigma$ (see Section~\ref{sect-skel-surfaces-duals}). Let $(\colr,\coleb)$ be a $\CM$\ti labeling of~$P$ compatible with $f$ in the sense that the two $\CM$-graph structures on $\partial P$ induced by~$(\colr,\coleb)$ (as above) and by $f$ (as in Section~\ref{sect-colored-Xi-surfaces}) coincide. Then this data determines a class $g_{\colr,\coleb,f} \in [M,B\CM]_\sigma$ such that $\partial g_{\colr,\coleb,f}=f$.
\end{lem}
\begin{proof}
Pick  a center in every $P$-ball so that the centers of half-open  $P$-balls are vertices of $\sigma$.
Let $\rr$ be the  set   of all regions of~$P$ contained in $\mathrm{Int}(M)$.
For each~$r\in\rr$, pick an oriented arc~$\gamma_r$ as in Section~\ref{sect-Xi-labelings} and let~$S_r$ be as in Section~\ref{sect-proof-lem-gauge-group-labelings}.
Let $\ee$ be the  set   of all (unoriented) interior edges of~$P$ with both endpoints in~$\mathrm{Int}(M)$.
For each $e\in\ee$, pick a $P$-ball branch~$B_e$ at~$e$, an orientation for~$e$, and a disk $\delta_e$ embedded in~$M$ as in Section~\ref{sect-proof-lem-gauge-group-labelings}.
Set $S=(\cup_{r\in \rr} S_r)\cup \sigma$.
Then the pair $(M, S)$ has a relative CW-decomposition with only 2-cells and 3-cells. The 2-cells are the disks~$\{\delta_e\}_{e\in \ee}$ together with the 2-cells of the cellular decomposition of $\partial M$ induced by $\sigma$ (i.e., the cellular decomposition dual to that induced by the skeleton~$\partial P$ of~$\partial M$, see Section~\ref{sect-skel-surfaces-duals}). The 3-cells are  ball neighborhoods  in $M$ of the interior vertices of~$P$. Define the map $\cup_{r\in\rr} S_r\to B\CM $ as in the proof of Claim~\ref{claim-filtered-map-to-Xi-labelings} and extend it to $S\cup\partial M$ using a representative of $f \in [\partial M,B\CM]_\sigma$. As in the proof of Claim~\ref{claim-filtered-map-to-Xi-labelings}, we further extend this map to the relative 2-skeleton $S\cup \partial M \cup  (\cup_{e\in\cup \ee} \delta_e)$ by using Condition~\eqref{precol1}, and then to all 3-cells  by using the  $1$-sphericality of  the $\CM$-graphs associated with the interior vertices of $P$ and the compatibility of $(\alpha,\beta)$ and $f$ on $\partial P$.
The conditions \eqref{precol2} and \eqref{precol3} ensure that the obtained filtered map $M_\sigma \to (B\CM)_*$ does not depend on the choices of the adjacent $P$-ball branches and orientations for the edges $e \in \ee$. Its $\sigma$-homotopy class $g_{\colr,\coleb,\partial g}$ satisfies $\partial g_{\colr,\coleb,f}=f$ and depends only on $P$, $(\colr,\coleb)$, and~$f$ since any two systems of arcs $\{\gamma_r\}_{r\in\rr}$, disks $\{ \delta_e\}_{e\in\ee}$, and suspensions $\{ S_r\}_{r\in\rr}$ as above are  isotopic  in $M$.
\end{proof}

\subsection{$\CM$-skeletons of $\CM$-manifolds}\label{sect-Xi-skeletons-Xi-manifolds-boundary}
A \emph{$\CM$-skeleton} of a $\CM$-manifold $(M,\sigma,g)$ is a skeleton $P$ of $M$ together with a $\CM$-labeling $(\colr,\coleb)$ such that $\partial P$ is dual to $\sigma$, the $\CM$-labeling $(\colr,\coleb)$ is compatible with $\partial g$ (as in Lemma~\ref{lem-map-col-boundary}), and $g$ is equal to the class $g_{\colr,\coleb,\partial g} \in [M,B\CM]_\sigma$ defined in Lemma~\ref{lem-map-col-boundary}.

The primary $\CM$-moves  $T_0^{\pm 1} -  T_3^{\pm 1}$ defined in Section~\ref{sect-Xi-moves} for $\CM$-skeletons of closed $\CM$-manifolds extend to $\CM$-skeletons of $\CM$-manifolds in the obvious way. All these moves proceed inside 3-balls  in the interior $\Int(M)$ of $M$ and do not modify the boundary of the skeletons. In particular, all vertices/edges/regions created or destroyed by the moves lie in $\Int(M)$.
The action of the moves on the $\CM$-labelings is determined by the
requirement that the labels of the big regions are preserved under
the moves. Under $T_0^{\pm 1}, T_1^{\pm 1}, T_2^{\pm 1},
T_3^{\pm 1} $, the labels of the regions and edges lying in $\Int(M)$ are transformed as in Section~\ref{sect-Xi-moves} and the labels of the regions and edges meeting $\partial M$ remain unchanged.
These moves as well as  label-preserving ambient isotopies of $\CM$-skeletons of $(M,\sigma,g)$ are called {\it primary $\CM$-moves}. The following lemma extends Lemma~\ref{lem-Xi-moves}.

\begin{lem}\label{lem-Xi-moves-boundary}
Any primary $\CM$-move transforms a $\CM$-skeleton of $\CM$-manifold $(M,\sigma,g)$ into a $\CM$-skeleton of $(M,\sigma,g)$.
Moreover, any two $\CM$-skeletons of $(M,\sigma,g)$ can be related by a finite sequence of primary $\CM$-moves.
\end{lem}

\begin{proof}
The proof is parallel to the proof of Lemma \ref{lem-Xi-moves} with the following changes. Given a skeleton $P$ of an oriented compact 3-dimensional manifold $M$, denote by $\Ball(P)$ the set of $P$-balls and by $\Reg(P)$ the set of regions of $P$.
The gauge group $\mathcal{G}_P$ associated with $P$ is the set of pairs $(\lambda,\mu) \in \mathrm{Map}(\Ball(P),H) \times \mathrm{Map}(\Reg(P),E)$ such that $\lambda(B)=1_H$ whenever $B$ is adjacent to $\partial M$ (i.e., $B$ is a half-open 3-ball) and  $\mu(r)=1_E$ whenever $r$ is adjacent to $\partial M$ (i.e., $r\cap \partial M \neq \emptyset$). The product in~$\mathcal{G}_P$ and the action of $\mathcal{G}_P$ on $\CM$-labelings of $P$ are defined as in Section~\ref{sect-Xi-labelings}. Also, we use that any two skeletons of $M$ whose boundaries are isotopic in $\partial M$ can be related by a finite sequence of primary moves (see \cite[Theorem 11.5]{TVi5}).
\end{proof}

\subsection{An invariant of $\CM$-manifolds}\label{sect-inv-Xi-man-boudary}
By an \emph{$I$-coloring} of a $\CM$-surface $\Sigma$, we mean a $\cc$-coloring of $\Sigma$ (see Section~\ref{sect-colored-Xi-surfaces}) taking values in $I$.

Let $(M,\sigma,g)$ be a $\CM$-manifold and $d$ be an $I$-coloring of the $\CM$-surface $\partial M=(\partial M, \sigma,\partial g)$.
We   define   a vector
$|M,\sigma,g,d|_\cc \in H(\partial M,d)$
where $H(\partial M,d)$ is the \kt module introduced in Section~\ref{sect-colored-Xi-surfaces}.

Pick a $\CM$-skeleton $P=(P,(\alpha,\beta))$ of $(M,\sigma,g)$, see Section~\ref{sect-Xi-skeletons-Xi-manifolds-boundary}.
An \emph{$I$-coloring} of~$P$ is a map $c \co \Reg(P) \to I$ such that $c(r)\in I_{\alpha(r)}$ for all region $r$ of $P$. Such an $I$-coloring $c$ is said to \emph{extend $d$} if $c(r_s)=d(s)$ for all edges $s$ of $\sigma$, where~$r_s$ is the region of $P$ adjacent to the edge of~$\partial P$ dual to~$s$.

Given an $I$-coloring $c$ of $P$ extending $d$, we    define  $\dim(c) \in \kk$ by~\eqref{dimcoloring}. We also associate with~$c$ a vector   $\vert c \vert \in H(\partial M,d) $   as follows. Let~$\ee$ be the set of all oriented interior edges of~$P$. For each $e \in \ee$, the coloring $c$ of $P$ turns the $\CM$-cyclic set $P_e$ of branches of $P$ at $e$ (see Section~\ref{sect-Xi-labelings-boundary}) into a $\CM$-cyclic $\cc$-set by assigning to each branch $b \in P_e$ the $c$-color of the region of~$P$ containing  $b$. Let $H_c(e)=H (P_e)$  be its multiplicity  module (see Section~\ref{sect-muliplicity-modules}). Each oriented edge~$e$ of~$P$ originates at a vertex of~$P$ called the \emph{tail} of~$e$. Let $\ee_0\subset \ee $ be the  set   of all oriented edges of~$P$ with tail in the interior of $M$ and $\ee_\partial=\ee \setminus \ee_0$.
Set
$$
H_c  = \bigotimes_{e\in \ee_0   }\, H_c(e) \quad \text{and} \quad H^\partial_c  = \bigotimes_{e\in \ee_\partial   }\, H_c(e),
$$
where  $\otimes$ is the unordered tensor product  of \kt modules.
Observe that the  tail  $ v(e)$ of any    $e\in   {\mathcal E}_\partial$  is a vertex of $\partial P$ and there is a canonical \kt linear isomorphism  $  H_c(e ) \simeq   H_{v(e)}((\partial P)_d)$. Here $(\partial P)_d$ is  the  $\CM$-graph in $\partial M$ induced by the $\CM$-labeling of $P$ (see Section~\ref{sect-Xi-labelings-boundary}) and is $\cc$-colored using $d$ (see Section~\ref{sect-colored-Xi-surfaces}). The formula $e \mapsto v(e)$ establishes a bijective correspondence between the edges $e\in {\mathcal E}_\partial$ and the  vertices~$v$ of~$\partial P$. This together with the cone isomorphism (see Section~\ref{sect-colored-Xi-surfaces}) gives a canonical \kt linear isomorphism
$$
H^\partial_c \simeq \bigotimes_{v }\, H_v((\partial P)_d) = H ((\partial P)_d) \simeq H(\partial M,d).
$$
As in Section~\ref{sec-computat}, an unoriented  interior edge~$e$ of  $P$ gives rise to two  opposite oriented edges $e_1,e_2 \in\ee$ and a vector $\ast_e \in H_c(e_1)  \otimes H_c(e_2) $ independent of the numeration of $e_1, e_2$. The unordered tensor product of  these vectors   over all  unoriented  interior edges  is a vector
\begin{equation*}
 \otimes_{e  }\, \ast_e \in \bigotimes_{e\in \ee   }\, H_c(e) .
\end{equation*}
We let $\ast_c \in H_c   \otimes  H(\partial M,d)$ be the image of $ \otimes_{e}\, \ast_e$  under the  isomorphism
\begin{equation*}\label{isoi}
\bigotimes_{e\in \ee}\, H_c(e) \simeq   H_c   \otimes   H_c^\partial  \simeq H_c  \otimes   H(\partial M,d).
\end{equation*}
As in Section~\ref{sec-computat}, an interior vertex $v$ of~$P$ determines a $1$-spherical $\cc$-colored $\CM$-graph $\Gamma_v^c$   and a vector
$$\inv_\cc (\Gamma_v^c) \in H(\Gamma_v^c)^\star \simeq \otimes_{e_v}\, H_c(e_v)^\star,$$ where~$e_v$ runs over all edges of~$P$ incident to $v$ and oriented away from $v$.  Let $V_c  \in H_c^\star$ be the image of the unordered tensor product $\otimes_v \,\inv_\cc (\Gamma^c_v)$ over all such~$v$   under the canonical \kt linear isomorphism
$$
\bigotimes_v  H(\Gamma_v^c)^\star \simeq \bigotimes_v \bigotimes_{e_v}   H_c(e_v)^\star \simeq H_c^\star .
$$
Finally set
$$
\vert c \vert=(V_c  \otimes  \id_{H(\partial M,d)})(\ast_c) \in \kk  \otimes  H(\partial M,d)\simeq H(\partial M,d).
$$

We  now define a vector in $H(\partial M,d)$ by
\begin{equation*}
|M,\sigma,g,d|_\cc=\dim (\cc^1_1)^{-\vert M\setminus P\vert} \sum_{c} \,\,  \dim (c) \,  \vert c \vert,
\end{equation*}
where   $\vert M\setminus P\vert$ is the number of connected components of $M\setminus P$ and $c$ runs over all $I$-colorings of $P$ extending $d$.

\begin{thm}\label{thm-state-3man-Xi-boundary}
The vector $|M,\sigma,g,d|_\cc\in H(\partial M,d)$ is a topological invariant of  $(M,\sigma,g,d)$ independent of the choice of~$P$.
\end{thm}

\begin{proof}
The  topological invariance here means   that if $\psi \co (M,\sigma,g) \to (M',\sigma',g')$ is a
$\CM$-homeomorphism between $\CM$-manifolds, then
$|M',\sigma',g',d'|_\cc=H(\psi) (|M,\sigma,g,d|_\cc)$,  where $d'$ is the composition of $d$ with the bijection $\{\text{edges of $\sigma'$}\} \to \{\text{edges of $\sigma$}\}$ induced by $\psi$ and $H(\psi)\co H(\partial M,d)    \to   H(\partial M',d')$ is the   isomorphism induced by~$\psi$.
This invariance follows from the independence of~$P$, and the latter is    verified exactly   as in the proof of Theorem~\ref{thm-state-3man-Xi} using  Lemma~\ref{lem-Xi-moves-boundary}.
\end{proof}

The invariant $|M,\sigma,g,d|_\cc$ is   multiplicative: for any $\CM$-manifolds $(M_1,\sigma_1,g_1)$, $(M_2,\sigma_2,g_2)$
and any $I$-colorings $d_1$ of $\partial M_1$ and $d_2$ of $\partial M_2$, we have:
$$
|(M_1,\sigma_1,g_1)\sqcup(M_2,\sigma_2,g_2), d_1\sqcup d_2|_\cc=|M_1,\sigma_1,g_1,d_1|_\cc  \otimes  |M_2,\sigma_2,g_2,d_2|_\cc
$$
up to the canonical isomorphism
$$
H (\partial (M_1\sqcup M_2), d_1\sqcup d_2 ) \simeq   H(\partial M_1,d_1)   \otimes   H(\partial M_2,d_2).
$$

\subsection{Functoriality}
Let $\Sigma_0=(\Sigma_0,\sigma_0,f_0)$ and $\Sigma_1=(\Sigma_1,\sigma_1,f_1)$ be $\CM$-surfaces and let $\xi\co \Sigma_0 \to \Sigma_1$ be a morphism in $\mathrm{Cob}_3^\CM$ represented by a $\CM$-cobordism $(M,h)$ where $M=(M,\sigma,g)$ is a $\CM$-manifold $M$ and $h \co (-\Sigma_0) \sqcup \Sigma_1 \to \partial M$ is a $\CM$-homeomorphism (see Section~\ref{sect-CM-cobordisms}).
Any     $I$-colorings $d_0$ of $\Sigma_0$ and $d_1$ of  $\Sigma_1$ yields an $I$-coloring $d_0\cup d_1$ of $(-\Sigma_0) \sqcup \Sigma_1$ in the obvious way and so an $I$-coloring $d$ of the $\CM$-surface $\partial M = h((-\Sigma_0) \sqcup \Sigma_1 )$. Theorem~\ref{thm-state-3man-Xi-boundary} gives  a vector $|M, \sigma,g, d |_\cc \in  H(\partial M,d)$. Now let $\Upsilon$ be the composition of the isomorphisms
\begin{gather*}
H(\partial M,d) \simeq
H((-\Sigma_0) \sqcup \Sigma_1, d_0\cup d_1)
\simeq  H(-\Sigma_0,d_0)   \otimes    H(\Sigma_1, d_1 ) \\
\simeq  H(\Sigma_0,d_0)^\star    \otimes     H( \Sigma_1, d_1 ) \simeq \Hom_\kk\bigl (H(\Sigma_0,d_0 ),H( \Sigma_1, d_1 )  \bigr ),
\end{gather*}
where the first isomorphism  is induced by $h^{-1}$, the second and fourth are obvious, and the third is induced by the canonical isomorphism $H(-\Sigma_0,d_0 ) \simeq  H(\Sigma_0,d_0 )^\star$ discussed in   Section~\ref{sect-colored-Xi-surfaces}.
Set
\begin{equation}\label{cobosm}
|\xi,   d_0,    d_1| =  \frac{\dim (\cc_1^1)^{v(\Sigma_1)}}{\dim (d_1)} \, \,  \Upsilon(|M, \sigma,g, d |_\cc) \co  H(\Sigma_0,d_0 ) \to H( \Sigma_1, d_1 ),
\end{equation}
where $v(\Sigma_1)$ is the number of vertices of $\sigma_1$ and $\dim (d_1)$
is   the product over all edges of~$\sigma_1$ of the dimensions of their colors.   Theorem~\ref{thm-state-3man-Xi-boundary} implies  that  the homomorphism $|f,   d_0,    d_1| $ does not depend on the choice of the pair $(M,h)$ representing~$f$.
The normalization factor in  the definition of $|f,   d_0,    d_1| $ is justified by the next lemma.

\begin{lem}\label{lem-compo-Xi-cobordisms-statesum}
Let $\xi_0\co \Sigma_0 \to \Sigma $ and  $\xi_1\co \Sigma  \to \Sigma_1$ be morphisms in $\mathrm{Cob}_3^\CM$.
For any $I$-colorings $d_0$ of $\Sigma_0$ and $d_1$ of $\Sigma_1$,
$$
|\xi_1  \xi_0,   d_0,    d_1| =\sum_{d}
 \, |\xi_1,   d,   d_0|  \circ |\xi_0,   d_0,d | \co H(\Sigma_0,d_0 ) \to H( \Sigma_1, d_1 ),
$$
where $d$ runs over all $I$-colorings of $\Sigma$.
\end{lem}

\begin{proof}
Represent   $\xi_0$   and   $\xi_1$, respectively,   by pairs $(M_0, h_0)$ and $(M_1, h_1)$ as above. The   morphism $\xi_1   \xi_0\co \Sigma_0 \to \Sigma_1 $ is    represented by the pair  $(M,{h})$, where~$M$ is obtained by  gluing  $M_0$ to~$M_1$ along $h_1 h_0^{-1} \co  h_0(\Sigma)   \to h_1({\Sigma})$ and $h=h_0 \vert_{\Sigma_0} \sqcup h_1 \vert_{\Sigma_2}$.
To simplify  notation, we identify    ${h}_0(\Sigma)  = {h}_1({\Sigma}) \subset \Int (M)$
with~$\Sigma$ via $h_0$.  Pick  $\CM$-skeletons $P_0$ of $M_0$ and  $P_1$ of $M_1$ such that $\partial P_0\cap \Sigma=\partial P_1 \cap \Sigma$. Then $P=P_0 \cup P_1 \subset M$ is a $\CM$-skeleton of $M$. Since $v(\Sigma)$ is equal to the number of connected components of $\Sigma \setminus A$ where $A=\partial P_0\cap \Sigma$ (see Section~\ref{sect-skel-surfaces-duals}), we have:
$$
\vert M\setminus P \vert = \vert M_0\setminus P_0 \vert + \vert M_1\setminus P_1 \vert - v(\Sigma).
$$
The term $-v(\Sigma)$ explains the need for the factor  $\dim(\cc_1^1)^{v(\Sigma)}$ in the definition of $|\xi_1,   d,   d_0|$. Similarly, given  a region  $r_0$  of $P_0$ and a region $r_1$ of $P_1$ adjacent to the same edge $e$ of $A$, the union $r= r_0 \cup r_1 \cup e$  is a region of $P$ with Euler characteristic $\Euler (r)=\Euler(r_0)+\Euler (r_1) -1$. The   term $-1$ explains the need for the factor $\dim (d)^{-1}$ in the definition of $|\xi_1,   d,   d_0|$.
\end{proof}

\subsection{The  HQFT $\vert \cdot \vert_{\cc}$}\label{sect-state-sum-HQFT-def}
For a $\CM$-surface $\Sigma$, denote by  $\mathrm{Col} (\Sigma)$ the set of all $I$-colorings of~$\Sigma$
and consider the \kt module
$$
  \vert \Sigma\vert^\circ  = \bigoplus_{d\in \mathrm{Col}   (\Sigma)}   \,  H(\Sigma,d).
$$
For a morphism $\xi\co \Sigma_0 \to \Sigma_1$  in $\mathrm{Cob}_3^\CM$, consider the \kt linear homomorphism
$$
|\xi|^\circ=\sum_{\substack{d_0 \in \mathrm{Col}(\Sigma_0)\\ d_1 \in
\mathrm{Col}(\Sigma_1)}} |\xi,d_0,d_1|\co \vert \Sigma_0 \vert^\circ \to \vert  \Sigma_1\vert^\circ.
$$
Lemma~\ref{lem-compo-Xi-cobordisms-statesum} implies that for any   $\xi_0, \xi_1$ as in this lemma,
\begin{equation}\label{eq-func}
|\xi_1  \xi_0|^\circ=|\xi_1|^\circ \circ  |\xi_0|^\circ.
\end{equation}
In particular, this implies that
$p_\Sigma=|\id_\Sigma|^\circ\co \vert\Sigma\vert^\circ \to \vert\Sigma\vert^\circ$
is a projector onto a direct summand $\vert  \Sigma\vert=\mathrm{Im}(p_\Sigma)$ of $\vert  \Sigma\vert^\circ$.
We next associate  with each  morphism $ \xi\co \Sigma_0   \to  \Sigma_1$  in   $\mathrm{Cob}_3^\CM$
a  homomorphism $ \vert \xi \vert  \co  \vert \Sigma_0 \vert  \to \vert \Sigma_1 \vert  $. Formulas~\eqref{eq-func} and $\id_{\Sigma_1} \circ \xi=\xi$   imply that $|\xi|^\circ=p_{\Sigma_1} \, |\xi|^\circ$,
and so the image of   $|\xi|^\circ$ is contained in  $\vert \Sigma_1\vert$. Denote by $|\xi|\co \vert   \Sigma_0\vert  \to \vert   \Sigma_1\vert$ the restriction of $|\xi|^\circ\co \vert   \Sigma_0\vert^\circ  \to \vert   \Sigma_1\vert^\circ$ to $\vert   \Sigma_0\vert $ and $\vert   \Sigma_1\vert$.
It is clear that  the rule  $\Sigma\mapsto \vert \Sigma \vert$, $\xi \mapsto \vert \xi \vert  $ defines a functor $\vert\cdot\vert  \co \mathrm{Cob}_3^\CM \to \mathrm{Mod}_\kk$. We endow this functor with  monoidal constraints as follows. By definition, $\vert  \emptyset \vert  =\vert  \emptyset \vert^\circ  = \kk$ for the empty $\CM$-surface $\emptyset$.
Given two $\CM$-surfaces $\Sigma$ and $\Sigma'$, the canonical isomorphisms $H(\Sigma \sqcup \Sigma', d \sqcup d' )\simeq H(\Sigma, d) \otimes  H(\Sigma', d')$, where $d$ runs over $I$-colorings of $\Sigma$ and~$d'$ runs over $I$-colorings of $\Sigma'$, yield an isomorphism  $|\Sigma \sqcup \Sigma'|^\circ \simeq |\Sigma|^\circ \otimes |\Sigma'|^\circ$. The latter factorizes into an  isomorphism $|\Sigma \sqcup \Sigma'| \simeq |\Sigma| \otimes |\Sigma'|$ which, together with the canonical isomorphism  $\vert \Sigma \vert \otimes_\kk \vert \Sigma' \vert \simeq \vert\Sigma \vert \otimes \vert \Sigma' \vert$, defines the monoidal constraint $\vert \Sigma\vert \otimes_\kk \vert \Sigma'\vert \simeq \vert \Sigma\sqcup\Sigma'\vert$. These monoidal constraints turn  $\vert\cdot\vert $ into a symmetric strong monoidal functor. We denote this functor by $\vert\cdot\vert_{\cc, I} $ or, shorter, by $\vert\cdot\vert_\cc $.  Combining the  considerations above, we obtain the following theorem.

\begin{thm}\label{thm-state-sum-Xi-HQFT}
The functor $\vert\cdot\vert_\cc$ is a 3-dimensional HQFT with target $B\CM$.
\end{thm}

The HQFT $\vert \cdot \vert_\cc$ is called the \emph{state sum HQFT} derived from $\cc$.  Considered up to isomorphism, the HQFT $\vert \cdot \vert_\cc$ does not depend on the choice of the $\CM$-representative
set $I$ for $\cc$. The scalar invariant of closed $\CM$-manifolds produced by this HQFT is
precisely the invariant  of Section~\ref{sect-state-sum-invariants-closed}.

\subsection{The case of push-forwards}\label{sect-push-forward-surfaces}
Let $\phi=(\psi \co E \to E',\varphi \co H \to H')$ be a morphism from a crossed module $\CM \co E \to H$ to a crossed module $\CM' \co E' \to H'$ such that $\psi$ and $\varphi$ are surjective, $\Ker(\psi)\cap \Ker(\CM)=1$, and both $\Ker(\varphi)$ and $\Ker(\CM')$ are finite with invertible cardinal in $\kk$. Consider a spherical $\CM$-fusion category~$\cc$ (over $\kk$) such that $\dim(\cc_1^1)$ is invertible.  Then the push-forward $\phi_*(\cc)$ of $\cc$ is a $\CM'$\ti fusion category such that $\dim\bigr(\phi_*(\cc)^1_1\bigl)$ is invertible (see  Section~\ref{sect-push-3man-Xi}).
By Theorem~\ref{thm-state-sum-Xi-HQFT}, the category $\phi_*(\cc)$ defines the 3-dimensional HQFT  $\vert \cdot \vert_{\phi_*(\cc)}$ with target~$B\CM'$ and the category $\cc$ defines the 3-dimensional HQFT  $\vert \cdot \vert_\cc$ with target~$B\CM$. Recall that Theorem~\ref{thm-pushforward-3man-Xi} computes the values of the former from the latter for connected closed $\CM'$-manifolds. There is a similar computation for $\CM'$-surfaces. More precisely, let $B\phi\co (B\CM)_* \to (B\CM')_*$ be the filtered map induced by $\phi$. Then for any $\CM'$-surface~$(\Sigma,\sigma,f')$,
$$
\vert\Sigma,\sigma,f'\vert_{\phi_*(\cc)}=\bigoplus_{ \substack{f \in [\Sigma,B\CM]_{\sigma} \\  B\phi \circ f=f'}}
 \vert\Sigma,\sigma,f\vert_{ \cc }.
$$
The proof of this equality follows that of Theorem~\ref{thm-pushforward-3man-Xi}.

\subsection*{Acknowledgements}
This work was supported in part by the FNS-ANR grant OCHoTop (ANR-18-CE93-0002) and the Labex CEMPI  (ANR-11-LABX-0007-01).


\begin{thebibliography}{EGNO}

\bibitem[BW]{BW} Barrett, J., Westbury, B., \emph{Invariants of piecewise-linear 3-manifolds},  Trans. Amer. Math. Soc.  348  (1996),   3997--4022.

\bibitem[BT]{BT}
Brightwell, M., Turner, P., \emph{Representations of the homotopy surface category of a
simply connected space}, J. Knot Theory Ramifications 9 (2000), 855--864.

\bibitem[BH]{BH1}
Brown, R., Higgins, P.J., \emph{The classifying space of a crossed complex}, Math. Proc. Camb. Phil.
Soc. 110 (1991), 95--120.

\bibitem[BHS]{BHS} Brown, R., Higgins, P.J., Sivera, R., \emph{Nonabelian Algebraic Topology}, EMS Tracts in Math.\
15, European Math.\ Soc.\ Publ.\ House, Z{\"{u}}rich 2011.

\bibitem[ENO]{ENO} Etingof, P.,  Nikshych, D.,  Ostrik, V.,  \emph{On fusion categories}, Ann.
of Math. (2)  162  (2005),   581--642.

\bibitem[EGNO]{EGNO}
Etingof, P.,   Gelaki, S., Nikshych, D.,  Ostrik, V.,  Tensor categories,  Mathematical Surveys and Monographs, 205. American Mathematical Society, Providence, RI, 2015.

\bibitem[FP]{FP} Faria Martins, J., Porter T., \emph{On Yetter's invariant and an extension of the Dijkgraaf-Witten invariant to categorical groups.} Theory Appl. Categ. 18 (2007) 118--150.

\bibitem[JS]{JS}  Joyal, A.,  Street, R., \emph{The geometry of tensor calculus I}, Adv. in Math. 88 (1991), 55--112.

\bibitem[{ML}]{ML1}
 {MacLane}, S., \emph{Categories for the working mathematician},
Second edition,   Springer-Verlag, New York, 1998.

\bibitem[MLW]{MLW}
Mac Lane, S. and Whitehead, J.H.C., \emph{On the 3-type of a complex}, Proc. Nat. Acad. Sci. (1950)
41--48.

\bibitem[Mat]{Mat1}   Matveev, S. V., \emph{Algorithmic topology and classification of
3-manifolds.} Second edition. Algorithms and Computation in Math., 9. Springer, Berlin, 2007.

\bibitem[Po]{P} Porter, T., \emph{Formal homotopy quantum field theories, II: {S}implicial formal maps}, Cont. Math. 431, 375 - 404.

\bibitem[PT]{PT} Porter, T., Turaev, V., \emph{Formal homotopy quantum field theories, {I}: {F}ormal maps and crossed $\mathcal{C}$-algebras}, J. Homotopy Relat. Struct. 3 (2008), 113--159.

\bibitem[So]{So} S\"{o}zer, K., \emph{Two-dimensional extended homotopy field theories.} arXiv:1909.04187. To appear in Algebr. Geom. Topol.

\bibitem[ST]{ST} Staic, M., Turaev, V., \emph{Remarks on 2-dimensional {HQFT}s}, Algebr. Geom. Topol. 10 (2010), 1367--1393.

\bibitem[SV]{SV} S\"{o}zer, K., Virelizier, A., \emph{Hopf crossed module (co)algebras}, preprint arXiv:2305.15485 (2023).

\bibitem[Tu]{Tu1} Turaev, V., \emph{Homotopy Quantum Field Theory}, EMS Tracts in Math.\ 10, European Math.\ Soc.\ Publ.\ House, Z{\"{u}}rich 2010.

\bibitem[TVi1]{TVi1} Turaev, V., Virelizier, A., \emph{On 3-dimensional homotopy quantum field theory, {I}},  Internat. J. Math. 23 (2012),   no. 9, 1250094, 28 pp.

\bibitem[TVi2]{TVi3} Turaev, V., Virelizier, A., \emph{On 3-dimensional homotopy quantum field theory  {II}:  The surgery approach}, Internat. J. Math. 25 (2014),   no. 4, 1450027, 66 pp.

\bibitem[TVi3]{TVi4} Turaev, V., Virelizier, A., \emph{On 3-dimensional homotopy quantum field theory  {III}:  Comparison of two approaches}, Internat. J. Math. 31 (2020),   no. 10, 2050076, 57 pp.

\bibitem[TVi4]{TVi5} Turaev, V., Virelizier, A., \emph{Monoidal Categories and Topological Field Theory}, Progress in  Mathematics, 322. Birkh\"auser, Basel, 2017. xii+523 pp.

\bibitem[Ye]{Yet} Yetter D. N., \emph{TQFTs from homotopy 2-types}, J. Knot Theory Ramifications 2 (1993), 113--123.

\end{thebibliography}
\end{document}